\documentclass[11pt,reqno]{amsart}
\usepackage{float}
\usepackage{fullpage}
\usepackage{amsmath,amsthm,verbatim,amssymb,amsfonts,amscd, graphicx,bbm,bm}
\usepackage{graphics}
\usepackage{mathtools}
\usepackage{mathrsfs}
\usepackage{esint}
\usepackage{appendix}
\usepackage{esint}
\usepackage{color}
\usepackage{tikz,tikz-cd}
\usepackage[all,cmtip]{xy}
\usetikzlibrary{decorations.pathreplacing,arrows.meta,calc,patterns,angles,quotes,decorations.markings}
\usetikzlibrary{decorations.pathmorphing}

\usepackage[pdfstartview=FitH, bookmarksnumbered=true,bookmarksopen=true, colorlinks=true, pdfborder={0 0 1}, citecolor=blue, linkcolor=blue,urlcolor=blue]{hyperref}

\newtheorem{theorem}{Theorem}
\newtheorem{proposition}[theorem]{Proposition}
\newtheorem{lemma}[theorem]{Lemma}

\theoremstyle{definition}
\newtheorem{remark}{Remark}

\newtheorem{definition}[theorem]{Definition}

\newcommand{\cref}[1]{Corollary~\ref{c.#1}}

\usepackage{xcolor}

\numberwithin{equation}{section}
\numberwithin{theorem}{section}

\newcommand{\R}{\mathbb{R}}
\newcommand{\N}{\mathbb{N}}

\newcommand{\eps}{\varepsilon}

\newcommand{\bC}{\mathbb{C}}

\newcommand{\Z}{\mathbb{Z}}

\newtheorem{Assumption}{Assumption}[section]

\newcommand{\mi}{\mathrm{i}}

\newcommand{\uinc}{u^{\rm in}}

\usepackage[textsize=tiny]{todonotes}

\title{Band-Edge Homogenization and the Sound-Soft Limit of Finite Bubbly Crystals}

\author{Habib Ammari}
\address[H. Ammari]{ETH Z\"urich, Department of Mathematics, Rämistrasse 101, 8092 Z\"urich, Switzerland and 
  Hong Kong Institute for Advanced Study, City University of Hong Kong, Kowloon Tong, Hong Kong Special Administrative Region}
\email{habib.ammari@math.ethz.ch}

\author{Yuxin Du}
\address[Y. Du]{Qiuzhen College, Tsinghua University, Beijing 100084}
\email{duyx23@mails.tsinghua.edu.cn}

\author{Xin Fu}
\address[X. Fu]{School of Science, Institute for Theoretical Sciences, Westlake University, Hangzhou,
Zhejiang Province, 310030, P.R. China}
\email{1550862645cf@gmail.com}

\author{Wenjia Jing}
\address[W. Jing]{Yau Mathematical Sciences Center, Tsinghua University, Beijing 100084 and Beijing Institute of Mathematical Sciences and Applications, Beijing 101408, P.R. China}
\email{wjjing@tsinghua.edu.cn}

\author{Moritz Melcher}
\address[M. Melcher]{ETH Z\"urich, Department of Mathematics, Rämistrasse 101, 8092 Z\"urich, Switzerland}
\email{moritz.melcher@math.ethz.ch}

\date{\today}



\begin{document}

\begin{abstract}
We establish a quantitative sound-soft scattering limit for a finite bubbly crystal near the upper edge of the first Bloch band of the corresponding infinite crystal. The inclusions form a dense periodic array of period $\eps$ in a bounded Lipschitz domain, and their density contrast is $\delta=\eps^2$, with fixed positive wave speeds. Under coordinate-reflection symmetry, we prove that the normalized capacitance symbol has a unique nondegenerate maximum on the Brillouin torus. Its Hessian defines a positive Dirichlet elliptic operator governing the limiting spectral detunings. A mean-constrained variational formulation allows us to compare the actual finite-array capacitance matrix with the truncated infinite-lattice operator and to prove exponential localization of their difference near the sample boundary. We obtain an $O(\eps)$ norm-resolvent approximation and connect it to the full acoustic scattering problem, retaining the second-order Bloch correction required by the frequency scaling. For rescaled detunings outside the effective Dirichlet spectrum, the exterior $L^2$ and far-field discrepancies are $O(\eps)$, while the interior $L^2$ field is $O(\eps^{1/2})$. Numerical experiments illustrate the far-field convergence and the effective spectral modes. Explicit one-dimensional calculations describe finer Fabry--P\'erot transmission windows and show that a fixed frequency strictly inside the first band need not have a unique scattering limit.
\end{abstract}

\maketitle

\tableofcontents

\section{Introduction}

Periodic assemblies of strongly contrasting inclusions can exhibit resonant wave phenomena at scales much smaller than the wavelength in the surrounding medium. In bubbly media, the contrast in density produces a low-frequency (also called subwavelength) resonant response that persists in the collective behavior of many inclusions. This mechanism underlies subwavelength band gaps and provides a setting in which microscopic geometry can strongly influence macroscopic propagation. Mathematical studies of these effects connect the acoustic transmission problem to capacitance operators and to the Bloch spectrum of an associated periodic medium; see, for example, \cite{sima-honeycomb,ammari2019bloch,ammari.davies.ea2024Functional}.

The behavior near a subwavelength band edge differs from the long-wavelength regime associated with quasi-momenta near the origin. At the corner of the Brillouin zone, neighboring cells carry alternating phases. The relevant large-scale description must therefore separate this rapidly varying cell structure from its slowly varying amplitude. For infinite periodic bubbly crystals, Bloch-wave analysis near the first band gap leads to an effective equation determined by the curvature of the dispersion relation \cite{ammari2019bloch}. However, for a bounded array, the termination of the lattice and its coupling to the surrounding homogeneous medium also influence the effective behavior. A quantitative scattering limit must control these boundary effects together with the simultaneous limits of small period and large contrast.

In this paper, we consider a periodic array contained in a bounded Lipschitz domain $\Omega\subset\R^d$, $d=2,3$, with a connected exterior. The reference inclusion is compactly contained in the unit cell and is invariant under reflections in the coordinate hyperplanes. We retain complete cells lying in $\Omega$ and impose the critical scaling $\delta=\eps^2$, where $\eps$ is the period and $\delta$ is the density contrast. The wave speeds in the inclusions and in the background remain fixed and positive. We tune the frequency according to
\[
 \omega_\eps=\omega_{*,\eps}-\tau\eps^2+O(\eps^4),
\]
where $\omega_{*,\eps}$ is the upper edge of the first Bloch band of the corresponding infinite periodic medium. In this scaling, the band-edge frequency has a positive limit, whereas the wavelength remains large relative to an individual cell.

Our analysis starts from a variational capacitance operator in which the unknowns are the averages of the field over the inclusions. This construction treats finite and infinite arrays within the same energy framework. It also allows us to compare the finite-array operator with the truncation of its infinite-lattice counterpart. The difference between the two operators is localized near the boundary of the array, with exponential decay in the lattice distance from the boundary. This localization is essential for deriving a continuum limit on a domain that does not need to be aligned with the lattice.

Let $\widehat{C}^\alpha$ denote the limiting capacitance symbol and let $M=(\pi,\cdots,\pi)$ denote the corner of the Brillouin torus. We prove that the assumed reflection symmetry makes $M$ the unique nondegenerate maximum of this symbol. The effective matrix and operator are
\[
 A_{\mathrm {eff}}=-\frac12\nabla_\alpha^2\widehat{C}^M,
 \qquad L_D=-\nabla\cdot A_{\mathrm{eff}}\nabla
 \quad\text{in }\Omega,
\]
with homogeneous Dirichlet boundary conditions. In particular, $A_{\mathrm{eff}}$ is positive definite. After removing the alternating phase and centering at the maximum, we prove an $O(\eps)$ norm-resolvent approximation of the finite-array capacitance operator by $L_D$. The proof uses boundary-layer estimates for continuum resolvent solutions and distinguishes the consistency estimates needed for the energy norm from those needed for the sharper $L^2$ operator norm. Endpoint regularity on Lipschitz domains provides the required control near $\partial\Omega$; see \cite{savare2002domain}.

We next relate this discrete limit to the full acoustic scattering problem. An exact decomposition separates the quasi-static lift of the bubble averages from a component with zero average in every inclusion. We prove uniform invertibility of the latter component and estimate its coupling to the bubble averages. A uniform expansion of the first Bloch eigenvalue is needed at this step: the exterior cell mass contributes at the same order as the detuning and must be retained to identify the limiting spectral parameter correctly. With $\omega_0=v_{\rm b}\sqrt{\widehat{C}^M}$, this parameter is
\[
 \sigma=\frac{2\omega_0\tau}{v_{\rm b}^2}.
\]

Our main scattering theorem applies when $\sigma\notin\mathrm{spec}(L_D)$. In this nonresonant regime, the total field converges outside $\Omega$ to the total field for scattering by a sound-soft obstacle occupying $\Omega$, while the interior field tends to zero in $L^2$. More precisely, for uniformly bounded incident-field trace norms, the exterior $L^2$ and far-field discrepancies are $O(\eps)$ and the interior $L^2$ norm is $O(\eps^{1/2})$. We also obtain an $O(\eps^{1/2})$ exterior $H^1$ estimate and an $O(\eps)$ local $H^1$ estimate away from the boundary. Thus, the effective Dirichlet operator 
determines the limiting spectral set excluded by the nonresonant convergence theorem, and the exterior nonresonant limit is described by a sound-soft boundary condition. The constants in these estimates are not asserted to remain bounded as $\sigma$ approaches the effective spectrum.

The distinction between nonresonant and exceptional detunings is also visible in our examples. Two-dimensional multiple-scattering computations illustrate the far-field convergence and compare demodulated bubble averages near selected discrepancy peaks with effective Dirichlet eigenfunctions. These comparisons provide numerical evidence for the role of the effective spectrum, without identifying finite-array peaks with limiting eigenvalues or complex scattering poles. In one dimension, an exact transfer-matrix calculation gives a more detailed description: it exhibits both the nonresonant sound-soft limit and a finer Fabry--P\'erot window with nonvanishing transmission and an enhanced interior field.

Finally, we complement our band-edge analysis in this paper with an interior-band analysis. We consider frequencies strictly inside the first Bloch band and separated from its upper edge. In one dimension, exact transfer formulas describe the transmission peaks and show that a fixed frequency can have different scattering limits along different sequences of periods. At a regular Bloch point in higher dimensions, Appendix \ref{append:B} formally derives a first-order equation for the leading bulk modulation, consistent with traveling-wave homogenization \cite{milton}. A complete scattering description for a finite crystal additionally requires boundary matching and uniform control of the resulting boundary-coupled resolvent. These higher-dimensional scattering questions remain open in the present work.


Previous work provides several complementary descriptions of resonant media. 
The effective properties of finite systems of high contrast resonators are established in \cite{hai2017,feppon2023,mourad2020}.
For infinite bubbly crystals, Bloch analysis near the first band gap yields an effective equation determined by the band curvature \cite{ammari2019bloch}. Capacitance approximations also describe finite resonator systems \cite{cbms}, and their spectral connection with infinite lattices has been studied through finite-array spectral convergence \cite{LMS25}. Quantitative homogenization for scattering by bounded periodic high-contrast composites is developed in \cite{du2026homogenization} for a soft-inclusion coefficient scaling. The present work concerns the acoustic density-and-bulk-modulus scaling of finite bubbly crystals, with a frequency approaching the moving upper first-band edge. The central issue is to retain the boundary termination of a dense finite array while passing from the microscopic transmission problem to a quantitative macroscopic scattering limit. Our main contribution is a quantitative scattering theory for dense finite bubbly crystals in the critical scaling $\delta=\eps^2$, at frequencies within $O(\eps^2)$ of the upper first-band edge. The analysis retains the boundary created by truncating the periodic medium and connects the resulting finite-array capacitance operator to an effective Dirichlet operator through an $O(\eps)$ norm-resolvent approximation. This yields a sound-soft limit for the full acoustic field away from the limiting Dirichlet spectrum. To the best of our knowledge, this work is the first to analyze the scattering responses of finite systems of periodically distributed high contrast subwavelength resonators at the subwavelength regime, closing an important gap in the field.

The paper is organized as follows. Section \ref{sec:mainresult} introduces the scattering problem, the scaling, and the main result. Section \ref{sec:prelimi} develops the variational capacitance operators and the uniform Bloch-band analysis. Section \ref{secreformulation} gives the exact operator decomposition and the boundary and coupling estimates. Section \ref{secresolventlimitofA} establishes the discrete-to-continuum resolvent approximation and identifies the reduced acoustic operator. Section \ref{secproofofmain} proves the scattering convergence theorem. Section \ref{sec:numerical-experiments} presents numerical illustrations. Appendix \ref{app:one-dimensional-transfer-matrix} contains the explicit one-dimensional transfer-matrix analysis, including the band-edge regime and the fixed interior-band regime. Appendix \ref{append:B} gives a formal derivation of the transport equation describing the effective behavior of the amplitudes at a regular Bloch point.

\section{Problem setup and main results}\label{sec:mainresult}

\subsection{Problem setup}

We first describe the finite-size bubble phononic crystal under consideration. Let $d \in \{2,3\}$. Let $Y=(-\frac{1}{2},\frac{1}{2})^d$ be the unit cell in $\R^d$, and let $D \Subset Y$ be a bounded and simply connected Lipschitz domain. Let $\Omega$ be a bounded and simply connected Lipschitz domain in $\mathbb{R}^d$ and assume that $\R^d \setminus \overline{\Omega}$ is connected. For any $\varepsilon \in (0,1)$, we define
\begin{equation}\label{def0D}
    \mathcal{I}_{\varepsilon}:=\{ n \in \mathbb{Z}^d : \varepsilon (n+Y) \Subset \Omega \} \qquad \mathrm{and} \qquad D_{\varepsilon}:= \bigcup_{n\in \mathcal{I}_{\varepsilon}} \varepsilon (n+D) .
\end{equation}
Physically, $D_{\varepsilon}$ denotes the union of a periodic array of bubbles with period $\varepsilon$ contained in $\Omega$, and the distance between $D_{\varepsilon}$ and $\partial \Omega$ is of order $\varepsilon$. See Figure~\ref{finitecrystal}.

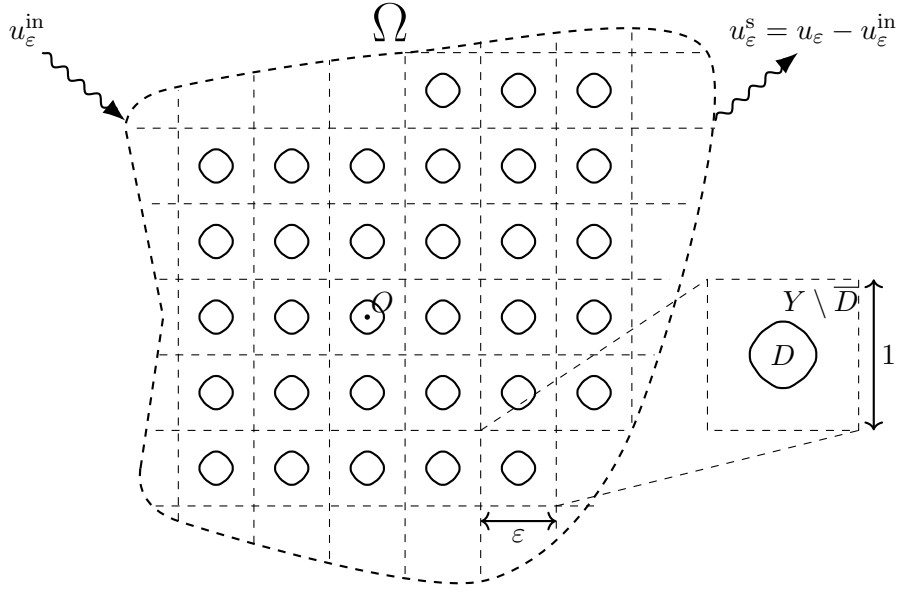
\begin{figure}
    \centering
    \begin{tikzpicture}[scale=1]

\def\N{5}
\def\L{5}


    \draw[dashed] (-0.3,0) -- (5.5,0);
    \draw[dashed] (-0.3,1) -- (5.9,1);
    \draw[dashed] (-0.25,2) -- (6.3,2);
    \draw[dashed] (-0.25,3) -- (6.5,3);
    \draw[dashed] (-0.35,4) -- (6.8,4);
    \draw[dashed] (-0.55,5) -- (7.1,5);
    \draw[dashed] (3,6) -- (7,6);

    \draw[dashed] (0,-0.2) -- (0,5.5);
    \draw[dashed] (1,-0.5) -- (1,5.7);
    \draw[dashed] (2,-0.7) -- (2,5.9);
    \draw[dashed] (3,-0.9) -- (3,6);
    \draw[dashed] (4,-0.9) -- (4,6.15);
    \draw[dashed] (5,-0.3) -- (5,6.2);
    \draw[dashed] (6,1.1) -- (6,6.3);

\foreach \i in {0,...,4}{
    \foreach \j in {0,...,4}{
        \begin{scope}[shift={(\i+0.5,\j+0.5)}]
            \draw[thick]
            plot[smooth cycle, tension=0.9] coordinates {
                ( 0.00,  0.22)
                ( 0.10,  0.18)
                ( 0.18,  0.10)
                ( 0.22,  0.00)
                ( 0.18, -0.10)
                ( 0.10, -0.18)
                ( 0.00, -0.22)
                (-0.10, -0.18)
                (-0.18, -0.10)
                (-0.22,  0.00)
                (-0.18,  0.10)
                (-0.10,  0.18)
            };
        \end{scope}
    }
}

\begin{scope}[shift={(3.5,5.5)}]
    \draw[thick]
        plot[smooth cycle, tension=0.9] coordinates {
            ( 0.00,  0.22)
            ( 0.10,  0.18)
            ( 0.18,  0.10)
            ( 0.22,  0.00)
            ( 0.18, -0.10)
            ( 0.10, -0.18)
            ( 0.00, -0.22)
            (-0.10, -0.18)
            (-0.18, -0.10)
            (-0.22,  0.00)
            (-0.18,  0.10)
            (-0.10,  0.18)
        };
\end{scope}

\begin{scope}[shift={(4.5,5.5)}]
    \draw[thick]
        plot[smooth cycle, tension=0.9] coordinates {
            ( 0.00,  0.22)
            ( 0.10,  0.18)
            ( 0.18,  0.10)
            ( 0.22,  0.00)
            ( 0.18, -0.10)
            ( 0.10, -0.18)
            ( 0.00, -0.22)
            (-0.10, -0.18)
            (-0.18, -0.10)
            (-0.22,  0.00)
            (-0.18,  0.10)
            (-0.10,  0.18)
        };
\end{scope}

\begin{scope}[shift={(5.5,5.5)}]
    \draw[thick]
        plot[smooth cycle, tension=0.9] coordinates {
            ( 0.00,  0.22)
            ( 0.10,  0.18)
            ( 0.18,  0.10)
            ( 0.22,  0.00)
            ( 0.18, -0.10)
            ( 0.10, -0.18)
            ( 0.00, -0.22)
            (-0.10, -0.18)
            (-0.18, -0.10)
            (-0.22,  0.00)
            (-0.18,  0.10)
            (-0.10,  0.18)
        };
\end{scope}

\begin{scope}[shift={(5.5,4.5)}]
    \draw[thick]
        plot[smooth cycle, tension=0.9] coordinates {
            ( 0.00,  0.22)
            ( 0.10,  0.18)
            ( 0.18,  0.10)
            ( 0.22,  0.00)
            ( 0.18, -0.10)
            ( 0.10, -0.18)
            ( 0.00, -0.22)
            (-0.10, -0.18)
            (-0.18, -0.10)
            (-0.22,  0.00)
            (-0.18,  0.10)
            (-0.10,  0.18)
        };
\end{scope}

\begin{scope}[shift={(5.5,3.5)}]
    \draw[thick]
        plot[smooth cycle, tension=0.9] coordinates {
            ( 0.00,  0.22)
            ( 0.10,  0.18)
            ( 0.18,  0.10)
            ( 0.22,  0.00)
            ( 0.18, -0.10)
            ( 0.10, -0.18)
            ( 0.00, -0.22)
            (-0.10, -0.18)
            (-0.18, -0.10)
            (-0.22,  0.00)
            (-0.18,  0.10)
            (-0.10,  0.18)
        };
\end{scope}

\begin{scope}[shift={(5.5,2.5)}]
    \draw[thick]
        plot[smooth cycle, tension=0.9] coordinates {
            ( 0.00,  0.22)
            ( 0.10,  0.18)
            ( 0.18,  0.10)
            ( 0.22,  0.00)
            ( 0.18, -0.10)
            ( 0.10, -0.18)
            ( 0.00, -0.22)
            (-0.10, -0.18)
            (-0.18, -0.10)
            (-0.22,  0.00)
            (-0.18,  0.10)
            (-0.10,  0.18)
        };
\end{scope}

\begin{scope}[shift={(5.5,1.5)}]
    \draw[thick]
        plot[smooth cycle, tension=0.9] coordinates {
            ( 0.00,  0.22)
            ( 0.10,  0.18)
            ( 0.18,  0.10)
            ( 0.22,  0.00)
            ( 0.18, -0.10)
            ( 0.10, -0.18)
            ( 0.00, -0.22)
            (-0.10, -0.18)
            (-0.18, -0.10)
            (-0.22,  0.00)
            (-0.18,  0.10)
            (-0.10,  0.18)
        };
\end{scope}

\draw[
    thick,
    -{Latex[length=3mm]},
    decorate,
    decoration={
        snake,
        amplitude=1.5pt,
        segment length=8pt
    }
]
(-1.8,6.0) -- (-0.7,5.1);

\draw[
    thick,
    -{Latex[length=3mm]},
    decorate,
    decoration={
        snake,
        amplitude=1.5pt,
        segment length=8pt
    }
]
(7.1,5.1) -- (8.2,6.0);

\node at (-2.0,6.3)
{$u^{\mathrm{in}}_{\varepsilon}$};
\node at (8.4,6.3)
{$u^{\mathrm{s}}_{\varepsilon}=u_\varepsilon - \uinc_\varepsilon$};
\node at (2.8,6.4)
{{\huge $\Omega$}};

\draw[dashed] (5,0) -- (9,1);
\draw[dashed] (4,1) -- (7,3);
\draw[dashed] (7,1) rectangle (9,3);
\draw[thick, cm={2,0,0,2,(8cm,2cm)}]
    plot[smooth cycle, tension=0.9] coordinates {
        ( 0.00,  0.22)
        ( 0.10,  0.18)
        ( 0.18,  0.10)
        ( 0.22,  0.00)
        ( 0.18, -0.10)
        ( 0.10, -0.18)
        ( 0.00, -0.22)
        (-0.10, -0.18)
        (-0.18, -0.10)
        (-0.22,  0.00)
        (-0.18,  0.10)
        (-0.10,  0.18)
    };

\node at (8,2)
{$D$};
\node at (8.5,2.7)
{$Y\setminus \overline{D}$};

\draw[thick, <->] (9.2,1) -- (9.2,3);
\node at (9.4,2)
{$1$};

\draw[thick, <->] (4,-0.2) -- (5,-0.2);
\node at (4.5,-0.4)
{$\varepsilon$};

\fill (2.5,2.5) circle (1pt);
\node at (2.7,2.7)
{$O$};

\draw[thick, dashed] (-0.5,0.5) -- (-0.2,2.5) -- (-0.7,5);
\draw[thick, dashed]
    plot[smooth] coordinates {
        (-0.5,0.5)
        (0,-0.2)
        (4,-1)
        (6,1)
        (7,6)
        (3,6)
        (-0.2,5.5)
        (-0.7,5)
    };
                
\end{tikzpicture}
    \caption{An illustration of the geometry, where $O$ denotes the origin.
    }
    \label{finitecrystal}
\end{figure}

We denote by $\rho_{\rm b}$ and $\kappa_{\rm b}$ the density and the bulk modulus of the air inside the bubbles $D_\varepsilon$,
respectively, and by $\rho$ and $\kappa$ the corresponding parameters for the matrix medium in $\mathbb{R}^d \setminus \overline{D_\varepsilon}$. We assume that $\rho_{\rm b},\kappa_{\rm b}, \rho,\kappa >0.$

For an incident frequency $\omega_\varepsilon>0$, which will be chosen later, let the wavenumber be
\begin{equation}
    k_\varepsilon := \sqrt{\frac{\rho}{\kappa}} \omega_\varepsilon,
\end{equation}
and $\uinc_\varepsilon$ be the incident plane wave propagating along some direction $\theta \in \mathbb{S}^{d-1}$:
\begin{equation}\label{inciplanewave}
    \uinc_\varepsilon (x) := e^{\mi  k_\varepsilon \theta \cdot x}, \qquad \forall x\in \mathbb{R}^d.
\end{equation}

In this context, a function $u$ is said to satisfy the Sommerfeld radiation condition with wavenumber $k$, denoted by $u\in\mathrm{SRC}(k)$, if
\begin{equation}\label{src}
    (\partial_{|x|}  - \mi k ) u(x) =\mathcal{O}(|x|^{-(d+1)/2}) \qquad \mathrm{as}\ |x| \rightarrow \infty,
\end{equation}
where $\partial_{|x|} := \frac{x}{|x|}\cdot\nabla$ denotes the radial derivative.

With the incident acoustic plane wave $\uinc_\varepsilon$ satisfying \eqref{inciplanewave}, the total field $u_\varepsilon$ propagating over the finite-size bubble phononic crystal satisfies the following scattering problem:
\begin{equation}\label{maineq}
\left\{
\begin{aligned}
\nabla\cdot\frac{1}{\rho}\nabla u_\varepsilon +\frac{\omega_\varepsilon^2}{\kappa}u_\varepsilon &=0 &&\text{in }\mathbb{R}^d\setminus \overline{D_\varepsilon},\\
\nabla\cdot\frac{1}{\rho_{\rm b}}\nabla u_\varepsilon +\frac{\omega_\varepsilon ^2}{\kappa_{\rm b}}u_\varepsilon &=0 &&\text{in } D_\varepsilon,\\
u_\varepsilon|_+-u_\varepsilon|_-&=0 &&\text{on }\partial D_\varepsilon,\\
\left.\frac{1}{\rho}\frac{\partial u_\varepsilon}{\partial \nu}\right|_+ - \left. \frac{1}{\rho_{\rm b}}\frac{\partial u_\varepsilon}{\partial \nu} \right|_- &=0
&&\text{on }\partial D_\varepsilon,\\
u_\varepsilon - \uinc_\varepsilon
&\in\text{SRC}(k_\varepsilon).
\end{aligned}
\right.
\end{equation}
Here, $\partial/\partial\nu$ denotes the outward normal derivative and $|_\pm$ denote the limits from outside and inside $D_\varepsilon$, respectively. 

The main goal of this paper is to establish the limiting behavior of $u_{\varepsilon}$ as $\varepsilon \rightarrow 0$ quantitatively.

\subsection{The incident frequency near the subwavelength band edge}

Let $\mathcal{B}_\varepsilon= \R^d /(2\pi\varepsilon^{-1} \Z^d)$ be the Brillouin zone for the corresponding periodic structure.
The Bloch eigenvalues and eigenfunctions are solutions to the following $\alpha$-quasiperiodic equations for each $\alpha \in \mathcal{B}_\varepsilon$:
\begin{equation}\label{blocheq}
\left\{
\begin{aligned}
\nabla\cdot\frac{1}{\rho}\nabla u+\frac{\omega^2}{\kappa}u&=0 &&\text{in }\varepsilon(Y\setminus \overline{D}),\\
\nabla\cdot\frac{1}{\rho_{\rm b}}\nabla u+\frac{\omega^2}{\kappa_{\rm b}}u&=0 &&\text{in }\varepsilon D,\\
u|_+-u|_-&=0 &&\text{on }\partial (\varepsilon D),\\
\left.\frac{1}{\rho}\frac{\partial u}{\partial \nu}\right|_+ - \left. \frac{1}{\rho_{\rm b}}\frac{\partial u}{\partial \nu} \right|_- &=0
&&\text{on }\partial (\varepsilon D),\\
e^{-\mi \alpha\cdot x}u
&\text{ is periodic}.
\end{aligned}
\right.
\end{equation}

Let
\[
v=\sqrt{\frac{\kappa}{\rho}},\qquad v_{\rm b}=\sqrt{\frac{\kappa_{\rm b}}{\rho_{\rm b}}},\qquad k=\frac{\omega}{v},\qquad k_{\rm b}=\frac{\omega}{v_{\rm b}}
\]
be, respectively, the speed of sound outside and inside the bubbles and the wavenumber outside and inside the bubbles. We also introduce
the dimensionless contrast parameter
\[
\delta=\frac{\rho_{\rm b}}{\rho}.
\]
For bubbly media, we assume that $\delta\ll1$, justifying the high contrast nature of the media. In a realistic setup, $\delta$ may
be of the order of $10^{-3}$. On the other hand, we assume that
\[
\frac{k^2}{k_{\rm b}^2} = \frac{v_{\rm b}^2}{v^2} = \frac{\rho\kappa_{\rm b}}{\rho_{\rm b}\kappa} = O(1),
\]
\emph{i.e.}, the wave numbers inside and outside the bubbles are comparable and are of order one.

It is known that \eqref{blocheq} has nontrivial solutions for discrete values of $\omega$ that are called Bloch eigenvalues. These eigenvalues can be arranged in the following increasing manner:
\begin{equation}
0\leq \omega_{1,\varepsilon}^{\alpha}\leq \omega_{2,\varepsilon}^{\alpha}\leq\cdots .
\end{equation}
We have the following band structure of propagating frequencies for the given periodic structure:
\begin{equation}
\left[0,\max_{\alpha \in \mathcal{B}_\varepsilon}\omega_{1,\varepsilon}^{\alpha}\right] \cup \left[\min_{\alpha \in \mathcal{B}_\varepsilon}\omega_{2,\varepsilon}^{\alpha}, \max_{\alpha \in \mathcal{B}_\varepsilon}\omega_{2,\varepsilon}^{\alpha }\right] \cup \left[\min_{\alpha \in \mathcal{B}_\varepsilon}\omega_{3,\varepsilon}^{\alpha}, \max_{\alpha \in \mathcal{B}_\varepsilon}\omega_{3,\varepsilon}^{\alpha}\right]
\cup\cdots .
\end{equation}

In \cite{ammari2017subwavelength}, it is shown that there is a subwavelength band-gap in the above band structure for fixed $\varepsilon$ (say $\varepsilon=1$) and sufficiently small $\delta$. More precisely, for $\varepsilon =1$, one has
\begin{equation}
\omega_{*,1} :=
\max_{\alpha \in \mathcal{B}_1}\omega_{1,1}^{\alpha} = O(\delta^{1/2}) \ll \min_{\alpha\in \mathcal{B}_1}\omega_{2,1}^{\alpha} = 
O(1).
\end{equation}
Thus, the first band $\alpha \mapsto \omega_{1,\varepsilon}^\alpha$ is called the \textit{subwavelength band}.
In this paper, we investigate the asymptotic properties of the scattering field when the incident frequency $\omega_\varepsilon$ is close to the critical frequency $\omega_{*,\varepsilon}$ where the subwavelength band gap opens for $\varepsilon\ll1$. By a scaling argument, it can be shown that
\begin{equation}
\omega_{*,\varepsilon}=\frac{1}{\varepsilon}\omega_{*,1}.
\end{equation}
In order to fix the critical frequency in the limit when $\varepsilon$ tends to zero, we rescale the contrast parameter $\delta$ as follows:
\begin{equation}
\delta=\varepsilon^2.
\end{equation}
Then the critical frequency remains of order one in the limiting process when $\varepsilon$ tends to zero. Thus, we are in a situation where the wavelength (of the free space) is of order one and the cell size is of order $\varepsilon\ll1$.

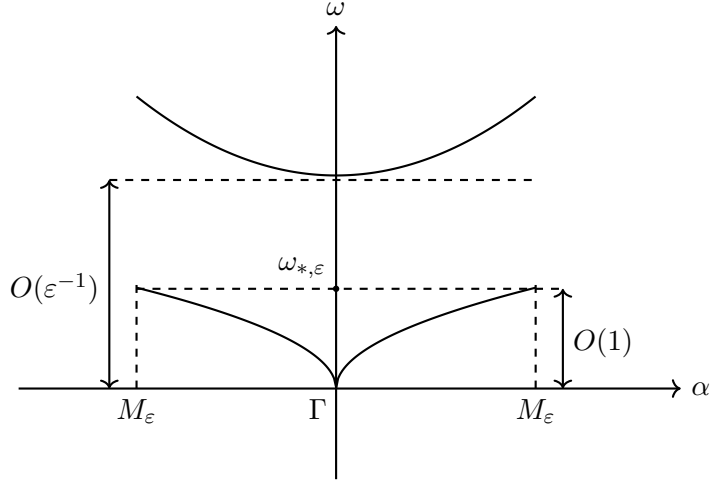
\begin{figure}
    \centering
    \begin{tikzpicture}[scale=1.2]

\draw[->, thick] (-3.5,0) -- (3.8,0) node[right] {$\alpha$};
\draw[->, thick] (0,-1) -- (0,4) node[above] {$\omega$};

\draw[thick]
    plot[domain=-2.2:0, samples=100]
    (\x,{0.75*sqrt(abs(\x))});

\draw[thick]
    plot[domain=0:2.2, samples=100]
    (\x,{0.75*sqrt(abs(\x))});

\draw[thick]
    plot[domain=-2.2:2.2, samples=100]
    (\x,{2.35+0.18*\x*\x});

\draw[dashed, thick]
    (-2.2,1.1) -- (2.5,1.1);

\draw[dashed, thick]
    (-2.5,2.3) -- (2.2,2.3);

\draw[dashed, thick]
    (-2.2,0) -- (-2.2,1.2);

\draw[thick, <->]
    (-2.5,0) -- (-2.5,2.3);
\draw[thick, <->]
    (2.5,0) -- (2.5,1.1);

\draw[dashed, thick]
    (2.2,0) -- (2.2,1.2);

\node[below] at (-2.2,0) {$M_\varepsilon$};
\node[below] at (-0.2,0) {$\Gamma$};
\node[below] at (2.2,0) {$M_\varepsilon$};
\node[above left] at (0,1.1) {$\omega_{*,\varepsilon}$};
\fill (0,1.1) circle (1pt);

\node[left] at (-2.5,1.1) {$O(\varepsilon^{-1})$};
\node[right] at (2.5,0.5) {$O(1)$};

\end{tikzpicture}
    \caption{A graph illustration of the band structure for bubbly media with period $\varepsilon$ and contrast $\delta = \varepsilon^2$, where $\Gamma$ denotes
    the origin of the Brillouin zone, and $M_\varepsilon$ denotes the corner of the Brillouin zone.}
    \label{blochband}
\end{figure}

Let $M_\varepsilon$ be the corner of the Brillouin zone $\mathcal{B}_\varepsilon$, which is unique as $\mathcal{B}_\varepsilon$ can be identified as a torus. We choose 
$M_\varepsilon =(\pi/\varepsilon,\pi/\varepsilon)$ for $d=2$ and $M_\varepsilon =(\pi/\varepsilon,\pi/\varepsilon,\pi/\varepsilon)$ for $d=3$. Due to time reversal symmetry, $M_\varepsilon$ is a critical point of $\alpha \mapsto \omega^\alpha_{1,\varepsilon}$.
To simplify the analysis, we impose the following assumption:
\begin{Assumption}\label{assumpM}
    The maximum $\omega_{*,\varepsilon}$ is attained at $M_\varepsilon$, and the Hessian matrix $\nabla_\alpha^2 \omega_{1,\varepsilon}^{M_\varepsilon}$ of the subwavelength band at $M_\varepsilon$ is negative definite.
\end{Assumption}

Assumption \ref{assumpM} is usually guaranteed by suitable symmetry conditions. A simple and verifiable sufficient condition is
\begin{equation}\label{symcond}
    D\subset \mathbb{R}^d \text{ is symmetric with respect to the coordinate planes}.
\end{equation}
That is, for $d=2$, for every $(x_1,x_2) \in D$, $(x_1,-x_2)$ and $(-x_1,x_2) $ are both in $D$. For $d=3$, for every $(x_1,x_2,x_3) \in D$, $(x_1,x_2,-x_3)$, $(x_1,-x_2,x_3)$ and $(-x_1,x_2,x_3)$ are in $D$. In \cite{ammari2019bloch}, it was shown that, when $d=3$ and \eqref{symcond} holds, the maximum $\omega_{*,\varepsilon}$ is attained at $M_\varepsilon$, and the Hessian matrix $\nabla_\alpha^2 \omega_{1,\varepsilon}^{M_\varepsilon}$ is negative semi-definite. In Section~\ref{secsecorder}, we strengthen this result by proving that $\nabla_\alpha^2 \omega_{1,\varepsilon}^{M_\varepsilon}$ is, in fact, negative definite. Consequently, the symmetry condition \eqref{symcond} implies Assumption~\ref{assumpM}. Throughout the paper,
we assume that \eqref{symcond} is satisfied.

We set the incident frequency to be
\begin{equation}
    \omega_{\varepsilon} = \omega_{*,\varepsilon} - \tau \varepsilon^2 +O(\varepsilon^4)
\end{equation}
for some $\tau\in \mathbb{R}$. So, $\omega_{\varepsilon}$ is near the upper edge of the subwavelength band. In \cite{ammari2017subwavelength}, the leading-order term of $\omega_{*,\varepsilon}$ with respect to $\varepsilon$ was derived and characterized in terms of the capacitance matrix. However, this approximation is insufficient for our purpose, as computation of the next-order correction is required. The derivation of this higher-order expansion will be carried out in Section \ref{secsecorder}.

\subsection{Main results}

In Proposition \ref{propexpansionofband}, we prove that
\begin{equation}
    \omega_{*,\varepsilon} = \omega_0 +  O(\varepsilon^2), \qquad \omega_0:= v_{\rm b} \sqrt{\widehat{C}^M},
\end{equation}
where $\widehat{C}^M$ is the capacitance defined in \eqref{minienerquasip1}. Let $k_0$ be the limiting wave number:
\begin{equation}
    k_0 := \frac{\omega_0}{v} = \frac{v_{\rm b} }{v} \sqrt{\widehat{C}^M}.
\end{equation}

Define the limiting soft-obstacle scattering problem
\begin{equation}\label{problimit}
\left\{
\begin{aligned}
    & (\Delta +k_0^2) \widehat{u}_0 =0 && \mathrm{in} \ \mathbb{R}^d\setminus \overline{\Omega}, \\
    & \widehat{u}_0|_{\partial\Omega} =0, \\ 
    & \widehat{u}_0 - u^{\mathrm{in}} \in \mathrm{SRC}(k_0),
\end{aligned}\right.
\end{equation}
where $u^{\mathrm{in}} $ is the limiting plane wave of $u^{\mathrm{in}}_\varepsilon$, \emph{i.e.},
\begin{equation}
    \uinc  (x) := e^{\mi  k_0 \theta \cdot x}, \qquad \forall x\in \mathbb{R}^d.
\end{equation}
Hereafter, we identify $\widehat{u}_0$ with its zero extension on $\mathbb{R}^d$. For the incident wave, define the data
\begin{equation}
    \mathcal{N}^{\mathrm{in}} := \| \uinc \|_{H^1(\partial \Omega)} + \left\|\left.\frac{\partial u^{\mathrm{in}}}{\partial\nu}\right|_{\partial\Omega}\right\|_{L^2(\partial\Omega)} .
\end{equation}

In Section \ref{secpropofcapa}, we prove that $M=(\pi,\cdots,\pi)$ is the unique maximal point of $\widehat{C}^\alpha$ over $\alpha \in \mathcal{B}_1$, and the Hessian of $\widehat{C}^\alpha$ at $M$ is negative-definite. Therefore, we are allowed to define
the second-order elliptic differential operator
\begin{equation}\label{defLDintro}
   L_D := -\nabla \cdot (A_{\mathrm{eff}} \nabla)  
\end{equation}
on $\Omega$ with the zero Dirichlet boundary condition on $\partial \Omega$, where
\begin{equation}\label{defAeffintro}
    A_{\mathrm{eff}}:= -\frac{1}{2}\nabla_\alpha^2 \widehat{C}^M >0 .
\end{equation}
We regard $L_D$ as an unbounded operator on $L^2(\Omega)$ with the domain
\begin{equation}
    D(L_D):= \{ u \in H^1_0(\Omega): L_D u \in L^2(\Omega)\}.
\end{equation}
The spectrum of $L_D$ is discrete and is denoted by $\mathrm{spec}\,(L_D)$.

Throughout this paper, we denote by $B_R$ the open ball of radius $R>0$ centered at the origin. The main result is as follows.

\begin{theorem}\label{themaintheorem}
    Fix $\tau \in \mathbb{R}$, and assume that 
    \begin{equation}
        \omega_\varepsilon = \omega_{*,\varepsilon} - \tau \varepsilon^2 + O(\varepsilon^4) .
    \end{equation}
    Here, the expansion formula of $\omega_{*,\varepsilon}$ up to second-order is given in \eqref{eq:omega-edge-expansion}. Assume that
    \begin{equation}
        \sigma :=\frac{2\omega_0}{v_{\rm b}^2} \tau = \frac{2\sqrt{\widehat{C}^M}}{v_{\rm b}}\tau\notin \mathrm{spec}\,(L_D).
    \end{equation}
    Let $u_\varepsilon$ be the solution of the problem \eqref{maineq}, and let $\widehat u_0$ be the solution of the limiting soft-obstacle scattering problem \eqref{problimit}. Then, for every fixed ball $B_R$ such that $\overline\Omega\Subset B_R$, there exists $C>0$, depending only on $D,\Omega,R,\rho,\rho_{\rm b},\kappa,\kappa_{\rm b},|\tau|$, and $\mathrm{dist}\,(\sigma,\mathrm{spec}\,(L_D))$, such that
\begin{equation}\label{eq:section-six-main-interior-and-exterior-L2-rates1}
    \|u_\varepsilon\|_{L^2(\Omega)}\leq C\varepsilon^{1/2}\mathcal N^{\mathrm{in}}, \qquad \|u_\varepsilon-\widehat u_0\|_{L^2(B_R\setminus\overline\Omega)}\leq C\varepsilon\mathcal N^{\mathrm{in}},
\end{equation}
and
\begin{equation}\label{eq:section-six-main-global-L2-and-exterior-H1-rates1}
     \|u_\varepsilon-\widehat u_0\|_{H^1(B_R\setminus\overline\Omega)}\leq C\varepsilon^{1/2}\mathcal N^{\mathrm{in}}.
\end{equation}
Moreover, for every compact set $K\Subset\mathbb R^d\setminus\overline\Omega$, there exists $C>0$, additionally depending on $K$, such that
\begin{equation}\label{eq:section-six-main-local-H1-rate1}
    \|u_\varepsilon-\widehat u_0\|_{H^1(K)}\leq C\varepsilon\mathcal N^{\mathrm{in}}.
\end{equation}
Consequently, if $u_\varepsilon^{\mathrm s}:=u_\varepsilon-u_\varepsilon^{\mathrm{in}}$ and $\widehat u_0^{\mathrm s}:=\widehat u_0-u^{\mathrm{in}}$, and if $u_\varepsilon^{\mathrm s,\infty}$ and $\widehat u_0^{\mathrm s,\infty}$ denote their far-field patterns, then, for any $m\in \mathbb{N}$, there exists $C>0$, additionally depending on $m$, such that
\begin{equation}\label{eq:section-six-main-far-field-rate1}
    \|u_\varepsilon^{\mathrm s,\infty}-\widehat u_0^{\mathrm s,\infty}\|_{C^m(\mathbb S^{d-1})}\leq C\varepsilon\mathcal N^{\mathrm{in}}. 
\end{equation}
\end{theorem}

\section{Preliminaries}\label{sec:prelimi}

In this section, we review the capacitance matrices, and prove some new properties of the capacitance matrices and the subwavelength band.

\subsection{Definition of capacitance matrices} \label{seccapacity}

We introduce the capacitance matrices by variational definitions. Define the admissible function space
\begin{equation}\label{energyspace}
    \mathcal{E}:=\left\{
    \begin{aligned}
        & \{u\in H^1_{\mathrm{loc}}(\mathbb{R}^2) : \nabla u \in L^2(\mathbb{R}^2)\} , && d=2, \\
        & \{u\in H^1_{\mathrm{loc}}(\mathbb{R}^3) \cap L^6(\mathbb{R}^3) : \nabla u \in L^2(\mathbb{R}^3)\} , && d=3.
    \end{aligned}\right.
\end{equation}
This space is usually used to study the behavior at infinity of harmonic functions \cite{arendt2015dirichlet}. Fix $\delta>0$ and a nodal set $\Lambda \subset \mathbb{Z}^d$. For every $n\in \Lambda$, we define $\mathscr{V}_{\delta,\Lambda}^n$ by
\begin{equation}\label{energyminprobN}
    \mathscr{V}_{\delta,\Lambda}^n := \operatorname*{arg\,min}_{u \in \mathcal{E}\atop \fint_{D+m} u(x)\,dx= \delta_{nm},\forall m \in \Lambda} \frac{1}{|D|}\int_{\mathbb{R}^d} a_{\delta,\Lambda} | \nabla u|^2\,dx ,
\end{equation}
where
\begin{equation}\label{weightbubleN}
    a_{\delta,\Lambda} := \delta^{-1} \mathbbm{1}_{D+\Lambda} + \mathbbm{1}_{\mathbb{R}^d\setminus (D+\Lambda)}.
\end{equation}
For every $m,n\in \Lambda$, we define the capacitance coefficient
\begin{equation}
    (\mathscr{C}_{\delta,\Lambda})_{mn} := \frac{1}{|D|} \int_{\mathbb{R}^d} a_{\delta,\Lambda} \nabla  \mathscr{V}_{\delta,\Lambda}^n \cdot \overline{ \nabla  \mathscr{V}_{\delta,\Lambda}^m } \,dx.
\end{equation}

When $\Lambda$ is finite, the $|\Lambda|\times |\Lambda|$ capacitance matrix $C^{\mathrm{f}}_{\delta,\Lambda}$ for the contrast $\delta>0$ and the bubble array $D+\Lambda$ is defined by
\begin{equation}
    C^{\mathrm{f}}_{\delta,\Lambda} := (\mathscr{C}_{\delta,\Lambda})_{mn} .
\end{equation}
When $\Lambda=\mathbb{Z}^d$, we define the capacitance operator $\mathfrak{C}_\delta :\ell^2(\mathbb{Z}^d)\rightarrow \ell^2(\mathbb{Z}^d)$ by
\begin{equation}
     (\mathfrak{C}_\delta  \phi)(m) := \sum_{n\in \mathbb{Z}^d} (\mathscr{C}_{\delta,\mathbb{Z}^d})_{mn}   \phi(n) , \qquad \phi  \in \ell^2(\mathbb{Z}^d).
\end{equation}

For each $\alpha \in \mathcal{B}_1$, let $H^1_\alpha(Y)\subset H^1(Y)$ be the space of $\alpha$-quasiperiodic functions. Define $V^\alpha_\delta$ by
\begin{equation}\label{minienerquasip}
    V^\alpha_\delta := \operatorname*{arg\,min}_{u \in H^1_\alpha(Y) \atop \fint_D u(x)\,dx=1 }  \frac{1}{|D|} \int_Y a_\delta  | \nabla u|^2\,dx  ,
\end{equation}
where
\begin{equation}
    a_\delta:= \delta^{-1} \mathbbm{1}_{D}+\mathbbm{1}_{Y\setminus D}.
\end{equation}
The minimal energy in the right-hand side of \eqref{minienerquasip} is denoted by $\widehat{C}^\alpha_\delta$, \emph{i.e.},
\begin{equation}
    \widehat{C}^\alpha_\delta := \frac{1}{|D|} \int_Y a_\delta | \nabla V^\alpha_\delta |^2 \,dx.
\end{equation}
With the discrete Fourier transform (also referred to as the Floquet transform)
\begin{equation}\label{defFourier}
    \mathcal{F}[\phi](\alpha) =\widehat{\phi}(\alpha) := \sum_{n \in \mathbb{Z}^d}  \phi(n) e^{\mi n\cdot \alpha}, \qquad \phi\in \ell^2(\mathbb{Z}^d)
\end{equation}
and the inverse transform
\begin{equation}
    \mathcal{F}^{-1}[\psi](n) = \check\psi(n):= \frac{1}{|\mathcal{B}_1|} \int_{\mathcal{B}_1} \psi(\alpha) e^{-\mi n\cdot\alpha} \,d\alpha, \qquad \psi \in L^2(\mathcal{B}_1),
\end{equation}
it is not difficult to get
\begin{equation}
    (\mathscr{C}_{\delta,\mathbb{Z}^d})_{mn}= \mathcal{F}^{-1}[\widehat{C}_\delta^\alpha] (m-n).
\end{equation}
Therefore, $\mathfrak{C}_\delta$ is a Fourier multiplier operator with the symbol $\mathcal{B}_1 \ni \alpha \mapsto \widehat{C}^\alpha_\delta$.

Let $E_{\Lambda}: \ell^2(\Lambda) \rightarrow \ell^2(\mathbb{Z}^d)$ be the zero extension operator. When $\Lambda$ is finite, we define the truncated capacitance matrix $C^{\mathrm{t}}_{\delta,\Lambda}$ by
\begin{equation}
    C^{\mathrm{t}}_{\delta,\Lambda} := E_{\Lambda}^* \mathfrak{C}_\delta E_{\Lambda}.
\end{equation}
Roughly speaking, $C^{\mathrm{t}}_{\delta,\Lambda}$ is obtained by keeping only the centering $|\Lambda| \times |\Lambda|$ capacitance  coefficients from $\mathfrak{C}_\delta$. 

We also define the capacitance matrices for $\delta =0$. When $\delta \rightarrow 0$, the weight $a_{\delta,\Lambda}$ goes to infinity on the bubbles $D+\Lambda$. This forces the minimizer in \eqref{energyminprobN} to be a constant in each bubble. Therefore, we define 
\begin{equation}
    \mathscr{V}_{\Lambda}^n := \operatorname*{arg\,min}_{u \in \mathcal{E} \atop u(x)= \delta_{nm},\forall x\in D+m,m\in \Lambda} \frac{1}{|D|}\int_{\mathbb{R}^d\setminus (D+\Lambda)}  | \nabla u|^2\,dx .
\end{equation}
For every $m,n\in \Lambda$, we define the capacitance coefficient
\begin{equation}
    (\mathscr{C}_{\Lambda})_{mn} := \frac{1}{|D|} \int_{\mathbb{R}^d\setminus(D+\Lambda)} \nabla  \mathscr{V}_{\Lambda}^n \cdot \overline{ \nabla  \mathscr{V}_{\Lambda}^m } \,dx.
\end{equation}
When $\Lambda$ is finite, the $|\Lambda|\times |\Lambda|$ capacitance matrix $C^{\mathrm{f}}_{\Lambda}$ is defined by
\begin{equation}
    C^{\mathrm{f}}_{\Lambda} := (\mathscr{C}_{\Lambda})_{mn} .
\end{equation}
When $\Lambda=\mathbb{Z}^d$, we define the capacitance operator $\mathfrak{C} :\ell^2(\mathbb{Z}^d)\rightarrow \ell^2(\mathbb{Z}^d)$ by
\begin{equation} \label{def:capacitanc_op}
     (\mathfrak{C}  \phi)(m) := \sum_{n\in \mathbb{Z}^d} (\mathscr{C}_{\mathbb{Z}^d})_{mn}   \phi(n) , \qquad \phi  \in \ell^2(\mathbb{Z}^d).
\end{equation}

For each $\alpha \in \mathcal{B}_1$, define $V^\alpha$ and $\widehat{C}^\alpha$ by
\begin{equation}\label{minienerquasip1}
    V^\alpha := \operatorname*{arg\,min}_{u \in H^1_\alpha(Y) \atop u|_D\equiv 1}  \frac{1}{|D|} \int_{Y\setminus D} | \nabla u|^2\,dx  , \qquad \widehat{C}^\alpha := \frac{1}{|D|} \int_{Y\setminus D} | \nabla V^\alpha |^2 \,dx.
\end{equation}
Still, $\mathfrak{C}$ is a Fourier multiplier operator with the symbol $\mathcal{B}_1 \ni \alpha \mapsto \widehat{C}^\alpha$:
\begin{equation}
    (\mathscr{C}_{\mathbb{Z}^d})_{mn}= \mathcal{F}^{-1}[\widehat{C}^\alpha] (m-n).
\end{equation}
Finally, the truncated capacitance matrix $C^{\mathrm{t}}_{\Lambda}$ for a finite $\Lambda$ is defined by
\begin{equation}
    C^{\mathrm{t}}_{\Lambda} := E_{\Lambda}^* \mathfrak{C}  E_{\Lambda}.
\end{equation}

\subsection{Properties of capacitance matrices}\label{secpropofcapa}

In this section, we establish some useful properties of capacitance matrices. We start with the notion of Banach-valued holomorphic functions.

On $\mathbb{C}^n$, $n\in \N$, we use the notation
    \begin{equation*}
    z \cdot \zeta = z_1 \zeta_1 + z_2\zeta_2 + \dots + z_n \zeta_n.
    \end{equation*}
There is a canonical embedding $\iota$ from $\R^n$ to $\bC^n$ given by $\iota(x) = x+\mi 0$. In the sequel of this section, for simplicity of notation, we identify subsets of $\R^n$ with their images under this embedding. Namely, an open set $U\subset \R^n$ is identified with $\widetilde{U} := \iota(U) = \{x+\mi 0\,:\,x\in U\}$. 
    
\begin{definition}
    Let $n\in \mathbb{N}$ and $X$ be a complex Banach space. A mapping $F:O\subset \mathbb{C}^n \rightarrow X$ is holomorphic if, near every $z_0\in O$, $F$ admits an
    $X$-norm convergent series
    \begin{equation}
        F(z) = \sum_{\gamma \in \mathbb{N}^n} F_\gamma (z-z_0)^\gamma.
    \end{equation}
    A mapping $G:U\subset \mathbb{R}^n \rightarrow X$ is (real) analytic if there exists a holomorphic mapping $F:O\subset \mathbb{C}^n \rightarrow X$ such that $U \subset O$ and $G=F|_U$. Such $F$ is called a holomorphic extension of $G$.
\end{definition}

In the following, we study analyticity with respect to the quasiperiodicity $\alpha$. However, the $\alpha$-quasiperiodic Sobolev spaces $H^1_\alpha(Y)$ depend on $\alpha$, and hence vary as $\alpha$ changes. To overcome this difficulty, we introduce the contrast-adapted gauge transform as follows. Choose $\Theta \in C^\infty(Y;\mathbb{R}^d)$ such that
\begin{equation}
    \Theta =0 \text{ near } \overline{D},  \qquad \Theta (x) =x \text{ near }\partial Y. 
\end{equation}
For every $z\in \mathbb{C}^d$, we define the gauge transform
\begin{equation}\label{gaugetran}
    U_z u(x) :=e^{\mi  z \cdot \Theta(x) } u(x), \qquad \forall\,u\in L^2(Y),\,x\in Y.
\end{equation}
One has
\begin{equation}
    \nabla (U_z u) = U_z(\nabla  + \mi Q_z ) u, \qquad \text{ where } Q_z:=\sum_{j=1}^d z_j \nabla \Theta_j.
\end{equation}

Given $\alpha \in \mathcal{B}_1$, let $H^\alpha g$ be the $\alpha$-quasiperiodic harmonic extension of $g\in H^{1/2}(\partial D)$ to $Y\setminus\overline{D}$. The exact definition is as follows. Let
\begin{equation}
    X_{\alpha}:=\left\{ u\in H_{\alpha}^1(Y\setminus\overline D):u|_+=0\text{ on }\partial D \right\}.
\end{equation}
Then, $H^\alpha g \in H^1_\alpha (Y\setminus\overline D)$ uniquely solves $(H^\alpha g)|_{\partial D+} =g$ and
\begin{equation}
    \int_{Y\setminus D} \nabla H^\alpha g \cdot \overline{ \nabla v } \,dx = 0, \qquad \forall\, v\in X_\alpha. 
\end{equation}
It is easy to show that $H^\alpha$ is a bounded operator from $H^{1/2}(\partial D)$ to $H^1(Y\setminus\overline D)$.

\begin{lemma}\label{lemharmoofH}
    There exists a holomorphic extension of $H^\alpha$, denoted by $H^z$, on a neighborhood $O\subset \mathbb{C}^d$ of $\mathcal{B}_1$.
\end{lemma}
\begin{proof}
    Define the sesquilinear form
    \begin{equation}\label{expressionbz}
    \begin{aligned}
        \mathfrak{b}_z (u,v)&:= \int_{Y\setminus D} \nabla (U_z u) \cdot \overline{ \nabla (U_{\overline{z}} v)} \,dx \\
        &\ = \int_{Y\setminus D} \{ \nabla u + \mi Q_z u  \}\cdot \{ \overline{\nabla v} - \mi Q_z\overline{v}  \} \,dx 
    \end{aligned}
    \end{equation}
for any $u,v\in H^1(Y\setminus \overline{D})$ and $z\in \mathbb{C}^d$. Expanding \eqref{expressionbz}, we get
\begin{equation*}
    \mathfrak{b}_z (u,v) = \mathfrak{b}_0 (u,v) + \mi\left( \int_{Y\setminus D} u  Q_z \cdot \overline{\nabla v}\,dx - \int_{Y\setminus D} \nabla u\cdot Q_z\overline{v}\,dx  \right) + \int_{Y\setminus D} (Q_z\cdot Q_z) u \overline{v}\,dx.
\end{equation*}
Hence, $z\mapsto \mathfrak{b}_z (u,v)$ is a polynomial in $z$, and
\begin{equation}\label{estibz0}
    |\mathfrak{b}_z (u,v) - \mathfrak{b}_0 (u,v)| \leq C (|z|+|z|^2) \|u \|_{H^1(Y\setminus \overline{D}) }
    \|v \|_{H^1(Y\setminus \overline{D}) }.
\end{equation}

Fix a reference point $\alpha_* \in \mathcal{B}_1$ and set $\mathcal{X}= X_{\alpha_*}$. We restrict $\mathfrak{b}_z$ in $\mathcal{X}\times \mathcal{X}$ so that the associated operator $\mathcal{B}(z) :\mathcal{X}\rightarrow \mathcal{X}^*$ is defined by
\begin{equation}
    \langle \mathcal{B}(z) u ,v \rangle_{\mathcal{X}^*,\mathcal{X}} = \mathfrak{b}_z(u,v),\qquad \forall\,u,v\in \mathcal{X}.
\end{equation}
Then, $z\mapsto \mathcal{B}(z)$ is an $\mathcal{L}(\mathcal{X},\mathcal{X}^*)$-valued polynomial in $z$. By Poincar\'{e}'s inequality,
\begin{equation}
    \langle \mathcal{B}(0) u ,u \rangle_{\mathcal{X}^*,\mathcal{X}} = \mathfrak{b}_0(u,u)\geq C\| u \|_{H^1(Y\setminus \overline{D})}^2 ,\qquad \forall\,u\in \mathcal{X}.
\end{equation}
Thus, $\mathcal{B}(0)$ is invertible by Lax-Milgram. By \eqref{estibz0}, 
\begin{equation}
    \| \mathcal{B}(z) - \mathcal{B}(0) \|_{\mathcal{L}(\mathcal{X}, \mathcal{X}^*)} \leq C(|z|+|z|^2).
\end{equation}
Therefore, for $|z|$ sufficiently small, the Neumann series
\begin{equation}\label{approxBz}
    \mathcal{B}(z)^{-1} = \mathcal{B}(0)^{-1} \sum_{n=0}^\infty \left\{- \big( \mathcal{B}(z) - \mathcal{B}(0) \big) \mathcal{B}(0)^{-1} \right\}^n
\end{equation}
converges uniformly in the operator norm. Since $\mathcal{B}(z)$ is a polynomial in $z$, \eqref{approxBz} shows that $\mathcal{B}(z)^{-1}$ is approximated by a locally uniformly convergent sequence of holomorphic functions. Consequently, $\mathcal{B}(z)^{-1} $ is holomorphic near $z=0$.

Let $Y'= (1-h)Y$, $0<h<1$, be a shrinking of $Y$ such that $D\Subset Y'$. The trace theorem provides an extension $\mathcal{E}:H^{1/2}(\partial D)\rightarrow H^1(Y\setminus\overline D)$ such that $(\mathcal{E}\phi)|_{Y\setminus Y'} =0$ and
    \begin{equation}
        \| \mathcal{E}\phi \|_{H^1(Y\setminus \overline{D})} \le C \| \phi \|_{H^{1/2}(\partial D)}, \qquad \forall \phi \in H^{1/2}(\partial D).
    \end{equation}
    In particular, $\mathrm{ran}\,\mathcal{E}\subset H^1_\alpha(Y\setminus\overline D)$ for every $\alpha \in \mathcal{B}_1$. Define $\mathcal{C}(z):H^{1/2}(\partial D) \rightarrow\mathcal{X}^*$ by
\begin{equation}
    \langle \mathcal{C}(z) g, v\rangle_{\mathcal{X}^*,\mathcal{X}} = \mathfrak{b}_z(\mathcal{E}g,v),\qquad \forall \,g \in H^{1/2}(\partial D) , \, v \in \mathcal{X}.
\end{equation}
This is also an operator-valued polynomial. Near $z=0$, set
\begin{equation}
    \mathcal{H}(z):= U_z \big( \mathcal{E} - \mathcal{B}(z)^{-1} \mathcal{C}(z) \big).
\end{equation}
Then, $\mathcal{H}(z)$ is an $\mathcal{L}(H^{1/2}(\partial D)  , H^1(Y\setminus\overline{D}))$-valued holomorphic function near $z=0$. Moreover, $\mathcal{H}(z)g|_{\partial D+} =g$ for every $g\in H^{1/2}(\partial D)$. For $z=\beta \in \mathbb{R}^d$, we have
\begin{equation}
    \mathcal{H}(\beta) g \in H^1_{\alpha_*+\beta}(Y\setminus \overline{D}) , \qquad \forall \, g \in H^{1/2}(\partial D),
\end{equation}
and
\begin{equation}
    \int_{Y\setminus D} \nabla \mathcal{H}(\beta) g  \cdot \overline{\nabla v}\,dx=0, \qquad \forall\, v\in X_{\alpha_*+\beta}.
\end{equation}
By uniqueness, we get $\mathcal{H}(\beta) = H^{\alpha_*+\beta}$. Since $\alpha_* \in \mathcal{B}_1$ is arbitrary, we may glue 
$\mathcal{H}(z)$ together to obtain a global holomorphic extension near $\mathcal{B}_1$.
\end{proof}

In our context, capacitance matrices may depend on $\delta$. Hence, we need to introduce the holomorphy uniformly in $\delta$.

\begin{definition}
     Let $I\subset \mathbb{R}$, $n\in \mathbb{N}$ and $X$ be a complex Banach space. A family of maps $\{F_\delta:O\subset \mathbb{C}^n\rightarrow X\}_{\delta\in I}$ is uniformly holomorphic in $\delta\in I$ if every $F_\delta$ is holomorphic and, near every $z_0\in O$, there exists $M,r>0$ independent of $\delta\in I$ such that
     \begin{equation}\label{uniformholdef}
         \sup_{\delta\in I \atop z\in O,\,|z-z_0|<r} \| F_\delta (z) \|_X \leq M.
     \end{equation}
     A family of maps $\{G_\delta:U\subset \mathbb{R}^n\rightarrow X\}_{\delta\in I}$ is uniformly analytic in $\delta\in I$ if there exists a uniformly holomorphic family of maps $\{F_\delta:O\subset \mathbb{C}^n\rightarrow X\}_{\delta\in I}$ such that $U\subset O$ and $G_\delta = F_\delta|_U$ for every $\delta \in I$. This $\{F_\delta\}_{\delta\in I}$ is called a uniformly holomorphic extension of $\{G_\delta\}_{\delta\in I}$.
\end{definition}

Let $H^z$ be the holomorphic extension of $H^\alpha$ on $O$ given by Lemma \ref{lemharmoofH}. Define
\begin{equation}
    T^z = -\left.\frac{\partial H^z }{\partial \nu} \right|_+ : H^{1/2}(\partial D) \rightarrow H^{-1/2}(\partial D), \qquad \forall \,z\in O.
\end{equation}
When $z=\alpha\in\mathcal{B}_1$, $T^\alpha$ is the exterior Dirichlet-to-Neumann operator for $\alpha$-quasiperiodic boundary condition. Let
\begin{equation}
    W:= \left\{u\in H^1(D):\int_D u\,dx=0 \right\}.
\end{equation}
Define the Neumann Laplacian $A:W \rightarrow W^*$ acting on the mean-zero Sobolev space by
\begin{equation}
    \langle A u ,v \rangle_{W^*,W} = \int_D \nabla u \cdot \overline{\nabla v}\,dx , \qquad \forall\, v\in W.
\end{equation}
By Lax-Milgram, $A$ is invertible. Denote by 
\begin{equation}
    \gamma: W\rightarrow H^{1/2}(\partial D)
\end{equation}
the trace operator, and by $\gamma^*: H^{-1/2}(\partial D) \rightarrow W^*$ the adjoint of $\gamma$. 

Since $T^z$ is holomorphic on $O$, there exists $\delta_0$ independent of $z$ such that, for any $\delta \in (0,\delta_0)$, 
\begin{equation}
    A + \delta \gamma^* T^z \gamma :W\rightarrow W^*
\end{equation}
is invertible. Moreover, the family
\begin{equation}\label{unifholominverse}
  \big\{ O\ni z\mapsto (A + \delta \gamma^* T^z \gamma  )^{-1}  \big\}_{\delta \in (0,\delta_0)}
\end{equation}
is uniformly holomorphic in $\delta \in (0,\delta_0)$.

\begin{lemma}\label{holoofV}
Let $H^z$ be the holomorphic extension of $H^\alpha$ on $O$ given by Lemma \ref{lemharmoofH}.
\begin{itemize}
    \item[(a)] For every $z\in O$, define
    \begin{equation}
        V^z := \left\{\begin{aligned}
            & 1, && \mathrm{\ in\ } D, \\
            & H^z(1), && \mathrm{\ in\ } Y\setminus D.
        \end{aligned} \right.
    \end{equation}
    Then, $V^z$ is a holomorphic extension of $V^\alpha$.
    \item[(b)] Let $\delta_0$ be given as above. For every $\delta \in (0,\delta_0)$ and $z\in O$, define
    \begin{equation}\label{formulaofVdelta}
        V_\delta^z := \left\{\begin{aligned}
            & 1 - \delta (A + \delta \gamma^* T^z \gamma )^{-1} \gamma^* T^z(1), && \mathrm{\ in\ } D, \\
            & H^z(V_\delta^z|_{\partial D-}), && \mathrm{\ in\ } Y\setminus D.
        \end{aligned} \right.
    \end{equation}
    Then, $\{V_\delta^z\}_{\delta \in (0,\delta_0)}$ is a uniformly holomorphic extension of $\{ V_\delta^\alpha \}_{\delta \in (0,\delta_0)}$.
\end{itemize}
\end{lemma}

\begin{proof}
    Item (a) is a direct corollary of Lemma \ref{lemharmoofH}. For (b), \eqref{unifholominverse} implies that $\{V_\delta^z\}_{\delta \in (0,\delta_0)}$ is a family of uniformly holomorphic functions. It remains to prove that
    \begin{equation}\label{formulaVope}
        1 - \delta (A + \delta \gamma^* T^\alpha \gamma )^{-1} \gamma^* T^\alpha(1) = V^\alpha_\delta |_D, \qquad \forall \, \alpha\in \mathcal{B}_1, \,\delta \in (0,\delta_0).
    \end{equation}

    By the variational definition of $V_\delta^\alpha$, we have
    \begin{equation}\label{variaofV}
        \delta^{-1} \int_D \nabla V_\delta^\alpha \cdot \overline {\nabla h}\,dx + \int_{Y\setminus D} \nabla V_\delta^\alpha \cdot 
        \overline{ \nabla H^\alpha( h|_{\partial D})}\,dx =0 , \qquad \forall \, h\in W.
    \end{equation}
    Let $w_\delta^\alpha =\delta^{-1} (1- V_\delta^\alpha|_D) \in W$. By \eqref{variaofV}, one has
    \begin{equation}\label{variawdelta}
        \begin{aligned}
            \int_D \nabla w_\delta^\alpha\cdot\overline {\nabla h}\,dx &+
            \delta \int_{Y\setminus D} \nabla H^\alpha(w_\delta^\alpha |_{\partial D})
            \cdot \overline{ \nabla H^\alpha( h|_{\partial D})}\,dx \\
            &=  \int_{Y\setminus D} \nabla V^\alpha\cdot \overline{ \nabla H^\alpha( h|_{\partial D})}\,dx, 
        \end{aligned}
        \qquad \forall \, h\in W.
    \end{equation}
Green’s identity gives
\begin{equation}
    \int_{Y\setminus D} \nabla H^\alpha g\cdot \overline{ \nabla H^\alpha \varphi }\,dx = \langle T^\alpha g,\varphi \rangle_{-\frac{1}{2},\frac{1}{2}}, \qquad \forall \,g,\varphi\in H^{1/2}(\partial D),
\end{equation}
where $\langle \cdot,\cdot \rangle_{-\frac{1}{2},\frac{1}{2}}$ denotes the $H^{-1/2}(\partial D)$-$H^{1/2}(\partial D)$ pair.
This, combined with \eqref{variawdelta}, implies that
\begin{equation}
    \begin{aligned}
        \langle A w_\delta^\alpha ,h \rangle_{W^*,W} = \int_D \nabla w_\delta^\alpha\cdot\overline {\nabla h}\,dx  &= \left\langle T^\alpha (1 - \delta \gamma w_\delta^\alpha), \gamma h \right\rangle_{-\frac{1}{2},\frac{1}{2}} \\
        &=\left\langle \gamma^*T^\alpha (1 - \delta \gamma w_\delta^\alpha), h \right\rangle_{W^*,W}.
    \end{aligned}
\end{equation}
This verifies \eqref{formulaVope}. Hence, the proof is complete.
\end{proof}

A direct corollary of Lemma \ref{holoofV} is the second-order expansion of $\widehat{C}^\alpha_\delta$, stated as follows.

\begin{proposition}\label{lem:analyticity-capacitance}
The map $\mathcal{B}_1 \ni \alpha \mapsto \widehat{C}^{\alpha}$ is an analytic function, and the map $\mathcal{B}_1 \ni \alpha \mapsto\widehat{C}_{\delta}^{\alpha}$ is uniformly analytic in $\delta \in (0,\delta_0)$. For every $\alpha \in \mathcal{B}_1$, let $w^\alpha \in H^1(D)$ be the weak solution of
\begin{equation}\label{correctorcapaci}
    \left\{
    \begin{aligned}
        &-\Delta w^\alpha  = \widehat{C}^\alpha &&\mathrm{in} \ D,\\
        &\left. \frac{\partial w^\alpha}{\partial \nu}\right|_- = \left. \frac{\partial V^\alpha}{\partial \nu}\right|_+ &&\mathrm{on} \ \partial D,\\
        &\int_D w^\alpha\,dx=0.
    \end{aligned}
    \right.
\end{equation} 
Then, the map $\mathcal{B}_1 \ni \alpha \mapsto w^\alpha$ is an analytic $H^1(D)$-valued function. Moreover, for every $\gamma\in\mathbb{N}^d$, 
    \begin{equation}\label{expansionofChat}
        \widehat{C}_{\delta}^{\alpha} =\widehat{C}^{\alpha} - \delta\fint_D|\nabla w^\alpha|^2\,dx + \mathcal{O}_{C^\gamma (\mathcal{B}_1)} (\delta^2) .
    \end{equation}
\end{proposition}

\begin{remark}
    The analyticity follows from the variational definition of the quasiperiodic capacitance matrix. In comparison, the layer-potential definition was used in \cite{cbms}, and only the analyticity over $\mathcal{B}_1\setminus \{0\}$ was proved.
\end{remark}

\begin{remark}
    The exponential decay of the off-diagonal terms of the capacitance operator $\mathfrak{C}$ follows from the analyticity of $\widehat{C}^\alpha$ over $\mathcal{B}_1$, as $\widehat{C}^\alpha$ is the symbol of $\mathfrak{C}$.
\end{remark}

\begin{proof}
    Since $\widehat{C}^{\alpha}$ and $\widehat{C}^{\alpha}_\delta$ are integrals of $V^\alpha$ and $V^\alpha_\delta$, respectively,  
    Lemma \ref{holoofV} implies that $\widehat{C}^{\alpha}$ is analytic and $\widehat{C}^{\alpha}_\delta$ is uniformly analytic.
    Using the definition \eqref{correctorcapaci} of $w^\alpha$ and integration by parts, one gets
    \begin{equation}
        w^\alpha = - A^{-1} \gamma^* T^\alpha(1).
    \end{equation}
    Therefore, $w^\alpha$ is also analytic. Moreover, by \eqref{uniformholdef} and the Cauchy's formula, for every $\gamma\in\mathbb{N}^d$, one has
    \begin{equation}
        V^\alpha_\delta = 
        \left\{ \begin{aligned}
            & 1 +  \delta w^\alpha  \\
            & V^\alpha + \delta H^\alpha (w^\alpha|_{\partial D})
        \end{aligned}\right.  + \mathcal{O}_{C^\gamma(\mathcal{B}_1;H^1(Y))}(\delta^2).
    \end{equation}
    Therefore,
    \begin{equation}\label{expanC1}
        \begin{aligned}
            |D| \widehat{C}^\alpha_\delta &= \delta \int_D |\nabla w^\alpha|^2\,dx + \int_{Y\setminus D}
            |\nabla V^\alpha|^2\,dx \\
            & \quad + 2 \delta \mathrm{Re}\, \int_{Y\setminus D} \nabla V^\alpha \cdot \overline{ \nabla H^\alpha  (w^\alpha|_{\partial D})} \,dx + \mathcal{O}_{C^\gamma(\mathcal{B}_1)}(\delta^2).
        \end{aligned}
    \end{equation}

Test \eqref{correctorcapaci} with $\overline{w^\alpha}$ yields
\begin{equation}\label{expanC2}
    \int_D |\nabla w^\alpha|^2\,dx = - \int_{Y\setminus D} \nabla V^\alpha \cdot \overline{ \nabla H^\alpha  (w^\alpha|_{\partial D})} \,dx.
\end{equation}
The desired conclusion follows from \eqref{expanC1} and \eqref{expanC2}.
\end{proof}

\begin{lemma}\label{maximumCalphalem}
    Under the symmetry condition \eqref{symcond} on $D$, for every $\delta>0$, we have
\begin{equation}\label{eq:capacity-maximum-M}
    \widehat{C}_\delta^\alpha < \widehat{C}_\delta^M, \qquad \forall\,\alpha \in \mathcal{B}_1 \setminus \{M\}.
\end{equation}
Moreover,
\begin{equation}\label{eq:capacity-maximum-M-0}
    \widehat{C}^\alpha < \widehat{C}^M, \qquad \forall\,\alpha \in \mathcal{B}_1 \setminus \{M\}.
\end{equation}
\end{lemma}

\begin{proof}
We only prove \eqref{eq:capacity-maximum-M}, as \eqref{eq:capacity-maximum-M-0} can be proved with the same argument. Let $V_\delta^M$ be a minimizer that defines $\widehat{C}_\delta^M$. For each $j\in\{1,\ldots,d\}$, the reflection
\begin{equation}
   R_j: (x_1,\ldots,x_j,\ldots,x_d) \mapsto (x_1,\ldots,-x_j,\ldots,x_d)
\end{equation}
preserves $Y$, $D$, the coefficient $a_\delta$, and the constraint that the average over $D$ is equal to one. Moreover, reflection changes the $j$-th component of the quasiperiodicity vector from $\pi$ to $-\pi$, and these two components are identical modulo $2\pi$. Therefore, the function $V^M_\delta \circ R_j$ is also a minimizer that defines $\widehat{C}_\delta^M$. However, one can easily show that the minimizer is unique. Indeed, the difference of two minimizers has zero weighted Dirichlet energy and is therefore constant, while a nonzero constant cannot satisfy the antiperiodic condition. It follows that $V_\delta^M =V^M_\delta \circ R_j $ for every $j$. This, combined with the antiperiodic condition, implies that $V_\delta^M \in H^1_0(Y)$.
Because $H_0^1(Y)\subset H_\alpha^1(Y)$ for any $\alpha \in \mathcal{B}_1$, we have
\begin{equation}
        \widehat{C}_\delta^\alpha=\frac{1}{|D|} \min_{\substack{u\in H_\alpha^1(Y)\\ \fint_Du\,dx=1}} \int_Ya_\delta|\nabla u|^2\,dx
        \leq \frac{1}{|D|} \int_Y a_\delta|\nabla V^M_\delta|^2\,dx =  \widehat{C}_\delta^M.
\end{equation}

Since $\overline{V^M_\delta}$ also minimizes the weighted energy, the uniqueness implies that $V^M_\delta$ is real-valued. Moreover, by the variational definition of $V^M_\delta$, we have
\begin{equation}\left\{
\begin{aligned}
    & -\nabla \cdot (a_\delta \nabla V^M_\delta) =\widehat{C}_\delta^M \mathbbm{1}_D \geq 0 && \mathrm{in} \ Y, \\
    & V^M_\delta =0 && \mathrm{on} \ \partial Y.
\end{aligned}\right.
\end{equation}
By the maximum principle, $V^M_\delta >0$ in $Y$. The Hopf lemma therefore gives $\frac{\partial V^M_\delta}{\partial \nu}<0$ on $\partial Y$. 

Now, we assume that $ \widehat{C}_\delta^{\alpha_0}= \widehat{C}_\delta^M$ for some $\alpha_0 \in \mathcal{B}_1$. Due to $V_\delta^M \in H^1_0(Y)$, the variational definition of $V^{\alpha_0}_\delta$ gives
\begin{equation}
    \int_Y a_\delta \nabla V^{\alpha_0}_\delta \cdot \overline{\nabla (V^{\alpha_0}_\delta - V^M_\delta)}\,dx=0.
\end{equation}
Consequently,
\begin{equation}
    0= |D| ( \widehat{C}_\delta^M - \widehat{C}_\delta^{\alpha_0} )=\int_Y a_\delta |\nabla V^{\alpha_0}_\delta - \nabla V^M_\delta|^2\,dx .
\end{equation}
Because $V^{\alpha_0}_\delta$ and $V^M_\delta$ have the same average on $D$, it follows that $V^{\alpha_0}_\delta = V^M_\delta$. Write 
\begin{equation}
    F^\pm_j := \left\{x \in \partial Y: x_j = \pm \frac{1}{2} \right\}, \qquad j=1,\cdots,d.
\end{equation}
By the $\alpha_0$-quasiperiodicity of $V^{\alpha_0}_\delta$, the identity $V^{\alpha_0}_\delta = V^M_\delta$, and the reflection symmetry $V_\delta^M =V^M_\delta \circ R_j $, one has
\begin{equation}
    \left. \frac{\partial V^M_\delta}{\partial \nu} \right|_{F_j^-}=\left. \frac{\partial V^M_\delta}{\partial \nu} \right|_{F_j^+} = - e^{\mi (\alpha_0)_j} \left. \frac{\partial V^M_\delta}{\partial \nu} \right|_{F_j^-} , \qquad j=1,\cdots,d.
\end{equation}
This, combined with $\frac{\partial V^M_\delta}{\partial \nu}<0$ on $\partial Y$, implies $\alpha_0=M$. The proof is complete.
\end{proof}

\begin{lemma}\label{lem:strict-capacity-hessian}
    Assume that the symmetry condition \eqref{symcond} for $D$ holds. There exists a constant $C>0$ such that
    \begin{equation}\label{nondegeneracy_capacity}
        \xi\cdot\nabla_{\alpha}^{2}\widehat{C}^{M}\xi\leq -C|\xi|^{2}, \qquad \xi\in\mathbb{R}^{d}.
    \end{equation}
    Moreover, there exists $\delta_0>0$ such that for every $ 0<\delta < \delta_0$, we have
    \begin{equation}
    \label{nondegeneracy_band}
        \xi\cdot\nabla_{\alpha}^{2}\widehat{C}_\delta^{M}\xi\leq -C|\xi|^{2}, \qquad \xi\in\mathbb{R}^{d}.
    \end{equation}
\end{lemma}

\begin{proof}
We only prove \eqref{nondegeneracy_capacity}, as the proof of \eqref{nondegeneracy_band} is similar.
    It was proved in \cite[Proposition 4.2 and Lemma 4.5]{ammari2019bloch} that $\widehat{C}^{\alpha}$ attains its maximum at $M$ and that its Hessian at $M$ is negative semidefinite. We therefore only need to show that the Hessian is nondegenerate.
    Denote $X:= \{h \in H^1_M(Y): h|_D \equiv 0\}$.

    By the proof of Lemma \ref{maximumCalphalem}, $V^M_\delta \in H^1_0(Y)$ for every $\delta>0$. By Lemma \ref{holoofV}, $V^M \in H^1_0(Y)$. Fix $\xi\in\mathbb{R}^{d}$. For sufficiently small $t$, the minimizer $V^{M+t\xi}$ can be written uniquely in the form
    \begin{equation}
        V^{M+t\xi}=V^{M}+e^{\mi t\xi\cdot x}z_{t},\qquad z_{t}\in X.
    \end{equation}
    Here $z_0=0$. By the variational definition of $V^{M+t\xi}$, we get
    \begin{equation}\label{eqforzt}
        \int_{Y\setminus D}\nabla\left(V^{M} + e^{\mi t\xi\cdot x}z_{t}\right)\cdot \overline{\nabla(e^{\mi t\xi\cdot x}h)}\,dx=0, \qquad h\in X.
    \end{equation}
    Applying $\nabla (e^{\mi t\xi\cdot x} f) = e^{\mi t\xi\cdot x}(\nabla f+\mi t\xi f)$ for $f=z_t$ and $f=h$, respectively, this equation becomes
    \begin{equation}
        \int_{Y\setminus D}  (\nabla z_t+\mi t\xi z_t) \cdot \overline{(\nabla h+\mi t\xi h)} \,dx= - \int_{Y\setminus D} 
        \nabla V^M \cdot \overline{e^{\mi t\xi\cdot x}(\nabla h+\mi t\xi h)} \,dx,
        \qquad h\in X.
    \end{equation}
    The right-hand side is an analytic function in $t$, while the left-hand side can be viewed as a quadratic form on $X\times X$. It is easy to use the Poincar\'{e} inequality to show that this form is coercive near $t=0$. Therefore, by the analytic inverse theorem, or equivalently a Neumann-series argument, one obtains that $z_t$ depends analytically on $t$ in a neighborhood of $t=0$.
    
    Taking $h=z_{t}$ in the equation \eqref{eqforzt} gives the exact identity
    \begin{equation}
        |D|\widehat{C}^{M+t\xi}=|D|\widehat{C}^{M}-\int_{Y\setminus D}\big|\nabla(e^{\mi t\xi\cdot x}z_{t})\big|^{2}\,dx.
    \end{equation}
    Differentiating this identity twice at $t=0$ yields
    \begin{equation}
        \xi\cdot\nabla_{\alpha}^{2}\widehat{C}^{M}\xi=-\frac{2}{|D|}\int_{Y\setminus D}\big|\nabla z_{0}'\big|^{2}\,dx.
    \end{equation}
    By contradiction, suppose that there is a unit vector $\xi $ such that $\xi\cdot\nabla_{\alpha}^{2}\widehat{C}^{M}\xi=0$. Then, $\nabla z_{0}'=0$ in $Y\setminus\overline{D}$. Since $z_0'|_D\equiv 0$, it follows that $z_{0}'$ is identically zero in $Y$.

    Differentiating the equation \eqref{eqforzt} at $t=0$ now gives
    \begin{equation}
        0=\int_{Y\setminus\overline{D}}\nabla V^{M}\cdot\nabla\big((\xi\cdot x)h\big)\,dx,
        \qquad h\in X.
    \end{equation}
    Because $V^{M}$ is harmonic in $Y\setminus\overline{D}$ and $h=0$ on $\partial D$, integration by parts reduces this identity to an integral over $\partial Y$. On each pair of opposite faces, the outward normal derivatives of $V^{M}$ agree by reflection symmetry, whereas $h$ have opposite signs by antiperiodicity. Moreover, the values of $\xi\cdot x$ on the two opposite faces differ by $\xi_{j}$. Pairing the opposite faces therefore gives
    \begin{equation}\label{integralface}
        \sum_{i=1}^{d}\xi_{i}\int_{F_i}\frac{\partial V^{M}}{\partial\nu} h\,d\sigma=0,
        \qquad h\in X, 
    \end{equation}
    where $F_i:= \partial Y\cap \{x_{i}=1/2\}$. Fix an index $j$. We may choose $h\in X$ so that $h|_{F_j}$ is a nonnegative, nonzero smooth function compactly supported in the interior of this face, $h|_{-F_j} = - h|_{F_j}$, and its trace on all the remaining faces is zero. Since $V^{M}=1$ on $\partial D$ and $V^{M}=0$ on $\partial Y$, the strong maximum principle gives $0<V^{M}<1$ in $Y\setminus\overline{D}$. By Hopf's lemma, we get
    \begin{equation}
        \frac{\partial V^{M}}{\partial\nu}<0\qquad\text{on the interior of }F_1,\cdots, F_d.
    \end{equation}
    For the preceding choice of $h$, \eqref{integralface} becomes
    \begin{equation}
        \xi_j \int_{F_j} \frac{\partial V^{M}}{\partial\nu} h\,d\sigma =0.
    \end{equation}
    This implies that $\xi_{j}=0$. Since the same argument applies to all $j$, we conclude that $\xi=0$. This contradicts the assumption that $\xi$ is a unit vector. Hence, we have finished the proof.
\end{proof}

\subsection{Properties of the subwavelength band}\label{secsecorder}

In this section, we establish some useful properties of the subwavelength band. For $\alpha\in \mathcal{B}_1$, consider the Bloch eigenvalue problem
\begin{equation}\label{blocheqdelta}
\left\{
\begin{aligned}
\nabla\cdot\frac{1}{\rho}\nabla u+\frac{\omega^2}{\kappa}u&=0 &&\text{in }Y\setminus \overline{D},\\
\nabla\cdot\frac{1}{\rho_{\rm b}}\nabla u+\frac{\omega^2}{\kappa_{\rm b}}u&=0 &&\text{in }  D,\\
u|_+-u|_-&=0 &&\text{on }\partial D,\\
\left.\frac{1}{\rho}\frac{\partial u}{\partial \nu}\right|_+ - \left. \frac{1}{\rho_{\rm b}}\frac{\partial u}{\partial \nu} \right|_- &=0
&&\text{on }\partial D,\\
e^{-\mi \alpha\cdot x}u
&\text{ is periodic}.
\end{aligned}
\right.
\end{equation}

Recall that $a_\delta= \delta^{-1} \mathbbm{1}_{D}+\mathbbm{1}_{Y\setminus D}$. Let
\begin{equation}
     \rho_\delta := \mathbbm{1}_D + \delta \left( \frac{v_{\rm b}^2}{v^2} \right) \mathbbm{1}_{Y\setminus D}.
\end{equation}
Define the unbounded differential operator on $L^2(Y)$:
\begin{equation}
    \mathcal{L}_\delta^\alpha := - \frac{1}{\rho_\delta} \nabla  \cdot a_\delta \nabla , \qquad D(\mathcal{L}_\delta^\alpha) :=\{ u\in H_\alpha^1(Y): \mathcal{L}_\delta^\alpha u \in L^2(Y) \}.
\end{equation}
Then, \eqref{blocheqdelta} is equivalent to
\begin{equation}
    \mathcal{L}_\delta^\alpha u = \lambda u, \qquad \lambda = \frac{\omega^2}{\delta v_{\rm b}^2}.
\end{equation}
Let the eigenvalues of $\mathcal{L}_\delta^\alpha$ be
\begin{equation}
    0\leq \lambda_{1,\delta}^\alpha \leq \lambda_{2,\delta}^\alpha \leq \cdots.
\end{equation}
By the min-max principle, 
\begin{equation}\label{min_max_charac}
    \lambda_{1,\delta}^\alpha = \inf_{0\neq v\in H_\alpha^1(Y)} \frac{ \int_Y a_\delta |\nabla v|^2 \,dx }{ \int_Y \rho_\delta |v|^2\,dx } , \qquad \lambda_{2,\delta}^\alpha = \inf_{ F \subset H^1_\alpha(Y) \atop \dim F=2} \sup_{0\neq v\in F} \frac{ \int_Y a_\delta |\nabla v|^2 \,dx }{ \int_Y \rho_\delta |v|^2\,dx } .
\end{equation}

We introduce the space
\begin{equation}
    Z_\alpha:=\left\{ \varphi\in H_{\alpha}^1(Y): \int_D \varphi\,dx=0 \right\}.
\end{equation}
The Poincar\'e inequality gives
\begin{equation}\label{eq:zero-mean-mass-estimate}
    \int_Y \rho_\delta |\varphi|^2\,dx\leq C\delta \int_Y a_\delta |\nabla \varphi|^2 \,dx , \qquad \forall\,\varphi\in H^1(Y) \mathrm{ \ with \ }\int_D \varphi \,dx=0.
\end{equation}

\begin{lemma}\label{lemsimple}
There exist $C_1=C_1(D,Y)$ and $C_2=C_2(D,Y)$ such that
\begin{equation}
    \sup_{\alpha \in \mathcal{B}_1 \atop \delta \in (0,1) } \lambda_{1,\delta}^\alpha \leq C_1 \qquad \mathrm{and} \qquad \inf_{\alpha \in \mathcal{B}_1}
    \lambda_{2,\delta}^\alpha \geq \frac{C_2}{\delta}, \quad \forall\, \delta \in (0,1).
\end{equation}
Moreover, any eigenfunction of $\lambda_{1,\delta}^\alpha$ has nonzero mean over $D$.
\end{lemma}

\begin{proof}
    Let $v = V^\alpha_\delta$ in the characterization of $\lambda_{1,\delta}^\alpha$ in \eqref{min_max_charac}, since $\int_Y a_\delta|\nabla V_\delta^\alpha|^2\,dx = |D|\widehat{C}_\delta^\alpha$,
    \begin{equation}
        \lambda_{1,\delta}^\alpha \leq \frac{ \widehat{C}_\delta^\alpha }{\fint_D |V^\alpha_\delta|^2\,dx } .
    \end{equation}
   By the results of Section~\ref{secpropofcapa}, $\lambda_{1,\delta}^\alpha $ is uniformly bounded in $\alpha\in \mathcal{B}_1 $ and $\delta\in (0,1)$. On the other hand, since any two dimensional subspace of $H^1_\alpha(Y)$ contains a nonzero element of $Z_\alpha$, the characterization of $\lambda_{2,\delta}^\alpha$ in \eqref{min_max_charac} and the estimate \eqref{eq:zero-mean-mass-estimate} yield
    \begin{equation}
        \inf_{\alpha \in \mathcal{B}_1}
    \lambda_{2,\delta}^\alpha \geq \frac{C}{\delta}, \qquad \forall\,\delta \in (0,1).
    \end{equation}
    
    Finally, the estimate \eqref{eq:zero-mean-mass-estimate} and the uniform boundedness of $\lambda_{1,\delta}^\alpha$ imply that any eigenfunction of $\lambda_{1,\delta}^\alpha$ has nonzero mean over $D$. 
\end{proof}

By Lemma \ref{lemsimple}, there is $\delta_0 = \delta_0(D,Y)>0$ such that, for $\delta \in (0,\delta_0)$, $\lambda_{1,\delta}^\alpha$
is simple and isolated. We denote by $u_{1,\delta}^\alpha$ the unique Bloch eigenfunction of $\lambda^\alpha_{1,\delta}$ such that $\fint_D u_{1,\delta}^\alpha\,dx=1$. 

On $L^2(Y)$, we define the contrast-weighted inner product $\langle \cdot,\cdot \rangle_{L^2_\delta(Y)}$ by
\begin{equation}
    \langle u,v \rangle_{L^2_\delta(Y)} := \int_Y \rho_\delta  u\overline{v}\,dx.
\end{equation}
We denote by $L^2_\delta(Y)$ the resulting Hilbert space. For every $z\in \mathbb{C}^d$, let $U_z$ be the gauge transform defined in 
\eqref{gaugetran}. For any $u,v\in H^1(Y)$, define the twisted form 
\begin{equation}
\mathfrak{a}^z_\delta (u,v):= \int_Y a_\delta \nabla (U_z u) \cdot \overline{\nabla (U_{\overline{z}} v)}\,dx = \int_Y a_\delta \{ \nabla u + \mi Q_z u  \}\cdot \{ \overline{\nabla v} - \mi Q_z\overline{v}  \} \,dx.
\end{equation}
Let $H^1_{\mathrm{per}}(Y)$ be the space of periodic functions in $H^1(Y)$, we define the associated twisted operator $L^z_\delta: D(L^z_\delta)\subset L^2(Y)\rightarrow L^2(Y)$ by
\begin{equation}
\begin{aligned}
    &D(L_\delta^z) :=\{u\in H^1_{\mathrm{per}}(Y): \text{there exists }f\in L^2(Y) \text{ such that }  \\
    &\qquad\qquad\qquad\qquad\qquad \mathfrak{a}^z_\delta (u,v)=\langle f,v\rangle_{L^2_\delta(Y)} \text{ for any }v\in H^1_{\mathrm{per}}(Y)\}
\end{aligned}
\end{equation}
and $L^z_\delta  u=f$. The differential form is
\begin{equation}
    L_\delta^z = - \frac{1}{\rho_\delta} (\nabla +\mi Q_z) \cdot a_\delta (\nabla +\mi Q_z),  \qquad D( L_\delta^z)=\{ u\in H^1_{\mathrm{per}}(Y): L_\delta^z u \in L^2(Y) \}.
\end{equation}

\begin{lemma}
    As bounded operators from $L^2_\delta(Y)$ to $H^1_{\mathrm{per}}(Y)$, $(L^z_\delta +1)^{-1}$ is uniformly holomorphic in $\delta \in (0,\delta_0)$ over a fixed neighborhood of $\mathcal{B}_1$. 
\end{lemma}
\begin{proof}
Fix a reference point $\alpha_* \in \mathcal{B}_1$. For $u,v\in H^1(Y)$, using the fact that $Q_z =0$ near $\overline{D}$, one has
\begin{equation}\label{azstar1}
    |\mathfrak{a}^z_\delta (u,v) - \mathfrak{a}^{\alpha_*}_\delta (u,v) | \leq C (|z-\alpha_*|+|z-\alpha_*|^2) \|u \|_{H^1(Y)} \|v\|_{H^1(Y)}.
\end{equation}
By the definition of $\mathfrak{a}^{\alpha_*}_\delta$, we have
\begin{equation}\label{azstar2}
    \begin{aligned}
        \mathfrak{a}^{\alpha_*}_\delta (u,u) + \| u \|_{L^2_\delta(Y)}^2 & \geq \int_Y |\nabla U_{\alpha_*} u|^2\,dx + \int_D |u|^2\,dx
        \\
        &= \int_Y |\nabla U_{\alpha_*} u|^2\,dx + \int_D | U_{\alpha_*}u|^2\,dx \\
        &\geq C\| U_{\alpha_*}u \|_{H^1(Y)}^2 \\
        &\geq C\|u\|_{H^1(Y)}^2,
    \end{aligned}
\end{equation}
where we used the following Poincar\'{e} inequality for $w= U_{\alpha_*}u$:
\begin{equation}
    \| w \|_{H^1(Y)}^2 \leq C \int_Y |\nabla w|^2\,dx + C\int_D |w|^2\,dx, \qquad \forall\,w\in H^1(Y).
\end{equation}

For $f\in L^2(Y)$, we inductively define $u_n$ by
\begin{equation}\label{iterativedef}
    \left\{
    \begin{aligned}
        &u_0 = (L^{\alpha_*}_\delta +1)^{-1}f , \\
        & \mathfrak{a}_\delta^{\alpha_*} (u_{n+1},v) + \langle u_{n+1},v\rangle_{L^2_\delta(Y)} = \mathfrak{a}_\delta^{\alpha_*} (u_{n},v)
        - \mathfrak{a}_\delta^{z} (u_{n},v), \qquad \forall v\in H^1_{\mathrm{per}}(Y),\,n\in \mathbb{N}.
    \end{aligned}
    \right.
\end{equation}
By estimates \eqref{azstar1}, \eqref{azstar2}, and the Lax-Milgram, we get
\begin{equation}
    \| u_0 \|_{H^1(Y)} \leq C\|f\|_{L^2_\delta(Y)}, 
\end{equation}
and
\begin{equation}
     \| u_{n+1} \|_{H^1(Y)} \leq C(|z-\alpha_*|+|z-\alpha_*|^2 )\|u_{n} \|_{H^1(Y)} , \qquad \forall\,n\in \mathbb{N}.
\end{equation}
Choose $r>0$ such that $C (r + r^2) < 1/2$. For every $z\in \mathbb{C}^d$ such that $|z-\alpha_*|<r$, one gets
\begin{equation}
    \left\| u:=\sum_{n=0}^\infty u_n \right\|_{H^1(Y)} \leq 2C \|f \|_{L^2_\delta(Y)} .
\end{equation}
Moreover, by \eqref{iterativedef}, it is easy to compute that 
\begin{equation}
    \mathfrak{a}_\delta^z (u,v) + \langle u,v \rangle_{L^2_\delta(Y)} = \langle f,v\rangle_{L^2_\delta(Y)}, \qquad \forall \,v\in H^1_{\mathrm{per}}(Y).
\end{equation}
Thus, $u = (L^z_\delta +1)^{-1} f$ is approximated by a locally uniformly convergent sequence of holomorphic functions 
with values in $\mathcal{L}(L_\delta^2(Y),H^1_{\mathrm{per}}(Y))$. Consequently, $(L^z_\delta +1)^{-1} $ is uniformly holomorphic in $\delta \in (0,\delta_0)$ over $|z-\alpha_*|<r$. Since $\alpha_*$ is arbitrary, $(L^z_\delta +1)^{-1}$ is uniformly holomorphic in $\delta \in (0,\delta_0)$ over a fixed neighborhood of $\mathcal{B}_1$.
\end{proof}

\begin{lemma}\label{holoofueigenfunction}
   The map $\{ \mathcal{B}_1\ni \alpha \mapsto u^\alpha_{1,\delta} \}_{\delta \in (0,\delta_0)}$ is $H^1(Y)$-valued uniformly analytic.
\end{lemma}

\begin{proof}
Fix $\alpha_* \in \mathcal{B}_1$. Choose a positively oriented circle $\Gamma$ enclosing $\lambda^{\alpha_*}_{1,\delta}$ and no other eigenvalue of $L^{\alpha_*}_\delta$. Let $\zeta \in \Gamma$. Since
\begin{equation}
    (\zeta+1) (L^{\alpha_*}_\delta + 1)^{-1} -1 = (\zeta - L^{\alpha_*}_\delta) (L^{\alpha_*}_\delta + 1 )^{-1} ,
\end{equation}
$(\zeta+1) (L^{\alpha_*}_\delta + 1)^{-1} -1 $ is invertible and the inverse is uniformly bounded in $\zeta \in \Gamma$. Since $(L^z_\delta + 1 )^{-1}$ is uniformly holomorphic, there exists $r>0$, independent of $\delta$, such that $(\zeta+ 1) (L^z_\delta + 1)^{-1} -1$ is invertible for $|z-\alpha_*|<r$, and the inverse is uniformly holomorphic in $\delta $ over $|z-\alpha_*|<r$.

By a direct computation, one has
\begin{equation}
    (\zeta - L^z_\delta)^{-1} = (L^z_\delta + 1)^{-1} \big\{ (\zeta + 1) (L^z_\delta + 1)^{-1} -1 \big\}^{-1}, \qquad \forall \, \zeta \in \Gamma, \,|z-\alpha_*|<r.
\end{equation}
  Consequently, $z\mapsto (\zeta - L^z_\delta)^{-1} $ is uniformly holomorphic in $(\zeta,\delta)\in \Gamma \times (0,\delta_0)$ over $|z-\alpha_*|<r$. Hence, the Riesz projection
\begin{equation}
    P_\delta(z) := \frac{1}{2\pi \mi}\int_{\Gamma} (\zeta - L^z_\delta)^{-1} \,d\zeta
\end{equation}
is uniformly holomorphic in $\delta \in (0,\delta_0)$ over $|z-\alpha_*|<r$. Since $\alpha_*$ is arbitrary, $P_\delta(z)$ is uniformly holomorphic in $\delta \in (0,\delta_0)$ over a fixed neighborhood of $\mathcal{B}_1$. Fix $\alpha_* \in \mathcal{B}_1$, define
\begin{equation}
    p^z_\delta := \frac{P_\delta (z) (U_{-\alpha_*} u^{\alpha_*}_{1,\delta} ) }{ \fint_D P_\delta (z) (U_{-\alpha_*} u^{\alpha_*}_{1,\delta} ) \,dx }.
\end{equation}
Then, due to the choice of the gauge $U_\alpha = I$ in $D$, $U_z p^z_\delta$ is a uniformly holomorphic extension of $u^\alpha_{1,\delta}$. The proof is complete.
\end{proof}

The following results hold. 

\begin{proposition}\label{propexpansionofband}
    For every $\gamma \in \mathbb{N}^d$, we have
\begin{equation}\label{eq:omega-square-expansion}
    \lambda_{1,\delta}^\alpha = \widehat{C}^{\alpha} -
    \delta \left( \fint_D |\nabla w^{\alpha}|^2\,dx + \frac{v_{\rm b}^2\widehat{C}^{\alpha}}{v^2|D|}
    \int_{Y\setminus\overline D}|V^{\alpha}|^2\,dx
    \right) + \mathcal{O}_{C^\gamma(\mathcal{B}_1)}(\delta^2).
\end{equation}
In particular, under the scaling $\delta=\varepsilon^2$, the upper edge $\omega_{*,\varepsilon}$ of the first Bloch band for the structure of period $\varepsilon$ satisfies
\begin{equation}\label{eq:omega-edge-expansion}
    \omega_{*,\varepsilon} = v_{\rm b}\sqrt{\widehat{C}^{M}} - \frac{v_{\rm b}\varepsilon^2}{2\sqrt{\widehat{C}^{M}}}
    \left(  \fint_D |\nabla w^{M}|^2\,dx + \frac{v_{\rm b}^2\widehat{C}^{M}}{v^2|D|} \int_{Y\setminus\overline D}|V^{M}|^2\,dx
    \right) + O(\varepsilon^4).
\end{equation}
\end{proposition}

\begin{proof}
We only prove \eqref{eq:omega-square-expansion}, which easily implies \eqref{eq:omega-edge-expansion}. The proof is divided into several steps. Let 
\begin{equation}
    r_\delta^z := u_{1,\delta}^z -V_\delta^z.
\end{equation}
Let $p_\delta^z = U_{-z} u_{1,\delta}^z $, $W_\delta^z = U_{-z} V_\delta^z$, and $s_\delta^z = U_{-z} r_\delta^z$. Then,
\begin{equation}
    s_\delta^z = p_\delta^z  - W_\delta^z.
\end{equation}
By Lemma \ref{holoofV} and Lemma \ref{holoofueigenfunction}, there exists a neighborhood $O\subset \mathbb{C}^d$ of $\mathcal{B}_1$, such that $s_\delta^z$ is uniformly holomorphic in $\delta\in (0,\delta_0) $ over $O$. 
The eigenvalue equation $L^z_\delta p_\delta^z  = \lambda_{1,\delta}^z p_\delta^z$ is equivalent to
\begin{equation}\label{testvaz}
    \mathfrak{a}^0_\delta (u_{1,\delta}^z  , U_{\overline{z}}v) =\mathfrak{a}^z_\delta (p_\delta^z , v) = \lambda_{1,\delta}^z  \langle p_\delta^z , v\rangle_{L^2_\delta(Y)}, \qquad \forall \, v \in H^1_{\mathrm{per}}(Y) .
\end{equation}
Let $v=s_\delta^z$, since $\int_D s_\delta^z\,dx=0$, we get
\begin{equation}\label{eqvaaz}
    \mathfrak{a}^z_\delta (s_\delta^z,s_\delta^z) =\lambda_{1,\delta}^z  \| s_\delta^z \|^2_{L^2_\delta(Y)} + 
    \lambda_{1,\delta}^z  \langle W_\delta^z , s_\delta^z\rangle_{L^2_\delta(Y)} .
\end{equation}

Fix $\alpha_* \in \mathcal{B}_1$. \eqref{azstar2} gives
\begin{equation}\label{3100z}
    \|s_\delta^z\|_{H^1(Y)}^2 \leq C \mathfrak a_\delta^{\alpha_*} (s_\delta^z,s_\delta^z) + C\|s_\delta^z\|_{L^2_\delta(Y)}^2.
\end{equation}
Applying \eqref{eq:zero-mean-mass-estimate} to $U_{\alpha_*} s_\delta^z$ yields
\begin{equation}\label{399sz}
   \| s_\delta^z \|^2_{L^2_\delta(Y)} \leq C\delta \mathfrak{a}^{\alpha_*}_\delta (s_\delta^z,s_\delta^z) .
\end{equation}
Using \eqref{399sz} and \eqref{3100z}, we obtain
\begin{equation}
    \|s_\delta^z\|_{H^1(Y)}^2 \leq C\mathfrak a_\delta^{\alpha_*} (s_\delta^z,s_\delta^z) .
\end{equation}

By \eqref{azstar1},
\begin{equation}
    \begin{aligned}
        \operatorname{Re}\mathfrak a_\delta^z(s_\delta^z,s_\delta^z) \ge \mathfrak a_\delta^{\alpha_*} (s_\delta^z,s_\delta^z) - C\bigl(|z-\alpha_*|+|z-\alpha_*|^2\bigr) \|s_\delta^z\|_{H^1(Y)}^2 .
    \end{aligned}
\end{equation}
Since $|\lambda_{1,\delta}^z|$ is uniformly bounded, using \eqref{399sz}, there exist $\delta_0>0$ and $r>0$ independent of $\delta$ such that
\begin{equation}\label{regeqsz}
    \operatorname{Re}\left\{ \mathfrak a_\delta^z(s_\delta^z,s_\delta^z) - \lambda_{1,\delta}^z\|s_\delta^z\|_{L^2_\delta(Y)}^2 \right\} \geq C \|s_\delta^z\|_{H^1(Y)}^2 , \qquad \forall \, \delta \in (0,\delta_0),\, |z-\alpha_*|<r.
\end{equation}

Next, since $\int_D s^z_\delta \,dx=0$, we get
\begin{equation}\label{estiVsnorm1}
    \begin{aligned}
       \big| \langle W_\delta^z,s_\delta^z\rangle_{L^2_\delta(Y)} \big|& \leq \int_D |W_\delta^z-1| |s_\delta^z|\,dx+ C\delta  \int_{Y\setminus\overline D} |W_\delta^z| |s_\delta^z| \,dx \\
       & \leq C\delta\|s_\delta^z\|_{H^1(Y)},
    \end{aligned}
\end{equation}
where we used $W_\delta^z = V_\delta^z$ in $D$, $|W_\delta^z |$ is comparable to $|V_\delta^z|$ in $Y\setminus D$, and
\begin{equation}
    \left\|V_\delta^z-1\right\|_{L^2(D)}\le C\delta, \qquad\left\|V_\delta^z\right\|_{L^2(Y\setminus\overline D)} \le C,
\end{equation}
which follows from formula \eqref{formulaofVdelta}.

Using \eqref{regeqsz} and \eqref{estiVsnorm1}, taking real parts in \eqref{eqvaaz}, we obtain
\begin{equation}
    \begin{aligned}
        \|s_\delta^z\|_{H^1(Y)}^2 &\leq C\operatorname{Re}\left\{ \mathfrak a_\delta^z(s_\delta^z,s_\delta^z) - \lambda_{1,\delta}^z\|s_\delta^z\|_{L^2_\delta(Y)}^2 \right\}  \\
        &\leq  C\big|\langle W_\delta^z , s_\delta^z\rangle_{L^2_\delta(Y)} \big| \\
        &\le C\delta\|s_\delta^z\|_{H^1(Y)}.
    \end{aligned}
\end{equation}
Hence, 
\begin{equation}\label{szh1norm}
    \|s_\delta^z\|_{H^1(Y)} \le C\delta.
\end{equation}
By the same argument as in \eqref{estiVsnorm1}, we get
\begin{equation}
     \begin{aligned}
       \big| \langle W_\delta^z,s_\delta^{\overline{z}}\rangle_{L^2_\delta(Y)} \big|& \leq \int_D |W_\delta^z-1| |s_\delta^{\overline{z}}|\,dx+ C\delta  \int_{Y\setminus\overline D} |W_\delta^z| |s_\delta^{\overline{z}}| \,dx \\
       & \leq C\delta\|s_\delta^{\overline{z}}\|_{H^1(Y)} \\
       &\leq C\delta^2,
    \end{aligned}
\end{equation}
where we used \eqref{szh1norm} with $z$ replaced by $\overline{z}$. Since $\langle W_\delta^z,s_\delta^{\overline{z}}\rangle_{L^2_\delta(Y)}$ is uniformly holomorphic, by Cauchy's formula, one gets
\begin{equation}\label{finalestimatealpha}
    \langle W_\delta^\alpha,s_\delta^{\alpha}\rangle_{L^2_\delta(Y)} = \mathcal{O}_{C^\gamma(\mathcal{B}_1)}(\delta^2).
\end{equation}

Let $z=\alpha \in \mathcal{B}_1$ and $v=W^\alpha_\delta$ in \eqref{testvaz}, we get
\begin{equation}
    |D| \widehat{C}^\alpha_\delta = \lambda^{\alpha}_{1,\delta} \left( \|V^\alpha_\delta\|_{L^2_\delta(Y)}^2 + \langle W_\delta^\alpha,s_\delta^\alpha\rangle_{L^2_\delta(Y)}  \right).
\end{equation}
Finally, the desired conclusion follows from Lemma \ref{holoofV}, Proposition \ref{lem:analyticity-capacitance}, and \eqref{finalestimatealpha}.
\end{proof}

\begin{lemma}\label{lem:maximum-capacity-M}
For sufficiently small $\delta>0$, let $\omega_{\mathrm{sub},\delta}^\alpha$ denote the first Bloch eigenfrequency for the unit cell with contrast $\delta$. We also have
\begin{equation}\label{eq:subeigen-maximum-M}
    \omega_{\mathrm{sub},\delta}^\alpha \leq \omega_{\mathrm{sub},\delta}^M.
\end{equation}
\end{lemma}

\begin{proof}
    The proof is by a symmetry argument analogous to the proof of Lemma \ref{maximumCalphalem}. We hence omit the proof.
\end{proof}

\begin{lemma}
    Assume that the symmetric condition \eqref{symcond} for $D$ holds. For sufficiently small $\delta >0$, let $\omega_{\mathrm{sub},\delta}^\alpha$ denote the first Bloch eigenfrequency for the unit cell with contrast $\delta$. There exists a constant $C>0$ such that
    \begin{equation}\label{nondegeneracy_band1}
        \xi\cdot\nabla_{\alpha}^{2}\omega_{\mathrm{sub},\delta}^M \xi\leq -C \sqrt{\delta} |\xi|^{2}, \qquad \xi\in\mathbb{R}^{d}.
    \end{equation}
\end{lemma}

\section{Reformulation of the scattering problem}\label{secreformulation}

\subsection{Reduction to the block system} 

In this section, we reduce the scattering problem \eqref{maineq} to a $2\times 2$ block system. We fix $R>0$ sufficiently large so that $\Omega $ is compactly contained in $B_R$.

For $k>0$, let $T_k:H^{1/2}(\partial B_R)\rightarrow H^{-1/2}(\partial B_R)$ be the exterior outgoing Dirichlet-to-Neumann operator on $\partial B_R$. More precisely, for any $\phi \in H^{1/2}(\partial B_R)$,
\begin{equation}
    T_k(\phi) :=\left.\frac{\partial u}{\partial \nu}\right|_{\partial B_R+},
\end{equation}
where $u \in H^1_{\mathrm{loc}}(\mathbb{R}^d\setminus \overline{B_R})$ is the weak solution of the scattering problem
\begin{equation}
    \left\{
    \begin{aligned}
        & \Delta u  + k^2 u =0  && \mathrm{in} \ \mathbb{R}^d\setminus \overline{B_R}, \\
        & u|_+ = \phi &&\mathrm{on} \ \partial B_R, \\
        & u\in \mathrm{SRC}(k).
    \end{aligned}
    \right.
\end{equation}

Define 
\begin{equation}
    a_\varepsilon :=\varepsilon^{-2} \mathbbm{1}_{D_\varepsilon} + \mathbbm{1}_{\mathbb{R}^d\setminus D_\varepsilon}, \qquad 
    b_\varepsilon := \varepsilon^{-2} \left(\frac{v}{v_{\rm b}}\right)^2 \mathbbm{1}_{D_\varepsilon} + \mathbbm{1}_{\mathbb{R}^d\setminus D_\varepsilon},
\end{equation}
and the outgoing scattering form $\mathcal{A}_\varepsilon(k)$ on $H^1(B_R)\times H^1(B_R)$ by
\begin{equation}\label{outgoingform}
    \mathcal{A}_\varepsilon(k)(u,v):= \int_{B_R} a_\varepsilon \nabla u \cdot \overline{\nabla v}\,dx -k^2\int_{B_R}
    b_\varepsilon u \overline{v}\,dx - \langle T_k u,v\rangle_{\partial B_R},
\end{equation}
where $\langle \cdot,\cdot\rangle_{\partial B_R}$ denotes the $H^{-1/2}(\partial B_R) -H^{1/2}(\partial B_R)$ pairing.

Thus, the solution $u_\varepsilon$ of \eqref{maineq}, restricting in $B_R$, is characterized by the variational form
\begin{equation}\label{maineqvariational}
    \mathcal{A}_\varepsilon(k_\varepsilon)(u_\varepsilon,v)= \left\langle \frac{\partial u^{\mathrm{in}}_\varepsilon }{\partial \nu} - T_{k_\varepsilon} u^{\mathrm{in}}_\varepsilon, v \right\rangle_{\partial B_R} , \qquad  \forall\,v\in H^1(B_R).
\end{equation}

Let $\widehat{u}_\varepsilon \in H^1_{\mathrm{loc}}(\mathbb{R}^d)$ denote the sound-soft total field generated by the same incident wave $u^{\mathrm{in}}_\varepsilon$, \emph{i.e.}, $\widehat{u}_\varepsilon$ is the weak solution of 
\begin{equation}\label{problimit1}
\left\{
\begin{aligned}
    & \Delta \widehat{u}_\varepsilon+k_\varepsilon^2 \widehat{u}_\varepsilon =0 && \mathrm{in} \ \mathbb{R}^d\setminus \overline{\Omega}, \\
    & \widehat{u}_\varepsilon \equiv 0,  && \mathrm{in} \ \Omega,\\
    & \widehat{u}_\varepsilon - u^{\mathrm{in}}_\varepsilon \in \mathrm{SRC}(k_\varepsilon).
\end{aligned}\right.
\end{equation}
Note that $\widehat{u}_\varepsilon$ is an intermediate comparison field, since it is the sound-soft total field at the $\varepsilon$-dependent wave number $k_\varepsilon$. By the difference $|k_\varepsilon-k_0|\leq C\varepsilon^2$ and a simple argument, one can easily show that, for any $R>0$, there exists $C=C(k_0,\Omega,R)>0$ such that 
\begin{equation}\label{difference_intermediate}
    \|\widehat{u}_\varepsilon-\widehat{u}_0\|_{H^1(B_R)} \leq C|k_\varepsilon-k_0|\leq C\varepsilon^2.
\end{equation}
Therefore, replacing the intermediate field $\widehat{u}_\varepsilon$ by the fixed limiting field $\widehat{u}_0$ does not change any of the convergence rates in the main theorem.

Hereafter, we only consider the discrepancy
\begin{equation}
    w_\varepsilon := u_\varepsilon - \widehat{u}_\varepsilon.
\end{equation}
Subtracting the variational formulation of \eqref{problimit1} from \eqref{maineqvariational}, we find that
\begin{equation}\label{formfordiscrep}
    \mathcal{A}_{\varepsilon}(k_\varepsilon)(w_\varepsilon,v)=\left\langle \left.\frac{\partial \widehat{u}_\varepsilon}{\partial \nu} \right|_{\partial \Omega+} , v \right\rangle_{\partial \Omega}  , \qquad \forall \,v\in H^1(B_R).
\end{equation}

The main goal of this paper is to prove that $w_\varepsilon$ is \emph{small} in some sense. The main difficulty in estimating $w_\varepsilon$ directly from \eqref{formfordiscrep} is that, at frequencies near the upper edge of the subwavelength band, the bubble array exhibits a strong collective response. At leading order, this response is described by the mean value of the field
inside each bubble. We therefore decompose $H^1(B_R)$ into a finite-dimensional component carrying the bubble averages and a complementary component having zero average in every bubble. This decomposition isolates the potentially
resonant part of the scattering problem and will subsequently lead to a $2\times2$ block system.

To construct this decomposition, define the bubble-average operator
\begin{equation}
    Q_\varepsilon:H^1(B_R) \rightarrow \ell^2( \mathcal{I}_\varepsilon )
\end{equation}
by
\begin{equation}
    (Q_\varepsilon u)(n) := \fint_{\varepsilon(D+n)} u (x)\,dx, \qquad \forall\,n \in \mathcal{I}_\varepsilon.
\end{equation}
For each $n\in\mathcal I_\varepsilon$, define
\begin{equation}
p_\varepsilon^n(x):=\mathscr V_{\varepsilon^2,\mathcal I_\varepsilon}^n(x/\varepsilon),\qquad \forall\,x\in\mathbb R^d.
\end{equation}
The function $p_\varepsilon^n$ describes the quasi-static microscopic field generated by the excitation of the $n$-th bubble and thus captures the leading-order interaction between the incident wave and the $n$-th bubble. By the definition of $\mathscr V_{\varepsilon^2,\mathcal I_\varepsilon}^n$ in \eqref{energyminprobN}, one has
\begin{equation}\label{biorthogonalityQp}
    (Q_\varepsilon p_\varepsilon^n)(m)=\delta_{mn}, \qquad \forall\, m,n\in\mathcal I_\varepsilon.
\end{equation}
We then define the projection $P_\varepsilon:H^1(B_R) \rightarrow H^1(B_R)$ by
\begin{equation}
    P_\varepsilon u = \sum_{n\in \mathcal{I}_\varepsilon} (Q_\varepsilon u)(n) p^n_\varepsilon  , \qquad \forall \, u \in H^1(B_R).
\end{equation}
It follows from \eqref{biorthogonalityQp} that
\begin{equation}
    Q_\varepsilon P_\varepsilon=Q_\varepsilon, \qquad P_\varepsilon^2=P_\varepsilon.
\end{equation}
Moreover, 
\begin{equation}\label{rangePepsilon}
    \operatorname{ran}P_\varepsilon = \operatorname{span} \bigl\{p_\varepsilon^n:n\in\mathcal I_\varepsilon\bigr\},
    \qquad \dim\operatorname{ran}P_\varepsilon = |\mathcal I_\varepsilon|.
\end{equation}

Consequently, every $u\in H^1(B_R)$ admits the unique decomposition
\begin{equation}\label{decompHepsilon}
    u=P_\varepsilon u+(I-P_\varepsilon)u,
    \qquad \mathrm{ \ where \ }
    P_\varepsilon u\in\operatorname{ran}P_\varepsilon \mathrm{ \ and \ } (I-P_\varepsilon)u\in\ker Q_\varepsilon.
\end{equation}
The first component is determined entirely by the vector of bubble averages $Q_\varepsilon u$ and contains the resonant collective response of bubbles, whereas the second component represents the zero-average and non-resonant remainder.

We now introduce a sesquilinear form such that the decomposition \eqref{decompHepsilon} is \emph{orthogonal} in some sense.
Given $\phi \in H^{1/2}(\partial B_R)$, let $u \in H^1_{\mathrm{loc}}(\mathbb{R}^d\setminus \overline{B_R})$ be the exterior harmonic extension of $\phi$, \emph{i.e.},
\begin{equation}\label{harmoextens}
    \left\{
\begin{aligned}
    &\Delta u=0 && \mathrm{in} \ \mathbb{R}^d\setminus \overline{B_R}, \\
    & u|_+=\phi && \mathrm{on} \ \partial B_R, \\
    & u (x)= O(|x|^{2-d}) && \mathrm{as} \ |x| \rightarrow \infty .
\end{aligned}
    \right.
\end{equation}
Using spherical harmonics, one can easily show that \eqref{harmoextens} is uniquely solvable and that $u \in \mathcal{E}$, which is defined in \eqref{energyspace} (see, for instance, \cite{grassle2025dirichlet}). We define the exterior Dirichlet-to-Neumann operator
$T_0:H^{1/2}(\partial B_R)\rightarrow H^{-1/2}(\partial B_R)$ by
\begin{equation}
    T_0(\phi):=\left.\frac{\partial u}{\partial \nu}\right|_{\partial B_R+},
\end{equation}
where $u$ solves \eqref{harmoextens}. The corresponding static sesquilinear form on $H^1(B_R)$ is
\begin{equation}\label{diriform}
    \mathcal{A}_\varepsilon(0)(u,v):= \int_{B_R} a_\varepsilon \nabla u \cdot \overline{\nabla v}\,dx - \langle T_0 u,v\rangle_{\partial B_R}, \qquad \forall\,u,v\in H^1(B_R).
\end{equation}

Indeed, letting $U$ and $V$ be the harmonic extensions of $u$ and $v$ to $\mathbb R^d\setminus\overline{B_R}$, respectively,
 Green's identity then shows that
 \begin{equation}\label{enerA0uv}
    \mathcal{A}_\varepsilon(0)(u,v)= \int_{\mathbb{R}^d} a_\varepsilon \nabla U\cdot \overline{\nabla V} \,dx.
\end{equation}
This is precisely the energy pairing associated with the capacitary basis functions $p_\varepsilon^n$, and therefore yields the orthogonality of the decomposition \eqref{decompHepsilon}. The precise statement is as follows.

\begin{lemma}\label{lemorthog}
    The decomposition $H^1(B_R) = \mathrm{ran}\, P_\varepsilon \oplus \ker Q_\varepsilon$ is orthogonal in the sense that, 
    \begin{equation}
        \mathcal{A}_\varepsilon(0)(u,v)=0
    \end{equation}
    for each $u \in \mathrm{ran}\, P_\varepsilon$ and $v\in \ker Q_\varepsilon$.
\end{lemma}
\begin{proof}
Since $u \in \mathrm{ran}\, P_\varepsilon$, $U$ is a linear combination of the basis functions
$p^n_\varepsilon$. Since $V\in \mathcal{E}$ and the average of $v$ in each bubble is zero, the variational definition \eqref{energyminprobN} and the identity \eqref{enerA0uv} implies that $\mathcal{A}_\varepsilon(0)(u,v)=0$.
\end{proof}

Motivated by Lemma \ref{lemorthog}, on the zero-mean space $\ker Q_\varepsilon$, we define the inner product
\begin{equation}
    \langle u ,v \rangle_{\mathcal{H}_\varepsilon} := \mathcal{A}_\varepsilon(0)(u,v), \qquad \forall\, u,v\in \ker Q_\varepsilon.
\end{equation}
We denote by $\mathcal{H}_\varepsilon$ the resulting Hilbert space.

Since the incident frequency is tuned near the upper edge of the subwavelength band, attained at the $M$ point, the resonant response is dominated by Bloch modes with quasimomentum $M=(\pi,\cdots,\pi)$. Because
\begin{equation}
    e^{\mi M\cdot e_j}=-1,  \qquad \forall\,j=1,\ldots,d,
\end{equation}
the corresponding bubble amplitudes alternate in sign between neighboring cells. To factor out this rapid $M$-point oscillation, we introduce the alternating operator
\begin{equation}
S_M:\ell^2(\mathcal I_\varepsilon)\rightarrow \ell^2(\mathcal I_\varepsilon), \qquad (S_M\phi)(n):=(-1)^{n_1+\cdots+n_d}\phi(n).
\end{equation}
Here $\phi\in\ell^2(\mathcal I_\varepsilon)$ and $n=(n_1,\ldots,n_d)\in\mathcal I_\varepsilon$. We then define the lift operator
\begin{equation}
   \Phi_{\varepsilon}^M : \ell^2(\mathcal{I}_\varepsilon) \rightarrow \mathrm{ran}\,P_\varepsilon, \qquad  \Phi_{\varepsilon}^M \phi := \sum_{n\in \mathcal{I}_\varepsilon} (S_M\phi)(n) p^n_\varepsilon.
\end{equation}

Finally, we define the residual form $\mathcal{R}_\varepsilon(k)$ on $H^1(B_R)$ by
\begin{equation}
\begin{aligned}
    \mathcal{R}_\varepsilon(k)(u,v)&:=   \mathcal{A}_\varepsilon(k)(u,v) - \mathcal{A}_\varepsilon(0) (u,v)\\
    &\ =-k^2\int_{B_R}
    b_\varepsilon u \overline{v}\,dx  -\langle (T_k -T_0)u,v\rangle_{\partial B_R}, 
\end{aligned}\qquad \forall \,u,v\in H^1(B_R).
\end{equation}
In particular, $\mathcal{R}_\varepsilon(k)(u,v) = \mathcal{A}_\varepsilon(k)(u,v) $ if $u \in \mathrm{ran}\, P_\varepsilon$ and $v\in\ker Q_\varepsilon$. Thus, $\mathcal R_\varepsilon(k)$ encodes all interactions between the potentially resonant bubble component and the zero-average background component.
 
Using Lemma \ref{lemorthog}, the discrepancy $w_\varepsilon$ admits the decomposition
\begin{equation}\label{decompofdiscre}
    w_\varepsilon = \Phi_\varepsilon^M \phi_\varepsilon + r_\varepsilon , \qquad \phi_\varepsilon \in \ell^2( \mathcal{I}_\varepsilon ) , \qquad r_\varepsilon  \in \mathcal{H}_\varepsilon = \ker Q_\varepsilon.
\end{equation}
Substituting \eqref{decompofdiscre} into \eqref{formfordiscrep} with test function in $\mathrm{ran}\,P_\varepsilon$ and $\ker Q_\varepsilon$ respectively, we obtain the $2\times 2$ block system
\begin{equation}\label{operatoreq}
    \begin{pmatrix}
        A & R_1 \\
        R_2 & B
    \end{pmatrix}
    \begin{pmatrix}
        \phi_\varepsilon \\ r_\varepsilon
    \end{pmatrix} = \begin{pmatrix}
        g_\varepsilon \\h_\varepsilon
    \end{pmatrix}.
\end{equation}
Here, $A=A_\varepsilon(k_\varepsilon):\ell^2( \mathcal{I}_\varepsilon ) \rightarrow \ell^2(\mathcal{I}_\varepsilon)$ is defined by
\begin{equation}\label{defofA}
    \langle A\phi,\psi \rangle_{\ell^2(\mathcal{I}_\varepsilon)} := \mathcal{A}_\varepsilon(k_\varepsilon)(\Phi_\varepsilon^M \phi, \Phi_\varepsilon^M \psi);
\end{equation}
$R_2 = R_{2,\varepsilon}(k_\varepsilon):\ell^2(\mathcal{I}_\varepsilon)\rightarrow\mathcal{H}_\varepsilon  $ is defined by
\begin{equation}\label{defofR2}
    \langle R_2 \phi, v\rangle_{\mathcal{H}_\varepsilon} :=  \mathcal{R}_\varepsilon(k_\varepsilon)(\Phi_\varepsilon^M \phi , v);
\end{equation}
$R_1 = R_{1,\varepsilon}(k_\varepsilon):\mathcal{H}_\varepsilon \rightarrow \ell^2(\mathcal{I}_\varepsilon)$ is defined by
\begin{equation}\label{defofR1}
    \langle R_1 u, \psi \rangle_{\ell^2( \mathcal{I}_\varepsilon )} :=  \mathcal{R}_\varepsilon(k_\varepsilon)(u,\Phi_\varepsilon^M \psi);
\end{equation}
$B:\mathcal{H}_\varepsilon \rightarrow \mathcal{H}_\varepsilon $ is defined by
\begin{equation}
    \langle Bu,v \rangle_{\mathcal{H}_\varepsilon} := \mathcal{A}_\varepsilon(k_\varepsilon)(u,v);
\end{equation}
and $g_\varepsilon\in \ell^2( \mathcal{I}_\varepsilon )$, $h_\varepsilon \in \mathcal{H}_\varepsilon$ are defined by
\begin{equation}\label{defgandh}
    \langle g_\varepsilon, \psi \rangle_{\ell^2( \mathcal{I}_\varepsilon )} := \left\langle \left.\frac{\partial \widehat{u}_\varepsilon}{\partial \nu} \right|_{\partial \Omega+} , \Phi_\varepsilon^M \psi \right\rangle_{\partial \Omega} , \qquad \langle h_\varepsilon, v\rangle_{\mathcal{H}_\varepsilon} :=  \left\langle \left.\frac{\partial \widehat{u}_\varepsilon}{\partial \nu} \right|_{\partial \Omega+} , v\right\rangle_{\partial \Omega} .
\end{equation}
Hereafter, our main task is to analyze the block system \eqref{operatoreq}.

\subsection{Some useful estimates related to the block system} 

In this section, we establish some useful estimates related to the quantities that appear in the block system.

The block system formulation separates the field into a zero-average component belonging to $\mathcal H_\varepsilon$ and a discrete component parametrized by the bubble averages in $\ell^2(\mathcal I_\varepsilon)$. We begin by establishing uniform estimates for the zero-average component.

\begin{lemma}\label{lemmaHepsprop}
    There exist constants $\varepsilon_0>0$ and $C>0$, independent of $\varepsilon$, such that, for every $0<\varepsilon<\varepsilon_0$ and $u\in\mathcal H_\varepsilon$, the following hold:
    \begin{itemize}
    \item[(a)] 
        \begin{equation}
            \int_\Omega b_\varepsilon |u|^2\,dx \leq C\varepsilon^2 \int_\Omega a_\varepsilon |\nabla u|^2\,dx.
        \end{equation}
    \item[(b)] 
        \begin{equation}
    \| u\|_{L^2(\partial\Omega)}^2 \leq C\varepsilon \int_\Omega a_\varepsilon|\nabla u|^2\,dx.
\end{equation}
    \item[(c)] Define 
    \begin{equation}
        \| u \|^2_{\mathcal{K}_\varepsilon}:= \int_\Omega a_\varepsilon |\nabla u|^2\,dx + \| u \|_{H^1(B_R\setminus \overline{\Omega})}^2 , \qquad \forall\,u\in \mathcal{H}_\varepsilon.
    \end{equation}
    Then, the norms $\|\cdot\|_{\mathcal H_\varepsilon}$ and $\|\cdot\|_{\mathcal K_\varepsilon}$ are uniformly equivalent:
    \begin{equation}
      C^{-1} \| u \|_{\mathcal{K}_\varepsilon} \leq  \| u\|_{\mathcal{H}_\varepsilon} \leq C \| u \|_{\mathcal{K}_\varepsilon}.  
    \end{equation}
    \end{itemize}
\end{lemma}
\begin{proof}
    Let $Y_\varepsilon = \bigcup_{n\in \mathcal{I}_\varepsilon} \varepsilon (Y+n)$ denote the union of the complete cells contained in $\Omega$, and set $U_\varepsilon = \Omega \setminus Y_\varepsilon$ to be the remainder boundary layer part of $\Omega$. 
    
    We first prove (a). Since
$u\in\mathcal H_\varepsilon=\ker Q_\varepsilon$, one has
\begin{equation}
    \int_{\varepsilon(D+n)}u\,dx=0, \qquad n\in\mathcal I_\varepsilon.
\end{equation}
For every $w \in H^1(Y)$ satisfying $\int_D w(x)\,dx =0$, the Poincar\'{e} inequality gives
\begin{equation}
    \int_Y |w|^2\,dx \leq C \int_Y |\nabla w|^2\,dx.
\end{equation}
Rescaling this inequality on each cell and summing over $n\in\mathcal I_\varepsilon$, we obtain
\begin{equation}\label{itemaesti1}
   \int_{Y_\varepsilon} |u|^2\,dx \leq C\varepsilon^2 \int_\Omega |\nabla u|^2\,dx.
\end{equation}

Applying the rescaled Poincar\'e inequality inside each bubble yields
\begin{equation}\label{itemaesti2}
    \varepsilon^{-2}\int_{D_\varepsilon} |u|^2\,dx \leq C\int_{D_\varepsilon} |\nabla u|^2\,dx \leq C\varepsilon^2 \int_\Omega a_\varepsilon|\nabla u|^2\,dx.
\end{equation}

It remains to estimate the $L^2$ norm of $u$ in the boundary layer $U_\varepsilon$. Since $\Omega$ is a bounded Lipschitz domain, for all sufficiently small $\varepsilon$, there exists a family of connected Lipschitz domains $\mathcal{P}_\varepsilon =\{P_{\varepsilon,i}\}_i$ such that
\begin{equation}
    P_{\varepsilon,i} \subset \Omega,  \qquad U_\varepsilon \subset \bigcup_i P_{\varepsilon,i}, \qquad \mathrm{diam}\,P_{\varepsilon,i}\leq C\varepsilon, \qquad \sum_i \mathbbm{1}_{P_{\varepsilon,i}}\leq C.
\end{equation}
Moreover, each $P_{\varepsilon,i}$ contains a bubble $\varepsilon(D+n_i)$ for some $n_i\in\mathcal I_\varepsilon$, and the
rescaled domains $\varepsilon^{-1}P_{\varepsilon,i}$ have uniformly bounded Lipschitz characters. The Poincar\'e inequality therefore gives
\begin{equation}\label{itemaesti3}
    \int_{U_\varepsilon} |u|^2\,dx \leq C\varepsilon^2\sum_i \int_{P_{\varepsilon,i}} |\nabla u|^2\,dx \leq C\varepsilon^2\int_{\Omega} |\nabla u|^2\,dx.
\end{equation}
Consequently, \eqref{itemaesti1}, \eqref{itemaesti2} and \eqref{itemaesti3} imply (a).

We next prove~(b). Let $\Gamma_{\varepsilon,i}:=\partial P_{\varepsilon,i}\cap\partial\Omega$. The patches may be chosen so that the sets $\Gamma_{\varepsilon,i}$ cover $\partial\Omega$. Since $\varepsilon^{-1}P_{\varepsilon,i}$ has uniformly bounded Lipschitz character and $\operatorname{diam}P_{\varepsilon,i}=O(\varepsilon)$, the scaled trace inequality \cite[Lemma A.4]{larson2024space} gives
\begin{equation}\label{traceinequ}
\int_{\Gamma_{\varepsilon,i}} |u|^2\,d\sigma \leq C\left( \varepsilon^{-1}\int_{P_{\varepsilon,i}} |u|^2\,dx
 +\varepsilon \int_{P_{\varepsilon,i}} |\nabla u|^2\,dx \right),
\end{equation}
Summing over $i$ and using the bounded overlap property of the patches, we find
\begin{equation}
    \int_{\partial \Omega} |u|^2\,d\sigma \leq C\left( \varepsilon^{-1} \int_\Omega |u|^2\,dx
 +\varepsilon   \int_\Omega |\nabla u|^2\,dx \right).
\end{equation}
This, combined with (a), proves (b).

Finally, we prove (c). Since $T_0 \leq 0$, one has
\begin{equation}\label{itemcesti1}
    \| u\|_{\mathcal{H}_\varepsilon}^2 = \mathcal{A}_\varepsilon(0)(u,u) \geq   \int_{B_R} a_\varepsilon |\nabla u|^2\,dx. 
\end{equation}
On the Lipschitz domain $B_R\setminus\overline{\Omega}$, the Poincar\'e--Friedrichs inequality gives
\begin{equation}\label{itemcesti2}
    \int_{B_R\setminus \Omega} |u|^2\,dx \leq C\left( \int_{B_R\setminus \Omega} |\nabla u|^2\,dx + \int_{\partial \Omega} |u|^2\,d\sigma \right).
\end{equation}
Combining \eqref{itemcesti1}, \eqref{itemcesti2}, and (b), we obtain $\| u \|_{\mathcal{K}_\varepsilon} \leq C \| u \|_{\mathcal{H}_\varepsilon}$. Conversely, the boundedness of $T_0$ and the trace theorem imply
\begin{equation}
    \|u\|_{\mathcal H_\varepsilon}^2 = \mathcal A_\varepsilon(0)(u,u)\leq \int_\Omega a_\varepsilon|\nabla u|^2\,dx
    + C\|u\|_{H^1(B_R\setminus\overline{\Omega})}^2\leq C\|u\|_{\mathcal K_\varepsilon}^2.
\end{equation}
This completes the proof.
\end{proof}

We next turn to the discrete component of the block decomposition. A vector $\phi\in\ell^2(\mathcal I_\varepsilon)$ records the average field carried by each bubble. Although the standard $\ell^2$-norm controls the overall size of these amplitudes, it does not control their variation between neighboring bubbles. Such control is needed to describe the slowly varying envelope of the wave packet near $M$ and to relate the discrete amplitude $\phi \in \ell^2(\mathcal I_\varepsilon)$ to the Sobolev estimates for the continuous lifted function $\Phi_\varepsilon^M \phi$. We therefore introduce a rescaled discrete $H^1$-norm on $\ell^2(\mathcal I_\varepsilon)$.

Let
\begin{equation}
    E:\ell^2(\mathcal I_\varepsilon)\rightarrow\ell^2(\mathbb Z^d)
\end{equation}
denote the extension by zero. We define the discrete Dirichlet form by
\begin{equation}\label{discredirienerg}
    \mathfrak{D}_\varepsilon(\psi,\phi):= \sum_{n \in \mathbb{Z}^d} \sum_{i=1}^d \big(E\psi(n+e_i)-E\psi(n) \big) \big( E\overline{\phi}(n+e_i)- E\overline{\phi}(n) \big) 
\end{equation}
for $\psi,\phi\in\ell^2(\mathcal I_\varepsilon)$, where $\{e_i\}_{i=1}^d$ denotes the standard basis of $\mathbb Z^d$.
By the Plancherel identity,
\begin{equation}\label{fourierDN}
    \mathfrak{D}_\varepsilon(\psi,\phi) = \frac{1}{|\mathcal{B}_1|} \int_{\mathcal{B}_1} \theta(\beta) \widehat{E\psi}(\beta) \overline{\widehat{E\phi}} (\beta) \,d\beta, \qquad \theta(\beta):= 4\sum_{i=1}^d\sin^2(\beta_i /2) .
\end{equation}
For simplicity, we denote $\mathfrak{D}_\varepsilon(\psi,\psi)$ by $\mathfrak{D}_\varepsilon(\psi)$.

Motivated by the scaling of the continuum gradient, we equip
$\ell^2(\mathcal I_\varepsilon)$ with the inner product
\begin{equation}\label{defXepsinner}
\langle\psi,\phi\rangle_{\mathcal X_\varepsilon}:=\langle\psi,\phi\rangle_{\ell^2(\mathcal I_\varepsilon)}
+\varepsilon^{-2}\mathfrak D_\varepsilon(\psi,\phi),
\end{equation}
and denote the resulting Hilbert space by $\mathcal X_\varepsilon$. 

In addition to the bulk lattice variation, we shall need to control the contribution of the lift $\Phi_\varepsilon^M$ near $\partial\Omega$. Bubble amplitudes supported several lattice layers away from the boundary have a weaker influence on the boundary trace. To quantify this localization, we introduce an exponentially weighted boundary mass. For a fixed $\gamma>0$, define
\begin{equation}
    \mathcal{M}_{\partial\Omega,\varepsilon} (\psi  ):=\sum_{n\in \mathcal{I}_\varepsilon }
        e^{-\gamma d_\varepsilon(n)} |\psi(n)|^2,  \qquad \forall\,\psi \in\ell^2(\mathcal I_\varepsilon),
\end{equation}
where
\begin{equation}
d_\varepsilon(n):=\operatorname{dist}_{\ell^\infty}\bigl(n,\mathbb Z^d\setminus\mathcal I_\varepsilon\bigr)
\end{equation}
is the lattice distance from $n$ to the boundary of $\mathcal I_\varepsilon$. The exponential weight reflects the
layer-by-layer decay of the capacitary fields that will be established below.

We first establish that the boundary mass is controlled by the discrete Dirichlet energy.

\begin{lemma}
\label{lem:discreteHardy}
Let $\gamma>0$. There exists a constant $C=C(d,\gamma)>0$ such that
\begin{equation}\label{hardyineq}
\mathcal{M}_{\partial\Omega,\varepsilon}(\psi) \leq C\mathfrak D_\varepsilon(\psi),
\qquad \forall\,\psi\in\ell^2(\mathcal I_\varepsilon).
\end{equation}
\end{lemma}

In the proof below, we use the following notations. A point $n\in \Z^d$ is called a site, and its (nearest-)neighbors are the sites $m$'s such that $|m-n|_{\ell^1} =1$. A pair $(n,m)$ of points in $\Z^d$ with $|n-m|_{\ell^1}=1$ forms an edge. We regard edges as undirected, i.e., $(n,m)=(m,n)$. A finite path in $\Z^d$ refers to a finite sequence of sites $n_0,n_1,\dots,n_L$, such that $|n_i-n_{i-1}|_{\ell^1} = 1$ for all $1\le i \le L$. It is non-self-intersecting if all sites in this path are different.

\begin{proof}
Given a site $n\in \mathcal{I}_\eps$, set $r = r(n)=d_\varepsilon(n)$. Choose $q(n)\in\mathbb Z^d\setminus\mathcal I_\varepsilon$ such that
\begin{equation}
    |q(n)-n|_{\ell^\infty}=r.
\end{equation}
Joining $n$ to $q(n)$ by a shortest non-self-intersecting path in $\Z^d$, and stop that path at step $L=L(n)$ when it first exits $\mathcal{I}_\eps$. The stopped path is then of the form
\begin{equation*}
    n=n_0 \rightarrow n_1 \rightarrow \cdots \rightarrow n_{L},
\end{equation*}
so that 
\begin{equation*}
    n_j\in\mathcal I_\varepsilon \mathrm{\ for \ }0\leq j< L ,
    \qquad n_L \notin\mathcal I_\varepsilon,
    \qquad L \leq |q(n)-n|_{\ell^1}\leq d r.
\end{equation*}

Since $E\psi(n_L)=0$, applying the Cauchy--Schwarz inequality and $L \le dr$ yield
\begin{equation}
    |\psi(n)|^2 =\left|\sum_{j=0}^{L-1}\big( E\psi(n_j)- E\psi(n_{j+1})\big)\right|^2 
    \leq dr\sum_{j=0}^{L-1}\left| E\psi(n_j)- E\psi(n_{j+1})\right|^2.
\label{hardyPathEstimate}
\end{equation}
Therefore, for each $n\in \mathcal{I}_\varepsilon$, we have constructed a path starting from $n$ satisfying the above properties. Let $P_n$ be the set $\{(n_{i-1},n_i)\,:\, i=1,\dots,L(n)\}$, that is, the set of edges contained in this path. 

For $k\geq1$ and an $e = (e_-,e_+)$ in $\Z^d$, define
\begin{equation}
    N_k(e):=\#\left\{n\in\mathcal I_\varepsilon:d_\varepsilon(n)=k,\ e\in P_n\right\}.
\end{equation}
Note that, if $e\in P_n$ and $d_\varepsilon(n)=k$, then $n$ lies within $\ell^1$ distance $dk$ from either site of $e$. Consequently,
\begin{equation}
    N_k(e)\leq (2dk+1)^d\leq C k^d.
\end{equation}
Thus, although the paths are not necessarily disjoint, their overlap on each distance layer is bounded polynomially. Summing \eqref{hardyPathEstimate} over the $k$-th layer gives
\begin{equation}\label{hardyLayerEstimate}
        \sum_{n\in\mathcal I_\varepsilon \atop d_\varepsilon(n)=k} |\psi(n)|^2 \leq
        d k\sum_{e} N_k(e) \big| E\psi(e_-) - E\psi(e_+) \big|^2 \leq C k^{d+1}\mathfrak D_\varepsilon(\psi).
\end{equation}
Multiplying \eqref{hardyLayerEstimate} by $e^{-\gamma k}$ and summing over $k\geq1$, we obtain
\begin{equation}
    \mathcal{M}_{\partial\Omega,\varepsilon}(\psi)\leq C\left(\sum_{k=1}^{\infty}k^{d+1}e^{-\gamma k}\right)
\mathfrak D_\varepsilon(\psi) \leq C\mathfrak D_\varepsilon(\psi)
\end{equation}
for any $\gamma>0$. This proves \eqref{hardyineq}.
\end{proof}

We next establish a boundary-layer energy decay estimate for general finite arrays. Let $\Lambda\subset\mathbb Z^d$ be finite and define
\begin{equation}
    Y_\Lambda:=\bigcup_{n\in\Lambda}(Y+n), \qquad d_\Lambda(n) := \operatorname{dist}_{\ell^\infty}\bigl(n,\mathbb Z^d\setminus\Lambda\bigr).
\end{equation}
For $r\geq0$, set
\begin{equation}\label{def-capacitary-layers}
    \mathcal O_{r}:=\mathbb R^d\setminus \bigcup_{n\in\Lambda\atop d_\Lambda(n)>r}(Y+n).
\end{equation}
Thus, $\mathcal O_{r}$ consists of the exterior of the array together with its first $\lfloor r\rfloor$ cell layers.

\begin{lemma}\label{lemenergydecayboun}
    Let $0<\delta\leq1$ and let $\xi\in\ell^2(\Lambda)$ satisfy
\begin{equation}\label{vanishing-boundary-data}
\mathrm{supp}\,\xi \subset \{n\in \Lambda: d_\Lambda(n) >m\}.
\end{equation}
for some $m>0$. Write
\begin{equation}
    V:= \sum_{n\in\Lambda}\xi(n)\mathscr V_{\delta,\Lambda}^n
\end{equation}
and define
\begin{equation}
    E(r):=\int_{\mathcal O_{r}}a_{\delta,\Lambda}|\nabla V|^2\,dx.
\end{equation}
There exist constants $C>0$ and $\rho\in(0,1)$, independent of $\Lambda$, $\delta$, $m$, and $\xi$, such that
\begin{equation}\label{capacitary-layer-decay}
    E(r)\leq C\rho^{m-r}\int_{\mathbb R^d}a_{\delta,\Lambda}|\nabla V|^2\,dx,\qquad \forall \,r\in[0,m].
\end{equation}
\end{lemma}

\begin{proof}
Fix integers $r\geq0$ and $K\geq1$ satisfying $r+K\leq m$. Choose a cutoff function $\chi \in C^\infty(\mathbb{R}^d)$ satisfying $0\leq \chi \leq 1$ and the following properties:
\begin{equation}
    \begin{aligned}
        & \chi \equiv 1 \ \mathrm{on} \  \mathcal{O}_r, && \qquad  \chi \equiv 0 \ \mathrm{on} \ \mathbb{R}^d\setminus \mathcal{O}_{r+K}, \\
        & \mathrm{supp}\,\nabla \chi \subset \mathcal{O}_{r+K} \setminus \mathcal{O}_r, && \qquad \|\nabla \chi \|_{L^\infty(\mathbb{R}^d)} \leq \frac{C}{K}.
    \end{aligned}
\end{equation}
Moreover, $\chi$ can be chosen such that
\begin{equation}
    \chi \mathrm{\ is \ a \ constant \ in \ the \ bubble\ }D+n \mathrm{ \ for \ each \ }n\in \Lambda. 
\end{equation}

By \eqref{vanishing-boundary-data}, the function $\chi^2 V$ satisfies $\int_{D+n} \chi^2 V\,dx=0$ for every $n \in \Lambda$. Therefore,
\begin{equation}
    \int_{\mathbb{R}^d} a_{\delta,\Lambda} \nabla V \cdot \overline{\nabla (\chi^2 V)} \,dx =0.
\end{equation}
Expanding $\nabla (\chi^2 V) =  \chi^2 \nabla V +2 \chi V \nabla \chi$, we get
\begin{equation}\label{testresult1}
    \int_{\mathbb{R}^d} \chi^2 a_{\delta,\Lambda} |\nabla V|^2 \,dx = -2 \mathrm{Re}\, \int_{\mathbb{R}^d} \chi a_{\delta,\Lambda}
    \overline{V} \nabla V \cdot \nabla \chi \,dx.
\end{equation}
Young's inequality gives
\begin{equation}
    2\left| \chi a_{\delta,\Lambda}
    \overline{V} \nabla V \cdot \nabla \chi \right| \leq \frac{1}{2}\chi^2 a_{\delta,\Lambda}|\nabla V|^2 + 2a_{\delta,\Lambda}|V|^2|\nabla \chi|^2.
\end{equation}
This, combined with \eqref{testresult1}, gives
\begin{equation}\label{testresult23}
    \int_{\mathbb{R}^d} \chi^2 a_{\delta,\Lambda} |\nabla V|^2 \,dx \leq 4 \int_{\mathbb{R}^d} a_{\delta,\Lambda} |V|^2|\nabla \chi|^2\,dx = 4 \int_{\mathbb{R}^d} |V|^2|\nabla \chi|^2\,dx,
\end{equation}
where we used the fact that the support of $\nabla \chi $ lies in $\{a_{\delta,\Lambda}=1\}$ for the last equality. 

For every $n$ such that $r< d_\Lambda(n) \leq  r+K$, we know that $\int_{D+n} V =0$. Therefore, by the Poincar\'{e} inequality, we get
\begin{equation}\label{data0poincvf}
    \int_{Y+n} |V|^2 \,dx  \leq C \int_{Y+n} |\nabla V|^2 \,dx \leq C \int_{Y+n} a_{\delta,\Lambda} |\nabla V|^2\, dx.
\end{equation}
Consequently, by \eqref{testresult23},
\begin{equation}
\begin{aligned}
    \int_{\mathbb{R}^d} \chi^2 a_{\delta,\Lambda} |\nabla V|^2 \,dx & \leq \frac{C}{K^2} \sum_{r< d_\Lambda(n)\leq r+K}\int_{Y+n} |V|^2\,dx
    \\
    & \leq \frac{C}{K^2} \sum_{r< d_\Lambda(n)\leq r+K}\int_{Y+n} a_{\delta,\Lambda}|\nabla V|^2\,dx \\
    & \leq \frac{C}{K^2}E(r+K).
\end{aligned}
\end{equation}
Choosing $K$ large enough yields
\begin{equation}
    E(r) \leq \frac{1}{2} E(r+K), \qquad \forall\,r+K\leq m.
\end{equation}
By iteration, we conclude that there exists $\rho \in (0,1)$ such that
\begin{equation}\label{expdecayEr}
    E(r)\leq C\rho^{m-r} \int_{\mathbb{R}^d} a_{\delta,\Lambda} |\nabla V|^2\,dx.
\end{equation}
The proof is complete.
\end{proof}

\begin{lemma}\label{lemtraceesti}
There exist constants $C,\gamma>0$, independent of $\varepsilon$, such that
    \begin{equation}\label{traceestimatePhi}
        \| \Phi_\varepsilon^M \psi \|^2_{L^2(\partial \Omega)} \leq C\varepsilon^{d-1} \mathcal{M}_{\partial\Omega,\varepsilon} (\psi  ),  \qquad \forall\,\psi \in \ell^2(\mathcal{I}_\varepsilon) .
    \end{equation}
\end{lemma}

\begin{proof}
For $m\geq 0$, let $Q_m:\ell^2(\mathcal{I}_\varepsilon)\rightarrow \ell^2(\mathcal{I}_\varepsilon)$ denote the projection onto the $(m+1)$-th cell layer, namely,
\begin{equation}\label{eq:definition-of-layer-projection}
    (Q_m\psi)(n):=\mathbbm{1}_{\{d_\varepsilon(n)=m+1\}}\psi(n), \qquad \forall\, n\in \mathcal{I}_\varepsilon.
\end{equation}
Since the array is finite, we have $\psi=\sum_{m\geq 0}Q_m\psi$. For each $m\geq 0$, define
\begin{equation}\label{eq:definition-of-layer-lift}
    W_{\varepsilon,m}:=\Phi^M_\varepsilon Q_m\psi.
\end{equation}
By linearity of $\Phi^M_\varepsilon$, it follows that
\begin{equation}\label{eq:decomposition-of-static-lift}
    \Phi^M_\varepsilon\psi=\sum_{m\geq 0}W_{\varepsilon,m}.
\end{equation} 
We divide the remainder of the proof into several steps.

\textit{Step 1. We establish the energy decay of $W_{\varepsilon,m}$ near the boundary.}
For each $m\geq 0$, let
\begin{equation}\label{eq:definition-of-unscaled-layer-lift}
    V_m:=\sum_{n\in \mathcal{I}_\varepsilon}(S_M Q_m\psi)(n) \mathscr{V}_{\varepsilon^2,\mathcal{I}_\varepsilon}^{n}.
\end{equation}
It is clear that
\begin{equation}\label{eq:definition-of-physical-layer-lift}
    W_{\varepsilon,m}(x) =  V_m\left(\frac{x}{\varepsilon}\right).
\end{equation}

Since the datum $S_M Q_m\psi$ is supported on the $(m+1)$-th layer, it vanishes on the first $m$ outer cell layers. Therefore, 
we can apply Lemma \ref{lemenergydecayboun} with $\delta = \varepsilon^2$ and $\Lambda = \mathcal{I}_\varepsilon$ to $V_m$, and obtain
\begin{equation}\label{intOsmallrho}
    \int_{\mathcal{O}_r}a_{\varepsilon^2,\mathcal{I}_\varepsilon}|\nabla V_m|^2\,dx \leq C\rho^{m-r }\int_{\mathbb R^d}a_{\varepsilon^2,\mathcal{I}_\varepsilon}|\nabla V_m|^2\,dx, \qquad \forall\, r\in [0,m].
\end{equation}

Choose a fixed function $\chi\in C_0^\infty(Y)$ such that $\chi\equiv 1$ near $\overline D$, and extend it by zero outside $Y$. Define
\begin{equation}\label{auxiVtilde}
    \widetilde V_m(x) := \sum_{n\in\mathcal I_\varepsilon}(S_MQ_m\psi)(n)\chi(x-n).
\end{equation}
The functions $\widetilde V_m$ and $V_m$ have the same average over each bubble:
\begin{equation}
    \fint_{D+n}\widetilde V_m\,dy = (S_MQ_m\psi)(n) = \fint_{D+n}V_m\,dy, \qquad \forall\, n\in\mathcal I_\varepsilon.
\end{equation}
Thus, by the energy-minimizing property of $V_m$,
\begin{equation}\label{VmsmallQ}
    \int_{\mathbb R^d}a_{\varepsilon^2,\mathcal I_\varepsilon}|\nabla V_m|^2\,dy
    \leq \int_{\mathbb R^d}a_{\varepsilon^2,\mathcal I_\varepsilon}|\nabla\widetilde V_m|^2\,dy \leq C \|Q_m\psi\|_{\ell^2(\mathcal I_\varepsilon)}^2.
\end{equation}

By the rescaling relation \eqref{eq:definition-of-physical-layer-lift}, \eqref{intOsmallrho}, and \eqref{VmsmallQ}, one gets, for any $K>0$,
\begin{equation}\label{eq:physical-boundary-layer-energy}
    \int_{U_{K\varepsilon}}a_\varepsilon|\nabla W_{\varepsilon,m}|^2\,dx
    \leq C_K \varepsilon^{d-2} \rho^m\|Q_m\psi\|_{\ell^2(\mathcal{I}_\varepsilon)}^2,
\end{equation}
where
\begin{equation}
    U_{K\varepsilon}:= (\mathbb{R}^d\setminus \Omega) \cup \{x\in \Omega: \mathrm{dist}\,(x,\partial \Omega)<K\varepsilon\}
\end{equation}
denotes the physical exterior region together with the $K\varepsilon$ thickness boundary layer of $\Omega$.

\textit{Step 2. We pass from the boundary-layer energy estimate to a trace estimate.} Let $\{P_{\varepsilon,j}\}_j$ be the family of boundary patches constructed in Lemma~\ref{lemmaHepsprop} and $\Gamma_{\varepsilon,j}=\partial P_{\varepsilon,j}\cap\partial\Omega$. Similarly to \eqref{traceinequ}, the scaled trace inequality gives
\begin{equation}\label{eq:scaled-trace-on-boundary-patch}
    \int_{\Gamma_{\varepsilon,j}}|W_{\varepsilon,m}|^2\,d\sigma\leq C\left(
    \varepsilon^{-1}\int_{P_{\varepsilon,j}}|W_{\varepsilon,m}|^2\,dx +\varepsilon\int_{P_{\varepsilon,j}}|\nabla W_{\varepsilon,m}|^2\,dx \right).
\end{equation}
On the other hand, applying the rescaled Poincar\'e inequality with the average over $\varepsilon(D+n_j)$, we have
\begin{equation}\label{eq:poincare-on-boundary-patch}
    \int_{P_{\varepsilon,j}}|W_{\varepsilon,m}|^2\,dx\leq C\varepsilon^2\int_{P_{\varepsilon,j}}|\nabla W_{\varepsilon,m}|^2\,dx
    +C\varepsilon^d|(Q_m\psi)(n_j)|^2.
\end{equation}
Substituting \eqref{eq:poincare-on-boundary-patch} into \eqref{eq:scaled-trace-on-boundary-patch} yields
\begin{equation}\label{eq:trace-estimate-with-bubble-average}
    \int_{\Gamma_{\varepsilon,j}}|W_{\varepsilon,m}|^2\,d\sigma
    \leq C\varepsilon\int_{P_{\varepsilon,j}}a_\varepsilon|\nabla W_{\varepsilon,m}|^2\,dx
    +C\varepsilon^{d-1}|(Q_m\psi)(n_j)|^2.
\end{equation}

If $m$ is larger than the fixed number of cell layers intersecting the boundary patches, then $(Q_m\psi)(n_j)=0$ for every $j$. For the finitely many remaining values of $m$, the second term in \eqref{eq:trace-estimate-with-bubble-average} is controlled directly by $\|Q_m\psi\|_{\ell^2(\mathcal{I}_\varepsilon)}^2$. Therefore, after changing $C$ and $\rho$ if necessary, summing over the boundary patches and using \eqref{eq:physical-boundary-layer-energy} gives
\begin{equation}\label{eq:trace-decay-for-one-layer}
    \|W_{\varepsilon,m}\|_{L^2(\partial\Omega)}^2
    \leq C\varepsilon^{d-1}\rho^m\|Q_m\psi\|_{\ell^2(\mathcal{I}_\varepsilon)}^2.
\end{equation}
It follows from \eqref{eq:decomposition-of-static-lift} that
\begin{equation}\label{eq:sum-of-layer-traces}
\begin{aligned}
    \|\Phi^M_\varepsilon\psi\|^2_{L^2(\partial\Omega)}
    &\leq C\varepsilon^{d-1}\sum_{m\geq 0}\rho^{m/2}\|Q_m\psi\|^2_{\ell^2(\mathcal{I}_\varepsilon)} \\
    & \leq C\varepsilon^{d-1}  \sum_{n\in \mathcal{I}_\varepsilon}\rho^{d_\varepsilon(n)/2} |\psi(n)|^2.
\end{aligned}
\end{equation}
Setting $\gamma:=-\frac{1}{2}\log\rho>0$, we obtain the desired estimate.
\end{proof}

\begin{lemma}\label{lemPhiML2}
    For every $\psi \in \ell^2(\mathcal{I}_\varepsilon)$ and $u\in \mathcal{H}_\varepsilon$, we have
    \begin{equation}\label{PhiMoutside}
      \| \Phi_\varepsilon^M \psi \|_{L^2(B_R\setminus\overline\Omega)} + \| \Phi_\varepsilon^M  \psi \|_{L^2(\partial B_R)}
    \leq C\varepsilon^{(d+1)/2}\|\psi\|_{\mathcal{X}_\varepsilon}, 
    \end{equation}
    and
    \begin{equation}\label{interiormassestimate}
        \left| \int_\Omega b_\varepsilon( \Phi_\varepsilon^M \psi)(x) \overline{u(x)}\,dx \right| \leq C\varepsilon^{(d+2)/2} \|\psi\|_{\mathcal{X}_\varepsilon} \|u \|_{\mathcal{H}_\varepsilon}.
    \end{equation}
\end{lemma}

\begin{proof}
    Set $W_\varepsilon = \Phi_\varepsilon^M \psi $. It follows from Lemma \ref{lem:discreteHardy} and Lemma \ref{lemtraceesti} that
\begin{equation}\label{boundaryofWeps}
    \|W_\varepsilon\|_{L^2(\partial\Omega)}^2 \leq C\varepsilon^{d-1} \mathcal{M}_{\partial\Omega,\varepsilon}(\psi)
    \leq C\varepsilon^{d-1} \mathfrak{D}_{\varepsilon}(\psi,\psi) \leq C\varepsilon^{d+1}\|\psi\|_{\mathcal{X}_\varepsilon}^2.
\end{equation}

By the same argument as \eqref{auxiVtilde}-\eqref{VmsmallQ}, we have
\begin{equation}\label{eq:energy-estimate-W-epsilon}
    \int_{\mathbb R^d}a_\varepsilon|\nabla W_\varepsilon|^2\,dx \leq C\varepsilon^{d-2}\|\psi\|_{\mathcal{X}_\varepsilon}^2.
\end{equation}

We now prove the estimate \eqref{interiormassestimate}. On each complete cell, the Poincar\'e inequality gives
\begin{equation}
    \int_{\varepsilon(Y+n)}|W_\varepsilon|^2\,dx \leq C\varepsilon^2\int_{\varepsilon(Y+n)}|\nabla W_\varepsilon|^2\,dx
    +C\varepsilon^d|\psi(n)|^2.
    \label{eq:cell-Poincare-W}
\end{equation}
The same estimate holds on the boundary patches introduced in the proof of Lemma~\ref{lemmaHepsprop}; see \eqref{eq:poincare-on-boundary-patch}. Summing the resulting estimates and using \eqref{eq:energy-estimate-W-epsilon}, we obtain
\begin{equation}\label{eq:matrix-L2-estimate-W}
    \|W_\varepsilon\|_{L^2(\Omega\setminus\overline{D_\varepsilon})} \leq C\varepsilon^{d/2}\|\psi\|_{\mathcal{X}_\varepsilon}.
\end{equation}

Since $u\in \mathcal{H}_\varepsilon$, Lemma~\ref{lemmaHepsprop} gives
\begin{equation}
    \|u\|_{L^2(\Omega\setminus\overline{D_\varepsilon})} \leq C\varepsilon\|u\|_{\mathcal H_\varepsilon},
    \qquad \int_\Omega a_\varepsilon|\nabla u|^2\,dx \leq \|u\|_{\mathcal H_\varepsilon}^2.
    \label{eq:matrix-L2-estimate-u}
\end{equation}
It follows from \eqref{eq:matrix-L2-estimate-W} and \eqref{eq:matrix-L2-estimate-u} that
\begin{equation}
    \left| \int_{\Omega\setminus\overline{D_\varepsilon}} W_\varepsilon\overline{u}\,dx \right|
    \leq C\varepsilon^{(d+2)/2}\|\psi\|_{\mathcal X_\varepsilon}\|u\|_{\mathcal H_\varepsilon}.
    \label{eq:matrix-mass-estimate}
\end{equation}

On the other hand, by the Poincar\'{e} and the Cauchy-Schwarz inequalities, one has
\begin{equation}
    \begin{aligned}
    \varepsilon^{-2} 
    \left| \int_{D_\varepsilon} W_\varepsilon  \overline{u}\,dx \right| &= \varepsilon^{-2} \left| \sum_{n\in \mathcal{I}_\varepsilon} \int_{\varepsilon(D+n)} \big(W_\varepsilon - (S_M\psi)(n)\big)\overline{u}\,dx \right| \\
    &\leq C \|\nabla W_\varepsilon\|_{L^2(D_\varepsilon)} \|\nabla u\|_{L^2(D_\varepsilon)} \\
    &\leq C\varepsilon^2 \left(\int_{D_\varepsilon}a_\varepsilon|\nabla W_\varepsilon|^2\,dx\right)^{1/2}
    \left(\int_{D_\varepsilon}a_\varepsilon|\nabla u|^2\,dx\right)^{1/2}.
    \end{aligned}
    \label{eq:bubble-scaled-Poincare-estimate}
\end{equation}
Because $b_\varepsilon=(v/v_{\rm b})^2\varepsilon^{-2}$ in $D_\varepsilon$, equations \eqref{eq:energy-estimate-W-epsilon} and \eqref{eq:matrix-L2-estimate-u} imply
\begin{equation}
    \left| \int_{D_\varepsilon}b_\varepsilon W_\varepsilon\overline{u}\,dx \right|
    \leq C\varepsilon^{(d+2)/2} \|\psi\|_{\mathcal X_\varepsilon}\|u\|_{\mathcal H_\varepsilon}.
    \label{eq:bubble-mass-estimate}
\end{equation}
Combining \eqref{eq:matrix-mass-estimate} and \eqref{eq:bubble-mass-estimate}, we obtain the desired estimate \eqref{interiormassestimate}.

It remains to prove the estimate \eqref{PhiMoutside}. The function $W_\varepsilon$ is harmonic in $\mathbb R^d\setminus\overline\Omega$. The standard $L^2$ nontangential maximal functions estimate for the exterior harmonic Dirichlet problem (see, for instance, \cite{verchota1984layer}) yields
\begin{equation}
    \|W_\varepsilon\|_{L^2(B_R\setminus\overline\Omega)} + \|W_\varepsilon\|_{L^2(\partial B_R)}
    \leq C\|W_\varepsilon\|_{L^2(\partial\Omega)}.
    \label{eq:exterior-harmonic-propagation}
\end{equation}
Therefore, \eqref{boundaryofWeps} gives
\begin{equation}
    \|W_\varepsilon\|_{L^2(B_R\setminus\overline\Omega)} + \|W_\varepsilon\|_{L^2(\partial B_R)}
    \leq C\varepsilon^{(d+1)/2}\|\psi\|_{\mathcal X_\varepsilon}.
    \label{eq:exterior-smallness-W}
\end{equation}
The proof is complete.
\end{proof}

\subsection{Smallness of the operators \texorpdfstring{$R_1$ and $R_2$}{R1 and R2}}

In this section, we prove that the operators $R_1$ and $R_2$ are both \emph{small} in some sense, which roughly means that the interactions between the quasi-static bubble responses and the zero-average background component are small.

\begin{proposition}\label{lem:off-diagonal-blocks}
Let $R_1=R_{1,\varepsilon}(k_\varepsilon)$ and $R_2=R_{2,\varepsilon}(k_\varepsilon)$ be defined by \eqref{defofR1} and \eqref{defofR2}, respectively. Let $\mathcal X_\varepsilon^*$ denote the dual of $\mathcal X_\varepsilon$ with respect to the $\ell^2$ pairing. Then there exists a constant $C>0$, independent of $\varepsilon$, such that
\begin{equation}
    \left\| R_1\right\|_{\mathcal L(\mathcal H_\varepsilon,\mathcal X_\varepsilon^*)} + \left\|R_2 \right\|_{\mathcal L(\mathcal X_\varepsilon,\mathcal H_\varepsilon)} \leq C\varepsilon^{(d+1)/2}.
    \label{eq:R2-off-diagonal-estimate}
\end{equation}
\end{proposition}

\begin{proof}
Fix $u\in \mathcal{H}_\varepsilon$ and $\psi \in \ell^2(\mathcal{I}_\varepsilon)$. Let $W_\varepsilon = \Phi^M_\varepsilon  \psi$. The equivalence between the $\mathcal{H}_\varepsilon$ norm and the $\mathcal{K}_\varepsilon$ norm in Lemma~\ref{lemmaHepsprop}, followed by the trace theorem, gives
\begin{equation}
    \|u\|_{L^2(B_R\setminus \Omega)} +\|u\|_{L^2(\partial B_R)} \leq C\|u\|_{\mathcal{H}_\varepsilon}.
    \label{eq:exterior-estimate-u}
\end{equation}
Thus, \eqref{PhiMoutside} gives
\begin{equation}
    \left| \int_{B_R\setminus \Omega} W_\varepsilon\overline{u}\,dx \right|
    \leq C\varepsilon^{(d+1)/2} \|\psi\|_{\mathcal{X}_\varepsilon}\|u\|_{\mathcal{H}_\varepsilon}.
    \label{eq:exterior-mass-both-orders}
\end{equation}

Since $\partial B_R$ is smooth, the operator $T_k-T_0$ is a pseudodifferential operator of order zero and is therefore bounded on $L^2(\partial B_R)$. The boundedness of $k_\varepsilon$ makes this estimate uniform:
\begin{equation}
    \sup_{0<\varepsilon<\varepsilon_0} \|T_{k_\varepsilon}-T_0\|_{\mathcal L(L^2(\partial B_R))} \leq C.
    \label{eq:uniform-DtN-bound}
\end{equation}
Using \eqref{eq:exterior-smallness-W} and \eqref{eq:exterior-estimate-u}, we consequently obtain the estimate
\begin{equation}
    \left|\big\langle (T_{k_\varepsilon}-T_0)W_\varepsilon,u \big\rangle_{\partial B_R}\right|
    + \left|\big\langle (T_{k_\varepsilon}-T_0)u,W_\varepsilon \big\rangle_{\partial B_R}\right|\leq C\varepsilon^{(d+1)/2} \|\psi\|_{\mathcal{X}_\varepsilon}\|u\|_{\mathcal{H}_\varepsilon}  .
    \label{eq:mixed-DtN-estimate}
\end{equation}

Recalling the definition
\begin{equation}
    \mathcal R_\varepsilon(k)(w,v) = -k^2\int_{B_R}b_\varepsilon w\overline v\,dx - 
    \left\langle (T_k-T_0)w,v\right\rangle_{\partial B_R},
    \label{eq:remainder-form-recalled}
\end{equation}
and combining \eqref{interiormassestimate}, \eqref{eq:exterior-mass-both-orders}, and \eqref{eq:mixed-DtN-estimate}, we obtain
\begin{equation}
    \left| \mathcal R_\varepsilon(k_\varepsilon)(W_\varepsilon,u) \right|
    +\left| \mathcal R_\varepsilon(k_\varepsilon)(u,W_\varepsilon) \right|
    \leq C\varepsilon^{(d+1)/2} \|\psi\|_{\mathcal X_\varepsilon}\|u\|_{\mathcal H_\varepsilon}.
    \label{eq:mixed-remainder-bound}
\end{equation}

By the definition of $R_2$, one has
\begin{equation}
    \begin{aligned}
        \|R_2\psi\|_{\mathcal H_\varepsilon} &=
        \sup_{0\neq u\in \mathcal H_\varepsilon}\frac{\left|\langle R_2\psi,u\rangle_{\mathcal H_\varepsilon}\right|} {\|u\|_{\mathcal H_\varepsilon}} \\
        &= \sup_{0\neq u\in \mathcal H_\varepsilon} \frac{\left|\mathcal R_\varepsilon(k_\varepsilon) (\Phi^M_\varepsilon  \psi,u)\right|}
        {\|u\|_{\mathcal H_\varepsilon}} \leq C\varepsilon^{(d+1)/2}\|\psi\|_{\mathcal X_\varepsilon}.
    \end{aligned}
    \label{eq:R2-vector-bound}
\end{equation}
This proves the desired estimate \eqref{eq:R2-off-diagonal-estimate} for $R_2$.

Finally, using the dual norm and applying \eqref{eq:mixed-remainder-bound}, we conclude that
\begin{equation}
    \|R_1u\|_{\mathcal X_\varepsilon^*}  =\sup_{\psi\neq0} \frac{|\langle R_1u,\psi\rangle_{\ell^2(\mathcal I_\varepsilon)}|}{\|\psi\|_{\mathcal X_\varepsilon}}  =\sup_{\psi\neq0} \frac{|\mathcal R_\varepsilon(k_\varepsilon) (u,\Phi^M_\varepsilon \psi )|}{\|\psi\|_{\mathcal X_\varepsilon}} \leq C\varepsilon^{(d+1)/2}\|u\|_{H_\varepsilon}.
    \label{eq:R1-vector-bound}
\end{equation}
Taking the supremum over nonzero $u\in \mathcal{H}_\varepsilon$ proves \eqref{eq:R2-off-diagonal-estimate} for $R_1$.
\end{proof}

\subsection{Uniform invertibility of the operator \texorpdfstring{$B$}{B}}

In this section, we prove that $B=B_\varepsilon(k_\varepsilon)$ is uniformly invertible, namely,
\begin{equation}
    \| B_\varepsilon(k_\varepsilon)^{-1} \|_{\mathcal{H}_\varepsilon \rightarrow \mathcal{H}_\varepsilon} \leq C.
\end{equation}

Before proving the inverse estimate, we first show that $B_\varepsilon(k_\varepsilon)$ is uniformly bounded. Indeed, by Lemma 
\ref{lemmaHepsprop}.(a) and the boundedness of $k_\varepsilon$, 
\begin{equation}
    \left| k_\varepsilon^2 \int_{\Omega} b_\varepsilon u\overline{v} \,dx \right| \leq C\varepsilon^2 \| u \|_{\mathcal{H}_\varepsilon} 
    \| v \|_{\mathcal{H}_\varepsilon}, \qquad \left| k_\varepsilon^2 \int_{B_R\setminus \overline{\Omega}} u\overline{v} \,dx \right| \leq C \| u \|_{\mathcal{H}_\varepsilon} 
    \| v \|_{\mathcal{H}_\varepsilon},
\end{equation}
By the boundedness of $T_{k_\varepsilon}$ and the trace theorem, we have
\begin{equation}
    |\langle T_{k_\varepsilon} u,v \rangle_{\partial B_R} | \leq C\| u\|_{H^{1/2}(\partial B_R)}\| v\|_{H^{1/2}(\partial B_R)} \leq 
    C \| u \|_{\mathcal{K}_\varepsilon} \| v \|_{\mathcal{K}_\varepsilon}.
\end{equation}
Therefore, 
\begin{equation}
    |\langle Bu,v \rangle_{\mathcal{H}_\varepsilon} |= |\mathcal{A}_\varepsilon(k_\varepsilon)(u,v)| \leq C \| u \|_{\mathcal{K}_\varepsilon} \| v \|_{\mathcal{K}_\varepsilon} \leq C \| u \|_{\mathcal{H}_\varepsilon} \| v \|_{\mathcal{H}_\varepsilon} ,
\end{equation}
where we used Lemma \ref{lemmaHepsprop}.(c) in the last inequality. Therefore, $B_\varepsilon(k_\varepsilon)$ is uniformly bounded.

We establish a uniform Gårding inequality.

\begin{lemma}
There exist two constants $c,C>0$ such that
\begin{equation}\label{garding}
    \mathrm{Re}\,\langle Bu,u \rangle_{\mathcal{H}_\varepsilon} + C\|u\|_{L^2(B_R\setminus \Omega)}^2 \geq c\|u\|_{\mathcal H_\varepsilon}^2, \qquad  \forall\, u\in\mathcal H_\varepsilon.
\end{equation}
\end{lemma}
\begin{proof}
    By definition, $\langle Bu,u \rangle_{\mathcal{H}_\varepsilon} = \|u\|_{\mathcal H_\varepsilon}^2 + \mathcal{R}_\varepsilon(k_\varepsilon)(u,u)$. By Lemma \ref{lemmaHepsprop}.(a), one has
    \begin{equation}
        \operatorname{Re}\mathcal{R}_\varepsilon(k_\varepsilon)(u,u) \geq - C\varepsilon^2 \int_\Omega a_\varepsilon |\nabla u|^2\,dx - 
        C\int_{B_R\setminus \Omega} |u|^2\,dx - \operatorname{Re}\langle (T_{k_\varepsilon} - T_0)u,u\rangle_{\partial B_R}.
    \end{equation}
    By Lemma \ref{lemmaHepsprop}.(c), we get
    \begin{equation}
        \operatorname{Re} \langle Bu,u \rangle_{\mathcal{H}_\varepsilon} + C\|u\|_{L^2(B_R\setminus \Omega)}^2 \geq c\|u\|_{\mathcal H_\varepsilon}^2 - \operatorname{Re} \langle (T_{k_\varepsilon} - T_0)u,u\rangle_{\partial B_R}.
    \end{equation}
Since $\partial B_R$ is smooth, $T_{k_\varepsilon}-T_0$ is a pseudodifferential operator of order zero. In particular,
\begin{equation}
    |\langle (T_{k_\varepsilon} - T_0)u,u\rangle_{\partial B_R}| \leq C\int_{\partial B_R}|u|^2\,d\sigma.
\end{equation}
Using the scaled trace inequality \cite[Lemma A.4]{larson2024space} again, for sufficiently small $\rho>0$, we get
\begin{equation}
    \int_{\partial B_R}|u|^2\,d\sigma \leq C\left( \rho^{-1} \int_{B_R\setminus B_{R-\rho}} | u|^2\,dx  +\rho \int_{B_R\setminus B_{R-\rho}} |\nabla u|^2\,dx   \right)
\end{equation}
Consequently, 
\begin{equation}
\begin{aligned}
    |\langle (T_{k_\varepsilon} - T_0)u,u\rangle_{\partial B_R}| & \leq C\left( \rho^{-1} \| u \|_{L^2(B_R\setminus \Omega)}^2  +\rho \| u \|_{H^1(B_R\setminus \Omega)}^2   \right) \\
    & \leq \frac{C}{\rho} \| u \|_{L^2(B_R\setminus \Omega)}^2 + C\rho \| u \|^2_{\mathcal{H}_\varepsilon},
\end{aligned}
\end{equation}
where we used Lemma \ref{lemmaHepsprop}.(c). Choosing $\rho$ sufficiently small, we get the desired conclusion.
\end{proof}

\begin{proposition}\label{prop:invertibilityofB}
There exists a constant $C>0$ such that for sufficiently small $\eps$, we have 
    \begin{equation}
    \| B_\varepsilon(k_\varepsilon)^{-1} \|_{\mathcal{L}(\mathcal{H}_\varepsilon) } \leq C.
\end{equation}
\end{proposition}
\begin{proof}
    We use the dual method. Given $f\in L^2(B_R\setminus \overline{\Omega})$, let $v_f$ be the solution of the dual scattering problem for soft obstacle:
\begin{equation}\label{dualvfproblem}
    \left\{
    \begin{aligned}
        &-\Delta v_f- k_\varepsilon^2 v_f =f && \mathrm{in} \ B_R\setminus \overline{\Omega}, \\
        & v_f = 0 && \mathrm{in} \ \overline{\Omega}, \\
        & \left.\frac{\partial v_f}{\partial \nu} \right|_-=T_{k_\varepsilon}^* v_f && \mathrm{on} \ \partial B_R,
    \end{aligned}
    \right.
\end{equation}
where $T_{k_\varepsilon}^*$ denotes the dual operator of $T_{k_\varepsilon}$. Then, we have the estimate
\begin{equation}\label{estivf}
    \| v_f \|_{H^1(B_R\setminus \overline{\Omega})} + \left\| \left.\frac{\partial v_f}{\partial \nu} \right|_+ \right\|_{L^2(\partial \Omega)} \leq C\|f\|_{L^2(B_R\setminus \overline{\Omega})}.
\end{equation}
The above estimate follows from the layer potential representation of $v_f$ and the $L^2$ regularity of non-tangential maximal functions on Lipschitz domains (see, for instance, \cite{verchota1984layer}). In particular,
\begin{equation}\label{estivf1}
    \| v_f \|_{\mathcal{H}_\varepsilon}\leq C\| v_f \|_{\mathcal{K}_\varepsilon} = C\| v_f \|_{H^1(B_R\setminus \overline{\Omega})}
    \leq C\|f\|_{L^2(B_R\setminus \Omega)}.
\end{equation}

For $u\in \mathcal{H}_\varepsilon$, Green's identity gives
\begin{equation}
    \langle Bu,v_f \rangle_{\mathcal{H}_\varepsilon} = \int_{B_R\setminus \Omega} u \overline{f} \,dx - \int_{\partial \Omega}
    u \left.\frac{\partial \overline{v_f} }{\partial \nu} \right|_+  \,d\sigma.
\end{equation}
Using \eqref{estivf}, \eqref{estivf1} and Lemma \ref{lemmaHepsprop}.(b), we get
\begin{equation}
    \left| \int_{B_R\setminus \Omega} u \overline{f} \,dx \right| \leq C\big(\| Bu\|_{\mathcal{H}_\varepsilon} + \varepsilon^{1/2} \| u\|_{\mathcal{H}_\varepsilon} \big)\|f\|_{L^2(B_R\setminus \Omega)} .
\end{equation}
Taking the supremum over $f \in L^2(B_R\setminus \overline{\Omega})$, we conclude that
\begin{equation}
    \| u \|_{L^2(B_R\setminus \overline{\Omega})} \leq C\big(\| Bu\|_{\mathcal{H}_\varepsilon} + \varepsilon^{1/2} \| u\|_{\mathcal{H}_\varepsilon} \big).
\end{equation}
Substituting this into the uniform Gårding inequality \eqref{garding} and using Cauchy-Schwarz, we obtain that, for sufficiently small $\varepsilon$, $\| u\|_{\mathcal{H}_\varepsilon}  \leq C\| Bu\|_{\mathcal{H}_\varepsilon} $. This proves that $B$ is uniformly bounded below.

It suffices to prove that $B$ is surjective, which is a consequence of the estimate $\| u\|_{\mathcal{H}_\varepsilon}  \leq C\| Bu\|_{\mathcal{H}_\varepsilon} $ and the fact that $B$ is a Fredholm operator of index zero. To show that $B$ is a Fredholm operator of index zero, we observe that $\mathcal{A}_\varepsilon(k) = \mathcal{A}_\varepsilon(0) + \mathcal{R}_\varepsilon(k)$. With respect to the inner product induced by $\mathcal{A}_\varepsilon(0)$, the form $\mathcal{A}_\varepsilon(0)$ corresponds to the identity operator on $\mathcal{H}_\varepsilon$, while $\mathcal{R}_\varepsilon(k)$ induces a compact operator $K_\varepsilon(k)$. Consequently, $B=I+K_\varepsilon(k_\varepsilon)$, and hence $B$ is a Fredholm operator of index zero.
The proof is complete.
\end{proof}

\section{Resolvent limits of the operator \texorpdfstring{$A$}{A}}
\label{secresolventlimitofA}

In this section, we establish the limit of the resolvent for the operator $A = A_\varepsilon(k_\varepsilon)$; see Theorem \ref{thm:physical-block-norm-resolvent-convergence}. Define
\begin{equation}
    C^{\mathrm{f}}_{\varepsilon} := C^{\mathrm{f}}_{\varepsilon^2, \mathcal{I}_\varepsilon}, \qquad C^{\mathrm{t}}_{\varepsilon} := C^{\mathrm{t}}_{\varepsilon^2, \mathcal{I}_\varepsilon}.
\end{equation}
By the definition \eqref{defofA} of the operator $A$, we get
\begin{equation}
    A_{mn} = |D|\varepsilon^{d-2} ( S_M^*C^{\mathrm{f}}_\varepsilon S_M)_{mn} + (-1)^{|m|+|n|}\mathcal{R}_\varepsilon(k_\varepsilon)(p^n_\varepsilon,p^m_\varepsilon),
\end{equation}
where $|m|=m_1+\cdots +m_d$ for $m=(m_1,\cdots,m_d)$. This expression motivates us to study the limit associated to $ S_M^*C^{\mathrm{f}}_\varepsilon S_M$. Therefore, we define the rescaled (and re-centered) capacitance matrix as follows:
\begin{equation}
    L^{\mathrm{f}}_\varepsilon := \frac{\widehat{C}^M_{\varepsilon^2} I - S_M^*C^{\mathrm{f}}_{\varepsilon} S_M}{ \varepsilon^2 },
    \qquad L^{\mathrm{t}}_\varepsilon := \frac{\widehat{C}^M_{\varepsilon^2} I - S_M^*C^{\mathrm{t}}_{\varepsilon} S_M}{ \varepsilon^2 }.
\end{equation}
We also define the limiting operator as follows. Let 
\begin{equation}
   L_D := -\nabla \cdot (A_{\mathrm{eff}} \nabla)  
\end{equation}
be the second-order differential operator on $\Omega$ with Dirichlet boundary condition on $\partial \Omega$, where
\begin{equation}\label{defAeff}
    A_{\mathrm{eff}}:= -\frac{1}{2}\nabla_\alpha^2 \widehat{C}^M >0 .
\end{equation}

To pass the limit from the matrix $L^\mathrm{f}_\varepsilon$ acting on vectors to the operator $L_D$ acting on functions defined on $\Omega$, we define the operator $J_\varepsilon : \ell^2(\mathcal{I}_\varepsilon) \rightarrow L^2(\Omega)$ by
\begin{equation}
    (J_\varepsilon \phi) (x):= \left\{
    \begin{aligned}
        & \varepsilon^{-d/2 } \phi(n) ,  && x\in \varepsilon(Y+n),\, n\in \mathcal{I}_\varepsilon , \\
        & 0, && \mathrm{otherwise}.
    \end{aligned}\right.
\end{equation}
The operator norm of $J_\varepsilon$ is $1$. The adjoint operator $J_\varepsilon^*:L^2(\Omega) \rightarrow  \ell^2(\mathcal{I}_\varepsilon) $ is
\begin{equation}
    (J_\varepsilon^* u) (n)= \varepsilon^{-d/2}\int_{\varepsilon(Y+n)} u(x)\,dx, \qquad \forall\, u\in L^2(\Omega), \,n\in \mathcal{I}_\varepsilon.
\end{equation}
It is clear that $J_\varepsilon^* J_\varepsilon = I$ and $\Pi_\varepsilon := J_\varepsilon J_\varepsilon^*$ is the cell-average projection on the lattice $\mathcal{I}_\varepsilon$:
\begin{equation}\label{projecPieps}
    \Pi_\varepsilon u(x)=\left\{
    \begin{aligned}
        & \fint_{\varepsilon(n+Y)} u(y)\,dy, && x\in \varepsilon (n+Y), n \in \mathcal{I}_\varepsilon , \\
        & 0, && \mathrm{otherwise}.
    \end{aligned}\right.
\end{equation}

\subsection{Some useful estimates} In this section, we collect some useful estimates that will be repeatedly used in this section.

The first lemma compares the finite-array capacitance matrix $C^{\mathrm{f}}_{\varepsilon}$ and the truncated capacitance matrix $C^{\mathrm{t}}_{\varepsilon}$.

\begin{lemma}\label{lemlocalbounR}
    We have
    \begin{equation}\label{monotoneofC}
        0\leq C^{\mathrm{f}}_{\varepsilon} \leq C^{\mathrm{t}}_{\varepsilon} \leq CI.
    \end{equation}
    Moreover, let $R_\varepsilon = C^{\mathrm{t}}_\varepsilon - C^{\mathrm{f}}_\varepsilon$, then,
    \begin{equation}\label{estimateRbyM}
    \big|\langle R_\varepsilon \psi , \psi \rangle_{\ell^2(\mathcal{I}_\varepsilon)} \big|\leq C\mathcal{M}_{\partial\Omega,\varepsilon} (\psi  ) , \qquad \forall\,\psi \in \ell^2(\mathcal{I}_\varepsilon).
    \end{equation}
\end{lemma}

\begin{proof}
It is easy to get \eqref{monotoneofC} by the variational definitions of the capacitance matrices. Hence, we only prove \eqref{estimateRbyM}. Let the notations be as in Lemma \ref{lemenergydecayboun}. Namely, for $r\geq0$, let 
\begin{equation}
    Y_{\mathcal{I}_\varepsilon}:=\bigcup_{n\in \mathcal{I}_\varepsilon}(Y+n) , \qquad \mathcal O_{r}:=\mathbb R^d\setminus \bigcup_{n\in \mathcal{I}_\varepsilon \atop d_\varepsilon(n)>r}(Y+n).
\end{equation}
We divide the proof into two steps.

\textit{Step 1. } Fix $m>0$, and let $\xi\in\ell^2(\mathcal{I}_\varepsilon)$ satisfy
\begin{equation}\label{vanishing-boundary-data1}
\mathrm{supp}\,\xi \subset \{n\in \mathcal{I}_\varepsilon : d_\varepsilon(n) >m\}.
\end{equation}
Let 
\begin{equation}
    V:= \sum_{n\in \mathcal{I}_\varepsilon}\xi(n)\mathscr V_{\varepsilon^2,\mathcal{I}_\varepsilon}^n \qquad \mathrm{and} 
    \qquad E(r):=\int_{\mathcal O_{r}}a_{\varepsilon^2,\mathcal{I}_\varepsilon}|\nabla V|^2\,dx.
\end{equation}
By the definition of $C^{\mathrm{f}}_\varepsilon$ and Lemma \ref{lemenergydecayboun}, one has
\begin{equation}
    |D| \langle C^{\mathrm{f}}_{\varepsilon } \xi,\xi \rangle_{\ell^2(\mathcal{I}_\varepsilon)}  = \int_{\mathbb{R}^d} a_{\varepsilon^2,\mathcal{I}_\varepsilon} |\nabla V|^2\,dx,
\end{equation}
and
\begin{equation}\label{estiErbyC}
    E(r) \leq C\rho^{m-r}\int_{\mathbb{R}^d} a_{\varepsilon^2,\mathcal{I}_\varepsilon} |\nabla V|^2\,dx =C\rho^{m-r} |D| \langle C^{\mathrm{f}}_{\varepsilon } \xi,\xi \rangle_{\ell^2(\mathcal{I}_\varepsilon)} .
\end{equation}

Let $r_1 = \lfloor m/3 \rfloor$ and $r_2=\lfloor 2m/3 \rfloor$. Choose a cutoff function $\chi \in C_0^\infty(\mathbb{R}^d)$ satisfying $0\leq \chi \leq 1$ with the following properties:
\begin{equation}\label{propertyofchi4}
    \begin{aligned}
        & \chi \equiv 0 \ \mathrm{on} \  \mathcal{O}_{r_1}, && \qquad  \chi \equiv 1 \ \mathrm{on} \ \mathbb{R}^d\setminus \mathcal{O}_{r_2}, \\
        & \mathrm{supp}\,\nabla \chi \subset \mathcal{O}_{r_2} \setminus \mathcal{O}_{r_1}, && \qquad \|\nabla \chi\|_{L^\infty(\mathbb{R}^d)} \leq \frac{C}{m}.
    \end{aligned}
\end{equation}
Moreover, $\chi$ can be chosen such that
\begin{equation}\label{constantofchi}
    \chi \mathrm{\ is \ a \ constant \ in \ the \ bubble\ }D+n \mathrm{ \ for \ each \ }n\in \mathbb{Z}^d. 
\end{equation}
We decompose $V$ into two parts:
\begin{equation}
    V= V_1 +V_2, \qquad V_1:=\chi V \mathrm{ \ and \ } V_2 :=(1-\chi)V.
\end{equation}

Let $E:\ell^2(\mathcal{I}_\varepsilon) \rightarrow \ell^2(\mathbb{Z}^d)$ be the zero extension. Then, by \eqref{vanishing-boundary-data1}, \eqref{propertyofchi4}, and \eqref{constantofchi}, one has
\begin{equation}
    \fint_{D+n} V_1(x)\,dx=E\xi(n), \qquad \forall\, n\in \mathbb{Z}^d .
\end{equation}
Therefore, 
\begin{equation}\label{estimateCnablaW}
    |D|\langle C^{\mathrm{t}}_{\varepsilon} \xi, \xi \rangle_{\ell^2(\mathcal{I}_\varepsilon)}  \leq  \int_{\mathbb{R}^d} a_{\varepsilon^2,\mathbb{Z}^d} |\nabla V_1|^2\,dx =  \int_{\mathbb{R}^d} a_{\varepsilon^2,\mathcal{I}_\varepsilon} |\nabla V_1|^2\,dx .
\end{equation}

Moreover, $\int_{D+n} V_2(x)\,dx=0$ for every $n \in \mathcal{I}_\varepsilon$. Therefore,
\begin{equation}
    \int_{\mathbb{R}^d} a_{\varepsilon^2,\mathcal{I}_\varepsilon} \nabla V \cdot \overline{\nabla V_2} \,dx =0.
\end{equation}
Consequently,
\begin{equation}\label{estinablaW}
\begin{aligned}
    \int_{\mathbb{R}^d} a_{\varepsilon^2,\mathcal{I}_\varepsilon} |\nabla V_1|^2\,dx &= \int_{\mathbb{R}^d} a_{\varepsilon^2,\mathcal{I}_\varepsilon} |\nabla V|^2\,dx + \int_{\mathbb{R}^d} a_{\varepsilon^2,\mathcal{I}_\varepsilon} |\nabla V_2|^2\,dx \\
    &=|D| \langle C^{\mathrm{f}}_{\varepsilon} \xi, \xi \rangle_{\ell^2(\mathcal{I}_\varepsilon)}  + \int_{\mathbb{R}^d} a_{\varepsilon^2,\mathcal{I}_\varepsilon} |\nabla V_2|^2\,dx .
\end{aligned}
\end{equation}
Substituting \eqref{estinablaW} into \eqref{estimateCnablaW} implies
\begin{equation}\label{decayR1}
    |D|\langle R_{\varepsilon} \xi, \xi \rangle_{\ell^2(\mathcal{I}_\varepsilon)}   = |D|\langle C^{\mathrm{t}}_{\varepsilon} \xi, \xi \rangle_{\ell^2(\mathcal{I}_\varepsilon)}   - |D|\langle C^{\mathrm{f}}_{\varepsilon} \xi, \xi \rangle_{\ell^2(\mathcal{I}_\varepsilon)}  \leq \int_{\mathbb{R}^d} a_{\varepsilon^2,\mathcal{I}_\varepsilon} |\nabla V_2|^2\,dx.
\end{equation}
We now estimate the right-hand side of \eqref{decayR1}.

Using the identity $\nabla V_2=(1-\chi)\nabla V -V\nabla \chi$ and the fact that $\chi$ is a constant in each bubble $D+n$, we obtain
\begin{equation}
    \int_{\mathbb{R}^d} a_{\varepsilon^2,\mathcal{I}_\varepsilon} |\nabla V_2|^2\,dx \leq C\int_{\mathcal{O}_{r_2}}
    a_{\varepsilon^2,\mathcal{I}_\varepsilon} |\nabla V|^2\,dx +C \int_{\mathbb{R}^d} |V|^2|\nabla \chi|^2\,dx.
\end{equation}
By the same reasoning of \eqref{data0poincvf}, the cell-wise Poincar\'{e} inequality gives 
\begin{equation}
    \int_{\mathbb{R}^d} |V|^2|\nabla \chi|^2\,dx \leq CE(r_2).
\end{equation}
Therefore,
\begin{equation}
   \int_{\mathbb{R}^d} a_{\varepsilon^2,\mathcal{I}_\varepsilon} |\nabla V_2|^2\,dx \leq CE(r_2).
\end{equation}
Applying \eqref{estiErbyC} we get
\begin{equation}\label{decayR2}
    \int_{\mathbb{R}^d} a_{\varepsilon^2,\mathcal{I}_\varepsilon} |\nabla V_2|^2\,dx  \leq C\rho^{m/3} \langle C^{\mathrm{f}}_{\varepsilon } \xi,\xi \rangle_{\ell^2(\mathcal{I}_\varepsilon)} \leq C\rho^{m/3} \| \xi\|^2_{\ell^2(\mathcal{I}_\varepsilon)}.
\end{equation}
Here, we used the uniform operator bound $C^{\mathrm{f}}_{\varepsilon } \leq C I$ in \eqref{monotoneofC}. Therefore, since the capacitance matrices are Hermitian and $R_\varepsilon$ is nonnegative, by taking the square root of $R_\varepsilon$, one gets
\begin{equation}\label{R1/2psi}
    \| R_\varepsilon^{1/2} \xi \|^2_{\ell^2(\mathcal{I}_\varepsilon)}    =\langle R_{\varepsilon} \xi, \xi \rangle_{\ell^2(\mathcal{I}_\varepsilon)}   \leq C\rho^{m/3} \| \xi\|^2_{\ell^2(\mathcal{I}_\varepsilon)}.
\end{equation}

\emph{Step 2. }
Let $\psi \in \ell^2(\mathcal{I}_\varepsilon)$ be arbitrary. For $m\geq 0$, let $Q_m:\ell^2(\mathcal{I}_\varepsilon)\rightarrow \ell^2(\mathcal{I}_\varepsilon)$ denote the projection onto the $(m+1)$-th cell layer; see \eqref{eq:definition-of-layer-projection}.
By Cauchy-Schwarz, one has
\begin{equation}
    \begin{aligned}
        \langle R_{\varepsilon} \psi, \psi \rangle_{\ell^2(\mathcal{I}_\varepsilon)}  &= \| R_\varepsilon^{1/2} \psi \|^2_{\ell^2(\mathcal{I}_\varepsilon)}  \leq \left( \sum_{m\in \mathbb{N}} \rho^{m/12} \rho^{-m/12}\| R_\varepsilon^{1/2}  Q_m \psi \|_{\ell^2(\mathcal{I}_\varepsilon)} \right)^2 \\
        & \leq \left( \sum_{m\in \mathbb{N}} \rho^{m/6}  \right) \left( \sum_{m\in \mathbb{N}} \rho^{-m/6}\| R_\varepsilon^{1/2}  Q_m \psi \|^2_{\ell^2(\mathcal{I}_\varepsilon)} \right) \\
        & \leq C\sum_{m\in \mathbb{N}} \rho^{m/6}\| Q_m \psi \|^2_{\ell^2(\mathcal{I}_\varepsilon)} ,
    \end{aligned}
\end{equation}
where we used \eqref{R1/2psi} with $\xi= Q_m \psi $ for the last inequality. Finally, since 
\begin{equation}
    \|Q_m\psi\|_{\ell^2(\mathcal{I}_\varepsilon)}^2 = \sum_{d_{\varepsilon}(n)=m+1} |\psi(n)|^2,
\end{equation}
therefore,
\begin{equation}
    \langle R_{\varepsilon} \psi, \psi \rangle_{\ell^2(\mathcal{I}_\varepsilon)} \leq C \sum_{n\in \mathcal{I}_\varepsilon} \rho^{d_{\varepsilon}(n)/6} |\psi(n)|^2 =C\mathcal{M}_{\partial\Omega,\varepsilon} (\psi  ) ,
\end{equation}
where $\gamma=-\frac{1}{6}\log \rho >0$ in the definition of $\mathcal{M}_{\partial\Omega,\varepsilon} (\psi  ) $. The proof is complete.
\end{proof}

We introduce the sesquilinear forms $\mathfrak{b}^{\mathrm{t}}_\varepsilon$ and $\mathfrak{b}^{\mathrm{f}}_\varepsilon$ by
\begin{equation}\label{defbeps}
    \mathfrak{b}^{j}_\varepsilon(\phi,\psi) := \big\langle (L^{j}_\varepsilon +I ) \phi, \psi \big\rangle_{\ell^2(\mathcal{I}_\varepsilon)},
    \qquad \forall\,j= \mathrm{t}, \mathrm{f}, \ \forall \,\phi, \psi \in \ell^2(\mathcal{I}_\varepsilon).
\end{equation}
For simplicity, we use the short notation $\mathfrak{b}^{j}_\varepsilon(\phi) = \mathfrak{b}^{j}_\varepsilon(\phi,\phi)$.

An immediate corollary of Lemma \ref{lemlocalbounR} is the following equivalence result.

\begin{lemma}\label{lem:bjequivXeps}
    There exists $c=c(\Omega,D),\,C=C(\Omega,D)>0$ such that
    \begin{equation}
    c \| \phi\|_{\mathcal{X}_\varepsilon}^2 \leq \mathfrak{b}^j_\varepsilon(\phi) \leq C  \| \phi\|_{\mathcal{X}_\varepsilon}^2 , \qquad \forall\, j=\mathrm{t},\mathrm{f}, \,\forall\,\phi \in \ell^2(\mathcal{I}_\varepsilon).
\end{equation}
\end{lemma}

\begin{proof}
    By the definition of $L^{\mathrm{t}}_\varepsilon$ and the Plancherel identity, we have
    \begin{equation}\label{fourierbeps}
\begin{aligned}
    \mathfrak{b}^{\mathrm{t}}_\varepsilon(\phi,\psi) &=\langle L^{\mathrm{t}}_\varepsilon  \phi, \psi \rangle_{\ell^2(\mathcal{I}_\varepsilon)} + \langle \phi, \psi \rangle_{\ell^2(\mathcal{I}_\varepsilon)} \\
    &= \frac{1}{\varepsilon^2} \left\langle \big( \widehat{C}^M_{\varepsilon^2} I- \mathfrak{C}_{\varepsilon^2}\big) E S_M \phi, E S_M \psi \right\rangle_{\ell^2( \mathbb{Z}^d )} + \langle E\phi, E\psi \rangle_{\ell^2( \mathbb{Z}^d )}\\
    &= \frac{1}{|\mathcal{B}_1|} \int_{\mathcal{B}_1} \left(\frac{ g_\varepsilon(\beta) }{\varepsilon^2} +1 \right) \widehat{E\phi}(\beta)  \overline{ \widehat{E \psi}(\beta)  }\,d\beta ,
\end{aligned} 
\end{equation}
where
\begin{equation}
    g_\varepsilon(\beta)=  \widehat{C}^M_{\varepsilon^2} - \widehat{C}^{M+\beta}_{\varepsilon^2}, \qquad \forall\, \beta \in \mathcal{B}_1.
\end{equation}

We further define functions
\begin{equation}\label{defg0}
    g_0(\beta)=\widehat{C}^M - \widehat{C}^{M+\beta}, \quad \forall\, \beta \in \mathcal{B}_1, \qquad \mathrm{and} \qquad  r_\varepsilon(\alpha)= \widehat{C}_{\varepsilon^2}^\alpha - \widehat{C}^\alpha, \quad \forall\, \alpha \in \mathcal{B}_1 .
\end{equation}
By Lemma \ref{maximumCalphalem}, $M$ is the unique common maximal point of $\widehat{C}_{\varepsilon^2}^\alpha $ and $\widehat{C}^\alpha$. It follows that
\begin{equation}
    \nabla_\alpha r_\varepsilon(M) =0.
\end{equation}
Moreover, Proposition \ref{lem:analyticity-capacitance} implies that $\| r_\varepsilon \|_{C^2(\mathcal{B}_1)} \leq C\varepsilon^2$. Then, Taylor's formula yields
\begin{equation}\label{differencegepsg0}
    \begin{aligned}
        |g_\varepsilon(\beta)-g_0(\beta)| &= | r_\varepsilon(M)-r_\varepsilon(M+\beta) | \leq C \| r_\varepsilon \|_{C^2(\mathcal{B}_1)}  |\beta|^2 \leq C\varepsilon^2 \theta(\beta),
    \end{aligned}
\end{equation}
where $\theta(\beta)$, defined in \eqref{fourierDN}, is the Fourier multiplier of the discrete Dirichlet energy $\mathfrak{D}_\varepsilon$.

Again, since $M$ is the unique maximal point of $\widehat{C}^\alpha$ and $\nabla^2_\alpha \widehat{C}^M <0$, it follows that
\begin{equation}\label{g0theta}
    c\theta(\beta) \leq g_0(\beta) \leq C\theta(\beta) , \qquad \forall\,\beta \in \mathcal{B}_1. 
\end{equation}
The function $g_0$ above can be replaced by $g_\varepsilon$ due to \eqref{differencegepsg0}. As a consequence of the representations \eqref{fourierDN} and \eqref{fourierbeps}, and the comparability between $g_\varepsilon$ and $\theta$, we conclude that
\begin{equation}
    c\| \phi\|_{\mathcal{X}_\varepsilon}^2 \leq \mathfrak{b}^{\mathrm{t}}_\varepsilon(\phi) \leq C  \| \phi\|_{\mathcal{X}_\varepsilon}^2, \qquad \forall\, \phi \in \ell^2(\mathcal{I}_\varepsilon).
\end{equation}
Moreover, by Lemma \ref{lemlocalbounR} and Lemma \ref{lem:discreteHardy}, one has
\begin{equation}
    0\leq \mathfrak{b}^{\mathrm{f}}_\varepsilon(\phi) - \mathfrak{b}^{\mathrm{t}}_\varepsilon(\phi) = \frac{1}{\varepsilon^2} \langle R_\varepsilon S_M \phi, S_M\phi \rangle_{\ell^2(\mathcal{I}_\varepsilon)} \leq C\| \phi \|_{\mathcal{X}_\varepsilon}^2.
\end{equation}
The proof is complete.
\end{proof}

We recall some regularity results and boundary layer estimates for elliptic equations on bounded Lipschitz domains. 

For any $h\in \mathbb{R}^d$, let $\tau_h$ be the translation operator defined by $\tau_h u(x) = u(x+h)$. The difference operator $\Delta_h$ is given by
\begin{equation}
    \Delta_h u(x) =(\tau_h - I) u(x) = u(x+h)-u(x).
\end{equation}
We use the following difference characterization of Besov spaces. For $0<s<1$, define
\begin{equation}\label{besov1}
    B^s_{2,\infty}(\mathbb{R}^d):= \left\{ u \in L^2(\mathbb{R}^d) : \sup_{0<|h|<1} \frac{\| \Delta_h u \|_{L^2(\mathbb{R}^d)}}{|h|^s} <\infty \right\},
\end{equation}
equipped with the norm
\begin{equation}
    \| u \|_{B^s_{2,\infty}(\mathbb{R}^d)} := \|u \|_{L^2(\mathbb{R}^d) } + \sup_{0<|h|<1} \frac{\| \Delta_h u \|_{L^2(\mathbb{R}^d)}}{|h|^s}.
\end{equation}
Similarly, we define
\begin{equation}\label{besov2}
    B^{1+s}_{2,\infty}(\mathbb{R}^d):= \left\{ u \in H^1(\mathbb{R}^d) : \frac{\partial u}{\partial x_j} \in B^s_{2,\infty}(\mathbb{R}^d), \,\forall \,j=1,\cdots,d \right\},
\end{equation}
equipped with the norm
\begin{equation}\label{defbesov1s}
    \| u \|_{B^{1+s}_{2,\infty}(\mathbb{R}^d)} := \|u \|_{H^1(\mathbb{R}^d) } + \sup_{0<|h|<1} \frac{\| \Delta_h \nabla u \|_{L^2(\mathbb{R}^d)}}{|h|^s}.
\end{equation}
These difference characterizations are equivalent to the standard definitions of Besov spaces based on dyadic decompositions; see \cite[2.5.12]{Triebel2010Theory} and \cite[(2.34)]{savare2002domain}.

The following endpoint regularity estimate will play a key role in our analysis.

\begin{lemma}\label{lemBesov}
    Let $\Omega\subset \mathbb{R}^d$ be a bounded Lipschitz domain, and let $u_D = (L_D +I)^{-1} f \in H^1_0(\Omega)$ for some $f \in L^2(\Omega)$. Denote by $U_D\in H^1(\mathbb{R}^d)$ the extension of $u_D$ by zero outside $\Omega$. Then there exists $C=C(L_D,\Omega)>0$ such that
    \begin{equation}
        \| U_D \|_{B^{3/2}_{2,\infty}(\mathbb{R}^d)} \leq C \| f \|_{L^2(\Omega)}.
    \end{equation}
\end{lemma}
\begin{proof}
    This follows from the endpoint regularity result established in \cite[Section 6]{savare2002domain}.
\end{proof}

An immediate corollary of Lemma \ref{lemBesov} is the following translation estimate and the boundary layer estimate.

\begin{lemma}\label{lemboundarylayerestiuD}
Let $\Omega\subset \mathbb{R}^d$ be a bounded Lipschitz domain, and let $u_D = (L_D +I)^{-1} f \in H^1_0(\Omega)$ for some $f \in L^2(\Omega)$. Denote by $U_D\in H^1(\mathbb{R}^d)$ the extension of $u_D$ by zero outside $\Omega$. Then there exists $C=C(L_D,\Omega)>0$ such that
\begin{equation}\label{translaesti}
    \|\Delta_{h}\nabla U_D\|_{L^2(\mathbb R^d)} \leq C|h|^{1/2}\|f\|_{L^2(\Omega)}, \qquad \forall h\in \mathbb{R}^d.
\end{equation}
Moreover, for every $0<t<1$, we have
    \begin{equation}\label{estinablauD}
        \int_{\Omega_t} |\nabla u_D|^2\,dx \leq Ct \int_\Omega |f|^2\,dx, \qquad \int_{\Omega_t} |u_D|^2\,dx \leq Ct^3 \int_\Omega |f|^2\,dx,
    \end{equation}
where $\Omega_t:= \{x\in \Omega: \mathrm{dist}\,(x,\partial \Omega) < t\}$.
\end{lemma}

\begin{proof}
    The estimate \eqref{translaesti} directly follows from Lemma \ref{lemBesov} and the definition \eqref{defbesov1s} of Besov spaces for $s= 1/2$. We now prove \eqref{estinablauD}.

    For large $t$, the desired estimates follow from the elementary energy estimate. Thus, it suffices to assume that $t$ is sufficiently small such that $\Omega_t$ is covered by a finite number of boundary cylinders. In each boundary cylinder $Q$, after suitable rotations and translations, one has
    \begin{equation}\label{inclusionVt}
        V_t\subset \Omega_t \cap Q \subset V_{Mt},
    \end{equation}
    where $M$ depends only on the Lipschitz character of $\Omega$, and $V_t$ denotes the vertical strip
    \begin{equation}
        V_t:= \{(x',x_d) \in \mathbb{R}^{d-1} \times \mathbb{R}: |x'|<R,\,g(x')+t>x_d > g (x')\}.
    \end{equation}
    Then, for every $x\in \Omega_t \cap Q$, $\tau_{- M t e_d} \nabla U_D(x)=0$. By \eqref{translaesti}, one gets
    \begin{equation}\label{estimatefirstcover}
        \int_{\Omega_t \cap Q} |\nabla u_D|^2\,dx = \int_{\Omega_t \cap Q} |\Delta_{-Mte_d}\nabla U_D|^2\,dx \leq Ct\int_\Omega |f|^2\,dx.
    \end{equation}
    Summing over the finite covering gives the first estimate of \eqref{estinablauD}. 

    By the Newton-Leibniz formula for $H^1$ functions, and the zero trace property of $u_D$, we have
    \begin{equation}
        u_D(x',g(x')+s) = \int_0^s \frac{\partial}{\partial x_d} u_D(x',g(x')+r) \,dr, \qquad \forall \,|x'|<R,\,0<s<Mt.
    \end{equation}
    By \eqref{inclusionVt}, Fubini, and Cauchy-Schwarz, one has
    \begin{equation}
        \begin{aligned}
            \int_{\Omega_t \cap Q} |u_D|^2\,dx  & \leq \int_{V_{Mt}} |u_D|^2\,dx = \int_{|x'|<R} \int_0^{Mt}|u_D(x',g(x')+s)|^2\,dsdx'
            \\
            & \leq \int_{|x'|<R} \int_0^{Mt} s \int_0^s \left|\frac{\partial}{\partial x_d}u_D(x',g(x')+r) \right|^2\,drdsdx' \\
            & \leq \frac{M^2 t^2}{2} \int_{|x'|<R} \int_0^{Mt} \left|\frac{\partial}{\partial x_d}u_D(x',g(x')+r) \right|^2\,drdx' \\
            & \leq \frac{M^2 t^2}{2} \int_{\Omega_{Mt}\cap Q} |\nabla u_D(x)|^2\,dx
        \end{aligned}
    \end{equation}
    After summing over the finite covering, and combined with \eqref{estimatefirstcover}, this gives the second estimate of \eqref{estinablauD}.
\end{proof}

Recall that $J_\varepsilon:\ell^2(\mathcal I_\varepsilon)\rightarrow L^2(\Omega)$ is the reconstruction operator. We identify $L^2(\Omega)$ with the closed subspace of $L^2(\mathbb R^d)$ consisting of functions that vanish almost everywhere outside $\Omega$, and we use the same notation $J_\varepsilon$ for the zero extension of the reconstruction operator:
\begin{equation}
    J_\varepsilon:\ell^2(\mathcal I_\varepsilon)\rightarrow L^2(\mathbb R^d).
\end{equation}
Under this identification, its adjoint is the operator $J_\varepsilon^*:L^2(\mathbb R^d)\rightarrow\ell^2(\mathcal I_\varepsilon)$ given by
\begin{equation}\label{eq:adjoint-of-reconstruction-on-full-space}
    (J_\varepsilon^*u)(n)=\varepsilon^{-d/2}\int_{\varepsilon(n+Y)}u(x)\,dx, \qquad \forall u \in L^2(\mathbb R^d),\,\forall\,n\in\mathcal I_\varepsilon.
\end{equation}
The composition $\Pi_\varepsilon=J_\varepsilon J_\varepsilon^*$ can therefore be regarded as the finite-lattice projection, which maps a function on $\mathbb{R}^d$ to its average on every cell $\varepsilon(n+Y)$, $n\in\mathcal I_\varepsilon$, and vanishes on all other cells.

By contrast, let $\widehat\Pi_\varepsilon:L^2(\mathbb R^d)\rightarrow L^2(\mathbb R^d)$ denote the full-lattice cell-average projection, defined by
\begin{equation}\label{eq:full-lattice-cell-average-projection}
    (\widehat\Pi_\varepsilon u)(x):=\fint_{\varepsilon(n+Y)}u(y)\,dy, \qquad \forall\,x\in\varepsilon(n+Y), \,\forall n\in\mathbb Z^d.
\end{equation}
It is easy to see that
\begin{equation}
    J_\varepsilon^*\widehat\Pi_\varepsilon=J_\varepsilon^*, \qquad \widehat\Pi_\varepsilon J_\varepsilon= J_\varepsilon.
\end{equation}

The following estimate is elementary.

\begin{lemma}
    We have
\begin{equation}\label{eq:cell-variance-estimate-for-full-lattice-projection}
    \|(I-\widehat{\Pi}_\varepsilon)H\|_{L^2(\mathbb R^d)}^2\leq\frac{1}{2\varepsilon^d}\int_{[-\varepsilon,\varepsilon]^d}\|\Delta_hH\|_{L^2(\mathbb R^d)}^2\,dh, \qquad \forall\, H\in L^2(\mathbb R^d).
\end{equation}
\end{lemma}

\begin{proof}
Note that $\widehat{\Pi}_\varepsilon H$ is the average of $H$ on each cell $\varepsilon(m+Y)$. The variance identity gives
\begin{equation}\label{eq:exact-cell-variance-identity}
    \int_{\varepsilon(m+Y)}|H(x)-\widehat{\Pi}_\varepsilon H|^2\,dx=\frac{1}{2\varepsilon^d}\int_{\varepsilon(m+Y)}\int_{\varepsilon(m+Y)}|H(y)-H(x)|^2\,dydx.
\end{equation}
Summing \eqref{eq:exact-cell-variance-identity} over $m\in\mathbb Z^d$, we obtain
\begin{equation}\label{eq:summed-cell-variance-identity}
    \|(I-\widehat{\Pi}_\varepsilon)H\|_{L^2(\mathbb R^d)}^2=\frac{1}{2\varepsilon^d}\sum_{m\in\mathbb Z^d}\int_{\varepsilon(m+Y)}\int_{\varepsilon(m+Y)}|H(y)-H(x)|^2\,dydx.
\end{equation}
For two points $x,y$ belonging to the same cell $\varepsilon(m+Y)$, their difference $h=y-x$ belongs to $[-\varepsilon,\varepsilon]^d$. Making the change of variables $y=x+h$ in \eqref{eq:summed-cell-variance-identity} therefore gives
\begin{equation}\label{eq:cell-variance-with-restricted-translations}
    \|(I-\widehat{\Pi}_\varepsilon)H\|_{L^2(\mathbb R^d)}^2=\frac{1}{2\varepsilon^d}\int_{[-\varepsilon,\varepsilon]^d}\sum_{m\in\mathbb Z^d}\int_{\varepsilon(m+Y)\cap(\varepsilon(m+Y)-h)}|\Delta_hH(x)|^2\,dxdh.
\end{equation}
For each fixed $h$, the sets $\varepsilon(m+Y)\cap(\varepsilon(m+Y)-h)$ are mutually disjoint for different $m\in \mathbb{Z}^d$. Consequently,
\begin{equation} \label{est564}
    \sum_{m\in\mathbb Z^d}\int_{\varepsilon(m+Y)\cap(\varepsilon(m+Y)-h)}|\Delta_hH(x)|^2\,dx\leq\int_{\mathbb R^d}|\Delta_hH(x)|^2\,dx.
\end{equation}
Substitution of \eqref{est564} into \eqref{eq:cell-variance-with-restricted-translations} proves \eqref{eq:cell-variance-estimate-for-full-lattice-projection}.
\end{proof}

\begin{lemma}
\label{lem:projection-and-boundary-depth}
Let $u_D = (L_D +I)^{-1} f \in H^1_0(\Omega)$ for some $f \in L^2(\Omega)$. Denote by $U_D\in H^1(\mathbb{R}^d)$ the extension of $u_D$ by zero outside $\Omega$. Then, we have
\begin{equation}\label{eq:projection-and-discrete-energy-estimates}
    \left\|\widehat{\Pi}_\varepsilon U_D-\Pi_\varepsilon u_D\right\|_{L^2(\mathbb{R}^d)}\leq
        C\varepsilon^{3/2}\|f\|_{L^2(\Omega)}, \qquad 
        \left\|\Pi_\varepsilon u_D-U_D\right\|_{L^2(\mathbb{R}^d)}\leq C\varepsilon\|f\|_{L^2(\Omega)}.
\end{equation}
Moreover, $J^*_\varepsilon u_D$ admits the estimates
\begin{equation}\label{eq:weighted-boundary-depth-estimate}
    \|J_\varepsilon^*u_D\|_{\mathcal X_\varepsilon}\leq C\|f\|_{L^2(\Omega)}, \qquad  \mathcal{M}_{\partial \Omega,\varepsilon} (J_\varepsilon^*u_D )\leq C\varepsilon^3\|f\|_{L^2(\Omega)}^2.
\end{equation}
\end{lemma}

\begin{proof}
The functions $\widehat{\Pi}_\varepsilon U_D$ and $\Pi_\varepsilon u_D$ agree on every cell $\varepsilon(n+Y)$ with $n\in\mathcal{I}_\varepsilon$. Therefore,
\begin{equation}\label{eq:unselected-cell-projection-estimate}
    \begin{aligned}
        \left\| \widehat{\Pi}_\varepsilon U_D-\Pi_\varepsilon u_D \right\|_{L^2(\mathbb{R}^d)}^2 &\leq \sum_{\varepsilon(n+\overline{Y})\cap\partial \Omega\neq\emptyset} \int_{\varepsilon(n+Y)}|U_D(x)|^2\,dx\\
        &\leq\int_{\Omega_{\sqrt{d} \varepsilon}}|u_D(x)|^2\,dx.
    \end{aligned}
\end{equation}
This, combined with the boundary layer estimate \eqref{estinablauD} for $u_D$, yields 
\begin{equation}\label{eq:boundary-cell-error-final}
    \left\|\widehat{\Pi}_\varepsilon U_D-\Pi_\varepsilon u_D \right\|_{L^2(\mathbb{R}^d)}^2\leq C\varepsilon^3\|f\|_{L^2(\Omega)}^2,
\end{equation}
which is the first estimate in \eqref{eq:projection-and-discrete-energy-estimates}. The rescaled Poincar\'e inequality and the standard energy estimate for $u_D$ give
\begin{equation}\label{eq:full-lattice-cell-poincare}
    \left\| \widehat{\Pi}_\varepsilon U_D-U_D \right\|_{L^2(\mathbb{R}^d)} \leq C\varepsilon \|f\|_{L^2(\Omega)}.
\end{equation}
Combining \eqref{eq:boundary-cell-error-final} and \eqref{eq:full-lattice-cell-poincare}, we obtain the second estimate in \eqref{eq:projection-and-discrete-energy-estimates}.

We next prove \eqref{eq:weighted-boundary-depth-estimate}. By the isometric property of $J_\varepsilon$, one has
\begin{equation}\label{eq:discrete-energy-difference-quotient-identity}
    \varepsilon^{-2} \mathfrak{D}_\varepsilon (J_\varepsilon^*u_D)= \sum_{i=1}^d\left\|
    \frac{\Delta_{\varepsilon e_i} }{\varepsilon} \Pi_\varepsilon u_D\right\|_{L^2(\mathbb{R}^d)}^2.
\end{equation}
Since $\widehat{\Pi}_\varepsilon$ commutes with $\Delta_{\varepsilon e_i}$, 
\begin{equation}\label{eq:full-lattice-difference-quotient-estimate}
    \left\|\frac{\Delta_{\varepsilon e_i} }{\varepsilon} \widehat{\Pi}_\varepsilon U_D \right\|_{L^2(\mathbb{R}^d)}^2
    =\left\| \widehat{\Pi}_\varepsilon \frac{\Delta_{\varepsilon e_i} }{\varepsilon} U_D \right\|_{L^2(\mathbb{R}^d)}^2
    \leq \|\nabla U_D\|_{L^2(\mathbb{R}^d)}^2 \leq C \|f\|_{L^2(\Omega)}^2.
\end{equation}
Thus, \eqref{eq:boundary-cell-error-final}, \eqref{eq:discrete-energy-difference-quotient-identity}, and \eqref{eq:full-lattice-difference-quotient-estimate} give
\begin{equation}\label{eq:selected-difference-quotient-estimate}
    \begin{aligned}
        \varepsilon^{-2} \mathfrak{D}_\varepsilon (J_\varepsilon^*u_D)
        &\leq C\sum_{i=1}^d \left\|\frac{\Delta_{\varepsilon e_i}}{\varepsilon}\widehat{\Pi}_\varepsilon U_D \right\|_{L^2(\mathbb{R}^d)}^2 +
        \frac{C}{\varepsilon^2} \left\|\widehat{\Pi}_\varepsilon U_D-\Pi_\varepsilon u_D\right\|^2_{L^2(\mathbb{R}^d)}\\
        &\leq C\|f\|^2_{L^2(\Omega)} + C\varepsilon \|f\|^2_{L^2(\Omega)}\leq C\|f\|^2_{L^2(\Omega)}.
    \end{aligned}
\end{equation}

Moreover, since $\Pi_\varepsilon$ is a projection, we have 
\begin{equation}\label{eq:ell2-part-of-Xepsilon-estimate}
    \left\|J_\varepsilon^*u_D \right\|_{\ell^2(\mathcal{I}_\varepsilon)}=\|\Pi_\varepsilon u_D\|_{L^2(\mathbb{R}^d)}
    \leq\|u_D\|_{L^2(\Omega)} \leq C\|f\|_{L^2(\Omega)}.
\end{equation}
Therefore, \eqref{eq:selected-difference-quotient-estimate} and \eqref{eq:ell2-part-of-Xepsilon-estimate} prove the first estimate in \eqref{eq:weighted-boundary-depth-estimate}. 

It remains to prove the second estimate in \eqref{eq:weighted-boundary-depth-estimate}. Suppose that $d_\varepsilon(n)=k$. It is clear that
\begin{equation}\label{eq:cell-contained-in-geometric-boundary-layer}
    \varepsilon(n+Y)\subset \Omega_{\sqrt{d}(k+1)\varepsilon}.
\end{equation}
Applying the boundary-layer estimate \eqref{estinablauD} for $u_D$, we obtain
\begin{equation}\label{eq:weighted-boundary-mass-calculation}
    \begin{aligned}
        \mathcal{M}_{\partial \Omega,\varepsilon} (J_\varepsilon^*u_D) 
        &= \sum_{k\geq1}e^{-\gamma k} \sum_{\substack{n\in\mathcal{I}_\varepsilon\atop  d_\varepsilon(n)=k}}
        \left|(J_\varepsilon^*u_D)(n)\right|^2
        \leq\sum_{k\geq1}e^{-\gamma k}\int_{\Omega_{\sqrt{d}(k+1)\varepsilon}}|u_D(x)|^2\,dx\\
        &\leq C\varepsilon^3\|f\|_{L^2(\Omega)}^2\sum_{k\geq1}e^{-\gamma k}(k+1)^3
        \leq C\varepsilon^3\|f\|_{L^2(\Omega)}^2.
    \end{aligned}
\end{equation}
This completes the proof.
\end{proof}

The piecewise-constant reconstruction $J_\varepsilon\varphi$ contains the exact discrete information of $\varphi$ at each bubble.
However, $J_\varepsilon\varphi$ generally has jumps across the cell interfaces and therefore belongs only to $L^2(\Omega)$. To prove the continuous limit of the capacitance matrix, we need a finer reconstruction that preserves the cell averages and lies in the energy space $H_0^1(\Omega)$. This reconstruction is given in the next lemma.

\begin{lemma}\label{lemKeps}
    There exist $C=C(\Omega)>0$ and a linear map $K_\varepsilon:\ell^2(\mathcal{I}_\varepsilon) \rightarrow H_0^1(\Omega)$ such that 
    \begin{equation}\label{estiK}
    \widehat  \Pi_\varepsilon  K_\varepsilon =  J_\varepsilon,\qquad \| K_\varepsilon\|_{\mathcal{X}_\varepsilon \rightarrow H^1(\Omega)} \leq C, \qquad  \| K_\varepsilon - J_\varepsilon \|_{\mathcal{X}_\varepsilon \rightarrow L^2(\Omega)} \leq C\varepsilon,
    \end{equation}
    where we identify $K_\varepsilon$ with its zero extension. In particular, $J_\varepsilon^* K_\varepsilon  = I$.
\end{lemma}

\begin{proof}
    Given $\varphi\in \ell^2(\mathcal{I}_\varepsilon)$, let $u_\varepsilon = J_\varepsilon \varphi $. Recall that $\Omega_t=\{x\in\Omega: \mathrm{dist}\,(x,\partial \Omega)<t\}$ is the $t$-boundary layer of $\Omega$. By \eqref{hardyLayerEstimate}, for any $L>0$, there exists $C_L>0$ such that 
    \begin{equation}
    \sum_{ \varepsilon(n+Y)\cap\Omega_{L\varepsilon}\neq \emptyset }
    |\varphi(n)|^2\leq C_L\mathfrak{D}_\varepsilon(\varphi).
\end{equation}
On the other hand, the definition of $u_\varepsilon$ gives $\int_{\varepsilon(n+Y)}|u_\varepsilon|^2\,dx=|\varphi(n)|^2$ for $n\in \mathcal{I}_\varepsilon$. Consequently,
\begin{equation}\label{boundarystripestimateueps}
    \int_{\Omega_{L\varepsilon}}|u_\varepsilon|^2\,dx\leq \sum_{ \varepsilon(n+Y)\cap\Omega_{L\varepsilon}\neq\varnothing }|\varphi(n)|^2\leq C_L\mathfrak{D}_\varepsilon(\varphi).
\end{equation}
We divide the remainder of the proof into several steps.

    \textit{Step 1. Mollification and cutoff near the boundary.} For $0 \leq r\leq  \varepsilon$, translating $u_\varepsilon$ by $re_i$ only changes it in slabs of thickness $r$ adjacent to cell faces orthogonal to $e_i$. Therefore,
    \begin{equation}
        \| \tau_{re_i} u_\varepsilon - u_\varepsilon \|_{L^2(\mathbb{R}^d)}^2 \leq C r\varepsilon^{d-1} \varepsilon^{-d} \sum_{n\in \mathbb{Z}^d} |E\varphi(n+e_i)- E\varphi(n)|^2.
    \end{equation}
    For a vector $h\in \mathbb{R}^d$ such that $|h|\leq \varepsilon$, it follows that
    \begin{equation}\label{translaDirich}
        \| \tau_{h} u_\varepsilon - u_\varepsilon \|_{L^2(\mathbb{R}^d)}^2 \leq C \mathfrak{D}_\varepsilon
        (\varphi).
    \end{equation}
Let $\rho\in C_0^\infty(\mathbb{R}^d;[0,1])$ such that $\int_{\mathbb{R}^d} \rho\,dx=1$, and $\rho_\varepsilon$ be the standard smooth mollifier:
\begin{equation}
    \rho_\varepsilon (x) = \varepsilon^{-d} \rho(x/\varepsilon), \qquad \forall\,x\in \mathbb{R}^d.
\end{equation}
Define the mollification
\begin{equation}
    v_\varepsilon := \rho_\varepsilon * u_\varepsilon \in C^\infty_0(\mathbb{R}^d).
\end{equation}
It is clear that
\begin{equation}\label{estimateofvepsphi}
    \| v_\varepsilon \|_{L^2(\mathbb{R}^d)} \leq \| u_\varepsilon \|_{L^2(\mathbb{R}^d)} \leq C\| \varphi \|_{\ell^2(\mathcal{I}_\varepsilon)}.
\end{equation}
Moreover, using \eqref{translaDirich} and
\begin{equation}
    v_\varepsilon - u_\varepsilon = \int_{\mathbb{R}^d} \rho_\varepsilon(y) ( \tau_{-y}u_\varepsilon - u_\varepsilon )\,dy,
\end{equation}
we get
\begin{equation}\label{translationesti}
   \| \nabla v_\varepsilon \|_{L^2(\mathbb{R}^d)}^2 \leq \frac{C}{\varepsilon^2} \mathfrak{D}_\varepsilon  (\varphi) \qquad \mathrm{and} \qquad \| v_\varepsilon - u_\varepsilon \|_{L^2(\mathbb{R}^d)}^2 \leq C\mathfrak{D}_\varepsilon  (\varphi).
\end{equation}

Choose $K>\sqrt{d}$, and $\chi \in C^\infty (\mathbb{R})$ such that $\chi=0$ on $(-\infty,K)$, $\chi=1$ on $(2K,\infty)$. Define
\begin{equation}
    \zeta_\varepsilon(x):=\chi\left( \frac{\operatorname{dist}(x,\mathbb{R}^d\setminus \Omega)}{\varepsilon}\right) \quad
    \mathrm{and} \quad 
    h:=\zeta_\varepsilon v_\varepsilon .
\end{equation}
Since $K>\sqrt{d}$, $h$ is supported on the union of cells $\bigcup_{n\in \mathcal{I}_\varepsilon}\varepsilon(n+Y)$.
Since the distance function is Lipschitz, we have
\begin{equation}\label{gradientofzeta}
    |\nabla\zeta_\varepsilon| \leq C\varepsilon^{-1} \qquad\text{ in }\mathbb{R}^d.
\end{equation}

To estimate the error introduced by the cutoff, we write
\begin{equation}
    h-u_\varepsilon = (v_\varepsilon-u_\varepsilon) + (\zeta_\varepsilon-1)v_\varepsilon.
\end{equation}
Since $v_\varepsilon=\rho_\varepsilon*u_\varepsilon$, and $\int_{\mathbb{R}^d}\rho_\varepsilon\,dx=1$, Jensen's inequality gives
\begin{equation}
    |v_\varepsilon(x)|^2\leq \int_{\mathbb{R}^d} \rho_\varepsilon(y)|u_\varepsilon(x-y)|^2\,dy.
\end{equation}
Integrating this inequality over $\operatorname{supp}\,(\zeta_\varepsilon-1)$ and applying \eqref{boundarystripestimateueps} and \eqref{translationesti}, we obtain
\begin{equation}\label{diffehueps}
    \|h-u_\varepsilon\|^2_{L^2(\mathbb{R}^d)} \leq C\mathfrak{D}_\varepsilon(\varphi).
\end{equation}
Furthermore, since $\nabla h= \zeta_\varepsilon\nabla v_\varepsilon + v_\varepsilon\nabla\zeta_\varepsilon$ and $\mathrm{supp}\,\nabla \zeta_\varepsilon \subset \Omega_{2K\varepsilon}$, we have
\begin{equation}
    \begin{aligned}
         \|\nabla h\|_{L^2(\mathbb{R}^d)} &\leq C\| \nabla v_\varepsilon  \|_{L^2(\mathbb{R}^d)} + \| v_\varepsilon\nabla\zeta_\varepsilon \|_{L^2(\Omega_{2K\varepsilon})} \\
         & \leq \frac{C}{\varepsilon} \sqrt{ \mathfrak{D}_\varepsilon  (\varphi) } + \frac{1}{\varepsilon}\| v_\varepsilon \|_{L^2(\Omega_{2K\varepsilon})} \\
         & \leq \frac{C}{\varepsilon} \sqrt{ \mathfrak{D}_\varepsilon  (\varphi) } + \frac{1}{\varepsilon} \big( \| v_\varepsilon - u_\varepsilon \|_{L^2(\mathbb{R}^d)} + \| u_\varepsilon \|_{L^2(\Omega_{2K\varepsilon})} \big) \\
         & \leq C\| \varphi \|_{\mathcal{X}_\varepsilon},
    \end{aligned}
\end{equation}
where we used \eqref{translationesti} and \eqref{gradientofzeta} in the second inequality, and \eqref{translationesti} and \eqref{boundarystripestimateueps} in the last inequality.

\textit{Step 2. Construction of $K_\varepsilon$.} Define the average defect by
\begin{equation}
    r:=\varphi- J_\varepsilon^*h \in \ell^2(\mathcal{I}_\varepsilon).
\end{equation}
Since $u_\varepsilon= J_\varepsilon\varphi$ and $ J^*_\varepsilon   J_\varepsilon =I$, we have $r =  J_\varepsilon^*(u_\varepsilon-h)$. It follows from \eqref{diffehueps} that
\begin{equation}
    \|r\|^2_{\ell^2(\mathcal I_\varepsilon)} \leq \|u_\varepsilon-h\|^2_{L^2(\mathbb{R}^d)}\leq C\mathfrak{D}_\varepsilon(\varphi) .
\end{equation}

Choose $b\in C_0^\infty(\mathbb{R}^d)$ such that $\operatorname{supp}b\Subset Y$ and $\int_Y b(y)\,dy=1$. We now define
\begin{equation}
    (K_\varepsilon\varphi)(x):= h(x) + \varepsilon^{-d/2}\sum_{n\in \mathcal I_\varepsilon}r(n) b\left(\frac{x}{\varepsilon}-n\right).
\end{equation}
It follows that
\begin{equation}
    \begin{aligned}
        J_\varepsilon^*K_\varepsilon\varphi=J_\varepsilon^*h+r =\varphi.
    \end{aligned}
\end{equation}
On every cell $\varepsilon(n+Y)$ with $n\notin \mathcal{I}_\varepsilon$, $K_\varepsilon \varphi$ vanishes. Therefore,
$\widehat \Pi_\varepsilon  K_\varepsilon\varphi= J_\varepsilon\varphi$.

Finally, the estimates for $h$ and $r$ give
\begin{equation}
    \|K_\varepsilon\varphi\|_{H^1(\Omega)} \leq C\left( \|\varphi\|_{\ell^2(\mathcal{I}_\varepsilon)} +
 \varepsilon^{-1} \mathfrak{D}_\varepsilon(\varphi)^{1/2}\right)  \leq
    C\|\varphi\|_{\mathcal{X}_\varepsilon}.
\end{equation}
Similarly,
\begin{equation}
    \begin{aligned}
        \|K_\varepsilon\varphi-J_\varepsilon\varphi\|_{L^2(\Omega)}
        &\leq
        \|h-u_\varepsilon\|_{L^2(\mathbb{R}^d)}
        +
        C\|r\|_{\ell^2(\mathcal{I}_\varepsilon)} \\
        &\leq
        C\mathfrak{D}_\varepsilon(\varphi)^{1/2} \\
        &\leq
        C\varepsilon\|\varphi\|_{\mathcal{X}_\varepsilon}.
    \end{aligned}
\end{equation}
The proof is complete.
\end{proof}

\subsection{Rescaled limit of the capacitance matrices}

Define the limiting sesquilinear form
\begin{equation}
    \mathfrak{b}_D (u,v) := \int_\Omega A_{\mathrm{eff}} \nabla u \cdot \overline{\nabla v}\,dx + \int_\Omega u \overline{v}\,dx, \qquad
    \forall\, u,v\in H_0^1(\Omega).
\end{equation}

We first prove the convergence of the sesquilinear forms.

\begin{lemma}
\label{lem:two-quantitative-consistency-estimates}
Let $f,g\in L^2(\Omega)$, and define
\begin{equation}\label{eq:definition-of-two-continuum-resolvent-solutions}
    u_D:=(L_D+I)^{-1}f, \qquad v_D:=(L_D+I)^{-1}g.
\end{equation}
There exists $C=C(L_D,\Omega)>0$ such that for $j\in\{\mathrm{t},\mathrm{f}\}$, and every $\psi\in\ell^2(\mathcal I_\varepsilon)$, 
\begin{equation}\label{eq:mixed-quantitative-consistency-estimate}
    \left| \mathfrak b_\varepsilon^j(J_\varepsilon^*u_D,\psi) -\mathfrak b_D(u_D,K_\varepsilon\psi) \right| \leq C\varepsilon^{1/2}\|f\|_{L^2(\Omega)}\|\psi\|_{\mathcal X_\varepsilon}.
\end{equation}
Moreover,
\begin{equation}\label{eq:two-resolvent-quantitative-consistency-estimate}
    \left| \mathfrak b_\varepsilon^j(J_\varepsilon^*u_D,J_\varepsilon^*v_D) -\mathfrak b_D(u_D,v_D) \right| \leq C\varepsilon\|f\|_{L^2(\Omega)}\|g\|_{L^2(\Omega)}.
\end{equation}
\end{lemma}

\begin{proof}
Let $U_D,V_D\in H^1(\mathbb R^d)$ be the zero extensions of $u_D$ and $v_D$, respectively. We will repeatedly use the following Newton-Leibniz formula: for every $n\in \mathbb{R}^d$,
\begin{equation}\label{eq:averaged-derivative-formula-for-resolvent-solutions}
    \frac{\Delta_{\varepsilon n}U_D}{\varepsilon}=\int_0^1\tau_{t\varepsilon n}(n\cdot\nabla U_D)\,dt, \qquad \frac{\Delta_{\varepsilon n}V_D}{\varepsilon}=\int_0^1\tau_{t\varepsilon n}(n\cdot\nabla V_D)\,dt.
\end{equation}
We divide the proof into three steps.

\textit{Step 1. Reduction to the desired estimates for the full-lattice reference form.} Recall that
\begin{equation}\label{eq:definition-of-reference-symbol-in-consistency-proof}
    g_0(\beta)=\widehat C^M-\widehat C^{M+\beta}.
\end{equation}
The function $g_0$ is even, real-valued, analytic, and satisfies $g_0(0)=0$. It therefore admits the Fourier representation
\begin{equation}\label{eq:fourier-representation-of-reference-symbol}
    g_0(\beta)=\sum_{r\in\mathbb Z^d}\kappa_r\big(1-\cos(r\cdot\beta)\big), \qquad \kappa_{-r}=\kappa_r\in\mathbb R.
\end{equation}
The analyticity of $g_0$ gives exponential decay of the Fourier coefficients $\kappa_r$ as $|r|\rightarrow \infty$. Moreover, the definition of $A_{\mathrm{eff}}$ gives
\begin{equation}\label{eq:effective-matrix-from-fourier-coefficients}
    A_{\mathrm{eff}}=\frac12\nabla^2g_0(0)=\frac12\sum_{r\in\mathbb Z^d}\kappa_r r\otimes r.
\end{equation}

Let $\mathscr C_\varepsilon$ denote the space of functions in $L^2(\mathbb R^d)$ that are constant on every cell $\varepsilon(m+Y)$, $m\in\mathbb Z^d$. Every $W,Z\in\mathscr C_\varepsilon$ can be written uniquely as
\begin{equation}\label{eq:cellwise-coefficient-representation}
    W(x)=\varepsilon^{-d/2}w(m), \qquad Z(x)=\varepsilon^{-d/2}\zeta(m), \qquad \forall\,x\in\varepsilon(m+Y),
\end{equation}
for some $w,\zeta\in\ell^2(\mathbb Z^d)$. The normalization $\varepsilon^{-d/2}$ makes this identification isometric:
\begin{equation}\label{eq:cellwise-isometric-identification}
    \langle W,Z\rangle_{L^2(\mathbb R^d)}=\langle w,\zeta\rangle_{\ell^2(\mathbb Z^d)}.
\end{equation}
Consequently,
\begin{equation}\label{eq:physical-and-discrete-difference-inner-products}
    \left\langle\frac{\Delta_{\varepsilon r}W}{\varepsilon},\frac{\Delta_{\varepsilon r}Z}{\varepsilon}\right\rangle_{L^2(\mathbb R^d)}
    =\frac{1}{\varepsilon^2}\sum_{m\in\mathbb Z^d}\bigl(w(m+r)-w(m)\bigr)\overline{\bigl(\zeta(m+r)-\zeta(m)\bigr)}.
\end{equation}

We define the full-lattice reference form by
\begin{equation}\label{eq:full-lattice-reference-form}
    \mathfrak b_\varepsilon^{(0)}(W,Z):=\frac12\sum_{r\in\mathbb Z^d}\kappa_r\left\langle\frac{\Delta_{\varepsilon r}W}{\varepsilon},\frac{\Delta_{\varepsilon r}Z}{\varepsilon}\right\rangle_{L^2(\mathbb R^d)}+\langle W,Z\rangle_{L^2(\mathbb R^d)}.
\end{equation}
This series is absolutely convergent due to the exponential decay of $\kappa_r$. 

Next, we verify that \eqref{eq:full-lattice-reference-form} has the Fourier symbol $g_0(\beta)/\varepsilon^2+1$. Let $\varphi,\psi\in\ell^2(\mathcal I_\varepsilon)$, let $E\varphi,E\psi\in\ell^2(\mathbb Z^d)$ denote their zero extensions, and set
\begin{equation}
    W= J_\varepsilon\varphi, \qquad Z=  J_\varepsilon\psi.
\end{equation}
The coefficient sequences associated with $W$ and $Z$ in \eqref{eq:cellwise-coefficient-representation} are therefore $E\varphi$ and $E\psi$, respectively.

By the definition \eqref{defFourier} of the discrete Fourier transform, the difference $\xi(m+r)-\xi(m)$ has the Fourier multiplier $e^{-\mi  r\cdot\beta}-1$. Hence, by the Plancherel identity,
\begin{equation}\label{eq:Fourier-representation-of-one-lattice-difference}
    \left\langle\frac{\Delta_{\varepsilon r}W}{\varepsilon},\frac{\Delta_{\varepsilon r}Z}{\varepsilon}\right\rangle_{L^2(\mathbb R^d)}
    =\frac{1}{|\mathcal B_1|}\int_{\mathcal B_1}\frac{|e^{-\mi r\cdot\beta}-1|^2}{\varepsilon^2}\widehat{E\varphi}(\beta)\overline{\widehat{E\psi}(\beta)}\,d\beta.
\end{equation}
Since
\begin{equation}\label{eq:phase-difference-and-cosine}
    |e^{- \mi r\cdot\beta}-1|^2=2\big(1-\cos(r\cdot\beta)\big),
\end{equation}
substituting into \eqref{eq:full-lattice-reference-form}, followed by the Fourier representation of $g_0$, gives the exact identity
\begin{equation}\label{eq:Fourier-symbol-of-full-lattice-reference-form}
    \mathfrak b_\varepsilon^{(0)}\big( J_\varepsilon\varphi, J_\varepsilon\psi\big)
    =\frac{1}{|\mathcal B_1|}\int_{\mathcal B_1}\left(\frac{g_0(\beta)}{\varepsilon^2}+1\right)\widehat{E\varphi}(\beta)\overline{\widehat{E\psi}(\beta)}\,d\beta.
\end{equation}

Comparing this with the Fourier representation \eqref{fourierbeps} of $\mathfrak{b}_\varepsilon^{\mathrm{t}}$, we get
\begin{equation}\label{comparebtandb0}
\begin{aligned}
    & \left| \mathfrak b_\varepsilon^{\mathrm{t}}(J_\varepsilon^*u_D,\psi) -\mathfrak b_\varepsilon^{(0)}(\Pi_\varepsilon u_D,J_\varepsilon\psi) \right|^2 \\
    & = \left| \frac{1}{|\mathcal B_1|}\int_{\mathcal B_1}\left(\frac{ g_\varepsilon(\beta)-g_0(\beta)}{\varepsilon^2}\right)\widehat{EJ_\varepsilon^*u_D}(\beta)\overline{\widehat{E\psi}(\beta)}\,d\beta \right|^2 \\
    & \leq C\left( \int_{\mathcal B_1} \theta(\beta) |\widehat{EJ_\varepsilon^*u_D}(\beta) |^2\,d\beta  \right) \left( \int_{\mathcal B_1} \theta(\beta) | \widehat{E\psi}(\beta) |^2 \,d\beta \right) \\
    & = C \mathfrak{D}_\varepsilon (J_\varepsilon^* u_D) \mathfrak{D}_\varepsilon (\psi). 
\end{aligned}
\end{equation}
where we used \eqref{differencegepsg0} for the third line and \eqref{fourierDN} for the last line. By Lemma \ref{lem:projection-and-boundary-depth}, the above is bounded by 
\begin{equation}
    C\varepsilon^4 \| f \|^2_{L^2(\Omega)} \| \psi \|^2_{\mathcal{X}_\varepsilon}.
\end{equation}
Moreover, we have 
\begin{equation}
\begin{aligned}
    \left| \mathfrak b_\varepsilon^{\mathrm{f}}(J_\varepsilon^*u_D,\psi) - \mathfrak b_\varepsilon^{\mathrm{t}}(J_\varepsilon^*u_D,\psi) \right| &=\varepsilon^{-2} \left| \langle R_\varepsilon S_M J_\varepsilon^*u_D, S_M \psi \rangle_{\ell^2(\mathcal{I}_\varepsilon)} \right| \\ 
    &\leq \varepsilon^{-2} \langle R_\varepsilon S_M J_\varepsilon^*u_D, S_M J_\varepsilon^*u_D \rangle_{\ell^2(\mathcal{I}_\varepsilon)}^{1/2} \langle R_\varepsilon S_M \psi, S_M \psi \rangle_{\ell^2(\mathcal{I}_\varepsilon)}^{1/2} \\
    & \leq C\varepsilon^{1/2} \| f\|_{L^2(\Omega)} \| \psi \|_{\mathcal{X}_\varepsilon},
\end{aligned}
\end{equation}
where we used Lemma \ref{lemlocalbounR}, Lemma \ref{lem:discreteHardy}, and Lemma \ref{lem:projection-and-boundary-depth} in the last line. 

Similarly, we have
\begin{equation}
    \left| \mathfrak b_\varepsilon^{\mathrm{t}}(J_\varepsilon^*u_D,J_\varepsilon^*v_D) -\mathfrak b_\varepsilon^{(0)}(\Pi_\varepsilon u_D,\Pi_\varepsilon v_D) \right|^2 \leq C\varepsilon^4 \| f \|^2_{L^2(\Omega)} \|g \|^2_{L^2(\Omega)} ,
\end{equation}
and
\begin{equation}\label{comparebtandbo2}
    \left| \mathfrak b_\varepsilon^{\mathrm{f}}(J_\varepsilon^*u_D,J_\varepsilon^*v_D) - \mathfrak b_\varepsilon^{\mathrm{t}}(J_\varepsilon^*u_D,J_\varepsilon^*v_D) \right| \leq C \varepsilon \| f \|_{L^2(\Omega)} \|g \|_{L^2(\Omega)}.
\end{equation}

The estimates \eqref{comparebtandb0}--\eqref{comparebtandbo2} imply that we can replace $\mathfrak{b}_\varepsilon^j$ in the desired estimates \eqref{eq:mixed-quantitative-consistency-estimate}--\eqref{eq:two-resolvent-quantitative-consistency-estimate} by $\mathfrak{b}^{(0)}_\varepsilon$. Namely, it suffices to prove that
\begin{equation}\label{eq:mixed-quantitative-consistency-estimateb0}
    \left| \mathfrak b_\varepsilon^{(0)}(\Pi_\varepsilon u_D,J_\varepsilon\psi)  -\mathfrak b_D(u_D,K_\varepsilon\psi) \right| \leq C\varepsilon^{1/2}\|f\|_{L^2(\Omega)}\|\psi\|_{\mathcal X_\varepsilon},
\end{equation}
and
\begin{equation}\label{eq:two-resolvent-quantitative-consistency-estimateb0}
    \left| \mathfrak b_\varepsilon^{(0)}(\Pi_\varepsilon u_D,\Pi_\varepsilon v_D)  -\mathfrak b_D(u_D,v_D) \right| \leq C\varepsilon\|f\|_{L^2(\Omega)}\|g\|_{L^2(\Omega)}.
\end{equation}
In the next steps, our aim is to prove \eqref{eq:mixed-quantitative-consistency-estimateb0}--\eqref{eq:two-resolvent-quantitative-consistency-estimateb0}.

\textit{Step 2. Convergence of the full-lattice reference form.} 
The Newton--Leibniz formula \eqref{eq:averaged-derivative-formula-for-resolvent-solutions} gives the exact identity
\begin{equation}\label{eq:symmetric-two-translation-identity}
    \begin{aligned}
        &\left\langle\frac{\Delta_{\varepsilon n}U_D}{\varepsilon},\frac{\Delta_{\varepsilon n}V_D}{\varepsilon}\right\rangle_{L^2(\mathbb R^d)}-\langle n\cdot\nabla U_D,n\cdot\nabla V_D\rangle_{L^2(\mathbb R^d)}\\
        &\qquad=-\frac12\int_0^1\int_0^1\left\langle\Delta_{(t-s)\varepsilon n} (n\cdot\nabla U_D),\Delta_{(t-s)\varepsilon n}(n\cdot\nabla V_D)\right\rangle_{L^2(\mathbb R^d)}\,dtds.
    \end{aligned}
\end{equation}
Using the translation estimate \eqref{translaesti} for both $U_D$ and $V_D$, we deduce that
\begin{equation}\label{eq:two-regular-finite-difference-error}
    \left|\left\langle\frac{\Delta_{\varepsilon n}U_D}{\varepsilon},\frac{\Delta_{\varepsilon n}V_D}{\varepsilon}\right\rangle_{L^2(\mathbb R^d)}-\langle n\cdot\nabla U_D,n\cdot\nabla V_D\rangle_{L^2(\mathbb R^d)}\right|\leq C\varepsilon|n|^3\|f\|_{L^2(\Omega)}\|g\|_{L^2(\Omega)}.
\end{equation}
This identity measures the difference between the discrete directional derivative $\Delta_{\varepsilon n}/\varepsilon$ and the continuous directional derivative $n\cdot \nabla$.

Since $\widehat{\Pi}_\varepsilon$ is an orthogonal projection and commutes with every translation $\tau_{\varepsilon n}$, $n\in \mathbb{Z}^d$, one has
\begin{equation}\label{eq:orthogonality-identity-for-projected-differences}
    \begin{aligned}
        &\left\langle\frac{\Delta_{\varepsilon n}\widehat{\Pi}_\varepsilon U_D}{\varepsilon},\frac{\Delta_{\varepsilon n}\widehat{\Pi}_\varepsilon V_D}{\varepsilon}\right\rangle_{L^2(\mathbb R^d)}-\left\langle\frac{\Delta_{\varepsilon n}U_D}{\varepsilon},\frac{\Delta_{\varepsilon n}V_D}{\varepsilon}\right\rangle_{L^2(\mathbb R^d)}\\
        &\qquad=-\left\langle(I-\widehat{\Pi}_\varepsilon)\frac{\Delta_{\varepsilon n}U_D}{\varepsilon},(I-\widehat{\Pi}_\varepsilon)\frac{\Delta_{\varepsilon n}V_D}{\varepsilon}\right\rangle_{L^2(\mathbb R^d)}.
    \end{aligned}
\end{equation}
Applying \eqref{eq:cell-variance-estimate-for-full-lattice-projection} to $H=\tau_a(n\cdot\nabla U_D)$ and using \eqref{translaesti}, we obtain
\begin{equation}\label{eq:projection-error-for-translated-gradient}
    \|(I-\widehat{\Pi}_\varepsilon)\tau_a(n\cdot\nabla U_D)\|_{L^2(\mathbb R^d)}\leq C\varepsilon^{1/2}|n|\|f\|_{L^2(\Omega)}, \qquad \forall\, a\in\mathbb R^d.
\end{equation}
The analogous estimate holds with $U_D,f$ replaced by $V_D,g$. It follows from \eqref{eq:averaged-derivative-formula-for-resolvent-solutions} and \eqref{eq:projection-error-for-translated-gradient} that the absolute value of \eqref{eq:orthogonality-identity-for-projected-differences} is bounded by
\begin{equation}\label{eq:projection-error-for-two-regular-differences}
    C\varepsilon|n|^2\|f\|_{L^2(\Omega)}\|g\|_{L^2(\Omega)}.
\end{equation}
On the other hand, the Poincar\'e inequality and the elementary energy estimate give
\begin{equation}\label{eq:zero-order-projection-identity}
\begin{aligned}
    |\langle\widehat{\Pi}_\varepsilon U_D,\widehat{\Pi}_\varepsilon V_D\rangle_{L^2(\mathbb R^d)}-\langle U_D,V_D\rangle_{L^2(\mathbb R^d)} |& =|\langle(I-\widehat{\Pi}_\varepsilon)U_D,(I-\widehat{\Pi}_\varepsilon)V_D\rangle_{L^2(\mathbb R^d)} | \\
    & \leq C\varepsilon^2\|f\|_{L^2(\Omega)}\|g\|_{L^2(\Omega)}.
\end{aligned}
\end{equation}
Combining \eqref{eq:effective-matrix-from-fourier-coefficients}, \eqref{eq:full-lattice-reference-form}, and \eqref{eq:two-regular-finite-difference-error}--\eqref{eq:projection-error-for-two-regular-differences}, and then summing with \eqref{eq:zero-order-projection-identity}, gives
\begin{equation}\label{eq:full-lattice-two-regular-consistency}
    \left|\mathfrak b_\varepsilon^{(0)}(\widehat{\Pi}_\varepsilon U_D,\widehat{\Pi}_\varepsilon V_D)-\mathfrak b_D(u_D,v_D)\right|\leq C\varepsilon\|f\|_{L^2(\Omega)}\|g\|_{L^2(\Omega)}.
\end{equation}

\textit{Step 3. Proof of \eqref{eq:mixed-quantitative-consistency-estimateb0}--\eqref{eq:two-resolvent-quantitative-consistency-estimateb0}.} It follows from Lemma~\ref{lem:projection-and-boundary-depth} that
\begin{equation}\label{eq:difference-quotients-of-boundary-projection-errors}
    \begin{aligned}
        \left\|\frac{\Delta_{\varepsilon n}(\widehat{\Pi}_\varepsilon U_D-\Pi_\varepsilon u_D)}{\varepsilon}\right\|_{L^2(\mathbb R^d)} &\leq C\varepsilon^{1/2}\|f\|_{L^2(\Omega)},\\
        \left\|\frac{\Delta_{\varepsilon n}(\widehat{\Pi}_\varepsilon V_D-\Pi_\varepsilon v_D)}{\varepsilon}\right\|_{L^2(\mathbb R^d)} &\leq C\varepsilon^{1/2}\|g\|_{L^2(\Omega)}.
    \end{aligned}
\end{equation}

Since $\Delta_{\varepsilon n}(\widehat{\Pi}_\varepsilon U_D-\Pi_\varepsilon u_D)$ is supported in $\{x\in \mathbb{R}^d:\mathrm{dist}\,(x,\partial\Omega)<\sqrt{d}(1+|n|)\varepsilon\}$, by the formula \eqref{eq:averaged-derivative-formula-for-resolvent-solutions} and the boundary layer estimates for $\nabla u_D$ and $\nabla v_D$ in \eqref{estinablauD}, we get
\begin{equation}\label{eq:localized-projected-difference-estimates}
    \begin{aligned}
    \left\|\frac{\Delta_{\varepsilon n}\widehat{\Pi}_\varepsilon U_D}{\varepsilon}\right\|_{L^2(\operatorname{supp}\Delta_{\varepsilon n}(\widehat{\Pi}_\varepsilon V_D-\Pi_\varepsilon v_D))}&\leq C\varepsilon^{1/2}|n|(1+|n|)^{1/2}\|f\|_{L^2(\Omega)}, \\
        \left\|\frac{\Delta_{\varepsilon n}\widehat{\Pi}_\varepsilon V_D}{\varepsilon}\right\|_{L^2(\operatorname{supp}\Delta_{\varepsilon n}(\widehat{\Pi}_\varepsilon U_D-\Pi_\varepsilon u_D))}&\leq C\varepsilon^{1/2}|n|(1+|n|)^{1/2}\|g\|_{L^2(\Omega)}.
    \end{aligned}
\end{equation}

Using \eqref{eq:difference-quotients-of-boundary-projection-errors} and \eqref{eq:localized-projected-difference-estimates}, we obtain
\begin{equation}\label{eq:removal-of-unselected-cells-for-gradient-terms}
    \begin{aligned}
        &\left|\left\langle\frac{\Delta_{\varepsilon n}\Pi_\varepsilon u_D}{\varepsilon},\frac{\Delta_{\varepsilon n}\Pi_\varepsilon v_D}{\varepsilon}\right\rangle_{L^2(\mathbb R^d)}-\left\langle\frac{\Delta_{\varepsilon n}\widehat{\Pi}_\varepsilon U_D}{\varepsilon},\frac{\Delta_{\varepsilon n}\widehat{\Pi}_\varepsilon V_D}{\varepsilon}\right\rangle_{L^2(\mathbb R^d)}\right|\\
        &\qquad\leq C\varepsilon(1+|n|)^2\|f\|_{L^2(\Omega)}\|g\|_{L^2(\Omega)}.
    \end{aligned}
\end{equation}
The boundary-layer estimates for $u_D$ and $v_D$ in \eqref{estinablauD} also yield
\begin{equation}\label{eq:localized-zero-order-projection-estimates}
    \begin{aligned}
        \|\widehat{\Pi}_\varepsilon V_D\|_{L^2(\operatorname{supp}(\widehat{\Pi}_\varepsilon U_D-\Pi_\varepsilon u_D))} &\leq C\varepsilon^{3/2}\|g\|_{L^2(\Omega)},\\
        \|\widehat{\Pi}_\varepsilon U_D\|_{L^2(\operatorname{supp}(\widehat{\Pi}_\varepsilon V_D-\Pi_\varepsilon v_D))} &\leq C\varepsilon^{3/2}\|f\|_{L^2(\Omega)}.
    \end{aligned}
\end{equation}
It follows from Lemma~\ref{lem:projection-and-boundary-depth} and \eqref{eq:localized-zero-order-projection-estimates} that
\begin{equation}\label{eq:removal-of-unselected-cells-for-zero-order-term}
    \left|\langle\Pi_\varepsilon u_D,\Pi_\varepsilon v_D\rangle_{L^2(\mathbb R^d)}-\langle\widehat{\Pi}_\varepsilon U_D,\widehat{\Pi}_\varepsilon V_D\rangle_{L^2(\mathbb R^d)}\right|\leq C\varepsilon^3\|f\|_{L^2(\Omega)}\|g\|_{L^2(\Omega)}.
\end{equation}
Combining \eqref{eq:full-lattice-two-regular-consistency}, \eqref{eq:removal-of-unselected-cells-for-gradient-terms}, and \eqref{eq:removal-of-unselected-cells-for-zero-order-term}, we conclude that \eqref{eq:two-resolvent-quantitative-consistency-estimateb0} holds:
\begin{equation} 
    \left|\mathfrak b_\varepsilon^{(0)}(\Pi_\varepsilon u_D,\Pi_\varepsilon v_D)-\mathfrak b_D(u_D,v_D)\right|\leq C\varepsilon\|f\|_{L^2(\Omega)}\|g\|_{L^2(\Omega)}.
\end{equation}

We next prove \eqref{eq:mixed-quantitative-consistency-estimateb0}. Fix $\psi\in\ell^2(\mathcal I_\varepsilon)$ and let
\begin{equation}\label{eq:zero-extension-of-reconstructed-test-function}
    W:=  K_\varepsilon \psi\in H^1(\mathbb R^d).
\end{equation}
Lemma \ref{lemKeps} gives
\begin{equation}\label{eq:reconstruction-properties-used-in-consistency-proof}
    \widehat{\Pi}_\varepsilon W=  J_\varepsilon \psi, \qquad J_\varepsilon^*K_\varepsilon \psi= \psi, \qquad \|W\|_{H^1(\mathbb R^d)}\leq C\| \psi\|_{\mathcal X_\varepsilon}.
\end{equation}

For the zero-order term, orthogonality and the cellwise Poincar\'e inequality give
\begin{equation}\label{eq:mixed-zero-order-full-lattice-error}
    \begin{aligned}
        \left|\langle\widehat{\Pi}_\varepsilon U_D,\widehat{\Pi}_\varepsilon W\rangle_{L^2(\mathbb R^d)}-\langle U_D,W\rangle_{L^2(\mathbb R^d)}\right|&=\left|\langle(I-\widehat{\Pi}_\varepsilon)U_D,(I-\widehat{\Pi}_\varepsilon)W\rangle_{L^2(\mathbb R^d)}\right|\\
        &\leq C\varepsilon^2\|f\|_{L^2(\Omega)}\| \psi\|_{\mathcal X_\varepsilon}.
    \end{aligned}
\end{equation}
For the gradient part, orthogonality of $\widehat{\Pi}_\varepsilon$ and the formula \eqref{eq:averaged-derivative-formula-for-resolvent-solutions} give the exact identity
\begin{equation}\label{eq:mixed-projected-difference-identity}
    \begin{aligned}
        \left\langle\frac{\Delta_{\varepsilon n}\widehat{\Pi}_\varepsilon U_D}{\varepsilon},\frac{\Delta_{\varepsilon n}\widehat{\Pi}_\varepsilon W}{\varepsilon}\right\rangle_{L^2(\mathbb R^d)}
        &=\left\langle\int_0^1\tau_{-s\varepsilon n}\widehat{\Pi}_\varepsilon\frac{\Delta_{\varepsilon n}U_D}{\varepsilon}\,ds,n\cdot\nabla W\right\rangle_{L^2(\mathbb R^d)}.
    \end{aligned}
\end{equation}
Again by \eqref{eq:averaged-derivative-formula-for-resolvent-solutions}, we split the difference between the first argument on the right-hand side and $n\cdot\nabla U_D$ into a cell-projection error and a translation error. \eqref{eq:projection-error-for-translated-gradient} and Lemma \ref{lemboundarylayerestiuD} then imply
\begin{equation}\label{eq:mixed-first-argument-error}
    \left\|\int_0^1\tau_{-s\varepsilon n}\widehat{\Pi}_\varepsilon\frac{\Delta_{\varepsilon n}U_D}{\varepsilon}\,ds-n\cdot\nabla U_D\right\|_{L^2(\mathbb R^d)}\leq C\varepsilon^{1/2}\bigl(|n|+|n|^{3/2}\bigr)\|f\|_{L^2(\Omega)}.
\end{equation}
Pairing \eqref{eq:mixed-first-argument-error} with $n\cdot\nabla W$, and using \eqref{eq:reconstruction-properties-used-in-consistency-proof}--\eqref{eq:mixed-zero-order-full-lattice-error}, we obtain 
\begin{equation}\label{eq:full-lattice-mixed-reference-consistency}
    \left|\mathfrak b_\varepsilon^{(0)}(\widehat{\Pi}_\varepsilon U_D,\widehat{\Pi}_\varepsilon W)-\mathfrak b_D(u_D,K_\varepsilon \psi)\right|\leq C\varepsilon^{1/2}\|f\|_{L^2(\Omega)}\| \psi\|_{\mathcal X_\varepsilon}.
\end{equation}

It remains to replace $\widehat{\Pi}_\varepsilon U_D$ by $\Pi_\varepsilon u_D$. By \eqref{eq:difference-quotients-of-boundary-projection-errors} and the $H^1$ estimate for $W$,
\begin{equation}\label{eq:mixed-removal-of-unselected-cells}
    \begin{aligned}
        \left\|\frac{\Delta_{\varepsilon n}(\widehat{\Pi}_\varepsilon U_D-\Pi_\varepsilon u_D)}{\varepsilon}\right\|_{L^2(\mathbb R^d)} &\leq C\varepsilon^{1/2}\|f\|_{L^2(\Omega)},\\
        \left\|\frac{\Delta_{\varepsilon n}\widehat{\Pi}_\varepsilon W}{\varepsilon}\right\|_{L^2(\mathbb R^d)} &\leq |n|\|\nabla W\|_{L^2(\mathbb R^d)}\leq C|n|\| \psi\|_{\mathcal X_\varepsilon}.
    \end{aligned}
\end{equation}
The zero-order replacement is bounded by
\begin{equation}\label{eq:mixed-zero-order-removal-of-unselected-cells}
    \|\widehat{\Pi}_\varepsilon U_D-\Pi_\varepsilon u_D\|_{L^2(\mathbb R^d)}\|\widehat{\Pi}_\varepsilon W\|_{L^2(\mathbb R^d)}\leq C\varepsilon^{3/2}\|f\|_{L^2(\Omega)}\| \psi\|_{\mathcal X_\varepsilon}.
\end{equation}
After summing over $n \in \mathbb{Z}^d$, estimates \eqref{eq:full-lattice-mixed-reference-consistency}--\eqref{eq:mixed-zero-order-removal-of-unselected-cells} yield the desired estimate \eqref{eq:mixed-quantitative-consistency-estimateb0}. The proof is complete.
\end{proof}

\begin{proposition}
\label{prop:capacitance-norm-resolvent-convergence}
There exists $C=C(L_D,\Omega)>0$ such that, for $j\in\{\mathrm{t},\mathrm{f}\}$, 
\begin{equation}\label{eq:capacitance-norm-resolvent-convergence}
    \left\|J_\varepsilon(L_\varepsilon^j+I)^{-1}J_\varepsilon^*-(L_D+I)^{-1}\right\|_{\mathcal L(L^2(\Omega))}\leq C\varepsilon,
\end{equation}
and
\begin{equation}\label{eq:energy-norm-resolvent-corrector-estimate}
    \left\| (L_\varepsilon^j+I)^{-1}J_\varepsilon^*-J_\varepsilon^*(L_D+I)^{-1}\right\|_{L^2(\Omega) \rightarrow \mathcal X_\varepsilon}\leq C\varepsilon^{1/2} .
\end{equation}
\end{proposition}

\begin{proof}
Fix $j\in\{\mathrm{t},\mathrm{f}\}$ and $f\in L^2(\Omega)$. Let $u_D = (L_D+I)^{-1} f$ and define
\begin{equation}\label{eq:definition-of-energy-error-in-resolvent-proof}
    e_\varepsilon^j:= (L_\varepsilon^j+I)^{-1}J_\varepsilon^* f -J_\varepsilon^*(L_D+I)^{-1} f.
\end{equation}
By the definition, one has
\begin{equation}\label{eq:error-equation-for-discrete-resolvent}
    \mathfrak b_\varepsilon^j(e_\varepsilon^j,\psi)=\left\langle f,J_\varepsilon\psi-K_\varepsilon\psi\right\rangle_{L^2(\Omega)}+\mathfrak b_D(u_D,K_\varepsilon\psi)-\mathfrak b_\varepsilon^j(J_\varepsilon^*u_D,\psi), \qquad \forall \psi\in\ell^2(\mathcal I_\varepsilon).
\end{equation}
The reconstruction estimate in Lemma~\ref{lemKeps} implies
\begin{equation}\label{eq:reconstruction-term-in-resolvent-error}
    \left|\left\langle f,J_\varepsilon\psi-K_\varepsilon\psi\right\rangle_{L^2(\Omega)}\right|\leq C\varepsilon\|f\|_{L^2(\Omega)}\|\psi\|_{\mathcal X_\varepsilon}.
\end{equation}
On the other hand, Lemma~\ref{lem:two-quantitative-consistency-estimates} gives
\begin{equation}\label{eq:consistency-term-in-resolvent-error}
    \left|\mathfrak b_D(u_D,K_\varepsilon\psi)-\mathfrak b_\varepsilon^j(J_\varepsilon^*u_D,\psi)\right|\leq C\varepsilon^{1/2}\|f\|_{L^2(\Omega)}\|\psi\|_{\mathcal X_\varepsilon}.
\end{equation}
Since $0<\varepsilon\leq1$, equations \eqref{eq:error-equation-for-discrete-resolvent}--\eqref{eq:consistency-term-in-resolvent-error} imply
\begin{equation}\label{eq:dual-error-bound-for-discrete-resolvent}
    \left|\mathfrak b_\varepsilon^j(e_\varepsilon^j,\psi)\right|\leq C\varepsilon^{1/2}\|f\|_{L^2(\Omega)}\|\psi\|_{\mathcal X_\varepsilon}.
\end{equation}
Taking $\psi=e_\varepsilon^j$ and using Lemma \ref{lem:bjequivXeps}, we obtain
\begin{equation}
    \|e_\varepsilon^j\|_{\mathcal X_\varepsilon}^2\leq C\varepsilon^{1/2}\|f\|_{L^2(\Omega)}\|e_\varepsilon^j\|_{\mathcal X_\varepsilon}.
\end{equation}
This proves \eqref{eq:energy-norm-resolvent-corrector-estimate}.

It remains to prove \eqref{eq:capacitance-norm-resolvent-convergence}. Define
\begin{equation}\label{eq:definition-of-embedded-resolvent-difference}
    \mathscr R_\varepsilon^j:=J_\varepsilon(L_\varepsilon^j+I)^{-1}J_\varepsilon^*-(L_D+I)^{-1}\in\mathcal L(L^2(\Omega)).
\end{equation}
By a direct computation, one has the identity
\begin{equation}\label{eq:exact-quadratic-resolvent-error-identity}
\begin{aligned}
    &\left\langle\mathscr R_\varepsilon^j f,f\right\rangle_{L^2(\Omega)}\\
    &=\mathfrak b_\varepsilon^j(e_\varepsilon^j,e_\varepsilon^j)-\left[\mathfrak b_\varepsilon^j(J_\varepsilon^*u_D,J_\varepsilon^*u_D)-\mathfrak b_D(u_D,u_D)\right]+2\operatorname{Re}\left\langle f,\Pi_\varepsilon u_D-u_D\right\rangle_{L^2(\Omega)}.
\end{aligned}
\end{equation}

We now estimate the three terms on the right-hand side of \eqref{eq:exact-quadratic-resolvent-error-identity}. The uniform boundedness of $\mathfrak b_\varepsilon^j$ with respect to the $\mathcal X_\varepsilon$ norm and \eqref{eq:energy-norm-resolvent-corrector-estimate} give
\begin{equation}\label{eq:quadratic-energy-error-estimate}
    \mathfrak b_\varepsilon^j(e_\varepsilon^j,e_\varepsilon^j)\leq C\|e_\varepsilon^j\|_{\mathcal X_\varepsilon}^2\leq C\varepsilon\|f\|_{L^2(\Omega)}^2.
\end{equation}
Applying the estimate in Lemma~\ref{lem:two-quantitative-consistency-estimates} with $g=f$, so that $v_D=u_D$, gives
\begin{equation}\label{eq:diagonal-two-resolvent-consistency-bound}
    \left|\mathfrak b_\varepsilon^j(J_\varepsilon^*u_D,J_\varepsilon^*u_D)-\mathfrak b_D(u_D,u_D)\right|\leq C\varepsilon\|f\|_{L^2(\Omega)}^2.
\end{equation}
Finally, the projection estimate in Lemma~\ref{lem:projection-and-boundary-depth} gives
\begin{equation}\label{eq:projected-limit-solution-error-in-quadratic-identity}
    \left|\left\langle f,\Pi_\varepsilon u_D-u_D\right\rangle_{L^2(\Omega)}\right|\leq\|f\|_{L^2(\Omega)}\|\Pi_\varepsilon u_D-u_D\|_{L^2(\Omega)}\leq C\varepsilon\|f\|_{L^2(\Omega)}^2.
\end{equation}
Combining \eqref{eq:exact-quadratic-resolvent-error-identity}--\eqref{eq:projected-limit-solution-error-in-quadratic-identity}, we conclude that
\begin{equation}\label{eq:quadratic-form-bound-for-embedded-resolvent-difference}
    \left|\left\langle\mathscr R_\varepsilon^j f,f\right\rangle_{L^2(\Omega)}\right|\leq C\varepsilon\|f\|_{L^2(\Omega)}^2, \qquad \forall\,f\in L^2(\Omega).
\end{equation}
Both $(L_\varepsilon^j+I)^{-1}$ and $(L_D+I)^{-1}$ are self-adjoint. Consequently, $J_\varepsilon(L_\varepsilon^j+I)^{-1}J_\varepsilon^*$ and hence $\mathscr R_\varepsilon^j$ are bounded self-adjoint operators on $L^2(\Omega)$. Therefore, the desired estimate \eqref{eq:capacitance-norm-resolvent-convergence} now follows from \eqref{eq:quadratic-form-bound-for-embedded-resolvent-difference}. The proof is complete.
\end{proof}

\subsection{Asymptotic expansion and resolvent limit of the operator \texorpdfstring{$A$}{A}}

In this section, we study the limit of the resolvent of the operator $A=A_\varepsilon(k_\varepsilon)$ defined in \eqref{defofA}. 

By Proposition \ref{propexpansionofband}, the maximum $\omega_{*,\varepsilon}$ of the subwavelength band has the limit
\begin{equation}
    \omega_0:= \lim_{\varepsilon \rightarrow 0} \omega_{*,\varepsilon}=v_{\rm b}\sqrt{\widehat C^M}.
\end{equation}
With the incident frequency $\omega_\varepsilon=\omega_{*,\varepsilon}-\tau\varepsilon^2+O(\varepsilon^4)$, set
\begin{equation}\label{eq:definition-of-rescaled-physical-block}
    \mathcal T_\varepsilon:=-\frac{1}{|D|\varepsilon^d}A, \qquad \sigma:=\frac{2\omega_0\tau}{v_{\rm b}^2}.
\end{equation}

\begin{lemma}
\label{lem:reduction-of-the-physical-block}
$\mathcal T_\varepsilon$ has the expansion
\begin{equation}\label{eq:reduction-of-rescaled-physical-block}
    \mathcal T_\varepsilon=L_\varepsilon^{\mathrm f}-\sigma I+\mathcal E_\varepsilon.
\end{equation}
Here, $\mathcal E_\varepsilon:\mathcal X_\varepsilon\rightarrow\mathcal X_\varepsilon^*$ satisfies $\|\mathcal E_\varepsilon\|_{\mathcal L(\mathcal X_\varepsilon,\mathcal X_\varepsilon^*)}\leq C\varepsilon$, where $\mathcal X_\varepsilon^*$ is the dual of $\mathcal X_\varepsilon$ with respect to the $\ell^2(\mathcal I_\varepsilon)$ pairing.
\end{lemma}

\begin{proof}
Given $\varphi, \psi \in \ell^2(\mathcal{I}_\varepsilon)$, let 
\begin{equation}
    \zeta_\varepsilon [\varphi](x):= \Phi^M_\varepsilon \varphi(x) - (S_M\varphi)(n), \qquad \forall \,x\in \varepsilon(n+D), \,n\in \mathcal{I}_\varepsilon.
\end{equation}
Then, one has
\begin{equation}\label{eq:exact-bubble-mass-decomposition}
    \int_{D_\varepsilon}\Phi_\varepsilon^M\varphi\,\overline{\Phi_\varepsilon^M\psi}\,dx=
    |D|\varepsilon^d\left\langle\varphi,\psi\right\rangle_{\ell^2(\mathcal I_\varepsilon)}+\int_{D_\varepsilon}\zeta_\varepsilon[\varphi]\,\overline{\zeta_\varepsilon[\psi]}\,dx.
\end{equation}

Moreover, from 
\begin{equation}\label{eq:static-physical-block-identity}
    \mathcal A_\varepsilon(0)(\Phi_\varepsilon^M\varphi,\Phi_\varepsilon^M\psi)=|D|\varepsilon^{d-2}\left\langle S_M^*C_\varepsilon^{\mathrm f}S_M\varphi,\psi\right\rangle_{\ell^2(\mathcal I_\varepsilon)},
\end{equation}
the definition
\begin{equation}\label{eq:definition-of-finite-capacitance-operator-recalled}
    L_\varepsilon^{\mathrm f}=\frac{\widehat C_{\varepsilon^2}^M I-S_M^*C_\varepsilon^{\mathrm f}S_M}{\varepsilon^2},
\end{equation}
and the decomposition \eqref{eq:exact-bubble-mass-decomposition}, we get
\begin{equation}\label{eq:exact-rescaled-physical-block-identity}
    \begin{aligned}
        \left\langle\mathcal T_\varepsilon\varphi,\psi\right\rangle_{\ell^2(\mathcal I_\varepsilon)}={}&\left\langle L_\varepsilon^{\mathrm f}\varphi,\psi\right\rangle_{\ell^2(\mathcal I_\varepsilon)}-\frac{\widehat C_{\varepsilon^2}^M-\omega_\varepsilon^2/v_{\rm b}^2}{\varepsilon^2}\left\langle\varphi,\psi\right\rangle_{\ell^2(\mathcal I_\varepsilon)}\\
        &+\frac{k_\varepsilon^2}{|D|\varepsilon^d}\int_{B_R\setminus\overline{D_\varepsilon}}\Phi_\varepsilon^M\varphi\,\overline{\Phi_\varepsilon^M\psi}\,dx+\frac{(v/v_{\rm b})^2k_\varepsilon^2}{|D|\varepsilon^{d+2}}\int_{D_\varepsilon}\zeta_\varepsilon[\varphi]\,\overline{\zeta_\varepsilon[\psi]}\,dx\\
        &+\frac{1}{|D|\varepsilon^d}\left\langle(T_{k_\varepsilon}-T_0)\Phi_\varepsilon^M\varphi,\Phi_\varepsilon^M\psi\right\rangle_{\partial B_R}.
    \end{aligned}
\end{equation}
Here we used $(v/v_{\rm b})^2k_\varepsilon^2=\omega_\varepsilon^2/v_{\rm b}^2$. We divide the remainder of the proof into two steps. 

\textit{Step 1. A good approximation of $\Phi_\varepsilon^M \varphi$.} Recall that $E:\ell^2(\mathcal I_\varepsilon)\rightarrow\ell^2(\mathbb Z^d)$ is the zero-extension operator. Let $\mathcal U_\varepsilon\varphi$ be the full-lattice lift of the data $ES_M\varphi$, \emph{i.e.},
\begin{equation}\label{eq:definition-of-infinite-array-lift-in-block-reduction}
    \mathcal U_\varepsilon\varphi(x):=\sum_{n\in\mathbb Z^d}(ES_M\varphi)(n) \mathscr{V}_{\varepsilon^2,\mathbb Z^d}^n(x/\varepsilon).
\end{equation}
The functions $\mathcal U_\varepsilon\varphi$ and $\Phi_\varepsilon^M\varphi$ have the same average on every bubble $\varepsilon(n+D)$, $n\in\mathcal I_\varepsilon$. Hence,
\begin{equation}\label{eq:finite-infinite-lift-energy-comparison}
    \begin{aligned}
        \int_{\mathbb R^d}a_\varepsilon\left|\nabla\bigl(\mathcal U_\varepsilon\varphi-\Phi_\varepsilon^M\varphi\bigr)\right|^2\,dx &=\int_{\mathbb R^d}a_\varepsilon|\nabla\mathcal U_\varepsilon\varphi|^2\,dx-\int_{\mathbb R^d}a_\varepsilon|\nabla\Phi_\varepsilon^M\varphi|^2\,dx\\
        &\leq |D|\varepsilon^{d-2}\left\langle R_\varepsilon S_M\varphi,S_M\varphi\right\rangle_{\ell^2(\mathcal I_\varepsilon)} \\
        &\leq C\varepsilon^d\|\varphi\|_{\mathcal X_\varepsilon}^2,
    \end{aligned}
\end{equation}
where $R_\varepsilon=C_\varepsilon^{\mathrm t}-C_\varepsilon^{\mathrm f}$. Therefore, the Poincar\'{e} inequality and \eqref{eq:finite-infinite-lift-energy-comparison} yield
\begin{equation}\label{eq:finite-infinite-lift-L2-comparison}
    \|\mathcal U_\varepsilon\varphi-\Phi_\varepsilon^M\varphi\|_{L^2(\Omega)}\leq C\varepsilon^{d/2+1}\|\varphi\|_{\mathcal X_\varepsilon}, \qquad \|\mathcal U_\varepsilon\varphi-\Phi_\varepsilon^M\varphi\|_{L^2(D_\varepsilon)}\leq C\varepsilon^{d/2+2}\|\varphi\|_{\mathcal X_\varepsilon}.
\end{equation}

Define the cellwise comparison function $Z_\varepsilon\varphi$ by
\begin{equation}\label{eq:definition-of-frozen-cell-mode}
    (Z_\varepsilon\varphi)(x):=(-1)^{n_1+\cdots+n_d}(E\varphi)(n)V_{\varepsilon^2}^M(x/\varepsilon-n), \qquad \forall\,x\in\varepsilon(n+Y), \, n\in\mathbb Z^d.
\end{equation}
We next compare $\mathcal U_\varepsilon\varphi$ with $Z_\varepsilon\varphi$. Let
\begin{equation}
    W(y):= (\mathcal U_\varepsilon\varphi)(\varepsilon y),
\end{equation}
and consider the sequence
\begin{equation}
    \varphi_y := \big\{W(y+n) \big\}_{n\in \mathbb{Z}^d}.
\end{equation}
It is clear that $\varphi_y \in \ell^2(\mathbb{Z}^d)$ for almost every $y \in \mathbb{R}^d$. We claim that
\begin{equation}\label{eq:discrete-Fourier-identity-for-infinite-array-lift}
    \widehat{\varphi_y} (\alpha)=V_{\varepsilon^2}^{-\alpha}(y)  \widehat{ES_M\varphi}(\alpha) .
\end{equation}
Here and below, the identities involving the discrete Fourier transform are understood in $L^2(\mathcal B_1\times Y)$.

To justify \eqref{eq:discrete-Fourier-identity-for-infinite-array-lift}, we observe that $y\mapsto \widehat{\varphi_y} (\alpha)$ is $(-\alpha)$-quasiperiodic. Moreover, since the average of $\mathcal U_\varepsilon\varphi$ over $\varepsilon(n+D)$ is $(ES_M\varphi)(n)$, one has
\begin{equation}
    \fint_D \widehat{\varphi_y} (\alpha) \,dy=\sum_{n\in\mathbb Z^d}(ES_M\varphi)(n)e^{\mi n\cdot\alpha}
    =\widehat{ES_M\varphi}(\alpha).
\end{equation}
Applying the discrete Fourier transform to the variational equation satisfied by $W$ shows that $\widehat{\varphi_y} (\alpha)$ is orthogonal to every $(-\alpha)$-quasiperiodic function having zero average over $D$. Uniqueness of the minimizer therefore gives
\eqref{eq:discrete-Fourier-identity-for-infinite-array-lift}.

Evaluating \eqref{eq:discrete-Fourier-identity-for-infinite-array-lift} at $\alpha=M-\beta$ gives
\begin{equation}\label{eq:Fourier-identity-at-M-for-infinite-lift}
    \widehat{\varphi_y}(M-\beta) = \widehat{E\varphi}(-\beta)V_{\varepsilon^2}^{M+\beta}(y).
\end{equation}
On the other hand, the definition of $Z_\varepsilon\varphi$ gives, for $y\in Y$,
\begin{equation}
    (Z_\varepsilon\varphi)\bigl(\varepsilon(y+n)\bigr)
    =
    e^{\mi M\cdot n}(E\varphi)(n)V_{\varepsilon^2}^{M}(y).
\end{equation}
Applying the discrete Fourier transform and evaluating it at $\alpha=M-\beta$, we obtain
\begin{equation}\label{eq:Fourier-identity-at-M-for-frozen-lift}
    \mathcal F\left\{ n\longmapsto(Z_\varepsilon\varphi)\bigl(\varepsilon(y+n)\bigr)\right\} (M-\beta)
    = \widehat{E\varphi}(-\beta)V_{\varepsilon^2}^{M}(y).
\end{equation}
Subtracting \eqref{eq:Fourier-identity-at-M-for-frozen-lift} from \eqref{eq:Fourier-identity-at-M-for-infinite-lift} yields
\begin{equation}\label{eq:Fourier-identity-for-lift-difference}
    \mathcal F\left\{ n\mapsto
        \bigl(\mathcal U_\varepsilon\varphi-Z_\varepsilon\varphi\bigr)
        \bigl(\varepsilon(y+n)\bigr)\right\} (M-\beta) =
        \widehat{E\varphi}(-\beta)
        \left(V_{\varepsilon^2}^{M+\beta}(y)-V_{\varepsilon^2}^{M}(y)\right).
\end{equation}

Applying the Plancherel identity, integrating over $y\in Y$, and then using the change of variables $x=\varepsilon y$, we obtain the identity
\begin{equation}\label{eq:infinite-array-cellwise-comparison}
    \left\|\mathcal U_\varepsilon\varphi-Z_\varepsilon\varphi\right\|_{L^2(\mathbb R^d)}^2
    =\frac{\varepsilon^d}{|\mathcal B_1|} \int_{\mathcal B_1}
    \left|\widehat{E\varphi}(-\beta)\right|^2 \left\|V_{\varepsilon^2}^{M+\beta}-V_{\varepsilon^2}^{M}\right\|_{L^2(Y)}^2\,d\beta.
\end{equation}
Lemma \ref{holoofV} gives
\begin{equation}
    \left\|V_{\varepsilon^2}^{M+\beta}-V_{\varepsilon^2}^{M}\right\|^2_{L^2(Y)}\leq C|\beta|^2\leq C\theta(\beta).
\end{equation}
It follows that
\begin{equation}\label{eqUepsandZeps}
    \left\|\mathcal U_\varepsilon\varphi-Z_\varepsilon\varphi\right\|_{L^2(\mathbb R^d)}^2
    \leq C\varepsilon^d\mathfrak D_\varepsilon(\varphi ) \leq C\varepsilon^{d+2}\|\varphi\|_{\mathcal X_\varepsilon}^2.
\end{equation}

Combining \eqref{eq:finite-infinite-lift-L2-comparison} and \eqref{eqUepsandZeps}, and using the uniform $L^2(Y)$ bound for $V_{\varepsilon^2}^M$, we get 
\begin{equation}\label{eq:finite-lift-frozen-mode-comparison}
    \|\Phi_\varepsilon^M\varphi-Z_\varepsilon\varphi\|_{L^2(\Omega)}\leq C\varepsilon^{d/2+1}\|\varphi\|_{\mathcal X_\varepsilon}, \qquad \|Z_\varepsilon\varphi\|_{L^2(\mathbb R^d)}\leq C\varepsilon^{d/2}\|\varphi\|_{\ell^2(\mathcal I_\varepsilon)}, 
\end{equation}
and hence,
\begin{equation}\label{Phivarestimate}
    \|\Phi_\varepsilon^M\varphi\|_{L^2(\Omega)}\leq C\varepsilon^{d/2}\|\varphi\|_{\mathcal X_\varepsilon} .
\end{equation}

\textit{Step 2. The expansion of $\mathcal{T}_\varepsilon$.} We now estimate the right-hand side of \eqref{eq:exact-rescaled-physical-block-identity} term by term. 

Define
\begin{equation}\label{eq:definition-of-cell-mass-constants-in-block-reduction}
    \mu^M:=\int_{Y\setminus\overline D}|V^M(y)|^2\,dy, \qquad \mu_\delta^M:=\int_{Y\setminus\overline D}|V_\delta^M(y)|^2\,dy.
\end{equation}
Lemma \ref{holoofV} implies
\begin{equation}\label{eq:convergence-of-cell-mass-constant}
    \mu_\delta^M=\mu^M+O(\delta).
\end{equation}

By the cellwise definition of $Z_\varepsilon$ and the unitarity of $S_M$, one has the identity
\begin{equation}\label{eq:exact-matrix-mass-of-frozen-mode}
    \int_{\Omega\setminus\overline{D_\varepsilon}}Z_\varepsilon\varphi\,\overline{Z_\varepsilon\psi}\,dx=\varepsilon^d\mu_{\varepsilon^2}^M\left\langle\varphi,\psi\right\rangle_{\ell^2(\mathcal I_\varepsilon)}.
\end{equation}
This, combined with \eqref{eq:finite-lift-frozen-mode-comparison}, \eqref{Phivarestimate}, and \eqref{eq:convergence-of-cell-mass-constant}, gives
\begin{equation}\label{eq:matrix-part-of-mass-expansion}
    \left|\varepsilon^{-d}\int_{\Omega\setminus\overline{D_\varepsilon}}\Phi_\varepsilon^M\varphi\,\overline{\Phi_\varepsilon^M\psi}\,dx-\mu^M\left\langle\varphi,\psi\right\rangle_{\ell^2(\mathcal I_\varepsilon)}\right|\leq C\varepsilon\|\varphi\|_{\mathcal X_\varepsilon}\|\psi\|_{\mathcal X_\varepsilon}.
\end{equation}

The estimate \eqref{PhiMoutside} gives
\begin{equation}\label{eq:exterior-lift-estimates-used-in-block-reduction}
    \|\Phi_\varepsilon^M\varphi\|_{L^2(B_R\setminus\overline\Omega)}+\|\Phi_\varepsilon^M\varphi\|_{L^2(\partial B_R)}\leq C\varepsilon^{(d+1)/2}\|\varphi\|_{\mathcal X_\varepsilon}.
\end{equation}
It follows that
\begin{equation}\label{eq:exterior-mass-estimate-in-block-reduction}
    \left|\int_{B_R\setminus\overline\Omega}\Phi_\varepsilon^M\varphi\,\overline{\Phi_\varepsilon^M\psi}\,dx\right|\leq C\varepsilon^{d+1}\|\varphi\|_{\mathcal X_\varepsilon}\|\psi\|_{\mathcal X_\varepsilon}.
\end{equation}
Since $\partial B_R$ is smooth, the operator $T_{k_\varepsilon}-T_0$ is uniformly bounded on $L^2(\partial B_R)$. Combining this with \eqref{eq:exterior-lift-estimates-used-in-block-reduction} yields
\begin{equation}\label{eq:radiation-term-estimate-in-block-reduction}
    \left|\left\langle(T_{k_\varepsilon}-T_0)\Phi_\varepsilon^M\varphi,\Phi_\varepsilon^M\psi\right\rangle_{\partial B_R}\right|\leq C\varepsilon^{d+1}\|\varphi\|_{\mathcal X_\varepsilon}\|\psi\|_{\mathcal X_\varepsilon}.
\end{equation}
 
We now estimate $\zeta_\varepsilon[\varphi]$. For $y\in D$, define the sequence
\begin{equation}
    q_y (n):=W(y+n)-ES_M\varphi(n).
\end{equation}
It follows from \eqref{eq:discrete-Fourier-identity-for-infinite-array-lift} that
\begin{equation}
    \widehat{q_y}(\alpha) = \big(V^{-\alpha}_{\varepsilon^2}(y)-1 \big) \widehat{ES_M\varphi} (\alpha).
\end{equation}
Applying the Plancherel identity and integrating over $D$, we obtain
\begin{equation}
    \sum_{n\in\mathbb Z^d}\int_D|W(y+n)- ES_M\varphi(n)|^2\,dy=\frac{1}{|\mathcal B_1|}\int_{\mathcal B_1}|\widehat{ES_M\varphi}(\alpha)|^2\|V_{\varepsilon^2}^{-\alpha}-1\|_{L^2(D)}^2\,d\alpha.
\end{equation}
Section \ref{secpropofcapa} gives the uniform estimate
\begin{equation}
    \|V_{\varepsilon^2}^{-\alpha}-1\|_{L^2(D)}^2\leq C\varepsilon^4, \qquad \forall \alpha \in \mathcal{B}_1.
\end{equation}
Consequently,
\begin{equation}
    \begin{aligned}
        \sum_{n\in\mathbb Z^d}\int_D|W(y+n)- ES_M\varphi(n)|^2\,dy &\leq C\varepsilon^4\frac{1}{|\mathcal B_1|}\int_{\mathcal B_1}|\widehat{ES_M\varphi}(\alpha)|^2\,d\alpha \\
        &=C\varepsilon^4\|\varphi\|_{\ell^2(\mathcal{I}_\varepsilon)}^2.
    \end{aligned}
\end{equation}
Returning to the variable $x=\varepsilon y$ introduces the factor $\varepsilon^d$. Therefore, by \eqref{eq:finite-infinite-lift-L2-comparison},
\begin{equation}\label{finalzetaepseis}
    \|\zeta_\varepsilon[\varphi]\|_{L^2(D_\varepsilon)}\leq C\varepsilon^{d/2+2}\|\varphi\|_{\mathcal X_\varepsilon}.
\end{equation}

Finally, \eqref{eq:exact-rescaled-physical-block-identity}, \eqref{eq:matrix-part-of-mass-expansion}, \eqref{eq:exterior-mass-estimate-in-block-reduction}, \eqref{eq:radiation-term-estimate-in-block-reduction}, and \eqref{finalzetaepseis} yield
\begin{equation}\label{eq:physical-block-reduction-before-band-edge-cancellation}
    \mathcal T_\varepsilon=L_\varepsilon^{\mathrm f}-q_\varepsilon I+\mathcal E_\varepsilon^{(0)}, \qquad q_\varepsilon:=\frac{\widehat C_{\varepsilon^2}^M-\omega_\varepsilon^2/v_{\rm b}^2}{\varepsilon^2}-\frac{k_\varepsilon^2\mu^M}{|D|},
\end{equation}
where
\begin{equation}\label{eq:preliminary-physical-block-error-bound}
    \left|\left\langle\mathcal E_\varepsilon^{(0)}\varphi,\psi\right\rangle_{\ell^2(\mathcal I_\varepsilon)}\right|\leq C\varepsilon\|\varphi\|_{\mathcal X_\varepsilon}\|\psi\|_{\mathcal X_\varepsilon}.
\end{equation}
Substitution of the expansion \eqref{expansionofChat} of $\widehat{C}^M_{\varepsilon^2}$ and the expansion \eqref{eq:omega-edge-expansion} of $\omega_{*,\varepsilon}$ give
\begin{equation}\label{eq:cell-mass-shift-cancellation}
    q_\varepsilon = \frac{2\omega_0\tau}{v_{\rm b}^2}+O(\varepsilon^2)=\sigma+O(\varepsilon^2).
\end{equation}
This completes the proof.
\end{proof}

\begin{theorem}
\label{thm:physical-block-norm-resolvent-convergence}
For every compact set $K\Subset\mathbb C\setminus\operatorname{spec}(L_D-\sigma I)$, there exist $\varepsilon_K>0$ and $C_K>0$ such that $\mathcal T_\varepsilon-z$ is invertible for every $0<\varepsilon<\varepsilon_K$ and every $z\in K$, and
\begin{equation}\label{eq:physical-block-norm-resolvent-convergence}
    \sup_{z\in K}\left\|J_\varepsilon(\mathcal T_\varepsilon-z)^{-1}J_\varepsilon^*-(L_D-\sigma -z)^{-1}\right\|_{\mathcal L(L^2(\Omega))}\leq C_K\varepsilon.
\end{equation}
In particular, if $\sigma\notin\operatorname{spec}(L_D)$, then $A$ is invertible for all sufficiently small $\varepsilon$ and
\begin{equation}\label{eq:physical-block-inverse-convergence}
    \left\||D|\varepsilon^dJ_\varepsilon A^{-1}J_\varepsilon^*-(\sigma -L_D)^{-1}\right\|_{\mathcal L(L^2(\Omega))}\leq C\varepsilon.
\end{equation}
\end{theorem}

\begin{proof}
\textit{Step 1.} Set $G_\varepsilon:=L_\varepsilon^{\mathrm{f}}+I$. Let $\mathcal E_\varepsilon$ be the remainder in Lemma~\ref{lem:reduction-of-the-physical-block}, we have
\begin{equation}
    \mathcal T_\varepsilon+\sigma+1 = G_{\varepsilon}^{1/2} ( I + \widetilde{\mathcal E}_\varepsilon)G_{\varepsilon}^{1/2}  , \qquad \widetilde{\mathcal E}_\varepsilon := G_{\varepsilon}^{-1/2} \mathcal E_\varepsilon G_{\varepsilon}^{-1/2}.
\end{equation}
Since the form of $G_\varepsilon$ is $\mathfrak b_\varepsilon^{\mathrm f}$, the uniform equivalence between $\mathfrak b_\varepsilon^{\mathrm f}(\varphi)$ and $\|\varphi\|_{\mathcal X_\varepsilon}^2$ yields
\begin{equation}\label{eq:square-root-and-discrete-energy-norm-equivalence}
    c\|\varphi\|_{\mathcal X_\varepsilon}^2\leq\|G_\varepsilon^{1/2}\varphi\|_{\ell^2(\mathcal I_\varepsilon)}^2=\mathfrak b_\varepsilon^{\mathrm f}(\varphi)\leq C\|\varphi\|_{\mathcal X_\varepsilon}^2.
\end{equation}
In particular,
\begin{equation}\label{eq:inverse-square-root-estimates}
    \|G_\varepsilon^{-1/2}\|_{\ell^2(\mathcal I_\varepsilon) \rightarrow \mathcal X_\varepsilon}\leq C.
\end{equation}
The estimate in \eqref{eq:reduction-of-rescaled-physical-block}, together with \eqref{eq:inverse-square-root-estimates}, gives
\begin{equation}\label{eq:sandwiched-physical-block-error-bound}
    \|\widetilde{\mathcal E}_\varepsilon\|_{\mathcal L(\ell^2(\mathcal I_\varepsilon))}\leq C\varepsilon.
\end{equation}
Consequently, for sufficiently small $\varepsilon$, the Neumann series gives the invertibility of $I+\widetilde{\mathcal E}_\varepsilon$ and
\begin{equation}\label{inverofTeps}
    (\mathcal T_\varepsilon+\sigma+1)^{-1}=G_\varepsilon^{-1/2}\bigl(I+\widetilde{\mathcal E}_\varepsilon\bigr)^{-1}G_\varepsilon^{-1/2}.
\end{equation}
Define
\begin{equation}
    R_\varepsilon:=J_\varepsilon(\mathcal T_\varepsilon+\sigma+1)^{-1}J_\varepsilon^*\in\mathcal L(L^2(\Omega)).
\end{equation}
The factorization \eqref{inverofTeps} gives
\begin{equation}
    R_\varepsilon-J_\varepsilon G_\varepsilon^{-1}J_\varepsilon^*=J_\varepsilon G_\varepsilon^{-1/2}\left[\bigl(I+\widetilde{\mathcal E}_\varepsilon\bigr)^{-1}-I\right]G_\varepsilon^{-1/2}J_\varepsilon^*,
\end{equation}
and hence
\begin{equation}
    \| R_\varepsilon-J_\varepsilon G_\varepsilon^{-1}J_\varepsilon^*\|_{\mathcal L(L^2(\Omega))}\leq C\varepsilon.
\end{equation}
Combining this estimate with Proposition~\ref{prop:capacitance-norm-resolvent-convergence}, we obtain
\begin{equation}\label{eq:full-reference-resolvent-comparison}
    \| R_\varepsilon-R_D\|_{\mathcal L(L^2(\Omega))}\leq C\varepsilon, \qquad R_D:=(L_D+I)^{-1}.
\end{equation}

For $z\in K$, set $t(z):=\sigma+z+1$. Since $I-t(z)R_D=(L_D-\sigma -z)R_D$,
the compactness of $K\Subset\mathbb C\setminus \mathrm{spec}\,(L_D-\sigma I)$ implies
\begin{equation}\label{I-tzRdinverse}
    \sup_{z\in K} \left\|\big( I-t(z)R_D\big)^{-1} \right\|_{\mathcal L(L^2(\Omega))}\leq C_K.
\end{equation}
This, together with the estimate \eqref{eq:full-reference-resolvent-comparison}, implies
\begin{equation}\label{eq:uniform-inverse-Q-epsilon}
    \sup_{z\in K} \left\| \big( I - t(z) R_\varepsilon \big)^{-1} \right\|_{\mathcal L(L^2(\Omega))}\leq C_K.
\end{equation}

\textit{Step 2.} Write
\begin{equation}
    I-t(z)R_\varepsilon=J_\varepsilon\left\{ I-t(z)(\mathcal T_\varepsilon+\sigma+1)^{-1}\right\} J_\varepsilon^*+(I-\Pi_\varepsilon).
\end{equation}
Consequently,
\begin{equation}
    \big( I - t(z) R_\varepsilon \big)^{-1} =J_\varepsilon\left\{ I-t(z)(\mathcal T_\varepsilon+\sigma+1)^{-1}\right\}^{-1} J_\varepsilon^*+(I-\Pi_\varepsilon).
\end{equation}
This implies that
\begin{equation}
    \begin{aligned}
        &J_\varepsilon(\mathcal T_\varepsilon-z)^{-1}J_\varepsilon^*\\
        &= J_\varepsilon (\mathcal T_\varepsilon+\sigma+1)^{-1}\left\{
        I - t(z) (\mathcal T_\varepsilon+\sigma+1)^{-1}\right\}^{-1} J_\varepsilon^*  \\
        &=J_\varepsilon (\mathcal T_\varepsilon+\sigma+1)^{-1} J_\varepsilon^* \Big\{ J_\varepsilon\left\{
        I - t(z) (\mathcal T_\varepsilon+\sigma+1)^{-1}\right\}^{-1} J_\varepsilon^*  + (I- \Pi_\varepsilon) \Big\} \\
        & = R_\varepsilon \big( I - t(z) R_\varepsilon \big)^{-1} .
    \end{aligned}
\end{equation}
Similarly, 
\begin{equation}
    (L_D-\sigma -z)^{-1}=R_D \big( I-t(z)R_D\big)^{-1} .
\end{equation}
As a consequence,
\begin{equation}\label{eq:exact-embedded-resolvent-identity}
    J_\varepsilon(\mathcal T_\varepsilon-z)^{-1}J_\varepsilon^*-(L_D-\sigma -z)^{-1}= \big( I - t(z) R_\varepsilon \big)^{-1}(R_\varepsilon-R_D) \big( I-t(z)R_D\big)^{-1}.
\end{equation}
Combining \eqref{eq:full-reference-resolvent-comparison}, \eqref{I-tzRdinverse}, \eqref{eq:uniform-inverse-Q-epsilon}, and \eqref{eq:exact-embedded-resolvent-identity}, we conclude that
\begin{equation}
    \sup_{z\in K}\left\|J_\varepsilon(\mathcal T_\varepsilon-z)^{-1}J_\varepsilon^*-(L_D-\sigma I-z)^{-1}\right\|_{\mathcal L(L^2(\Omega))}\leq C_K\varepsilon.
\end{equation}

Finally, suppose that $\sigma\notin\operatorname{spec}(L_D)$ and take $K=\{0\}$. Then $\mathcal T_\varepsilon$ and hence $A$ are invertible for all sufficiently small $\varepsilon$. Since
\begin{equation}
    \mathcal T_\varepsilon^{-1}=-|D|\varepsilon^dA^{-1},\qquad -(L_D-\sigma I)^{-1}=(\sigma I-L_D)^{-1},
\end{equation}
the estimate \eqref{eq:physical-block-inverse-convergence} follows from \eqref{eq:physical-block-norm-resolvent-convergence} with $z=0$.
\end{proof}

\section{Proof of the main theorem}\label{secproofofmain}

Recall that
\begin{equation}\label{eq:section-six-discrepancy-and-Neumann-datum}
    w_\varepsilon=u_\varepsilon-\widehat u_\varepsilon
\end{equation}
is the difference between the original total field $u_\varepsilon$ and the intermediate comparison field $\widehat u_\varepsilon$.

For the incident field, define
\begin{equation}\label{eq:section-six-incident-data-norm}
    \mathcal N_\varepsilon^{\mathrm{in}}:=\|u_\varepsilon^{\mathrm{in}}\|_{H^1(\partial\Omega)}+\left\|\left.\frac{\partial u_\varepsilon^{\mathrm{in}}}{\partial\nu}\right|_{\partial\Omega}\right\|_{L^2(\partial\Omega)}.
\end{equation}
Since $|k_\varepsilon - k_0|<C\varepsilon^2$, we have
\begin{equation}
    | \mathcal N_\varepsilon^{\mathrm{in}} - \mathcal N^{\mathrm{in}} | <C\varepsilon^2.
\end{equation}
Moreover, it follows from \eqref{difference_intermediate} that
the main result (Theorem \ref{themaintheorem}) of this paper is a direct implication of the following theorem.

\begin{theorem}\label{thm:main-scattering-convergence-rate}
Assume that
    \begin{equation}
        \sigma \notin \mathrm{spec}\,(L_D).
    \end{equation}
    For every fixed ball $B_R$ such that $\overline\Omega\Subset B_R$, there exists $C>0$, depending only on $D,\Omega,R,\rho,\rho_{\rm b},\kappa,\kappa_{\rm b},|\tau|$, and $\mathrm{dist}\,(\sigma,\mathrm{spec}\,(L_D))$, such that
    \begin{equation}\label{eq:section-six-main-interior-and-exterior-L2-rates}
    \|w_\varepsilon\|_{L^2(\Omega)}\leq C\varepsilon^{1/2}\mathcal N_\varepsilon^{\mathrm{in}}, \qquad \|w_\varepsilon\|_{L^2(B_R\setminus\overline\Omega)}\leq C\varepsilon\mathcal N_\varepsilon^{\mathrm{in}},
    \end{equation}
        and
    \begin{equation}\label{eq:section-six-main-global-L2-and-exterior-H1-rates}
    \|w_\varepsilon\|_{H^1(B_R\setminus\overline\Omega)}\leq C\varepsilon^{1/2}\mathcal N_\varepsilon^{\mathrm{in}}.
    \end{equation}
    Moreover, for every compact set $K\Subset\mathbb R^d\setminus\overline\Omega$, there exists $C>0$, additionally depending on $K$, such that
\begin{equation}\label{eq:section-six-main-local-H1-rate}
    \|w_\varepsilon\|_{H^1(K)}\leq C\varepsilon\mathcal N_\varepsilon^{\mathrm{in}}.
\end{equation}
\end{theorem}

Hereafter, we prove Theorem \ref{thm:main-scattering-convergence-rate}. We first solve the block system quantitatively.

\begin{lemma}\label{lem:section-six-block-component-estimates}
Let
\begin{equation}\label{eq:section-six-block-decomposition}
    w_\varepsilon=\Phi_\varepsilon^M\varphi_\varepsilon+r_\varepsilon, \qquad \varphi_\varepsilon\in\mathcal X_\varepsilon, \qquad r_\varepsilon\in\mathcal H_\varepsilon,
\end{equation}
be the decomposition in \eqref{decompofdiscre}. Assume that
    \begin{equation}
        \sigma \notin \mathrm{spec}\,(L_D).
    \end{equation}
    For every fixed ball $B_R$ such that $\overline\Omega\Subset B_R$, there exists $C>0$, depending only on $D,\Omega,R,\rho,\rho_{\rm b},\kappa,\kappa_{\rm b},|\tau|$, and $\mathrm{dist}\,(\sigma,\mathrm{spec}\,(L_D))$, such that
\begin{equation}\label{eq:section-six-component-rates}
    \|\varphi_\varepsilon\|_{\mathcal X_\varepsilon}\leq C\varepsilon^{(1-d)/2}\mathcal N_\varepsilon^{\mathrm{in}}, \qquad \|r_\varepsilon\|_{\mathcal H_\varepsilon}\leq C\varepsilon^{1/2}\mathcal N_\varepsilon^{\mathrm{in}} .
\end{equation}
\end{lemma}

\begin{proof}
We regard the operator $A$ as an operator from $\mathcal X_\varepsilon$ to $\mathcal X_\varepsilon^*$. We first prove that
\begin{equation}\label{eq:section-six-energy-space-inverse-of-A}
    \|A^{-1}\|_{\mathcal{L}( \mathcal{X}_\varepsilon^*, \mathcal{X}_\varepsilon )}\leq C\varepsilon^{-d}.
\end{equation}
We use the notation from the proof of Theorem \ref{thm:physical-block-norm-resolvent-convergence}. The estimate \eqref{eq:square-root-and-discrete-energy-norm-equivalence} and its dual give
\begin{equation}\label{eq:section-six-square-root-norm-equivalences}
    \|G_\varepsilon^{-1/2}\|_{\mathcal{L}(\ell^2(\mathcal I_\varepsilon),\mathcal{X}_\varepsilon )} +  \|G_\varepsilon^{-1/2}\|_{\mathcal{L}(\mathcal{X}^*_\varepsilon ,\ell^2(\mathcal I_\varepsilon))}\leq C.
\end{equation}
Moreover, one has
\begin{equation}
    \mathcal{T}_\varepsilon = G_\varepsilon^{1/2} \big[ I - (\sigma+1) G_\varepsilon^{-1} + \widetilde{\mathcal{E}}_\varepsilon \big] G_\varepsilon^{1/2}.
\end{equation}
To prove \eqref{eq:section-six-energy-space-inverse-of-A}, it remains to prove that
\begin{equation}\label{Isigma1uniformbound}
    \big\{ I - (\sigma+1) G_\varepsilon^{-1} \big\}^{-1} =J_\varepsilon^* \{ I - (\sigma+1) J_\varepsilon G_\varepsilon^{-1} J_\varepsilon^* \big\}^{-1} J_\varepsilon
\end{equation}
is uniformly bounded on $\ell^2(\mathcal{I}_\varepsilon)$. Since, by Proposition \ref{prop:capacitance-norm-resolvent-convergence},
\begin{equation}
     I - (\sigma+1) J_\varepsilon G_\varepsilon^{-1} J_\varepsilon^* 
\end{equation}
converges to $I-(\sigma+1) R_D$ in the $\mathcal{L}(L^2(\Omega))$ operator norm, and $I-(\sigma+1) R_D$ is invertible, we find that \eqref{Isigma1uniformbound} is uniformly bounded on $\ell^2(\mathcal{I}_\varepsilon)$. Hence, \eqref{eq:section-six-energy-space-inverse-of-A} holds.

Since the scattered field $\widehat u_\varepsilon-u_\varepsilon^{\mathrm{in}}$ has Dirichlet datum $-u_\varepsilon^{\mathrm{in}}|_{\partial\Omega}$, the $L^2$ regularity estimate for the exterior Dirichlet problem on a Lipschitz domain gives
\begin{equation}\label{eq:section-six-sound-soft-Neumann-estimate}
    \left\|\left.\frac{\partial\widehat u_\varepsilon}{\partial\nu}\right|_{\partial\Omega+} \right\|_{L^2(\partial\Omega)}\leq C\mathcal N_\varepsilon^{\mathrm{in}}.
\end{equation}

By the definitions in \eqref{defgandh}, Lemmas~\ref{lemmaHepsprop} and \ref{lemtraceesti}, and \eqref{eq:section-six-sound-soft-Neumann-estimate}, for every $\psi\in\mathcal X_\varepsilon$ and $v\in\mathcal H_\varepsilon$ we have
\begin{equation}\label{eq:section-six-source-estimates-before-duality}
    \left|\langle g_\varepsilon,\psi\rangle_{\ell^2(\mathcal I_\varepsilon)}\right|\leq C\varepsilon^{(d+1)/2}\mathcal N_\varepsilon^{\mathrm{in}}\|\psi\|_{\mathcal X_\varepsilon}, \qquad \left|\langle h_\varepsilon,v\rangle_{\mathcal H_\varepsilon}\right|\leq C\varepsilon^{1/2}\mathcal N_\varepsilon^{\mathrm{in}}\|v\|_{\mathcal H_\varepsilon}.
\end{equation}
Consequently,
\begin{equation}\label{eq:section-six-source-norm-estimates}
    \|g_\varepsilon\|_{\mathcal X_\varepsilon^*}\leq C\varepsilon^{(d+1)/2}\mathcal N_\varepsilon^{\mathrm{in}}, \qquad \|h_\varepsilon\|_{\mathcal H_\varepsilon}\leq C\varepsilon^{1/2}\mathcal N_\varepsilon^{\mathrm{in}}.
\end{equation}
Propositions~\ref{lem:off-diagonal-blocks} and \ref{prop:invertibilityofB} give
\begin{equation}\label{eq:section-six-off-diagonal-and-background-estimates}
    \|R_1\|_{\mathcal L(\mathcal H_\varepsilon,\mathcal X_\varepsilon^*)}+\|R_2\|_{\mathcal L(\mathcal X_\varepsilon,\mathcal H_\varepsilon)}\leq C\varepsilon^{(d+1)/2}, \qquad \|B^{-1}\|_{\mathcal L(\mathcal H_\varepsilon)}\leq C.
\end{equation}
It follows from \eqref{eq:section-six-energy-space-inverse-of-A} and \eqref{eq:section-six-off-diagonal-and-background-estimates} that
\begin{equation}\label{eq:section-six-small-Schur-perturbation}
    \|A^{-1}R_1B^{-1}R_2\|_{\mathcal L(\mathcal X_\varepsilon)}\leq C\varepsilon^{-d}\varepsilon^{(d+1)/2}\varepsilon^{(d+1)/2}=C\varepsilon.
\end{equation}
After decreasing $\varepsilon$, the operator $I-A^{-1}R_1B^{-1}R_2$ is invertible and its inverse is bounded uniformly in $\varepsilon$. Eliminating $r_\varepsilon$ from the block equation \eqref{operatoreq} gives
\begin{equation}\label{eq:section-six-Schur-complement-equation}
    \bigl(A-R_1B^{-1}R_2\bigr)\varphi_\varepsilon=g_\varepsilon-R_1B^{-1}h_\varepsilon.
\end{equation}

It follows from \eqref{eq:section-six-energy-space-inverse-of-A} and \eqref{eq:section-six-small-Schur-perturbation} that
\begin{equation}\label{eq:section-six-Schur-complement-inverse}
    \|\bigl(A-R_1B^{-1}R_2\bigr)^{-1}\|_{\mathcal L(\mathcal X_\varepsilon^*,\mathcal X_\varepsilon)}\leq C\varepsilon^{-d},
\end{equation}
whereas \eqref{eq:section-six-source-norm-estimates} and \eqref{eq:section-six-off-diagonal-and-background-estimates} imply
\begin{equation}\label{eq:section-six-effective-source-estimate}
    \|g_\varepsilon-R_1B^{-1}h_\varepsilon\|_{\mathcal X_\varepsilon^*}\leq C\left(\varepsilon^{(d+1)/2}+\varepsilon^{(d+2)/2}\right)\mathcal N_\varepsilon^{\mathrm{in}}\leq C\varepsilon^{(d+1)/2}\mathcal N_\varepsilon^{\mathrm{in}}.
\end{equation}
Applying \eqref{eq:section-six-Schur-complement-inverse} to \eqref{eq:section-six-Schur-complement-equation} proves the estimate for $\varphi_\varepsilon$. Since
\begin{equation}\label{eq:section-six-background-component-formula}
    r_\varepsilon=B^{-1}\bigl(h_\varepsilon-R_2\varphi_\varepsilon\bigr),
\end{equation}
we have
\begin{equation}
    \|r_\varepsilon\|_{\mathcal H_\varepsilon}\leq C\left(\varepsilon^{1/2}+\varepsilon^{(d+1)/2}\varepsilon^{(1-d)/2}\right)\mathcal N_\varepsilon^{\mathrm{in}}\leq C\varepsilon^{1/2}\mathcal N_\varepsilon^{\mathrm{in}}.
\end{equation}
This completes the proof.
\end{proof}

We now prove Theorem \ref{thm:main-scattering-convergence-rate}.

\begin{proof}
Set $W_\varepsilon:=\Phi_\varepsilon^M\varphi_\varepsilon$, so that $w_\varepsilon=W_\varepsilon+r_\varepsilon$. We first estimate the two terms in $\Omega$. The estimate \eqref{eq:matrix-L2-estimate-W} gives
\begin{equation}\label{eq:section-six-matrix-L2-estimate-for-lift}
    \|W_\varepsilon\|_{L^2(\Omega\setminus\overline{D_\varepsilon})}\leq C\varepsilon^{d/2}\|\varphi_\varepsilon\|_{\mathcal X_\varepsilon}.
\end{equation}
Inside a bubble $\varepsilon(n+D)$, the average of $W_\varepsilon$ is $(S_M\varphi_\varepsilon)(n)$. The rescaled Poincar\'e inequality, the fact that $S_M$ is unitary on $\ell^2(\mathcal I_\varepsilon)$, and the estimate \eqref{eq:energy-estimate-W-epsilon} therefore give
\begin{equation}\label{eq:section-six-bubble-L2-estimate-for-lift}
    \begin{aligned}
        \|W_\varepsilon\|_{L^2(D_\varepsilon)}^2 &\leq C\varepsilon^d\|\varphi_\varepsilon\|_{\ell^2(\mathcal I_\varepsilon)}^2+C\varepsilon^2\|\nabla W_\varepsilon\|_{L^2(D_\varepsilon)}^2 \\
        &\leq C\varepsilon^d\|\varphi_\varepsilon\|_{\mathcal X_\varepsilon}^2+C\varepsilon^4\int_{D_\varepsilon}a_\varepsilon|\nabla W_\varepsilon|^2\,dx\leq C\varepsilon^d\|\varphi_\varepsilon\|_{\mathcal X_\varepsilon}^2.
    \end{aligned}
\end{equation}
Combining \eqref{eq:section-six-matrix-L2-estimate-for-lift} and \eqref{eq:section-six-bubble-L2-estimate-for-lift}, we obtain
\begin{equation}\label{eq:section-six-interior-L2-estimate-for-lift}
    \|W_\varepsilon\|_{L^2(\Omega)}\leq C\varepsilon^{d/2}\|\varphi_\varepsilon\|_{\mathcal X_\varepsilon}\leq C\varepsilon^{1/2}\mathcal N_\varepsilon^{\mathrm{in}}.
\end{equation}
For the zero-average component, Lemma~\ref{lemmaHepsprop}(a) and Lemma~\ref{lemmaHepsprop}(c) imply
\begin{equation}\label{eq:section-six-interior-L2-estimate-for-background}
    \|r_\varepsilon\|_{L^2(\Omega)}\leq C\varepsilon\|r_\varepsilon\|_{\mathcal H_\varepsilon}\leq C\varepsilon^{3/2}\mathcal N_\varepsilon^{\mathrm{in}}.
\end{equation}
Estimates \eqref{eq:section-six-interior-L2-estimate-for-lift} and \eqref{eq:section-six-interior-L2-estimate-for-background} prove the first estimate in \eqref{eq:section-six-main-interior-and-exterior-L2-rates}.

We next establish the exterior $H^1$ estimate. Let $Q_m$ and $W_{\varepsilon,m}:=\Phi_\varepsilon^M Q_m\varphi_\varepsilon$ be the layer projections and layer lifts introduced in the proof of Lemma~\ref{lemtraceesti}. The boundary-layer energy estimate \eqref{eq:physical-boundary-layer-energy}, followed by Cauchy--Schwarz in $m$, and Lemma~\ref{lem:discreteHardy}, gives
\begin{equation}\label{eq:section-six-exterior-gradient-estimate-for-lift}
    \begin{aligned}
        \|\nabla W_\varepsilon\|_{L^2(\mathbb R^d\setminus\overline\Omega)} &\leq \sum_{m\geq0}\|\nabla W_{\varepsilon,m}\|_{L^2(\mathbb R^d\setminus\overline\Omega)} \\
        &\leq C\varepsilon^{(d-2)/2}\sum_{m\geq0}\rho^{m/2}\|Q_m\varphi_\varepsilon\|_{\ell^2(\mathcal I_\varepsilon)} \\
        &\leq C\varepsilon^{(d-2)/2}\left(\sum_{m\geq0}\rho^{m/2}\right)^{1/2}\left(\sum_{m\geq0}\rho^{m/2}\|Q_m\varphi_\varepsilon\|_{\ell^2(\mathcal I_\varepsilon)}^2\right)^{1/2} \\
        &\leq C\varepsilon^{(d-2)/2}\mathcal M_{\partial\Omega,\varepsilon}(\varphi_\varepsilon)^{1/2}\leq C\varepsilon^{(d-2)/2}\mathfrak D_\varepsilon(\varphi_\varepsilon)^{1/2} \\
        &\leq C\varepsilon^{d/2}\|\varphi_\varepsilon\|_{\mathcal X_\varepsilon}.
    \end{aligned}
\end{equation}
Here, the exponential weight in $\mathcal M_{\partial\Omega,\varepsilon}$ is chosen with $\gamma=-\frac12\log\rho$, up to the harmless shift between the indices $m$ and $d_\varepsilon(n)=m+1$. Together with Lemma~\ref{lemPhiML2}, this implies
\begin{equation}\label{eq:section-six-exterior-H1-estimate-for-lift}
    \|W_\varepsilon\|_{H^1(B_R\setminus\overline\Omega)}\leq C\varepsilon^{d/2}\|\varphi_\varepsilon\|_{\mathcal X_\varepsilon}\leq C\varepsilon^{1/2}\mathcal N_\varepsilon^{\mathrm{in}}.
\end{equation}
Moreover, we have
\begin{equation}\label{eq:section-six-exterior-H1-estimate-for-background}
    \|r_\varepsilon\|_{H^1(B_R\setminus\overline\Omega)}\leq C\|r_\varepsilon\|_{\mathcal H_\varepsilon}\leq C\varepsilon^{1/2}\mathcal N_\varepsilon^{\mathrm{in}}.
\end{equation}
Estimates \eqref{eq:section-six-exterior-H1-estimate-for-lift} and \eqref{eq:section-six-exterior-H1-estimate-for-background} prove \eqref{eq:section-six-main-global-L2-and-exterior-H1-rates}.

It remains to improve the exterior $L^2$ rate. Lemma~\ref{lemPhiML2} and Lemma~\ref{lem:section-six-block-component-estimates} give
\begin{equation}\label{eq:section-six-exterior-L2-estimate-for-lift}
    \|W_\varepsilon\|_{L^2(B_R\setminus\overline\Omega)}\leq C\varepsilon^{(d+1)/2}\|\varphi_\varepsilon\|_{\mathcal X_\varepsilon}\leq C\varepsilon\mathcal N_\varepsilon^{\mathrm{in}}.
\end{equation}
To obtain the same rate for $r_\varepsilon$, take $f\in L^2(B_R\setminus\overline\Omega)$ and let $v_f$ solve the dual sound-soft problem used in \eqref{dualvfproblem}. It is clear that $v_f$ belongs to $\mathcal H_\varepsilon$, because $v_f$ has zero trace on $\partial\Omega$ and vanishes in $\Omega$. The estimate \eqref{estivf}, together with Lemma~\ref{lemmaHepsprop}(c), yields
\begin{equation}\label{eq:section-six-dual-sound-soft-estimate}
    \|v_f\|_{\mathcal H_\varepsilon}+\left\|\left.\frac{\partial v_f}{\partial\nu}\right|_{\partial\Omega+}\right\|_{L^2(\partial\Omega)}\leq C\|f\|_{L^2(B_R\setminus\overline\Omega)}.
\end{equation}
Green's identity and the definition of $B$ give
\begin{equation}\label{eq:section-six-duality-identity-for-background}
    \int_{B_R\setminus\overline\Omega}r_\varepsilon\overline f\,dx=\langle Br_\varepsilon,v_f\rangle_{\mathcal H_\varepsilon}+\int_{\partial\Omega}r_\varepsilon\overline{\left.\dfrac{\partial v_f}{\partial\nu}\right|_+}\,d\sigma.
\end{equation}
Since $v_f=0$ on $\partial\Omega$, the definition of $h_\varepsilon$ gives $\langle h_\varepsilon,v_f\rangle_{\mathcal H_\varepsilon}=0$. Consequently,
\begin{equation}\label{eq:section-six-background-equation-on-dual-test}
    \langle Br_\varepsilon,v_f\rangle_{\mathcal H_\varepsilon}=-\langle R_2\varphi_\varepsilon,v_f\rangle_{\mathcal H_\varepsilon}.
\end{equation}
Moreover, Proposition~\ref{lem:off-diagonal-blocks} and Lemma~\ref{lem:section-six-block-component-estimates} give
\begin{equation}\label{eq:section-six-R2-component-rate}
    \|R_2\varphi_\varepsilon\|_{\mathcal H_\varepsilon}\leq C\varepsilon^{(d+1)/2}\|\varphi_\varepsilon\|_{\mathcal X_\varepsilon}\leq C\varepsilon\mathcal N_\varepsilon^{\mathrm{in}},
\end{equation}
while Lemma~\ref{lemmaHepsprop}(b) gives
\begin{equation}\label{eq:section-six-background-boundary-trace-rate}
    \|r_\varepsilon\|_{L^2(\partial\Omega)}\leq C\varepsilon^{1/2}\|r_\varepsilon\|_{\mathcal H_\varepsilon}\leq C\varepsilon\mathcal N_\varepsilon^{\mathrm{in}}.
\end{equation}
Combining \eqref{eq:section-six-dual-sound-soft-estimate}--\eqref{eq:section-six-background-boundary-trace-rate}, we obtain
\begin{equation}
    \left|\int_{B_R\setminus\overline\Omega}r_\varepsilon\overline f\,dx\right|\leq C\varepsilon\mathcal N_\varepsilon^{\mathrm{in}}\|f\|_{L^2(B_R\setminus\overline\Omega)}.
\end{equation}
Taking the supremum over $f$ yields
\begin{equation}\label{eq:section-six-exterior-L2-estimate-for-background}
    \|r_\varepsilon\|_{L^2(B_R\setminus\overline\Omega)}\leq C\varepsilon\mathcal N_\varepsilon^{\mathrm{in}}.
\end{equation}
Equations \eqref{eq:section-six-exterior-L2-estimate-for-lift} and \eqref{eq:section-six-exterior-L2-estimate-for-background} prove the second estimate in \eqref{eq:section-six-main-interior-and-exterior-L2-rates}. 

Finally, $w_\varepsilon=u_\varepsilon-\widehat u_\varepsilon$ satisfies the homogeneous Helmholtz equation in $\mathbb R^d\setminus\overline\Omega$ and the outgoing radiation condition. Choose a sphere $\partial B_{R_1}$ that is strictly between $\overline\Omega$ and $\partial B_R$. The exterior $L^2$ estimate and interior elliptic estimates on a fixed annulus containing $\partial B_{R_1}$ give Cauchy data of size $O(\varepsilon\mathcal N_\varepsilon^{\mathrm{in}})$ on that sphere. Interior estimates between $\partial\Omega$ and $\partial B_{R_1}$, followed outside $B_{R_1}$ by the standard outgoing representation, prove \eqref{eq:section-six-main-local-H1-rate} for every compact $K\Subset\mathbb R^d\setminus\overline\Omega$. Applying the bounded far-field map to the same Cauchy data gives \eqref{eq:section-six-main-far-field-rate1}. This completes the proof.
\end{proof}

\section{Numerical experiments}
\label{sec:numerical-experiments}

We investigate the quantitative sound-soft limit of Theorem~\ref{themaintheorem} for a two-dimensional array of circular bubbles contained in a disk. The experiments address two complementary questions: the decay of the far-field discrepancy at a fixed nonresonant detuning, and the relation between enhanced scattering discrepancies and the Dirichlet spectrum of the effective operator $L_D$. The latter comparison includes the spatial structure of the bubble-averaged field.

\subsection{Configuration and numerical method}

We take $\Omega=B_1$, $Y=(-\frac{1}{2},\frac{1}{2})^2$, and $D=B_{1/5}$. For $\varepsilon=1/n$, the bubbles are $D_{\varepsilon,m}=\varepsilon(m+D)$, with centers $\varepsilon m$ indexed by
\begin{equation}\label{eq:num-index-set} \mathcal I_\varepsilon=\left\{m\in\mathbb Z^2:\left(|\varepsilon m_1|+\frac{\varepsilon}{2}\right)^2+\left(|\varepsilon m_2|+\frac{\varepsilon}{2}\right)^2<1\right\}. \end{equation}
Thus, only bubbles whose entire centered lattice cell lies in $\Omega$ are retained, according to the geometric convention of the paper. We set $\rho=\kappa=1$ and $\rho_{\rm b}=\kappa_{\rm b}=\varepsilon^2$, so that $\delta=\varepsilon^2$ and $v=v_{\rm b}=1$. The incident field and the frequency are chosen as
\begin{equation}\label{eq:num-frequency} u_\varepsilon^{\mathrm{in}}(x)=e^{\mi k_\varepsilon x_1},\qquad k_\varepsilon=\omega_\varepsilon=\omega_{*,\varepsilon}-\tau\varepsilon^2. \end{equation}
Here, $\omega_{*,\varepsilon}$ is computed from the periodic cell problem at the corner of the Brillouin zone for each $\varepsilon$. Retaining this finite-$\varepsilon$ band edge is essential because an $O(\varepsilon^2)$ error in the reference frequency produces an $O(1)$ error in $\tau$.

The transmission problem is solved by a multiple-scattering expansion. Outside the bubbles, its angular truncation has the form
\begin{equation}\label{eq:num-multipole} u_\varepsilon^{(P)}(x)=e^{\mi k_\varepsilon x_1}+\sum_{m\in\mathcal I_\varepsilon}\sum_{p=-P}^{P}a_{m,p}H_p^{(1)}(k_\varepsilon|x-\varepsilon m|)e^{\mi p\arg(x-\varepsilon m)}. \end{equation}
Regular Bessel expansions are used inside the bubbles. Continuity of the field and of the density-weighted normal derivative determines the coefficients. All pairwise cylindrical-wave interactions are retained; zero-padded FFT convolutions accelerate the matrix-vector products. The nonresonant computations use $P=4$, with a true relative linear-system residual below $10^{-10}$. The detuning scan uses $P=4$ and a relative linear-system residual tolerance of $10^{-12}$. All six detected maxima are subsequently refined with $P=6$. The bubble-averaged envelope comparisons are performed with $P=6$ at the first two maxima.

The reference $\widehat u_\varepsilon$ is the sound-soft total field for the unit disk at the same wavenumber $k_\varepsilon$, extended by zero inside $\Omega$. In polar coordinates, its exterior representation is
\begin{equation}\label{eq:num-soft-reference} \widehat u_\varepsilon(r,\theta)=e^{\mi k_\varepsilon r\cos\theta}-\sum_{p\in\mathbb Z}\mi^p\frac{J_p(k_\varepsilon)}{H_p^{(1)}(k_\varepsilon)}H_p^{(1)}(k_\varepsilon r)e^{\mi p\theta},\qquad r>1. \end{equation}
We retain $|p|\leq40$ in this reference expansion. Our principal observable is the relative discrepancy between the complex far-field patterns,
\begin{equation}\label{eq:num-error} E_\varepsilon(\tau)=\frac{\|u_\varepsilon^{s,\infty}-\widehat u_\varepsilon^{s,\infty}\|_{L^2(\mathbb S^1)}}{\|\widehat u_\varepsilon^{s,\infty}\|_{L^2(\mathbb S^1)}}. \end{equation}
The angular integrals are evaluated by the periodic trapezoidal rule, using $512$ equally spaced directions for the convergence study and $360$ for the detuning and peak computations. This observable compares both amplitude and phase; it is not a difference of scattering intensities.

\subsection{Nonresonant convergence to the sound-soft disk}

We fix $\tau_{\mathrm{nr}}=0.3248382263$, which lies below the first limiting spectral detuning computed in the next subsection. For this incident plane wave, the incident-field norms in Theorem~\ref{themaintheorem} remain bounded as $\varepsilon\to0$. Since the reference far-field norm has a nonzero limit, the far-field estimate in that theorem yields $E_\varepsilon(\tau_{\mathrm{nr}})=O(\varepsilon)$.

Table~\ref{tab:num-convergence} and Figure~\ref{fig:num-convergence} report the results for eight lattice spacings. The discrepancy decreases from $0.5260$ at $\varepsilon=1/8$ to $0.04153$ at $\varepsilon=1/128$. Throughout this range, $E_\varepsilon/\varepsilon$ remains between $4.21$ and $5.49$, consistent with first-order decay. 

\begin{table}[tb]
\centering
\caption{Nonresonant far-field discrepancy at $\tau=\tau_{\mathrm{nr}}$. The bubble radius is $\varepsilon/5$ and the angular truncation is $P=4$.}
\label{tab:num-convergence}
\begin{tabular}{rrrr}
\hline
$\varepsilon^{-1}$ & $|\mathcal I_\varepsilon|$ & $E_\varepsilon$ & $E_\varepsilon/\varepsilon$ \\
\hline
8   & 177   & 0.526033 & 4.2083 \\
16  & 749   & 0.292385 & 4.6782 \\
24  & 1709  & 0.228754 & 5.4901 \\
32  & 3101  & 0.151237 & 4.8396 \\
48  & 7041  & 0.113296 & 5.4382 \\
64  & 12637 & 0.075701 & 4.8448 \\
96  & 28585 & 0.052913 & 5.0797 \\
128 & 50957 & 0.041534 & 5.3163 \\
\hline
\end{tabular}
\end{table}

\begin{figure}[tb]
\centering
\includegraphics[width=\linewidth]{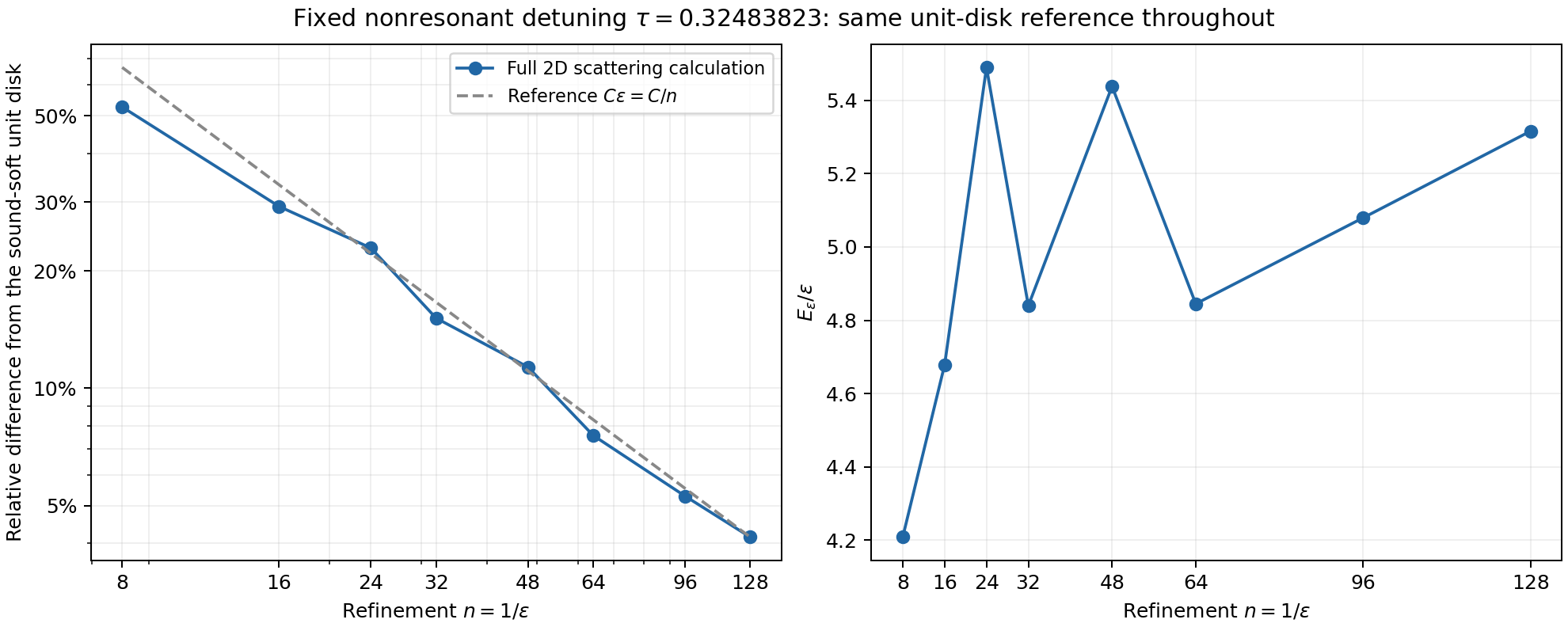}
\caption{Far-field convergence at the fixed nonresonant detuning $\tau_{\mathrm{nr}}$. Left: the computed discrepancy, shown as a percentage, and a reference line $C\varepsilon$, with $C\approx5.3163$ chosen to match the finest-lattice data point. Right: the scaled discrepancy $E_\varepsilon/\varepsilon$. Both horizontal axes use $n=\varepsilon^{-1}$. The reference line illustrates first-order scaling.}
\label{fig:num-convergence}
\end{figure}

To assess angular truncation, the cases $\varepsilon=1/24$ and $1/64$ were repeated with $P=6$. The changes in the complex far field, normalized by the sound-soft far-field norm, were $1.22\times10^{-8}$ and $4.47\times10^{-9}$, respectively. The largest measured linear-system residual in Table~\ref{tab:num-convergence} was $9.67\times10^{-11}$. These checks support the numerical resolution of the reported discrepancies; they are not rigorous forward-error bounds.

Figure~\ref{fig:num-total-field} compares the real parts of the total fields for $\varepsilon=1/128$, where $k_\varepsilon\approx7.09905585$. The exterior interference patterns agree closely, while the sampled field inside the array is small on the scale of the incident wave. The difference is displayed separately to resolve its spatial structure. This field plot is a qualitative comparison; the quantitative convergence test above concerns the far-field estimate and does not measure all the volume and Sobolev norms in Theorem~\ref{themaintheorem}.

\begin{figure}[tb]
\centering
\includegraphics[width=\linewidth]{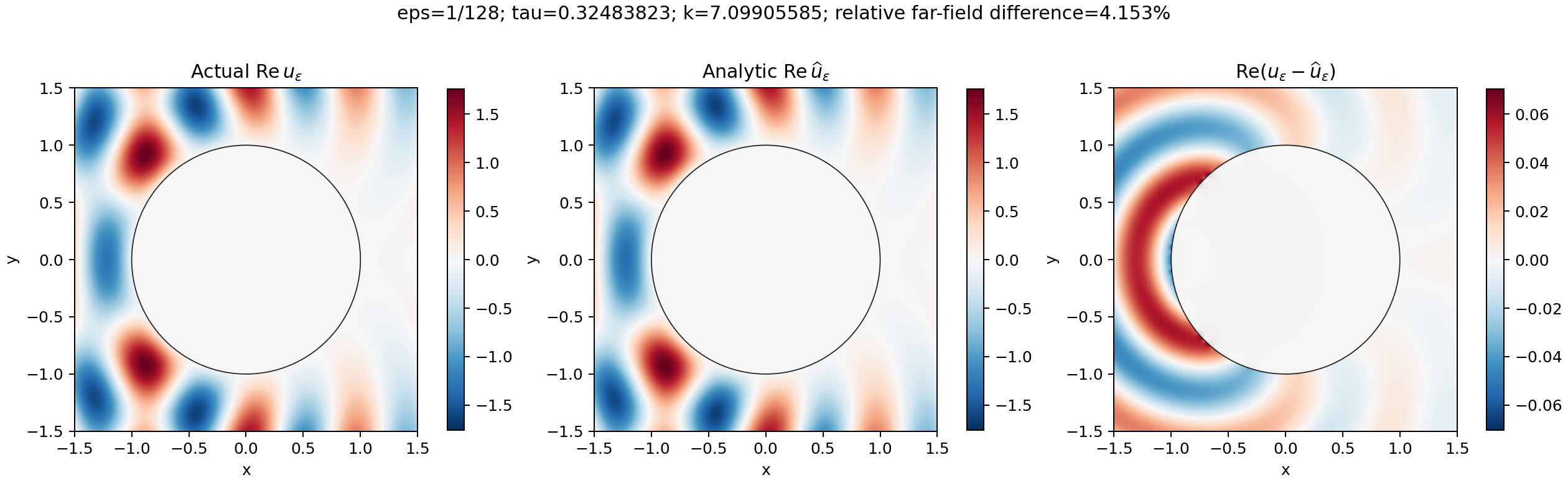}
\caption{Total-field comparison at $\varepsilon=1/128$ and $\tau=\tau_{\mathrm{nr}}$: the bubbly-medium field, the analytic sound-soft reference extended by zero inside the disk, and their difference. The first two panels share a color scale; the difference has its own scale. The circle marks $\partial\Omega$. The fields are sampled on $[-1.5,1.5]^2$ with spacing $\varepsilon/4$.}
\label{fig:num-total-field}
\end{figure}

\subsection{Detuning peaks and the effective Dirichlet spectrum}

For the square cell with a centered circular bubble, symmetry gives $A_{\mathrm{eff}}=a_{\mathrm{eff}}I$. With the normalization of \eqref{defLDintro}--\eqref{defAeffintro}, an independent quasiperiodic cell calculation gives
\begin{equation}\label{eq:num-effective-constants} a_{\mathrm{eff}}\approx3.19014805,\qquad \widehat C^{M}\approx50.40130208,\qquad \omega_0=v_{\rm b}\sqrt{\widehat C^{M}}\approx7.09938744. \end{equation}
The capacitance symbol is evaluated using harmonic expansions in $Y\setminus\overline D$, with quasiperiodic matching on opposite cell faces; its curvature at $M$ determines $A_{\mathrm{eff}}=-\tfrac12\nabla_\alpha^2\widehat C^{M}$. 

Consequently, $L_D=-a_{\mathrm{eff}}\Delta$ on the unit disk with homogeneous Dirichlet boundary conditions. If $j_{m,\ell}$ denotes the $\ell$th positive zero of $J_m$, its eigenvalues and the corresponding limiting spectral detunings are
\begin{equation}\label{eq:num-spectral-detuning} \lambda_{m,\ell}=a_{\mathrm{eff}}j_{m,\ell}^2,\qquad \tau^D_{m,\ell}=\frac{v_{\rm b}^2\lambda_{m,\ell}}{2\omega_0}. \end{equation}
This conversion follows from $\sigma=2\omega_0\tau/v_{\rm b}^2$ in Section~\ref{secproofofmain}. In particular, the first two distinct spectral detunings are $\tau^D_{0,1}\approx1.299353$ and $\tau^D_{1,1}\approx3.298711$.

For $\varepsilon=1/16$, Figure~\ref{fig:num-detuning} shows the relative far-field discrepancy over the extended detuning interval $0\leq\tau\leq12$. The initial scan uses $601$ equally spaced points, corresponding to a spacing $\Delta\tau=0.02$, supplemented by local sampling and peak refinement. Six local maxima are detected and subsequently refined with $P=6$. Their locations and heights are reported in Table~\ref{tab:num-detected-peaks}. Since narrow maxima can be missed by a finite sampling grid, this scan is not asserted to provide an exhaustive list of scattering resonances.

The six maxima are stable under angular refinement: increasing the truncation order from $P=4$ to $P=6$ changes their locations by less than $2\times10^{-5}$ and their heights by less than $3.3\times10^{-8}$. The relative linear-system residual at each of the six refined peaks is below $10^{-12}$. These checks support the numerical resolution of the detected maxima.

\begin{figure}[tb]
\centering
\includegraphics[width=\linewidth]{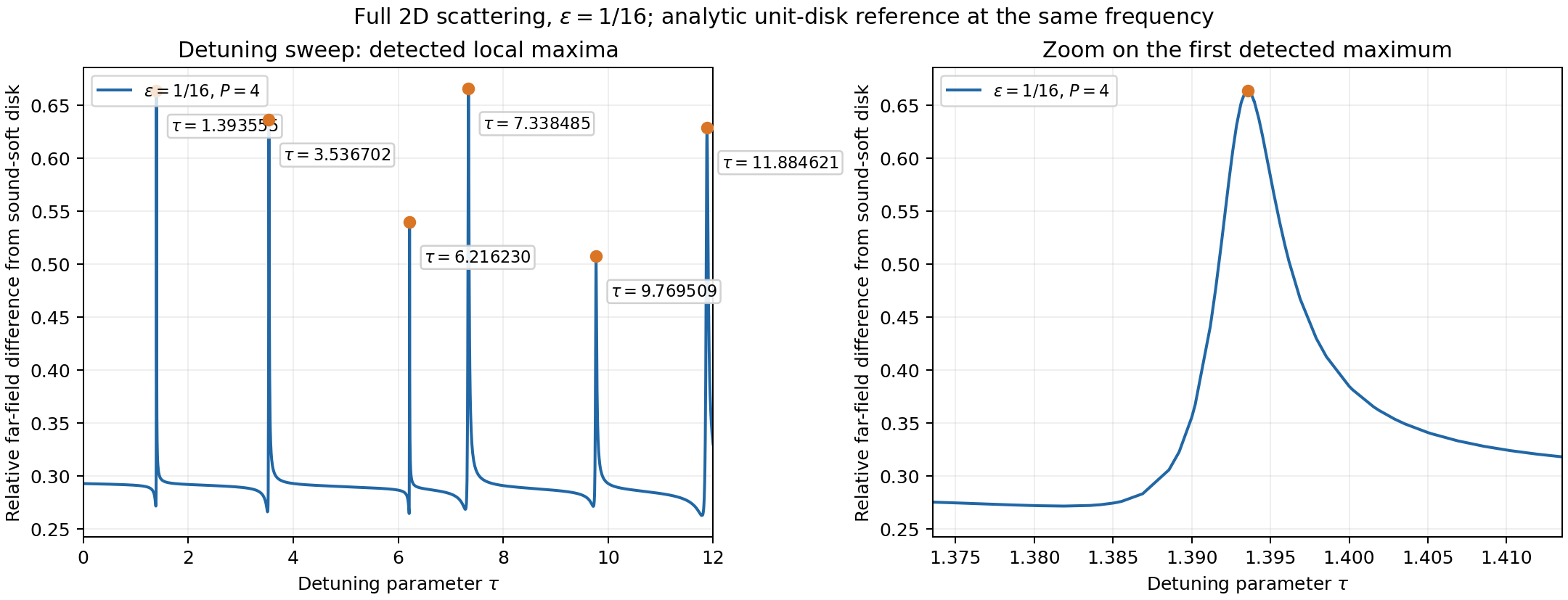}
\caption{Relative far-field discrepancy as a function of detuning for $\varepsilon=1/16$. Left: the extended scan over $0\leq\tau\leq12$, with six detected maxima marked in orange. Right: a magnified view of the first maximum. The curves, markers, and annotations show the original $P=4$ computation. Independently refined $P=6$ locations and heights are reported in Table~\ref{tab:num-detected-peaks}.}
\label{fig:num-detuning}
\end{figure}

\begin{table}[tb]
\centering
\caption{Detected maxima of the relative far-field discrepancy for $\varepsilon=1/16$ in the interval $0\leq\tau\leq12$. The two location columns compare angular truncations $P=4$ and $P=6$; the final column gives the discrepancy evaluated at the refined $P=6$ maximum.}
\label{tab:num-detected-peaks}
\begin{tabular}{rrrr}
\hline
Peak & $\tau_{\mathrm{peak}}^{(4)}$ & $\tau_{\mathrm{peak}}^{(6)}$ & $E_{1/16}^{(6)}$ \\
\hline
1 & 1.393555  & 1.393553  & 0.663078 \\
2 & 3.536702  & 3.536696  & 0.635901 \\
3 & 6.216230  & 6.216219  & 0.539318 \\
4 & 7.338485  & 7.338473  & 0.665508 \\
5 & 9.769509  & 9.769494  & 0.507107 \\
6 & 11.884621 & 11.884602 & 0.628728 \\
\hline
\end{tabular}
\end{table}

We examine the spectral association of the first two maxima in greater detail. Table~\ref{tab:num-peaks} and Figure~\ref{fig:num-peak-comparison} compare their refined locations with the limiting spectral detunings $\tau^D_{0,1}$ and $\tau^D_{1,1}$. Their relative upward shifts are approximately $7.25\%$ and $7.21\%$, respectively. The corresponding envelope comparisons are presented below. The remaining four maxima are reported as resolved scattering-discrepancy peaks; their association with particular eigenspaces of $L_D$ is not established by the present envelope data.

\begin{table}[tb]
\centering
\caption{Limiting spectral detunings and the first two refined discrepancy peaks at $\varepsilon=1/16$, using $P=6$. The shift is $100(\tau_{\mathrm{peak}}/\tau^D_{m,1}-1)\%$. The correlation $\mathcal C$ and shape error $\eta$ are defined in \eqref{eq:num-shape-metrics}.}
\label{tab:num-peaks}
\small
\setlength{\tabcolsep}{4pt}
\begin{tabular}{lrrrrrr}
\hline
Mode & $\tau^D_{m,1}$ & $\tau_{\mathrm{peak}}$ & $E_{1/16}$ & Shift & $\mathcal C$ & $\eta$ \\
\hline
Radial $(0,1)$ & 1.299353 & 1.393553 & 0.663078 & $7.25\%$ & 0.999048 & 0.043624 \\
Dipole $(1,1)$ & 3.298711 & 3.536696 & 0.635901 & $7.21\%$ & 0.997922 & 0.064433 \\
\hline
\end{tabular}
\end{table}

\begin{figure}[tb]
\centering
\includegraphics[width=0.85\linewidth]{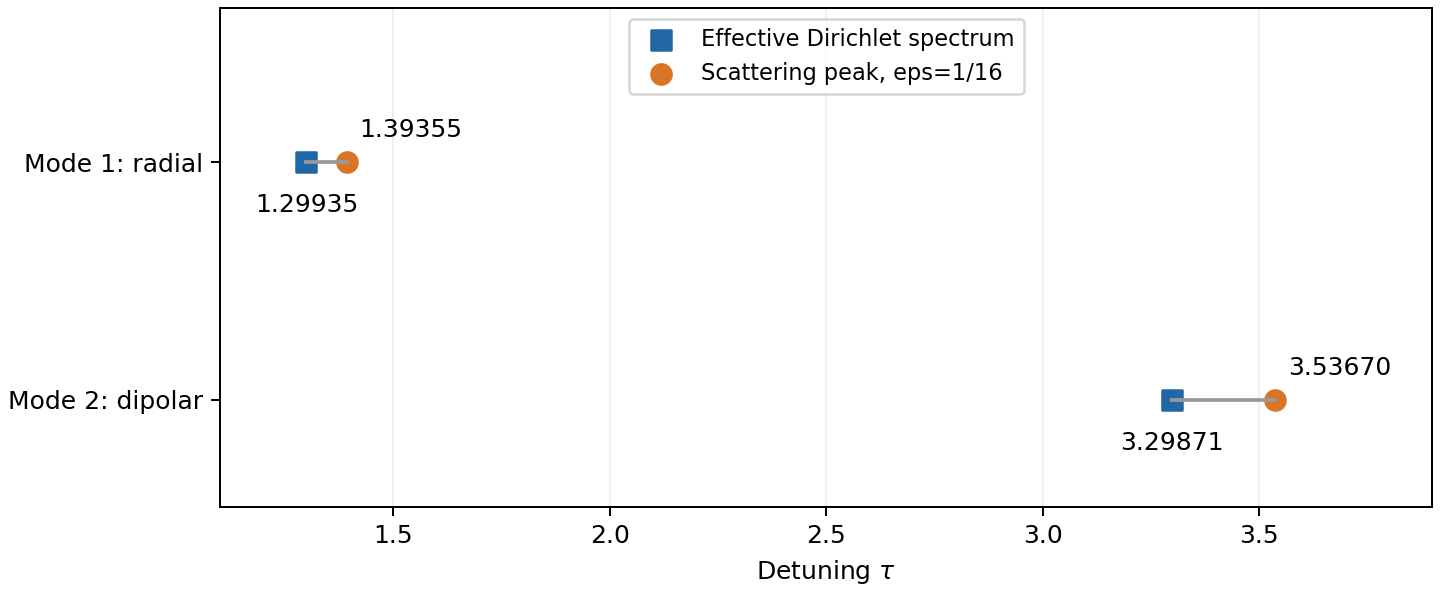}
\caption{Comparison of the first two limiting spectral detunings (blue squares) with the corresponding refined scattering-discrepancy peaks for $\varepsilon=1/16$ and $P=6$ (orange circles). The radial and dipole pairs correspond to the two rows of Table~\ref{tab:num-peaks}.}
\label{fig:num-peak-comparison}
\end{figure}

The detected maxima are finite-array scattering features. Theorem~\ref{themaintheorem} assumes a fixed $\sigma\notin\operatorname{spec}(L_D)$ and establishes convergence as $\varepsilon\to0$; it does not assert that finite-$\varepsilon$ discrepancy maxima occur exactly at the limiting spectral detunings. The shifts of the first two maxima in Table~\ref{tab:num-peaks} are therefore compatible with the theorem. Establishing convergence of peak locations would require tracking the maxima as $\varepsilon$ decreases. The extended scan documents additional frequency-dependent structure, while the envelope comparisons below provide spatial evidence for the spectral association of the first two maxima.

\subsection{Bubble-averaged envelopes at the first two peaks}

To extract the slow component from the oscillatory field, we average over each bubble and remove the alternating Bloch carrier:
\begin{equation}\label{eq:num-envelope} U_m=\frac{1}{|D_{\varepsilon,m}|}\int_{D_{\varepsilon,m}}u_\varepsilon(x)\,dx,\qquad W_m=(-1)^{m_1+m_2}U_m,\qquad m\in\mathcal I_\varepsilon. \end{equation}
The integrals are evaluated analytically from the interior Bessel expansions. This construction corresponds to applying the demodulation $S_M$ to the bubble-average operator $Q_\varepsilon$ used in the paper. In particular, it compares a slow envelope with the effective eigenfunctions rather than comparing the rapidly alternating pressure directly with them.

For the first two peaks, the candidate eigenfunctions are
\begin{equation}\label{eq:num-target-modes} \Phi_{0,1}(r,\theta)=J_0(j_{0,1}r),\qquad \Phi_{1,1}(r,\theta)=J_1(j_{1,1}r)\cos\theta. \end{equation}
The eigenvalue associated with the second function has multiplicity two. The incident field and the array are even under $x_2\mapsto-x_2$, selecting the cosine component rather than its sine counterpart.

For each candidate, let $\phi_m=\Phi_{m_0,1}(\varepsilon m)$, where $m_0=0$ or $1$ labels the angular order, and let $W$ and $\phi$ denote the vectors of sampled values. We allow one global complex amplitude and phase by setting $c=(\sum_m\overline{\phi_m}W_m)/(\sum_m|\phi_m|^2)$. The comparison metrics are
\begin{equation}\label{eq:num-shape-metrics} \mathcal C=\frac{\left|\sum_m\overline{\phi_m}W_m\right|}{\|\phi\|_2\|W\|_2},\qquad \eta=\frac{\|W-c\phi\|_2}{\|W\|_2}=\sqrt{1-\mathcal C^2}. \end{equation}
The equal bubble areas make these normalized quantities identical to their bubble-area-weighted versions. No spatial deformation or position-dependent rescaling is fitted.

Figure~\ref{fig:num-envelopes} shows the aligned envelopes and the analytic eigenfunctions. The radial and dipole correlations are $0.999048$ and $0.997922$, with relative shape errors of $4.36\%$ and $6.44\%$, respectively. Together with the detuning comparison, these spatial agreements support the association of the two maxima with the first two distinct eigenspaces of $L_D$. The nonresonant refinement study provides numerical evidence for the far-field conclusion of Theorem~\ref{themaintheorem}, while the peak and envelope comparisons illustrate the spectral structure associated with its nonresonance condition.

\begin{figure}[tb]
\centering
\includegraphics[width=0.83\linewidth]{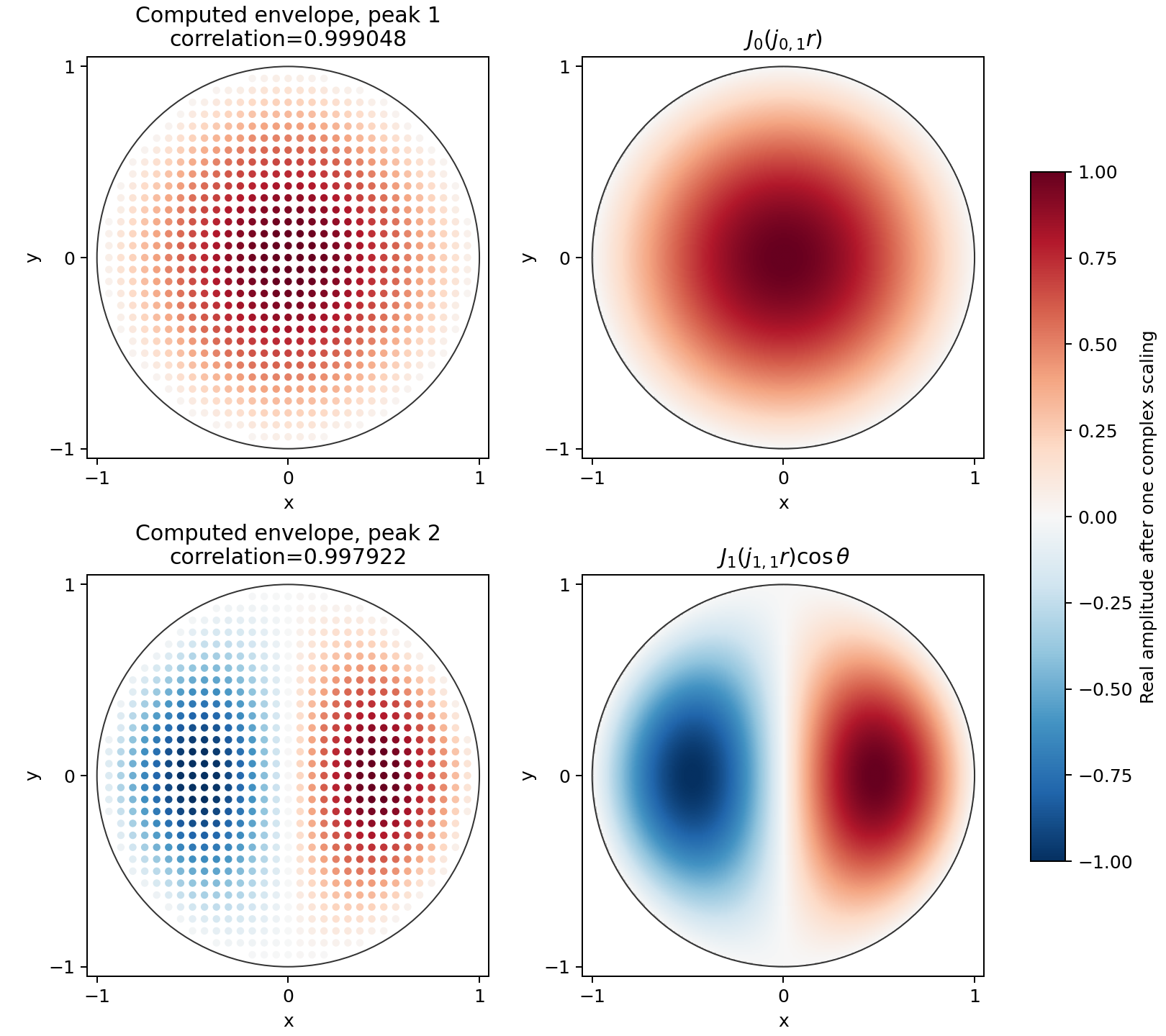}
\caption{Demodulated bubble-average envelopes at the first two refined peaks for $\varepsilon=1/16$. Left: $\operatorname{Re}(W_m/c)$ at the bubble centers. Right: the corresponding analytic Dirichlet eigenfunctions. Each row is divided by $\max_m|\phi_m|$, using the same normalization for the computed and analytic fields. The complex-valued correlations and errors in Table~\ref{tab:num-peaks} are computed before taking real parts.}
\label{fig:num-envelopes}
\end{figure}

\section{Concluding remarks}
We have derived a quantitative sound-soft limit for dense periodic arrays of high-contrast inclusions at frequencies within $O(\eps^2)$ of the first Bloch-band edge. Two operators enter the effective description. The interior Dirichlet operator determines the spectral set excluded by the uniform approximation, while the exterior sound-soft problem gives the nonresonant scattering limit. For positive detunings, the operating frequency lies just below the upper edge of the first band of the infinite crystal. Thus, the existence of propagating Bloch modes in the infinite medium does not determine by itself the transmission of a finite crystal. In one dimension, the exceptional detunings lead to finer Fabry--P\'erot windows, within which nonvanishing transmission and a large interior standing wave persist. The variational capacitance formulation makes it possible to connect them while controlling the boundary created by truncating the periodic medium. In particular, the exponentially localized difference between the finite-array and truncated infinite-array operators leads, together with boundary regularity estimates, to an $O(\eps)$ norm-resolvent approximation.   


The interior-band calculations also clarify the limits of the band-edge description. The exact one-dimensional example shows that a fixed frequency strictly inside the first band can produce different scattering limits along different sequences of periods. At a regular Bloch point in higher dimensions, Appendix \ref{append:B} formally derives a first-order bulk transport equation, in agreement with traveling-wave homogenization \cite{milton}. This local equation does not determine scattering by a finite crystal: the necessary boundary matching and uniform resolvent estimates remain to be established.

Several questions remain open. The first is to generalize our analysis beyond the subwavelength regime; see \cite{fabryperot1,fabryperot3}. Another direction is to consider arrays with several inclusions per cell, in particular, honeycomb structures. Multiple inclusions would naturally lead to matrix-valued Bloch symbols and potentially coupled envelope equations near Dirac degeneracies; see \cite{arma-honeycomb1}. A third challenging question is to extend the interior-band analysis to higher dimensions and rigorously derive the effective behavior of scattered fields at a frequency that is strictly inside the first Bloch band of the corresponding infinitely periodic structure. To simplify the problem, we can consider symmetry-plane slabs and attempt to establish the corresponding boundary matching and uniform resolvent estimates. All of these questions based on the operator framework developed in this work while requiring information beyond the nonresonant convergence theorem.

\section*{Acknowledgments} This work was initiated while H.\,A., X.\,F., and W.\,J. were visiting the Hong Kong Institute for Advanced Study in April 2025. We are very grateful to HKIAS for their hospitality. The work of Y.\,D.\,and W.\,J.\,is partially supported by the NSFC Grant No.\,12571220 and by the New Cornerstone Investigator Program 100001127.

\appendix

\section{Explicit computations in one dimension: a transfer matrix approach}
\label{app:one-dimensional-transfer-matrix}

This appendix gives a self-contained one-dimensional calculation illustrating both the sound-soft limit proved in the main part of the paper and the exceptional Fabry--P\'erot regime excluded by the non-resonance assumption. The model is the matched-wave-speed specialization of the acoustic problem: the density and bulk modulus have the same contrast, so the wave number is $k$ in both phases. The calculation also fixes the phase conventions in the Bloch, transmission, and reflection coefficients.

\subsection{The cell problem and its transfer matrix}

Let $0<a<1$, set $\ell:=(1-a)/2$, and suppose that $L/\varepsilon=:N_\varepsilon\in\mathbb N$. The periodic slab occupies $(0,L)$, the $j$-th cell is $\varepsilon(j+(0,1))$, and its bubble is
\begin{equation}\label{eq:one-dimensional-bubble-array}
    D_{\varepsilon,j}:=\varepsilon\bigl(j+(\ell,a+\ell)\bigr), \qquad j=0,\ldots,N_\varepsilon-1.
\end{equation}
Also, define $D_\varepsilon := \bigcup_{j=0}^{N_\varepsilon -1} D_{\varepsilon,j}$ as the union of all bubbles.

For a real wave number $k>0$, define the coefficient $\mu_\varepsilon$ to equal $\varepsilon^{-2}$ in the bubbles and $1$ in the matrix and exterior. After normalization, the total field satisfies
\begin{equation}\label{eq:one-dimensional-divergence-form-equation}
    (\mu_\varepsilon u_\varepsilon' )'+k^2\mu_\varepsilon u_\varepsilon=0 \qquad\text{in }\mathbb R,
\end{equation}
and its scattered part satisfies the one-dimensional outgoing radiation condition. Thus, $u_\varepsilon''+k^2u_\varepsilon=0$ away from the interfaces, while $u_\varepsilon$ and $\mu_\varepsilon u_\varepsilon'$ are continuous at every interface. Equivalently, at the two interfaces of every cell one has
\begin{equation}\label{eq:one-dimensional-interface-conditions}
    u_\varepsilon'(\varepsilon(j+\ell)-)=\varepsilon^{-2}u_\varepsilon'(\varepsilon(j+\ell)+), \qquad \varepsilon^{-2}u_\varepsilon'(\varepsilon(j+a+\ell)-)=u_\varepsilon'(\varepsilon(j+a+\ell)+).
\end{equation}
We use the state vector $X=(u,J)^{\mathsf T}$, where $J=u'$ in the matrix and $J=\varepsilon^{-2}u'$ in a bubble. Thus, both components of $X$ are continuous across every interface.

The propagation through a matrix interval of length $t$ and through the bubble of length $a\varepsilon$ is described, respectively, by
\begin{equation}\label{eq:one-dimensional-elementary-transfer-matrices}
    \mathsf P_{\mathrm m}(t):=\begin{pmatrix}\cos(kt)&k^{-1}\sin(kt)\\-k\sin(kt)&\cos(kt)\end{pmatrix}, \qquad \mathsf P_{\mathrm b}(a\varepsilon):=\begin{pmatrix}\cos(ka\varepsilon)&\varepsilon^2k^{-1}\sin(ka\varepsilon)\\-\varepsilon^{-2}k\sin(ka\varepsilon)&\cos(ka\varepsilon)\end{pmatrix}.
\end{equation}
Consequently, the transfer matrix of one symmetric cell is
\begin{equation}\label{eq:one-dimensional-cell-transfer-matrix}
    \mathsf M_\varepsilon(k):=\mathsf P_{\mathrm m}(\ell\varepsilon)\mathsf P_{\mathrm b}(a\varepsilon)\mathsf P_{\mathrm m}(\ell\varepsilon)=\begin{pmatrix}A_\varepsilon(k)&B_\varepsilon(k)\\C_\varepsilon(k)&A_\varepsilon(k)\end{pmatrix}.
\end{equation}
Direct multiplication gives
\begin{equation}\label{eq:one-dimensional-transfer-coefficients}
    \begin{aligned}
        A_\varepsilon(k)&=\cos(ka\varepsilon)\cos(k(1-a)\varepsilon)-\frac{\varepsilon^2+\varepsilon^{-2}}{2}\sin(ka\varepsilon)\sin(k(1-a)\varepsilon),\\
        B_\varepsilon(k)&=\frac{1}{k}\left[2\cos(k\ell\varepsilon)\sin(k\ell\varepsilon)\cos(ka\varepsilon)+\sin(ka\varepsilon)\left(\varepsilon^2\cos^2(k\ell\varepsilon)-\varepsilon^{-2}\sin^2(k\ell\varepsilon)\right)\right],\\
        C_\varepsilon(k)&=-k\left[2\cos(k\ell\varepsilon)\sin(k\ell\varepsilon)\cos(ka\varepsilon)+\sin(ka\varepsilon)\left(\varepsilon^{-2}\cos^2(k\ell\varepsilon)-\varepsilon^2\sin^2(k\ell\varepsilon)\right)\right].
    \end{aligned}
\end{equation}
Since both matrices in \eqref{eq:one-dimensional-elementary-transfer-matrices} have determinant one, one has
\begin{equation}\label{eq:one-dimensional-unimodularity}
    A_\varepsilon(k)^2-B_\varepsilon(k)C_\varepsilon(k)=1.
\end{equation}

Let $T_N$ and $U_N$ denote the Chebyshev polynomials of the first and second kind, respectively. The Cayley--Hamilton identity and \eqref{eq:one-dimensional-unimodularity} imply the branch-independent identity
\begin{equation}\label{eq:one-dimensional-power-of-cell-matrix}
    \mathsf M_\varepsilon(k)^{N_\varepsilon}=\begin{pmatrix}T_{N_\varepsilon}(A_\varepsilon)&B_\varepsilon U_{N_\varepsilon-1}(A_\varepsilon)\\C_\varepsilon U_{N_\varepsilon-1}(A_\varepsilon)&T_{N_\varepsilon}(A_\varepsilon)\end{pmatrix}.
\end{equation}
If \begin{equation} \label{def:thetaeps} A_\varepsilon(k)=\cos\theta_\varepsilon(k),
\end{equation}
then
\begin{equation}\label{eq:one-dimensional-Chebyshev-phase-representation}
    T_{N_\varepsilon}(A_\varepsilon)=\cos(N_\varepsilon\theta_\varepsilon), \qquad U_{N_\varepsilon-1}(A_\varepsilon)=\frac{\sin(N_\varepsilon\theta_\varepsilon)}{\sin\theta_\varepsilon},
\end{equation}
where the quotient is understood by continuity when $\sin\theta_\varepsilon=0$. Its limiting value at such a point is generally nonzero.

For a unit-amplitude wave incident from the left, write
\begin{equation}\label{eq:one-dimensional-scattering-ansatz}
    u_\varepsilon(x)=e^{\mi  kx}+R_\varepsilon e^{-\mi  kx} \quad\text{for }x<0, \qquad u_\varepsilon(x)=T_\varepsilon e^{\mi  k(x-L)} \quad\text{for }x>L.
\end{equation}
The normalization on the right makes $T_\varepsilon=u_\varepsilon(L)$. Substitution of the endpoint states into \eqref{eq:one-dimensional-power-of-cell-matrix} yields
\begin{equation}\label{eq:one-dimensional-exact-scattering-coefficients}
    \begin{aligned}
        T_\varepsilon&=\frac{2\mi  k}{2\mi  k\,T_{N_\varepsilon}(A_\varepsilon)+(k^2B_\varepsilon-C_\varepsilon)U_{N_\varepsilon-1}(A_\varepsilon)},\\
        R_\varepsilon&=\frac{(k^2B_\varepsilon+C_\varepsilon)U_{N_\varepsilon-1}(A_\varepsilon)}{2\mi  k\,T_{N_\varepsilon}(A_\varepsilon)+(k^2B_\varepsilon-C_\varepsilon)U_{N_\varepsilon-1}(A_\varepsilon)}.
    \end{aligned}
\end{equation}
For real $k$, the quantity $\operatorname{Im}(\overline{u_\varepsilon}\,\mu_\varepsilon u_\varepsilon')$ is constant. Evaluating it on the two sides of the sample gives
\begin{equation}\label{eq:one-dimensional-flux-conservation}
    |T_\varepsilon|^2+|R_\varepsilon|^2=1.
\end{equation}

\subsection{Dispersion relation and the upper edge of the first band}

A Bloch wave with dimensionless phase $\theta\in[-\pi,\pi]$ satisfies $X(\varepsilon)=e^{ \mi \theta}X(0)$. Therefore, $e^{ \mi \theta}$ is an eigenvalue of $\mathsf M_\varepsilon(k)$, and \eqref{eq:one-dimensional-unimodularity} gives
\begin{equation}\label{eq:one-dimensional-dispersion-relation}
    A_\varepsilon(k)=\cos\theta.
\end{equation}
For a propagating band, we take the principal phase $\theta_\varepsilon(k)\in[0,\pi]$. The corresponding physical quasi-momentum is $\alpha_\varepsilon=\theta_\varepsilon/\varepsilon$. Introduce
\begin{equation}\label{eq:one-dimensional-p-and-r}
    p:=a(1-a), \qquad r:=a^2+(1-a)^2.
\end{equation}
Taylor expansion of \eqref{eq:one-dimensional-transfer-coefficients}, uniformly for $k$ in compact subsets of $(0,\infty)$, gives
\begin{equation}\label{eq:one-dimensional-expansion-of-A}
    A_\varepsilon(k)=1-\frac{p}{2}k^2-\frac{r}{12}k^2\bigl(6-pk^2\bigr)\varepsilon^2+O(\varepsilon^4).
\end{equation}
Let $k_{*,\varepsilon}$ denote the solution near $k_0$ of $A_\varepsilon(k_{*,\varepsilon})=-1$. Since the derivative of the leading term in \eqref{eq:one-dimensional-expansion-of-A} with respect to $k^2$ equals $-p/2\neq0$, the implicit function theorem gives
\begin{equation}\label{eq:one-dimensional-upper-edge-expansion}
    k_{*,\varepsilon}^2=k_0^2-\frac{4r}{3p^2}\varepsilon^2+O(\varepsilon^4), \qquad k_0^2:=\frac{4}{p}.
\end{equation}
This is the upper edge of the first Bloch band.

We first tune the frequency below this edge according to
\begin{equation}\label{eq:one-dimensional-first-detuning}
    k_\varepsilon^2=k_{*,\varepsilon}^2-\tau\varepsilon^2+O(\varepsilon^4), \qquad \tau>0,
\end{equation}
and set
\begin{equation}\label{eq:one-dimensional-effective-wave-number}
    q:=\sqrt{p\tau}.
\end{equation}
Equations \eqref{eq:one-dimensional-dispersion-relation} and \eqref{eq:one-dimensional-expansion-of-A} imply
\begin{equation}\label{eq:one-dimensional-Bloch-phase-expansion}
    \theta_\varepsilon(k_\varepsilon)=\pi-q\varepsilon+O(\varepsilon^3), \qquad \alpha_\varepsilon=\frac{\pi}{\varepsilon}-q+O(\varepsilon^2).
\end{equation}

Direct expansion of the remaining matrix entries gives
\begin{equation}\label{eq:one-dimensional-expansions-of-B-and-C}
    B_\varepsilon(k_\varepsilon)=\frac{a(1-a)^2\tau}{4}\varepsilon^3+O(\varepsilon^5)=\frac{(1-a)q^2}{4}\varepsilon^3+O(\varepsilon^5), \qquad C_\varepsilon(k_\varepsilon)=-\frac{ak_0^2}{\varepsilon}+O(\varepsilon).
\end{equation}
In particular, the nominal $O(\varepsilon)$ coefficient in the expansion of $B_\varepsilon(k)$ vanishes at $k=k_0$. This cancellation must be retained in the band-edge calculation.

All asymptotic statements below are taken as $\varepsilon\to0$ along values for which $N_\varepsilon=L/\varepsilon$ is an integer. Their remainder estimates are uniform with respect to the cell index.

\subsection{The non-resonant regime and the sound-soft limit}

In this section, we consider the non-resonant regime as in the main part of this paper.

\begin{proposition}\label{prop:one-dimensional-non-resonant-limit}
Assume \eqref{eq:one-dimensional-first-detuning} and suppose that
\begin{equation}\label{eq:one-dimensional-non-resonance-condition}
    \sin(qL)\neq0, \qquad\text{equivalently,}\qquad \frac{qL}{\pi}\notin\mathbb Z.
\end{equation}
Then
\begin{equation}\label{eq:one-dimensional-non-resonant-scattering-expansions}
    T_\varepsilon=(-1)^{N_\varepsilon+1}\frac{2\mi  q}{ak_0\sin(qL)}\varepsilon^2+O(\varepsilon^4), \qquad R_\varepsilon=-1-\frac{2\mi  q}{ak_0}\cot(qL)\varepsilon^2+O(\varepsilon^4).
\end{equation}
Thus $T_\varepsilon\to0$ and $R_\varepsilon\to-1$. Outside the slab, the limiting total field is
\begin{equation}\label{eq:one-dimensional-sound-soft-limit-field}
    u_0(x)=e^{\mi  k_0x}-e^{-\mi  k_0x}=2\mi \sin(k_0x) \quad\text{for }x<0, \qquad u_0(x)=0 \quad\text{for }x>L.
\end{equation}
This is precisely the one-dimensional sound-soft scattering field for the obstacle $[0,L]$.

Define the cell profile $\beta_M:[0,1]\to\mathbb R$ by
\begin{equation}\label{eq:one-dimensional-cell-profile}
    \beta_M(y):=\begin{cases}y,&0\leq y\leq\ell,\\\ell,&\ell\leq y\leq a+\ell,\\1-y,&a+\ell\leq y\leq1,\end{cases}
\end{equation}
and extend it antiperiodically by $\beta_M(y+j)=(-1)^j\beta_M(y)$ for $y\in[0,1]$ and $j\in\mathbb Z$. Let
\begin{equation}\label{eq:one-dimensional-non-resonant-envelope}
    \mathcal A(x):=2\mi  k_0\frac{\sin(q(L-x))}{\sin(qL)}.
\end{equation}
Then $\mathcal A$ is the unique solution of
\begin{equation}\label{eq:one-dimensional-non-resonant-envelope-equation}
    \mathcal A''+q^2\mathcal A=0 \quad\text{in }(0,L), \qquad \mathcal A(0)=2 \mi  k_0, \qquad \mathcal A(L)=0.
\end{equation}
Moreover, there exists a constant $C$, independent of $\varepsilon$, such that
\begin{equation}\label{eq:one-dimensional-non-resonant-two-scale-estimate}
    \left\|u_\varepsilon-\varepsilon\mathcal A(\cdot)\beta_M(\cdot/\varepsilon)\right\|_{L^\infty(0,L)}\leq C\varepsilon^2, \qquad \left\|u_\varepsilon-\varepsilon\mathcal A(\cdot)\beta_M(\cdot/\varepsilon)\right\|_{W^{1,\infty}(0,L)}\leq C\varepsilon.
\end{equation}
\end{proposition}

\begin{proof}
Since $N_\varepsilon=L/\varepsilon$, the phase expansion \eqref{eq:one-dimensional-Bloch-phase-expansion} gives
\begin{equation}\label{eq:one-dimensional-non-resonant-trigonometric-expansions}
    \cos(N_\varepsilon\theta_\varepsilon)=(-1)^{N_\varepsilon}\cos(qL)+O(\varepsilon^2), \qquad U_{N_\varepsilon-1}(A_\varepsilon)=(-1)^{N_\varepsilon+1}\frac{\sin(qL)}{q\varepsilon}+O(\varepsilon).
\end{equation}
Inserting \eqref{eq:one-dimensional-expansions-of-B-and-C} and \eqref{eq:one-dimensional-non-resonant-trigonometric-expansions} into \eqref{eq:one-dimensional-exact-scattering-coefficients} proves \eqref{eq:one-dimensional-non-resonant-scattering-expansions}. The exterior conclusion follows from \eqref{eq:one-dimensional-scattering-ansatz} and $k_\varepsilon\to k_0$.

It remains to prove the two-scale estimate. Write $x_j=j\varepsilon$, $u_j=u_\varepsilon(x_j)$, and $d_j=u_\varepsilon'(x_j+)$. The cell transfer formula and the left scattering data give
\begin{equation}\label{eq:one-dimensional-nodal-state-formulas}
    \begin{aligned}
        u_j&=\cos(j\theta_\varepsilon)(1+R_\varepsilon)+ \mi  k_\varepsilon B_\varepsilon\frac{\sin(j\theta_\varepsilon)}{\sin\theta_\varepsilon}(1-R_\varepsilon),\\
        d_j&=C_\varepsilon\frac{\sin(j\theta_\varepsilon)}{\sin\theta_\varepsilon}(1+R_\varepsilon)+ \mi  k_\varepsilon\cos(j\theta_\varepsilon)(1-R_\varepsilon).
    \end{aligned}
\end{equation}
Uniformly for $0\leq j\leq N_\varepsilon$,
\begin{equation}\label{eq:one-dimensional-intermediate-phase-expansions}
    \cos(j\theta_\varepsilon)=(-1)^j\cos(qx_j)+O(\varepsilon^2), \qquad \frac{\sin(j\theta_\varepsilon)}{\sin\theta_\varepsilon}=(-1)^{j+1}\frac{\sin(qx_j)}{q\varepsilon}+O(\varepsilon).
\end{equation}
Because
\begin{equation}\label{eq:one-dimensional-left-boundary-expansion}
    1+R_\varepsilon=-\frac{2\mi  q}{ak_0}\cot(qL)\varepsilon^2+O(\varepsilon^4), \qquad 1-R_\varepsilon=2+O(\varepsilon^2),
\end{equation}
substitution into \eqref{eq:one-dimensional-nodal-state-formulas} yields
\begin{equation}\label{eq:one-dimensional-nodal-state-expansions}
    u_j=O(\varepsilon^2), \qquad d_j=(-1)^j\mathcal A(x_j)+O(\varepsilon^2).
\end{equation}
Indeed,
\begin{equation}\label{eq:one-dimensional-expanded-nodal-derivative}
    d_j=(-1)^j2 \mi  k_0\big(\cos(qx_j)-\sin(qx_j)\cot(qL)\big)+O(\varepsilon^2)=(-1)^j\mathcal A(x_j)+O(\varepsilon^2).
\end{equation}

For $x=\varepsilon(j+y)$ with $0\leq y\leq1$, define the state at the left bubble endpoint by
\begin{equation}\label{eq:one-dimensional-state-at-left-bubble-endpoint}
    u_{j,\ell}:=u_j\cos(k_\varepsilon\ell\varepsilon)+\frac{d_j}{k_\varepsilon}\sin(k_\varepsilon\ell\varepsilon), \qquad d_{j,\ell}:=-k_\varepsilon u_j\sin(k_\varepsilon\ell\varepsilon)+d_j\cos(k_\varepsilon\ell\varepsilon).
\end{equation}
The exact solution in the left matrix part and in the bubble is
\begin{equation}\label{eq:one-dimensional-exact-field-in-left-and-bubble-parts}
    u_\varepsilon(\varepsilon(j+y))=\begin{cases}u_j\cos(k_\varepsilon\varepsilon y)+\dfrac{d_j}{k_\varepsilon}\sin(k_\varepsilon\varepsilon y),&0\leq y\leq\ell,\\u_{j,\ell}\cos(k_\varepsilon\varepsilon(y-\ell))+\dfrac{\varepsilon^2d_{j,\ell}}{k_\varepsilon}\sin(k_\varepsilon\varepsilon(y-\ell)),&\ell\leq y\leq a+\ell.\end{cases}
\end{equation}
At the right bubble endpoint, set
\begin{equation}\label{eq:one-dimensional-state-at-right-bubble-endpoint}
    u_{j,a+\ell}:=u_{j,\ell}\cos(k_\varepsilon a\varepsilon)+\frac{\varepsilon^2d_{j,\ell}}{k_\varepsilon}\sin(k_\varepsilon a\varepsilon), \qquad d_{j,a+\ell}:=-\frac{k_\varepsilon}{\varepsilon^2}u_{j,\ell}\sin(k_\varepsilon a\varepsilon)+d_{j,\ell}\cos(k_\varepsilon a\varepsilon).
\end{equation}
The exact solution in the right matrix part is
\begin{equation}\label{eq:one-dimensional-exact-field-in-right-part}
    u_\varepsilon(\varepsilon(j+y))=u_{j,a+\ell}\cos(k_\varepsilon\varepsilon(y-a-\ell))+\frac{d_{j,a+\ell}}{k_\varepsilon}\sin(k_\varepsilon\varepsilon(y-a-\ell)), \qquad a+\ell\leq y\leq1.
\end{equation}
Using \eqref{eq:one-dimensional-nodal-state-expansions} in these formulas, together with
\begin{equation}\label{eq:one-dimensional-critical-flux-reversal}
    k_0^2a\ell=2,
\end{equation}
gives
\begin{equation}\label{eq:one-dimensional-cellwise-non-resonant-expansion}
    u_\varepsilon(\varepsilon(j+y))=\varepsilon\mathcal A(x_j)\beta_M(j+y)+O(\varepsilon^2), \qquad u_\varepsilon'(\varepsilon(j+y))=\mathcal A(x_j)\beta_M'(j+y)+O(\varepsilon).
\end{equation}
The errors are uniform in $j$ and $y$. Since $|\mathcal A(x)-\mathcal A(x_j)|\leq C\varepsilon$ on the $j$-th cell, \eqref{eq:one-dimensional-non-resonant-two-scale-estimate} follows. 
\end{proof}

\subsection{Fabry--P\'erot resonances and the next detuning scale}

The exact formulas already show what is special about a Fabry--P\'erot resonance. If
\begin{equation}\label{eq:one-dimensional-exact-Fabry-Perot-condition}
    U_{N_\varepsilon-1}(A_\varepsilon)=0,
\end{equation}
or equivalently if $\theta_\varepsilon\in(0,\pi)$ and $N_\varepsilon\theta_\varepsilon=s\pi$ for some $s\in\mathbb Z$, then
\begin{equation}\label{eq:one-dimensional-exact-Fabry-Perot-matrix}
    \mathsf M_\varepsilon^{N_\varepsilon}=(-1)^sI.
\end{equation}
Consequently,
\begin{equation}\label{eq:one-dimensional-exact-Fabry-Perot-transparency}
    R_\varepsilon=0, \qquad T_\varepsilon=(-1)^s.
\end{equation}
Thus the slab is perfectly transmitting in modulus, although the transmitted wave may acquire a phase $\pi$. The restriction $\theta_\varepsilon\in(0,\pi)$ is important: at a cell band edge, $\sin(N_\varepsilon\theta_\varepsilon)=0$, but the continuously extended quotient in \eqref{eq:one-dimensional-Chebyshev-phase-representation} need not vanish.

At the first detuning scale \eqref{eq:one-dimensional-first-detuning}, the limiting Fabry--P\'erot condition is
\begin{equation}\label{eq:one-dimensional-limiting-Fabry-Perot-condition}
    qL=m\pi \quad\text{for some integer }m\geq1.
\end{equation}
This is only a leading-order condition and does not imply exact transparency at a fixed $\varepsilon$. The $O(\varepsilon^2)$ correction to $N_\varepsilon\theta_\varepsilon$ contributes at leading order to scattering, so the $O(\varepsilon^4)$ frequency detuning must be retained.

\begin{proposition}[The Fabry--P\'erot window and resonant field enhancement]\label{prop:one-dimensional-Fabry-Perot-window}
Assume \eqref{eq:one-dimensional-limiting-Fabry-Perot-condition} and let
\begin{equation}\label{eq:one-dimensional-second-detuning}
    k_\varepsilon^2=k_{*,\varepsilon}^2-\tau\varepsilon^2+\sigma\varepsilon^4+O(\varepsilon^6),
\end{equation}
where $\sigma\in\mathbb R$ is fixed. Define
\begin{equation}\label{eq:one-dimensional-resonant-parameters}
    \eta:=\frac{p\sigma/2+r\tau/6-p^2\tau^2/24}{q}, \qquad \Gamma:=\frac{ak_0\eta L}{q}, \qquad t_*:=\frac{2 \mi }{2 \mi +\Gamma}.
\end{equation}
Then
\begin{equation}\label{eq:one-dimensional-resonant-scattering-law}
    T_\varepsilon=(-1)^{N_\varepsilon+m}\frac{2 \mi }{2 \mi +\Gamma}+O(\varepsilon^2), \qquad R_\varepsilon=-\frac{\Gamma}{2\mi +\Gamma}+O(\varepsilon^2).
\end{equation}
In particular,
\begin{equation}\label{eq:one-dimensional-resonant-energy-law}
    |T_\varepsilon|^2\rightarrow\frac{4}{4+\Gamma^2}, \qquad |R_\varepsilon|^2\rightarrow\frac{\Gamma^2}{4+\Gamma^2}.
\end{equation}
The phase-corrected quantity $(-1)^{N_\varepsilon+m}T_\varepsilon$ has a limit, whereas $T_\varepsilon$ itself need not converge unless the parity of $N_\varepsilon$ is fixed. The transparent center of the resonant window is determined by $\eta=0$, or equivalently by
\begin{equation}\label{eq:one-dimensional-transparent-second-detuning}
    \sigma=\sigma_{\mathrm{tr}}:=-\frac{r\tau}{3p}+\frac{p\tau^2}{12}.
\end{equation}
At this value,
\begin{equation}\label{eq:one-dimensional-transparent-coefficients}
    R_\varepsilon=O(\varepsilon^2), \qquad T_\varepsilon=(-1)^{N_\varepsilon+m}+O(\varepsilon^2).
\end{equation}

The interior field is nevertheless resonantly amplified. More precisely, define
\begin{equation}\label{eq:one-dimensional-resonant-envelope}
    \mathcal A_{\mathrm{res}}(x):=\frac{ak_0^2t_*}{q}\sin(qx).
\end{equation}
Then $\mathcal A_{\mathrm{res}}''+q^2\mathcal A_{\mathrm{res}}=0$ in $(0,L)$ and $\mathcal A_{\mathrm{res}}(0)=\mathcal A_{\mathrm{res}}(L)=0$, while
\begin{equation}\label{eq:one-dimensional-resonant-two-scale-estimate}
    \left\|\varepsilon u_\varepsilon-\mathcal A_{\mathrm{res}}(\cdot)\beta_M(\cdot/\varepsilon)\right\|_{L^\infty(0,L)}+\left\|\varepsilon^2u_\varepsilon'-\mathcal A_{\mathrm{res}}(\cdot)\beta_M'(\cdot/\varepsilon)\right\|_{L^\infty(0,L)}\leq C\varepsilon.
\end{equation}
Consequently,
\begin{equation}\label{eq:one-dimensional-resonant-L2-amplification}
    \begin{aligned}
        \varepsilon\|u_\varepsilon\|_{L^2(0,L)}&\rightarrow\left(\int_0^L|\mathcal A_{\mathrm{res}}(x)|^2\,dx\right)^{1/2}\left(\int_0^1|\beta_M(y)|^2\,dy\right)^{1/2},\\
        \int_0^1|\beta_M(y)|^2\,dy&=\frac{(1-a)^2(1+2a)}{12}.
    \end{aligned}
\end{equation}
Thus $\|u_\varepsilon\|_{L^2(0,L)}$ and $\|u_\varepsilon\|_{L^\infty(0,L)}$ are of order $\varepsilon^{-1}$. In particular, near perfect exterior transmission at $\eta=0$ coexists with a giant interior standing wave.
\end{proposition}

\begin{proof}
Taylor expansion with respect to $k^2$ around the exact root $k_{*,\varepsilon}^2$ gives
\begin{equation}\label{eq:one-dimensional-A-at-second-detuning}
    A_\varepsilon(k_\varepsilon)+1=\frac{p\tau}{2}\varepsilon^2+\left(-\frac{p\sigma}{2}-\frac{r\tau}{6}\right)\varepsilon^4+O(\varepsilon^6).
\end{equation}
Indeed,
\begin{equation}\label{eq:one-dimensional-derivative-at-edge}
    \partial_{k^2}A_\varepsilon(k_{*,\varepsilon})=-\frac{p}{2}+\frac{r}{6}\varepsilon^2+O(\varepsilon^4),
\end{equation}
and the quadratic Taylor remainder is $O(\varepsilon^6)$ because the leading term in \eqref{eq:one-dimensional-expansion-of-A} is affine in $k^2$. On the other hand,
\begin{equation}\label{eq:one-dimensional-cosine-at-second-detuning}
    \cos\bigl(\pi-q\varepsilon+\eta\varepsilon^3\bigr)=-1+\frac{q^2}{2}\varepsilon^2+\left(-q\eta-\frac{q^4}{24}\right)\varepsilon^4+O(\varepsilon^6).
\end{equation}
Using $q^2=p\tau$ and comparing the last two expansions gives
\begin{equation}\label{eq:one-dimensional-phase-at-second-detuning}
    \theta_\varepsilon=\pi-q\varepsilon+\eta\varepsilon^3+O(\varepsilon^5).
\end{equation}
Since $N_\varepsilon\varepsilon=L$ and $qL=m\pi$, it follows that
\begin{equation}\label{eq:one-dimensional-resonant-phase-ratio}
    \cos(N_\varepsilon\theta_\varepsilon)=(-1)^{N_\varepsilon+m}+O(\varepsilon^4), \qquad U_{N_\varepsilon-1}(A_\varepsilon)=(-1)^{N_\varepsilon+m}\frac{\eta L}{q}\varepsilon+O(\varepsilon^3).
\end{equation}
Using \eqref{eq:one-dimensional-expansions-of-B-and-C}, the denominator and reflection numerator in \eqref{eq:one-dimensional-exact-scattering-coefficients} satisfy
\begin{equation}\label{eq:one-dimensional-resonant-denominator-and-numerator}
    \begin{aligned}
        2 \mi  k_\varepsilon T_{N_\varepsilon}(A_\varepsilon)+(k_\varepsilon^2B_\varepsilon-C_\varepsilon)U_{N_\varepsilon-1}(A_\varepsilon)&=(-1)^{N_\varepsilon+m}k_0(2 \mi +\Gamma)+O(\varepsilon^2),\\
        (k_\varepsilon^2B_\varepsilon+C_\varepsilon)U_{N_\varepsilon-1}(A_\varepsilon)&=-(-1)^{N_\varepsilon+m}k_0\Gamma+O(\varepsilon^2).
    \end{aligned}
\end{equation}
This proves \eqref{eq:one-dimensional-resonant-scattering-law}; \eqref{eq:one-dimensional-resonant-energy-law} and \eqref{eq:one-dimensional-transparent-second-detuning} follow immediately.

It remains to prove the field estimate. Equations \eqref{eq:one-dimensional-nodal-state-formulas}, \eqref{eq:one-dimensional-phase-at-second-detuning}, and \eqref{eq:one-dimensional-resonant-scattering-law} imply, uniformly for $0\leq j\leq N_\varepsilon$,
\begin{equation}\label{eq:one-dimensional-resonant-nodal-state}
    u_j=(-1)^jt_*\cos(qx_j)+O(\varepsilon^2), \qquad d_j=\varepsilon^{-2}(-1)^j\mathcal A_{\mathrm{res}}(x_j)+O(1).
\end{equation}
Substituting \eqref{eq:one-dimensional-resonant-nodal-state} into the exact cellwise formulas \eqref{eq:one-dimensional-state-at-left-bubble-endpoint}--\eqref{eq:one-dimensional-exact-field-in-right-part}, and using \eqref{eq:one-dimensional-critical-flux-reversal}, gives, uniformly for $x=\varepsilon(j+y)$ and $0\leq y\leq1$,
\begin{equation}\label{eq:one-dimensional-resonant-cellwise-expansion}
    \varepsilon u_\varepsilon(x)=\mathcal A_{\mathrm{res}}(x_j)\beta_M(j+y)+O(\varepsilon), \qquad \varepsilon^2u_\varepsilon'(x)=\mathcal A_{\mathrm{res}}(x_j)\beta_M'(j+y)+O(\varepsilon).
\end{equation}
The Lipschitz continuity of $\mathcal A_{\mathrm{res}}$ proves \eqref{eq:one-dimensional-resonant-two-scale-estimate}. Finally, $|\beta_M|^2$ is one-periodic, so the uniform field estimate and a cellwise Riemann-sum argument yield \eqref{eq:one-dimensional-resonant-L2-amplification}. Direct integration over the three parts of the cell gives the stated value of $\int_0^1|\beta_M|^2$. This completes the proof.
\end{proof}

Propositions~\ref{prop:one-dimensional-non-resonant-limit} and \ref{prop:one-dimensional-Fabry-Perot-window} distinguish the two possible limits. Away from $qL\in\pi\mathbb Z$, the exterior coupling forces an $O(\varepsilon)$ interior field and the slab converges to a sound-soft obstacle. When $qL\in\pi\mathbb Z$, a fourth-order frequency correction survives in the scattering limit and the microscopic field is amplified to order $\varepsilon^{-1}$. The non-resonance condition \eqref{eq:one-dimensional-non-resonance-condition} is therefore the one-dimensional analogue of excluding the Dirichlet spectrum of the effective operator in the main theorem.

\subsection{Frequencies strictly inside the first Bloch band}

In this subsection, we examine frequencies strictly inside the first Bloch band and show that at regular Bloch points the leading bulk modulation is governed by a \emph{first-order transport equation} and explicitly show that a fixed interior frequency does not need to have a unique scattering limit as the period $\eps$ tends to zero.

Fix a compact interval
\begin{equation}
 K\Subset(0,k_0),\qquad k_0=\frac2{\sqrt{p}}, 
 \label{eq:Ainterval}
\end{equation}
where $p$ is defined in \eqref{eq:one-dimensional-p-and-r}
and let 
\begin{equation}
    \theta_0(k)=\arccos\left(1-\frac{pk^2}{2}\right),\qquad
 \theta_0'(k)=\frac{pk}{\sin\theta_0(k)},\qquad
 \theta_2(k)=\frac{rk^2(6-pk^2)}{12\sin\theta_0(k)}.
\end{equation}
Introduce
\begin{equation}
 \mathcal{Z}_\eps:=\frac{\sin\theta_\eps}{B_\eps},\qquad
 \eta_\eps:=\frac{\mathcal{Z}_\eps}{k},\qquad
 \Phi_\eps:=N_\eps\theta_\eps,
\end{equation}
with $\theta_\eps \in [0,\pi]$ being defined in \eqref{def:thetaeps}.
For $k \in K$, we have
\begin{align}
 \mathcal{Z}_\eps&=\frac{ak^2}{\eps\sin\theta_0}\bigl(1+O(\eps^2)\bigr),\qquad
 \eta_\eps=\frac{ak}{\eps\sin\theta_0}\bigl(1+O(\eps^2)\bigr),\label{eq:AZ}\\
 \Phi_\eps&=N_\eps\theta_0+L\eps\theta_2+O(\eps^3).\label{eq:Aphase}
\end{align}
The $O(\eps)$ correction to the total phase in \eqref{eq:Aphase} is of the same order as the resonance phase width. 

\subsubsection{Exact scattering formulas and nonresonant convergence}
The determinant identity gives $C_\eps=- \mathcal{Z}_\eps\sin\theta_\eps$, and hence
\begin{equation}
 \mathsf M_\eps^{N_\eps}=
 \begin{pmatrix}
 \cos\Phi_\eps& \mathcal{Z}_\eps^{-1}\sin\Phi_\eps\\
 -\mathcal{Z}_\eps\sin\Phi_\eps&\cos\Phi_\eps
 \end{pmatrix}.
\end{equation}
Substitution of the endpoint states $(1+R_\eps, \mi k(1-R_\eps))^{\mathsf T}$ and $(T_\eps, \mi kT_\eps)^{\mathsf T}$ yields the exact identities
\begin{align}
 T_\eps&=\frac{2\mi}{2\mi\cos\Phi_\eps+
 (\eta_\eps+\eta_\eps^{-1})\sin\Phi_\eps},\label{eq:AT}\\
 R_\eps&=\frac{(\eta_\eps^{-1}-\eta_\eps)\sin\Phi_\eps}
 {2\mi\cos\Phi_\eps+(\eta_\eps+\eta_\eps^{-1})\sin\Phi_\eps},\label{eq:AR}\\
 |T_\eps|^2&=\frac1{1+\tfrac14(\eta_\eps-\eta_\eps^{-1})^2\sin^2\Phi_\eps},
 \qquad |R_\eps|^2+|T_\eps|^2=1.\label{eq:Apower}
\end{align}

It can be easily seen that the following result holds. 
\begin{proposition}[Interior-band sound-soft limit]\label{prop:Asoft}
For $k\in K$,
\begin{equation}
 |T_\eps|+|R_\eps+1|
 \le \frac{C\eps}{\eps+|\sin\Phi_\eps|}.\label{eq:Auniform}
\end{equation}
If $|\sin\Phi_\eps|\ge d_0>0$, then
\begin{align}
 T_\eps&=\frac{2\mi\eps\sin\theta_0}{ak\sin\Phi_\eps}+O(\eps^2),\\
 R_\eps&=-1+\frac{2\mi\eps\sin\theta_0}{ak}\cot\Phi_\eps+O(\eps^2).
\end{align}
More generally, $|\sin\Phi_\eps|/\eps\to\infty$ suffices for $T_\eps\to0$ and $R_\eps\to-1$.
\end{proposition}

\subsubsection{Interior fields and the transport description}
At $x_j=j\eps$, the backward propagation from $x=L$ gives
\begin{align}
 u_j&=T_\eps\left[\cos((N_\eps-j)\theta_\eps)
 - \mi\eta_\eps^{-1}\sin((N_\eps-j)\theta_\eps)\right],\label{eq:Anodes}\\
 J_j&=T_\eps\left[\mathcal{Z}_\eps\sin((N_\eps-j)\theta_\eps)
 + \mi k\cos((N_\eps-j)\theta_\eps)\right].
\end{align}
Propagation through each matrix segment and bubble gives
\begin{equation}
 \|{u_\eps}\|_{L^\infty(0,L)}\le C|T_\eps|.
 \label{eq:Ainterior}
\end{equation}

Moreover, we have
\begin{equation}
 \|{u_\eps'}\|_{L^\infty((0,L)\setminus D_\eps)} \le C\eps^{-1}|T_\eps|,
 \quad \|{u_\eps'}\|_{L^\infty(D_\eps)}\le C\eps|T_\eps|.\label{eq:Ainteriorprime}
\end{equation}

To express the resonance criterion in terms of the accumulated Bloch phase, we let
\begin{equation*}
 d_\eps(k):=\mathrm{dist}\, \left(\frac{N_\eps\theta_\eps(k)}\pi,\Z\right).
\end{equation*}
Equations \eqref{eq:AZ}, \eqref{eq:Apower}, and \eqref{eq:Ainterior}, together with $u_\eps(L)=T_\eps$, imply uniformly for $k\in K$ that
\begin{equation*}
 \|{u_\eps}\|_{L^\infty(0,L)} = O(|T_\eps|)= O(\frac{\eps}{\eps+d_\eps(k)}).
\end{equation*}
Consequently, the interior field tends uniformly to zero when $d_\eps(k)/\eps\to\infty$, whereas its amplitude remains of order one when $d_\eps(k)=O(\eps)$. For a fixed wave number $k_c$, setting $\theta_c=\theta_0(k_c)$, these two alternatives can equivalently be expressed using
\begin{equation*}
 d_{N_\eps}:=\mathrm{dist}\, \left(\frac{N_\eps\theta_c}\pi,\Z\right),
\end{equation*}
since \eqref{eq:Aphase} gives $d_\eps(k_c)=d_{N_\eps}+O(\eps)$. However, the phase correction in \eqref{eq:Aphase} must be retained to determine the limiting response within a resonance window. This order $O(1)$ interior-band response differs from the order $O(\eps^{-1})$ field enhancement in the band-edge resonance regime.

Furthermore, the corresponding bulk capacitance symbol is
\begin{equation}
 \widehat{C}^\theta=\frac{2-2\cos\theta}{p}.
\end{equation}
Therefore, for a slowly modulated carrier $q_j=e^{\mi  j\theta_c}A(\eps j)$ and
$k_\eps=k_c+\eps\nu+O(\eps^2)$, where $\widehat{C}^{\theta_c}=k_c^2$ and $k_c \in (0,k_0)$, Taylor expansion of the discrete operator gives immediately the leading envelope equation
\begin{equation}
 - \mi\partial_\theta\widehat{C}^{\theta_c} A'(x)-2k_c\nu A(x)=0.\label{eq:Atransport}
\end{equation}
There are two independent bulk carriers at $\pm\theta_c$. Their coefficients are determined by matching at both ends of the sample. At an interior regular point, \eqref{eq:Atransport} is of first order. Note that the obtained second-order equation near the band-edge arises because the first derivative of the band vanishes there.

For the two carriers $\pm\theta_c$, the transport equations take the following form:
\begin{equation*}
 - \mi v_c A_+'=\nu A_+,\qquad \mi v_c A_-'=\nu A_-,\qquad
 v_c:=\frac{\partial_\theta\widehat{C}^{\theta_c}}{2k_c}
 =\frac{\sin\theta_c}{pk_c}>0.
\end{equation*}
Thus, $A_\pm(x)=a_\pm e^{\pm \mi\nu x/v_c}$, with coefficients determined by matching at the two endpoints. For fixed $k_c$, the transport velocities $\pm v_c$ are independent of the sequence $\eps=L/N_\eps$. The sequence dependence enters through the accumulated phase, the selected detuning, and the envelope coefficients. Indeed, if $k_\eps=k_c+\eps\nu+O(\eps^2)$, then
\begin{equation*}
 N_\eps\theta_\eps(k_\eps)
 =N_\eps\theta_c+L\theta_0'(k_c)\nu+O(\eps).
\end{equation*}
Therefore, an order $O(1)$ interior amplitude requires this phase to lie within $O(\eps)$ of $\pi\Z$. Equivalently, $k_\eps$ must lie within $O(\eps^2)$ of an exact resonance center. In fact, arbitrary $O(\eps)$ detuning does not need to preserve the resonant response.

\subsubsection{Resonance centers, spacing, and linewidths}
Let $k_{m,\eps}$ be defined by
\begin{equation}
 N_\eps\theta_\eps(k_{m,\eps})=m\pi,
\end{equation}
for integers $m$ such that the corresponding wave numbers remain in a fixed compact subinterval of $(0,k_0)$. At these exact centers $R_\eps=0$ and $T_\eps=(-1)^m$. Put
\begin{equation}
 k_{m,0}=\frac2{\sqrt p}\sin\frac{m\pi}{2N_\eps}.
\end{equation}
The implicit function theorem and 
\begin{equation}
    \theta_\eps(k) =\theta_0(k)+\eps^2\theta_2(k)+O(\eps^4),\qquad
 \theta_2(k)=\frac{rk^2(6-pk^2)}{12\sin\theta_0(k)}
\end{equation}
give
\begin{align}
 k_{m,\eps}&=k_{m,0}
 -\frac{r k_{m,0}(6-pk_{m,0}^2)}{12p}\eps^2+O(\eps^4),\label{eq:Acenters}\\
 k_{m+1,\eps}-k_{m,\eps}
 &=\frac{\pi\eps\sin\theta_0(k_{m,\eps})}{Lp k_{m,\eps}}+O(\eps^2).
\end{align}
We immediately obtain the following result.
\begin{proposition}[Interior-band resonance window]\label{prop:Ares}
Suppose $k_{m,\eps}\to k_c\in(0,k_0)$ and
$k=k_{m,\eps}+\eps^2 z$, with $z$ in a bounded set. Define
\begin{equation}
 G_c(z)=\frac{La}{1-pk_c^2/4}\,z.
\end{equation}
Then
\begin{equation}
 T_\eps=(-1)^m\frac{2\mi}{2\mi+G_c(z)}+o(1),\quad
 R_\eps=-\frac{G_c(z)}{2 \mi+G_c(z)}+o(1),\quad
 |T_\eps|^2\rightarrow\frac4{4+G_c(z)^2}.\label{eq:Alorentz}
\end{equation}
The half width at half maximum in the variable $k$ is 
\begin{equation}
\frac{2(1-pk_c^2/4)}{La}\eps^2,\label{eq:Awidth}
\end{equation}
up to $o(\eps^2)$. If $k_{m,\eps}-k_c=O(\eps)$, the $o(1)$ errors in \eqref{eq:Alorentz} are $O(\eps)$.
\end{proposition}

The separation of neighboring centers is $O(\eps)$, whereas the width of each peak is $O(\eps^2)$. These are widths in $k$, not in $k^2$. 

\subsubsection{A fixed-frequency example without a unique limit}
Choose
\begin{equation}
 k_c=\sqrt{\frac2p},\qquad \theta_0(k_c)=\frac\pi2,
 \qquad G_*:=\frac{2Lar k_c}{3p}>0.\label{eq:Aexample}
\end{equation}
When $N_\eps$ is even, \eqref{eq:Aphase} gives
$\Phi_\eps=N_\eps\pi/2+L\eps\theta_2(k_c)+O(\eps^3)$ and
$\eta_\eps L\eps\theta_2(k_c)\to G_*$. Hence,
\begin{equation}
 R_\eps\rightarrow-\frac{G_*}{2\mi+G_*},\qquad
 |T_\eps|^2\rightarrow\frac4{4+G_*^2}>0
 .
 \label{eq:Aeven}
\end{equation}
When $N_\eps$ is odd, $|\sin\Phi_\eps|\to1$, so
\begin{equation}
 R_\eps\rightarrow-1,\qquad |T_\eps|^2\rightarrow0.
 \label{eq:Aodd}
\end{equation}
Thus, a fixed interior frequency generally does not have a unique scattering limit as $\eps=L/N_\eps\to0$. The complex transmission on the even subsequence has the further factor $(-1)^{N_\eps/2}$, while the power in \eqref{eq:Aeven} does not have such an ambiguity.

For example, $a=0.3$ and $L=1$ give $k_c\simeq3.086067$ and
$4/(4+G_*^2)\simeq0.579211$. This example rules out extending the sound-soft theorem to every fixed interior frequency without a phase or resolvent nonresonance condition. It does not contradict the effective band-edge behavior, whose frequency scaling and nonresonance condition are different.

\section{Interior-band behavior in higher dimensions: a local transport equation}
\label{append:B}

In this section, we formally show that at regular Bloch points, the bulk amplitudes satisfy a transport equation. For doing so, we use the notation 
\begin{equation}
 c(\alpha):=\widehat{C}^\alpha
 =\frac1{|D|}\int_{Y\setminus D}|\nabla V^\alpha|^2 \, d y,
 \qquad s_\eps:=\omega_\eps^2/v_{\rm b}^2.\label{eq:Bsymbol}
\end{equation}

We consider the physical first-band squared frequency, that is,
\begin{equation}
 h_\eps(\alpha):=\lambda_{1,\eps^2}^{\alpha}
 =c(\alpha)+O(\eps^2).
 \label{eq:Bbandexp}
\end{equation}
In the physical period-$\eps$ medium,
$\omega_{1,\eps}^{\alpha/\eps}=v_{\rm b}\sqrt{h_\eps(\alpha)}$. The term $O(\eps^2)$ can be made explicit using \eqref{eq:omega-square-expansion}.

Here, we fix an interior value $s_c=\omega_c^2/v_{\rm b}^2$, rather than taking an $O(\eps^2)$ detuning from $c(M)$. Suppose that
\begin{equation}
 c(\alpha_c)=s_c,\qquad \nabla_\alpha c(\alpha_c)\ne0.
\end{equation}
The scalar symbol is real analytic and reciprocal, $c(-\alpha)=c(\alpha)$. Its Fourier coefficients are exponentially summable. Let $q_\eps:= Q_\eps u_\eps$. For
the ansatz $q_\eps(n)=e^{\mi  n\cdot\alpha_c}A(\eps n)$, with the slowly varying envelope $A$ being a Schwartz function and $A(\eps n)$ being the envelope sampled at cell $n$, the Taylor expansion of the convolution
\begin{equation}
    (\mathfrak{C} q_\eps)(n)= \sum_{r\in \Z^d} c_r e^{\mi  (n-r) \cdot \alpha_c} A(\eps (n- r)),
\end{equation}
where $\mathfrak{C}$ is defined by \eqref{def:capacitanc_op}
and $c(\alpha)= \sum_{r\in \Z^d} c_r e^{-\mi  r\cdot \alpha}$, yields the following expansion:
\begin{align}
 (\mathfrak{C} q_\eps)(n)=e^{\mi  n\cdot\alpha_c}\bigg[
 &c(\alpha_c)A(\eps n)- \mi\eps\nabla c(\alpha_c)\cdot\nabla A(\eps n)\notag\\
 &-\frac{\eps^2}{2}\sum_{l,j}\partial_{\alpha_l\alpha_j}c(\alpha_c)
 \partial_{lj}A(\eps n)\bigg]+\mathcal R_\eps(n),\label{eq:BbulkTaylor}
\end{align}
where $\eps^{d/2}\|{\mathcal R_\eps}\|_{\ell^2}\le C\eps^3$ for some constant $C$ independent of $\eps$.

If $\omega_\eps=\omega_c+\eps\nu+O(\eps^2)$, then, from \eqref{eq:Bbandexp}, the leading equation follows:
\begin{equation}
 \mi \nabla c(\alpha_c)\cdot\nabla A
 + \frac{2\omega_c\nu}{v_{\rm b}^2}A=0.\label{eq:Btransport}
\end{equation}

Equation \eqref{eq:Btransport} is a local bulk modulation equation for a selected simple regular Bloch mode. It does not constitute an effective scattering problem for a finite sample. Such a problem also requires boundary matching, the contributions of other propagating and evanescent modes, and uniform control of the resulting boundary-coupled resolvent. In particular, regularity of the dispersion surface alone does not exclude vanishing surface coupling or finite-sample resonances. The establishment of these properties for higher-dimensional arrays, including the simpler case of symmetry-plane slabs, remains open in the present work.

\bibliographystyle{alpha}
\bibliography{mybib}

@article {arma-honeycomb1,
    AUTHOR = {Ammari, Habib and Hiltunen, Erik Orvehed and Yu, Sanghyeon},
     TITLE = {A high-frequency homogenization approach near the {D}irac
              points in bubbly honeycomb crystals},
   JOURNAL = {Arch. Ration. Mech. Anal.},
  FJOURNAL = {Archive for Rational Mechanics and Analysis},
    VOLUME = {238},
      YEAR = {2020},
    NUMBER = {3},
     PAGES = {1559--1583},
       DOI = {10.1007/s00205-020-01572-w},
       URL = {https://doi.org/10.1007/s00205-020-01572-w},
}

@article {sima-honeycomb,
    AUTHOR = {Ammari, Habib and Fitzpatrick, Brian and Hiltunen, Erik
              Orvehed and Lee, Hyundae and Yu, Sanghyeon},
     TITLE = {Honeycomb-lattice {M}innaert bubbles},
   JOURNAL = {SIAM J. Math. Anal.},
  FJOURNAL = {SIAM Journal on Mathematical Analysis},
    VOLUME = {52},
      YEAR = {2020},
    NUMBER = {6},
     PAGES = {5441--5466},
       DOI = {10.1137/19M1281782},
       URL = {https://doi.org/10.1137/19M1281782},
}

@book{cbms,
  author    = {Ammari, Habib and Davies, Bryn and Hiltunen, Erik Orvehed},
  title     = {Mathematical Theories for Metamaterials: {From} Condensed Matter Theory to Subwavelength Physics},
  series    = {{NSF-CBMS} Regional Conference Series in the Mathematical Sciences},
  publisher = {American Mathematical Society},
  year     = {2026},
  volume  = {136},
}

@article{fabryperot1,
    author = {Ammari, Habib and Li, Bowen and Liu, Ping and Shao, Yingjie},
    title = {Frequency-dependent capacitance matrix formulation for {F}abry-{P}érot-type Resonances. Part I: one-dimensional finite systems},
    journal = {arXiv preprint arXiv:2604.01159},
    year = {2026}
}

@article{fabryperot3,
      title={Frequency-dependent capacitance matrix formulation for {F}abry-{P}\'erot resonances in two and three dimensional systems}, 
      author={Habib Ammari and Bowen Li and Ping Liu and Yingjie Shao and Alexander Uhlmann},
      year={2026},
     journal={arXiv preprint arXiv:2605.27572},
}

@article{ammari.davies.ea2024Functional,
  title = {Functional Analytic Methods for Discrete Approximations of Subwavelength Resonator Systems},
  author = {Ammari, Habib and Davies, Bryn and Hiltunen, Erik Orvehed},
  year = 2024,
  journal = {Pure and Applied Analysis},
  volume = {6},
  number = {3},
  pages = {873--939},
  doi = {10.2140/paa.2024.6.873},
}

@article {LMS25,
    AUTHOR = {Ammari, Habib and Davies, Bryn and Hiltunen, Erik Orvehed},
     TITLE = {Spectral convergence in large finite resonator arrays: the
              essential spectrum and band structure},
   JOURNAL = {Bull. Lond. Math. Soc.},
  FJOURNAL = {Bulletin of the London Mathematical Society},
    VOLUME = {57},
      YEAR = {2025},
    NUMBER = {3},
     PAGES = {730--747},
}

@article{du2026homogenization,
  title={Homogenization of the scattered wave and scattering resonances for periodic high-contrast subwavelength resonators},
  author={Du, Yuxin and Fu, Xin and Jing, Wenjia},
  journal={Communications in Mathematical Physics},
  volume={407},
  number={8},
  pages={162},
  year={2026},
  publisher={Springer}
}

@article{ammari2017subwavelength,
  title={Subwavelength phononic bandgap opening in bubbly media},
  author={Ammari, Habib and Fitzpatrick, Brian and Lee, Hyundae and Yu, Sanghyeon and Zhang, Hai},
  journal={Journal of Differential Equations},
  volume={263},
  number={9},
  pages={5610--5629},
  year={2017},
  publisher={Elsevier}
}

@article{ammari2019bloch,
  title={Bloch waves in bubbly crystal near the first band gap: a high-frequency homogenization approach},
  author={Ammari, Habib and Lee, Hyundae and Zhang, Hai},
  journal={SIAM Journal on Mathematical Analysis},
  volume={51},
  number={1},
  pages={45--59},
  year={2019},
  publisher={SIAM}
}

@article {feppon2023,
    AUTHOR = {Feppon, Florian and Ammari, Habib},
     TITLE = {Homogenization of sound-soft and high-contrast acoustic
              metamaterials in subcritical regimes},
   JOURNAL = {ESAIM Math. Model. Numer. Anal.},
  FJOURNAL = {ESAIM. Mathematical Modelling and Numerical Analysis},
    VOLUME = {57},
      YEAR = {2023},
    NUMBER = {2},
     PAGES = {491--543},
}

@article {milton,
    AUTHOR = {Harutyunyan, Davit and Milton, Graeme W. and Craster, Richard
              V.},
     TITLE = {High-frequency homogenization for travelling waves in periodic
              media},
   JOURNAL = {Proc. A},
  FJOURNAL = {Proceedings A},
    VOLUME = {472},
      YEAR = {2016},
    NUMBER = {2191},
     PAGES = {20160066, 18},
}

@article {mourad2020,
    AUTHOR = {Ammari, Habib and Challa, Durga Prasad and Choudhury, Anupam
              Pal and Sini, Mourad},
     TITLE = {The equivalent media generated by bubbles of high contrasts:
              volumetric metamaterials and metasurfaces},
   JOURNAL = {Multiscale Model. Simul.},
  FJOURNAL = {Multiscale Modeling \& Simulation. A SIAM Interdisciplinary
              Journal},
    VOLUME = {18},
      YEAR = {2020},
    NUMBER = {1},
     PAGES = {240--293},
}

@article {hai2017,
    AUTHOR = {Ammari, Habib and Zhang, Hai},
     TITLE = {Effective medium theory for acoustic waves in bubbly fluids
              near {M}innaert resonant frequency},
   JOURNAL = {SIAM J. Math. Anal.},
  FJOURNAL = {SIAM Journal on Mathematical Analysis},
    VOLUME = {49},
      YEAR = {2017},
    NUMBER = {4},
     PAGES = {3252--3276},
}

@article{arendt2015dirichlet,
  title={The {D}irichlet-to-{N}eumann operator on exterior domains},
    AUTHOR = {Arendt, W. and ter Elst, A. F. M.}, 
  journal={Potential analysis},
  volume={43},
  number={2},
  pages={313--340},
  year={2015}
}

@article{grassle2025dirichlet,
  title={{D}irichlet-to-{N}eumann operator for the {H}elmholtz problem with general wavenumbers on the $ n $-sphere},
  author={Gr{\"a}{\ss}le, Benedikt and Sauter, Stefan A},
  journal={arXiv preprint arXiv:2503.18837},
  year={2025}
}

@article{larson2024space,
  title={Space-time CutFEM on overlapping meshes I: simple continuous mesh motion},
  author={Larson, Mats G and Logg, Anders and Lundholm, Carl},
  journal={Numerische Mathematik},
  volume={156},
  number={3},
  pages={1015--1054},
  year={2024},
  publisher={Springer}
}

@article{verchota1984layer,
  title={Layer potentials and regularity for the {D}irichlet problem for {L}aplace's equation in {L}ipschitz domains},
  author={Verchota, Gregory},
  journal={Journal of functional analysis},
  volume={59},
  number={3},
  pages={572--611},
  year={1984},
  publisher={Elsevier}
}

@book{Triebel2010Theory,
  author    = {Triebel, Hans},
  title     = {Theory of Function Spaces},
  series    = {Modern Birkh{\"a}user Classics},
  publisher = {Birkh{\"a}user Basel},
  year      = {1983},
}

@article{savare2002domain,
  title={Domain perturbations and estimates for the solutions of second order elliptic equations},
  author={Savar{\'e}, Giuseppe and Schimperna, Giulio},
  journal={Journal de math{\'e}matiques pures et appliqu{\'e}es},
  volume={81},
  number={11},
  pages={1071--1112},
  year={2002},
  publisher={Elsevier}
}

\end{document}